\documentclass[11pt,reqno]{amsart}
\usepackage{amssymb,mathrsfs,color,mathtools,empheq, verbatim, epstopdf}
\usepackage{pinlabel}
\mathtoolsset{showonlyrefs}
\usepackage{hyperref} 

\usepackage{enumerate, bbm}
\usepackage{tensor}

\usepackage{graphicx}
\usepackage{xcolor} 
\usepackage{tensor}
\usepackage{slashed}
\usepackage{cite}

\usepackage[shortlabels]{enumitem}

\usepackage{geometry}

\numberwithin{equation}{section}

\newtheorem{mainthm}{Theorem}
\newtheorem{thm}{Theorem}[section]
\newtheorem{cor}[thm]{Corollary}
\newtheorem{lem}[thm]{Lemma}
\newtheorem{prop}[thm]{Proposition}

\theoremstyle{definition} 
\newtheorem{rem}[thm]{Remark}
\newtheorem{defn}[thm]{Definition}

\theoremstyle{remark}

\def\bN {\mathbb{N}}

\def\bR {\mathbb{R}}

\def\bZ {\mathbb{Z}}

\def\fv {\mathfrak{v}}

\def\cE {\mathcal{E}}

\def\grad {{\nabla}}

\newcommand{\la}{\langle}
\newcommand{\ra}{\rangle}
\newcommand{\La}{\big\langle}
\newcommand{\Ra}{\big\rangle}

\newcommand{\tx}[1]{\mathrm{#1}}

\newcommand{\wt}[1]{\widetilde{#1}}
\newcommand{\bs}[1]{\boldsymbol{#1}}
\newcommand{\vd}{\mathrm{d}}

\newcommand{\vD}{\mathrm{D}}

\newcommand{\uln}[1]{{\underline{ #1 }}}
\newcommand{\lin}{_{\textsc{l}}}

\newcommand{\Lin}{\textsc{l}} 

\newcommand{\sh}[1]{#1^\sharp}
\newcommand{\fl}[1]{#1^\flat}

\newcommand{\dd}[1]{\frac{\ud}{\ud{#1}}}

\definecolor{deepgreen}{cmyk}{1,0,1,0.5}

\newcommand{\E}{\mathcal{E}}

\newcommand{\LL}{\mathcal{L}}

\newcommand{\cS}{\mathcal{S}}

\newcommand{\N}{\mathbb{N}}
\newcommand{\R}{\mathbb{R}}

\newcommand{\Z}{\mathbb{Z}}

\newcommand{\al}{\alpha}
\newcommand{\be}{\beta}

\newcommand{\de}{\delta}

\newcommand{\fy}{\varphi}
\newcommand{\om}{\omega}
\newcommand{\lam}{\lambda}
\newcommand{\te}{\theta}

\newcommand{\s}{\sigma}

\newcommand{\si}{\varsigma}

\newcommand{\De}{\Delta}
\newcommand{\Om}{\Omega}

\newcommand{\Lam}{\Lambda}

\newcommand{\p}{\partial}
\newcommand{\na}{\nabla}

\makeatletter

\newcommand{\Rmnum}[1]{\expandafter\@slowromancap\romannumeral #1@}
\makeatother

\newcommand{\ti}{\widetilde}

\newcommand{\U}{\underline}

\newcommand{\ang}[1]{\left\langle{#1}\right\rangle}
\newcommand{\abs}[1]{\left\lvert{#1}\right\rvert}

\newcommand{\EQ}[1]{\begin{equation}\begin{split} #1 \end{split}\end{equation}}
\newcommand{\pmat}[1]{\begin{pmatrix} #1 \end{pmatrix}}

\newcommand{\Del}[1]{}

\newcommand{\mand}{{\ \ \text{and} \ \  }}
\newcommand{\mor}{{\ \ \text{or} \ \ }}
\newcommand{\mif}{{\ \ \text{if} \ \ }}
\newcommand{\mfor}{{\ \ \text{for} \ \ }}
\newcommand{\mas}{{\ \ \text{as} \ \ }}

\newcommand{\uD}{\operatorname{D}}

\newcommand{\rest}{\!\!\restriction}

\definecolor{green}{rgb}{0,0.8,0} 

\newcommand{\ud}{\mathrm{d}}

\newcommand{\eps}{\epsilon}

\newcommand{\bfd}{{\bf d}}

\newcommand{\bfi}{{\bf i}}

\newcommand{\bfp}{{\bf p}}
\newcommand{\bfq}{{\bf q}}

\newcommand{\calC}{\mathcal C}

\newcommand{\calE}{\mathcal E}

\newcommand{\calJ}{\mathcal J}

\newcommand{\calL}{\mathcal L}

\newcommand{\calQ}{\mathcal Q}

\newcommand{\calS}{\mathcal S}

\newcommand{\calW}{\mathcal W}

\newcommand{\calY}{\mathcal Y}
\newcommand{\calZ}{\mathcal Z}

\newcommand{\ULam}{\U{\Lam}}

\begin{document}

\title[Soliton resolution for the radial energy-critical NLW]{Soliton resolution for the energy-critical nonlinear wave equation in the radial case}
\author{Jacek Jendrej}
\author{Andrew Lawrie}

\begin{abstract}
We consider the focusing energy-critical nonlinear wave equation for radially symmetric initial data in space dimensions  $D  \ge 4$. This equation has a unique (up to sign and scale) nontrivial, finite energy stationary solution $W$, called the ground state. 
We prove that every finite energy solution with bounded energy norm resolves, continuously in time, into a finite superposition of asymptotically decoupled copies of the ground state and free radiation.
\end{abstract}


\thanks{J.Jendrej is supported by  ANR-18-CE40-0028 project ESSED.  A. Lawrie is supported by NSF grant DMS-1954455, a Sloan Research Fellowship, and the Solomon Buchsbaum Research Fund}

\maketitle

\tableofcontents

\section{Introduction} 

%
\subsection{Setting of the problem}

We study the Cauchy problem for the focusing nonlinear wave equation in the energy-critical case and under the assumption of radial symmetry, i.e., 
\EQ{ \label{eq:nlw} 
\p_t^2 u - \De u - \abs{u}^{\frac{4}{D-2}} u   &= 0,  \\ 
 (u(T_0), \p_t u(T_0)) &= (u_0, \dot u_0), 
}
where here $D \in \{ 4, 5,  6, \dots\}$ is the underlying spatial dimension, $u = u(t, r) \in \R$, where  $r = \abs{x} \in (0, \infty)$ is the radial coordinate in $\R^D$, $\De := \p_r^2 + (D-1)r^{-1} \p_r$ is the radial Laplacian in $\R^D$, and $T_0 \in\R$. 

The conserved energy for~\eqref{eq:nlw} is given by 
\EQ{
E( u(t), \p_t u(t)) :=  \int_0^\infty \frac{1}{2} \Big[ (\p_t u(t))^2 + ( \p_r u(t))^2\Big]\, r^{D-1} \, \ud r  -  \int_0^\infty\frac{D-2}{2D}\abs{u(t)}^{\frac{2D}{D-2}} \, r^{D-1} \, \ud r. 
}

The Cauchy problem for~\eqref{eq:nlw} can be rephrased as a Hamiltonian system. To formulate it as such,  we will write pairs of functions using boldface,  $\bs v = (v, \dot v)$, noting that the notation $\dot v$ does not, in general, refer to the time derivative of $v$ but  just to the second component of the vector $\bs v$. We see that~\eqref{eq:nlw} is equivalent to 
\begin{equation} \label{eq:u-ham} 
\partial_t \bs u(t) =  J \circ \vD E( \bs u(t)),\qquad \bs u(T_0) = \bs u_0,
\end{equation}
where 
\begin{equation} \label{eq:DE}
J = \begin{pmatrix} 0 &1 \\ -1 &0\end{pmatrix}, \quad \vD E( \bs u(t))  =  \begin{pmatrix}- \De u(t)  - f(u) \\ \partial_t u(t) \end{pmatrix}.
\end{equation}
Above we have introduced the notation $f(u) := \abs{u}^{\frac{4}{D-2}} u$. 

Solutions to~\eqref{eq:nlw} are invariant under the scaling, 
\EQ{
\bs u(t, r) \mapsto \bs u_{\lam}(t, r) := ( \lambda^{-\frac{D-2}{2} }u(t/\lambda, r/ \lambda), \lambda^{-\frac{D}{2} } \p_t u(t/ \lambda, r/ \lambda), \quad \lambda >0, 
}
and~\eqref{eq:nlw} is called \emph{energy-critical} because $E( \bs u) = E( \bs u_{\lam})$.  

The linearization of~\eqref{eq:nlw} about the zero solution is the free scalar wave equation, 
\EQ{\label{eq:lin} 
\p_t^2 v - \De v = 0. 
}

In this paper we study solutions with initial data in the energy space $\E$, which is defined via the norm, 
\EQ{
\| \bs u_0 \|_{\E}^2 := \int_0^\infty  \Big[(\dot u_0(r))^2 + (\p_r u_0(r))^2 + \frac{(u_0(r))^2}{r^2}\Big] \, r^{D-1} \, \ud r . 
} 
Using Hardy's inequality, functions $\bs u(r)$ in $\E$ can be identified with radially symmetric functions $\bs v (x)$ in the space $\dot H^1 \times L^2( \R^D)$ in the usual way.  We will sometimes use the notation, 
\EQ{
\| u_0 \|_{H}^2 := \int_0^\infty  \Big[ (\p_r u_0(r))^2 + \frac{(u_0(r))^2}{r^2}\Big] \, r^{D-1} \, \ud r ,  
}
and write $\E = H \times L^2$. 

It is a classical result of  Ginibre and Velo~\cite{GV89} that~\eqref{eq:nlw} is well-posed in the space $\E$. In particular, to every $\bs u_0 \in\E$ we can associate  maximal forward and backward times of existence $T_+ \in (0, \infty]$ and $T_- \in [-\infty, 0)$. We will only consider solutions $\bs u(t) \in \E$ to~\eqref{eq:nlw} that satisfy, 
\EQ{ \label{eq:type-ii-intro} 
\limsup_{t \to T_{+}} \| \bs u(t) \|_{\E} < \infty \mor \limsup_{t \to T_{-}} \| \bs u(t) \|_{\E} < \infty.
}
Solutions for which $\lim_{t \to T_{+}} \| \bs u(t) \|_{\E} =\infty$ are known to exist and are called type-I (or ODE-type) blow up solutions; see e.g.,~\cite{Levine, BCT-04, Donn-Duke}.

We define the Aubin-Talenti solution, $\bs W(x) := (W(x), 0)$  where $W: \R^D \to \R$, by 
\EQ{
W(x) :=  \Big( 1 + \frac{\abs{x}^2}{D(D-2)}\Big)^{-\frac{D-2}{2}}, 
}
and note that $W(x)$ is the unique (up to sign, scaling, and translation), \emph{non-negative} and nontrivial $C^2$ solution to 
\EQ{
-\De W (x)= - \abs{W(x)}^{\frac{4}{D-2}} W(x) , \quad x \in \R^d. 
}
Abusing notation slightly and writing $W(x) = W(r)$ with $r = \abs{x}$, we see that $\bs W(r)$ is a stationary solution to~\eqref{eq:nlw}. In fact, it is the unique (up to sign and scaling) static radial nontrivial solution to~\eqref{eq:nlw} in $\E$. For each $\lam>0$,  we write $\bs W_\lam(r):= (\lam^{-\frac{D-2}{2}} W(r/ \lambda), 0)$.

\subsection{Statement of the results}  Our main result is formulated as follows. 

\begin{mainthm}[Soliton Resolution] \label{thm:main}  Let $D \ge 4$ and let $\bs u(t)$ be a finite energy solution to~\eqref{eq:nlw} with initial data $\bs u(0) = \bs u_0 \in \E$,  defined on its maximal forward interval of existence $[0,T_+)$. Suppose that, 
\EQ{ \label{eq:type-II-thm} 
\limsup_{t \to T_{+}} \| \bs u(t) \|_{\E} < \infty. 
}
 Then, 
 
 \emph{({Global solution})} if $T_+ = \infty$, there exist a time $T_0>0$, a solution $\bs u^*\lin(t) \in \E$ to the linear wave equation~\eqref{eq:lin}, an integer $N \ge 0$, continuous functions $\lam_1(t), \dots,  \lam_N(t) \in C^0([T_0, \infty))$, signs $\iota_1, \dots, \iota_N \in \{-1, 1\}$,  and $\bs g(t) \in \E$ defined by 
 \EQ{
 \bs u(t) =  \sum_{j =1}^N \iota_j \bs W_{\lam_j(t)}  + \bs u^*\lin(t) + \bs g(t) , 
 }
 such that 
 \EQ{
 \| \bs g(t)\|_{\E} + \sum_{j =1}^{N} \frac{\lam_{j}(t)}{\lam_{j+1}(t)}  \to 0 \mas t \to \infty,  
 }
 where above we use the convention that $\lam_{N+1}(t) = t$;

 \emph{({Blow-up solution})} if $T_+ < \infty$, there exists a time $T_0< T_+$, a function $\bs u_0^*\in \E$, an integer $N \ge 1$, continuous functions $\lam_1(t), \dots,  \lam_N(t) \in C^0([T_0, T_+))$, signs $\iota_1, \dots, \iota_N \in \{-1, 1\}$, and $\bs g(t) \in \E$ defined by 
 \EQ{
 \bs u(t) =  \sum_{j =1}^N \iota_j\bs W_{\lam_j(t)}  + \bs u^*_0 + \bs g(t) , 
 }
 such that 
 \EQ{
 \| \bs g(t)\|_{\E} + \sum_{j =1}^{N} \frac{\lam_{j}(t)}{\lam_{j+1}(t)}  \to 0 \mas t \to T_+, 
 }
 where above we use the convention that $\lam_{N+1}(t) = T_+-t$. 

Analogous statements hold for the backwards-in-time evolution. 
\end{mainthm} 


\begin{rem} This type of behavior is referred to as soliton resolution. Theorem~\ref{thm:main} has been proved for~\eqref{eq:nlw} in a series of remarkable works by Duyckaerts, Kenig, and Merle in odd space dimensions $D \ge 3$; see~\cite{DKM3} for dimension $D = 3$ and see ~\cite{DKM7, DKM8, DKM9} for odd space dimensions $D \ge 5$. The space dimension $D=4$ was treated by Duyckaerts, Kenig, Martel,  and Merle in~\cite{DKMM}, which also covers the $k=1$-equivariant wave maps equation, and dimension $D=6$ was treated by  Collot, Duyckaerts, Kenig, and Merle~\cite{CDKM}. All of these papers use, in some fashion, 
the method of energy channels.  
Roughly, energy channels refer to measurements of the portion of energy that a linear or nonlinear wave radiates outside fattened light cones.
The approach we take to prove Theorem~\ref{thm:main} is independent of the method of energy channels. Our method of proof follows closely our recent preprint~\cite{JL6}, which proved the analogue of Theorem~\ref{thm:main} for the $k$-equivariant wave maps equation in all equivariance classes $k \in \N$. 
\end{rem}

\begin{rem}
The soliton resolution problem is inspired by the theory of completely integrable systems, e.g.,~\cite{Eck-Sch, Sch, SA},  motivated by numerical simulations,~\cite{FPU, ZK}, and by the bubbling theory of harmonic maps in the elliptic and parabolic settings~\cite{Struwe85, Qing, QT, Top-JDG, Topping04}; see also~\cite{DJKM1, DKM9, DKMM} for discussions on the history of the problem. 
\end{rem} 

\begin{rem} 
Equation~\ref{eq:nlw}, its counterpart in defocusing case, as well as the sub- and super-critical versions, have been classically studied; see for example the articles~\cite{Jorgens, Segal, Segal-Annals, Morawetz68, Strauss68, MoSt72, Rauch, Struwe88, SS93, SS94, BS98, Grillakis90, Pecher78, Pecher88, Kap94} and the books by Strauss~\cite{Strauss}, Sogge~\cite{Sogge}, and Statah, Struwe~\cite{SSbook}.

\end{rem} 


\begin{rem} 
Theorem~\ref{thm:main} is a qualitative description of the dynamics of all finite energy solutions to~\eqref{eq:nlw} with bounded critical norm.  A natural, challenging question is to ask which types of configurations of solitons and radiation are realized in solutions. The first results in this direction were by Krieger and Schlag~\cite{KS07} who found a manifold of global-in-time solutions that decoupled into a static $\bs W$ and free radiation; see also the improvement by Beceanu~\cite{Bec-14}.  Kenig and Merle~\cite{KM08} gave the first general description of dynamics in a non-perturbative setting for the focusing energy critical NLW (non-radial), proving that solutions $\bs u$ with energy below the ground state energy  scatter in both directions if $\| \na u_0 \|_{L^2}^2 < \| \na W \|_{L^2}^2$ or blow up in finite time in both directions if $\| \na u_0 \|_{L^2}^2 < \| \na W \|_{L^2}^2$.  In~\cite{DM08} Duyckaerts and Merle classified the dynamics of solutions at the threshold energy $E = E(\bs W)$. Characterizations of the dynamics for energies slightly above the ground state energy were given by Krieger, Nakanishi, and Schlag in~\cite{KNS13, KNS15}. The first construction of a solution developing a bubbling singularity (with one concentrating bubble) in finite time was done by Krieger, Schlag, and Tataru~\cite{KST3}; see also Hillairet and Raphael~\cite{HiRa12} for a different construction in dimension $D=4$,  and also~\cite{KrSc14}. 

In~\cite{JJ-AJM}, the first author constructed a solution exhibiting more than one bubble in the decomposition, showing the existence of a solution that forms a $2$-bubble in infinite time with zero radiation in dimension $D = 6$. We expect that solutions that form $2$-bubble in forward infinite time also exist in dimensions $D \ge 7$. 
\end{rem}


%

When multi-bubble solutions do occur in one time direction, it is natural to ask about the dynamics  of those solutions   in the opposite time direction.
We answered this question in~\cite{JL1} in the setting of equivariant wave maps  for the pure $2$-bubble solution $\bs u_{(2)}$ constructed by the first author in~\cite{JJ-AJM}. We showed that any $2$-bubble in forward time must scatter freely in backwards time.  When the scales of the bubbles become comparable, this ‘collision’ completely annihilates the $2$-bubble structure and the entire solution becomes free radiation, i.e., the collision is \emph{inelastic}. Viewed in forward time, this means that the $2$-soliton structure emerges from pure radiation, and constitutes an orbit connecting two different dynamical behaviors. We later showed in~\cite{JL2-regularity, JL2-uniqueness} that $\bs u_{(2)}(t)$ is the unique $2$-bubble solution up to sign, translation, and scaling in equivariance classes $k \ge 4$. 
 While we do not consider such refined two-directional analysis here, a relatively straightforward corollary of the proof of Theorem~\ref{thm:main} is that there can be \emph{no elastic collisions of pure multi-bubbles}, which we formulate as a proposition below. 
 
 Before stating the result, we define pure multi-bubbles in forward or backward time. 
\begin{defn}
\label{def:pure}
With the notations from the statement of Theorem~\ref{thm:main}, we say that
$\bs u$ is a \emph{pure multi-bubble} in the forward time direction if
$\bs u\lin^* = 0$ in the case $T_+ = +\infty$, and $\bs u_0^* = 0$ in the case $T_+ < +\infty$.

We say that $\bs u$ is a pure multi-bubble in the backward time direction
if $t \mapsto \bs u(-t)$ is a pure multi-bubble in the forward time direction.
\end{defn}


\begin{prop}[No elastic collisions of pure multi-bubbles]
\label{prop:inelastic}
Stationary solutions are the only pure multi-bubbles in both time directions. 
\end{prop}

\begin{rem}
 We note that Proposition~\ref{prop:inelastic} was also proved for odd dimensions $D \ge 3$ in~\cite{DKM4, DKM7, DKM8, DKM9} and in dimensions $D =4, 6$ in~\cite{DKMM, CDKM} by the method of energy channels. The case of $k=1$ equivariant wave maps was treated in~\cite{DKMM} using energy channels. Here the proof of Proposition~\ref{prop:inelastic} follows from the method introduced by the authors in~\cite{JL6} where we treated equivariant wave maps for all equivariance classes $k \ge 1$.  See ~\cite{MM11, MM11-2, MM18} for more regarding the inelastic soliton collision problem for non-integrable PDEs. 
\end{rem}

\subsection{Summary of the proof}  

The proof of Theorem~\ref{thm:main} is built on two significant partial results;  (1) that the radiation term, $\bs u^*\lin$ in the global case and $\bs u^*_0$ in the blow-up setting,  can be identified continuously in time, and (2) that the resolution is known to hold along a  well-chosen sequence of times (at least in the case of certain space dimensions). Because the existing literature does not cover all space dimensions, we sketch a unified proof of the sequential soliton resolution (see Theorem~\ref{thm:seq} below) as a consequence of what we call the Compactness Lemma (see Lemma~\ref{lem:compact}, which is also used crucially in the proof of the main theorem), the identification of the radiation, and the fact that no energy can concentrate in the self-similar region of the light cone. 

We discuss these  results in more detail. 
To unify the blow-up and global-in-time settings we make the following conventions. Consider a finite energy solution  $\bs u(t) \in \E$. We assume that either $\bs u(t)$ blows up in backwards time at $T_-=0$ and is defined on an interval $I_*:=(0, T_0]$, or $\bs u(t)$ is global in forward time and defined on the  interval $I_* := [T_0, \infty)$ where in both cases $T_0>0$. We let $T_* := 0$ in the blow-up case and $T_* := \infty$ in the global case. We assume that $\bs u(t)$ exhibits type II behavior in that, 
\EQ{ \label{eq:typeII} 
\lim_{t \to T_*} \| \bs u(t) \|_{\E} < \infty. 
}

\textbf{Step 1: Extraction of the radiation.} 
%
Below we will use the notation $\cE(r_1, r_2)$ to denote the localized energy norm
\begin{equation} \label{eq:local-E-norm} 
\|\bs g\|_{\cE(r_1, r_2)}^2 := \int_{r_1}^{r_2} \Big((\dot g)^2 + (\partial_r g)^2 + \frac{g^2}{r^2}\Big)\,r^{D-1}\vd r,
\end{equation}
By convention, $\cE(r_0) := \cE(r_0, \infty)$ for $r_0 > 0$. The local nonlinear energy is denoted $E(\bs u_0; r_1, r_2)$. 
We adopt similar conventions as for $\cE$ regarding the omission
of $r_2$, or both $r_1$ and $r_2$.

\begin{thm}[Properties of the radiation] \emph{\cite{DKM1, DKM3, CKLS3, JK, Rod16-adv}}  \label{thm:radiation} 
Let $\bs u(t) \in \E$ be a finite energy solution to~\eqref{eq:nlw} on an interval $I_*$ as above and such that~\eqref{eq:typeII} holds.  Then, 
there is a  finite energy solution $\bs u^*(t) \in \E$ to~\eqref{eq:nlw} called the radiation, and a function $\rho : I_* \to (0,\infty)$ that satisfies, 
\begin{equation} \label{eq:radiation} 
\lim_{t \to T_*} \big((\rho(t) / t)^{\frac{D-2}{2}} + \|\bs u(t) -  \bs u^*(t)\|_{\cE(\rho(t))}^2\big) = 0.
\end{equation}
Moroever, for any $\al \in (0, 1)$,  
\EQ{ \label{eq:rad-rho} 
\| \bs u^*(t)\|_{\E( 0, \al t)} \to 0 \mas t \to T_*. 
}
\end{thm}

\begin{rem} 
In the global setting, i.e., $I_* = [T_0, \infty)$  the linear wave $\bs u\lin^*(t) \in \E$ that appears in Theorem~\ref{thm:main} is the unique solution to the linear equation~\eqref{eq:lin} satisfying,
\EQ{
\| \bs u^*(t) - \bs u\lin^*(t) \|_{\E} \to 0 \mas n \to \infty, 
} 
see Proposition~\ref{prop:rad-global}. In the finite time blow-up setting the final radiation $\bs u^*_0 \in \E$ that appears in Theorem~\ref{thm:main} can be viewed as initial data for $\bs u^*(t)$, i.e., the radiation $\bs u^*(t)$ in Theorem~\ref{thm:radiation} satisfies $\bs u(t, r)  =   \bs u^*(t, r)$ for $r > t$. We refer the reader to Section~\ref{ssec:seq} for a sketch of the proof of Theorem~\ref{thm:radiation} following the scheme of Duyckaerts, Kenig, and Merle~\cite{DKM3} (see also the preliminary results in Sections~\ref{ssec:rad} for the identification of the radiation and Section~\ref{ssec:ssim} for non-concentration of energy in the self-similar region of the cone, which follow the methods of~\cite{CKLS3, JK} by C\^ote, Kenig, the second author, and Schlag, and by Jia and Kenig). 


\end{rem}

\begin{rem} 
 The radiation field for~\eqref{eq:nlw} can be identified even outside radial symmetry; see the work of Duyckaerts, Kenig, and Merle~\cite{DKM19}. The radiation field can be identified in several other contexts and by different means. For example, Tao accomplished this in~\cite{Tao-07} for certain high dimensional NLS.  
\end{rem} 

\textbf{Step 2: Sequential soliton resolution}. 
A deep insight of Duyckaerts, Kenig, and Merle, proved in~\cite{DKM3} for $D=3$, is that once the linear radiation is subtracted from the solution, the entire remainder should exhibit strong sequential compactness -- it decomposes into a finite sum of asymptotically decoupled elliptic objects, in our case these are copies of the ground state, along at least one time sequence, up to an error that vanishes in the energy space. A crucial tool in proving such a compactness statement is the remarkable theory of profile decompositions for dispersive equations developed by Bahouri and G\'erard~\cite{BG}. However, after finding the profiles and their space-time concentration properties (in our case their scales) via the main result in~\cite{BG}, one must identify them as elliptic objects (solitons) by some means, and then prove that the error vanishes in the energy space,  rather the weaker form of compactness (vanishing of certain Strichartz norms) given by~\cite{BG}.  This was proved in the breakthrough paper~\cite{DKM3} using  linear energy channels (amongst other techniques).  The program from~\cite{DKM3} was used to extend the sequential resolution result to all  all odd dimensions $D \ge 3$ in~\cite{Rod16-adv}. It was shown in~\cite{CKS, CKLS3} that scheme of proof from~\cite{DKM3} could be extended to the subset of even space dimensions $D = 0 \mod(4)$. Jia and Kenig then proved the sequential soliton resolution result for dimension $D =6$ using a different scheme based on virial inequalities rather than energy channels. We follow the Jia-Kenig approach here to prove a general result, which we call the Compactness Lemma~\ref{lem:compact}, which we then combine with Theorem~\ref{thm:radiation} to give a unified proof of the sequential resolution in all space dimensions $D \ge 4$; for the latter, see Section~\ref{ssec:seq}. 

Before stating the sequential resolution result, we introduce some notation. 

%

\begin{defn}[Multi-bubble configuration] \label{def:multi-bubble} \label{def:M-bubble} 
Given $M \in \{0, 1, \ldots\}$, $\vec\iota = (\iota_1, \ldots, \iota_M) \in \{-1, 1\}^M$
and an increasing sequence $\vec\lambda = (\lambda_1, \ldots, \lambda_M) \in (0, \infty)^M$,
a \emph{multi-bubble configuration} is defined by the formula
\EQ{
\bs\calW( \vec\iota, \vec\lambda; r) := \sum_{j=1}^M\iota_j\bs W_{\lambda_j}(r) .
}
\end{defn}
\begin{rem}
If $M = 0$, it should be understood that $\bs \calW( \vec\iota, \vec\lambda; r) = 0$
for all $r \in (0, \infty)$, where $\vec\iota$ and $\vec\lambda$ are $0$-element sequences,
that is the unique functions $\emptyset \to \{-1, 1\}$ and $\emptyset \to (0, \infty)$, respectively. 
\end{rem}


\begin{thm}[Sequential soliton resolution] \emph{\cite[Theorems 1 and 4]{DKM3},~\cite[Theorems 1.1 and 1.3]{CKLS3}~\cite[Theorem 1.1]{JK}\cite[Theorems 1.1 and 1.3]{Rod16-adv}} 
 \label{thm:seq}  
Let $\bs u(t) \in \E$ be a finite energy solution to~\eqref{eq:nlw} on an interval $I_*$ as above. Let the radiation $\bs u^*(t) \in \E$ be as in Theorem~\ref{thm:radiation}. Then, there exists an integer $N \ge 0$, a sequence of times $t_n \to T_*$,  a vector of signs $\vec \iota \in \{-1, 1\}^N$, and a sequence of scales $\vec \lam_n \in (0, \infty)^N$  such that, 
 \EQ{
 \lim_{n \to \infty} \Big( \| \bs u(t_n) - \bs u^*(t_n)  - \bs \calW( \vec \iota, \vec \lam_n) \|_{\E} + \sum_{j =1}^{N}  \frac{\lam_{n, j}}{\lam_{n, j+1}} \Big)  = 0, 
 }
 where above we use the convention $\lam_{n, N+1} := t_n$. 
\end{thm}

%

\begin{rem} 

The Duyckaerts, Kenig, and Merle approach from~\cite{DKM3} to sequential soliton resolution has been successful in other settings. The same authors with Jia proved the sequential decomposition for the full energy critical NLW (i.e., not assuming radial symmetry) in~\cite{DJKM1} and for wave maps outside equivariant symmetry for data with energy slightly above the ground state~\cite{DJKM2}. 
\end{rem}

\textbf{Step 3: Collision intervals and no-return analysis}
The challenging nature of bridging the  gap between Theorem~\ref{thm:seq}, which is the resolution along one sequence of times, and Theorem~\ref{thm:main} is apparent from the following consideration. The sequence $t_n \to T_*$ in Theorem~\ref{thm:seq} gives no relationship between the lengths of the time intervals $[t_n, t_{n+1}]$ and the concentration scales $\vec \lam_n$ of the bubbles  in the decomposition. One immediate enemy 
is then the  possibility of \emph{elastic collisions}.
If colliding solitons could recover their shape after a collision, then one could potentially
encounter the following scenario: the solution approaches a multi-soliton configuration for a sequence of times,
but in between infinitely many collisions take place, so that there is no soliton resolution in continuous time.

We describe our approach.  Fix $\bs u(t) \in \E$,  a finite energy solution to~\eqref{eq:nlw} on the time interval $I_*$ as defined above. Let $N \ge 0$ and  the radiation $\bs u^*(t) \in \E$ be as in Theorem~\ref{thm:seq}.  We define a 
\emph{multi-bubble proximity function} at each $t \in I_*$ by
\begin{equation} \label{eq:d-intro} 
\bfd(t) := \inf_{\vec \iota, \vec\lam}\bigg( \| \bs u(t) - \bs u^*(t) - \bs\calW( \vec\iota, \vec\lambda) \|_{\E}^2 + \sum_{j=1}^{N}\Big(\frac{ \lam_{j}}{\lam_{j+1}}\Big)^{\frac{D-2}{2}} \bigg)^{\frac{1}{2}},
\end{equation}
where $\vec\iota := (\iota_{1}, \ldots, \iota_N) \in \{-1, 1\}^{N}$, $\vec\lambda := (\lambda_{1}, \ldots, \lambda_N) \in (0, \infty)^{N}$,  and $\lambda_{N+1} := t$. Note that $d(t)$ is continuous on $I_*$.  

With this notation, we see that Theorem~\ref{thm:seq} gives a monotone sequence of times $t_n \to T_*$ such that, 
\begin{equation} \label{eq:seq} 
\lim_{n \to \infty} \bfd(t_n) = 0.
\end{equation}
Theorem~\ref{thm:main} is a direct consequence of showing that $\lim_{t \to T_*} \bfd(t) = 0$. We argue by contradiction, assuming that $\limsup_{t  \to T_*} \bfd(t) >0$.  This means that there is  some sequence of times where $\bs u - \bs u^*$ approaches an $N$-bubble and another sequence of times for which it stays bounded away from $N$-bubble configurations. It is natural to rule out this  behavior  by proving what is called a \emph{no-return} lemma. In this generality, our approach  is inspired by no-return results for one soliton by Duyckaerts and Merle~\cite{DM08, DM09}, Nakanishi and Schlag~\cite{NaSc11-1, NaSc11-2}, and Krieger, Nakanishi and Schlag~\cite{KNS13, KNS15}. In those works a key role is played by exponential instability, where here we have in addition  attractive nonlinear interactions between the solitons. This latter consideration, and indeed the overall scheme of the proof is based on our previous works \cite{JL1, JL6}. We remark that the argument in~\cite{JL1} marks the first time where modulation analysis of bubble interactions was used
 in the context of the soliton resolution problem. 

The basic tool we use is the standard virial functional 
\EQ{
\fv(t) :=  \Big\langle  \p_t u(t) \mid \chi_{\rho(t)} \, \Big(r \p_r u(t) + \frac{D-2}{2} u(t)\Big)\Big\rangle,  
}
where the cut-off $\chi$ is placed along a Lipschitz curve $r = \rho(t)$ that will be carefully chosen (note that a time-dependent cut-off of the virial functional was also used in~\cite{NaSc11-1, NaSc11-2}). Here the inner product is, 
 \EQ{ \label{eq:inner-prod}
\ang{\phi \mid \psi}  := \int_0^\infty \phi(r) \psi(r) \, r^{D-1} \ud r, \qquad \text{for }\phi, \psi : (0, \infty) \to \bR.
}
 Differentiating $\fv(t)$ in time we have, 
\EQ{ \label{eq:v'} 
\fv'(t) = - \int_0^\infty \abs{\p_t u(t, r)}^2\chi_{\rho(t)}(r) \, r^{D-1}\,   \ud r + \Om_{\rho(t)}(\bs u(t)), 
}
where $\Om_{\rho(t)}(\bs u(t))$ is the error created by the cut-off. Importantly, this error has structure, see Lemmas~\ref{lem:vir} and \ref{lem:virial-error}, and satisfies the estimates, 
\EQ{
\Om_{\rho(t)}(\bs u(t)) \lesssim (1+ \abs{ \rho'(t)}) \min \{ \|\bs u(t)\|_{\E( \rho(t), 2 \rho(t))}, \bfd(t)\} . 
}
Roughly, this allows us to think of $\fv(t)$ as a Lyapunov functional for our problem, localized to scale $\rho(t)$, with ``almost'' critical points given by multi-bubbles $\bs\calW( \vec \iota, \vec \lam)$. Indeed, if $\bs u(t)$ is close to a multi-bubble up to scale $\rho(t)$, and $\abs{\rho'(t)} \lesssim 1$,  then $\abs{\fv'(t)} \lesssim \bfd(t)$. 

Our first result is a localized compactness lemma. In Section~\ref{sec:compact} we prove the following: given a sequence of nonlinear waves $\bs u_n(t) \in\E$  on time intervals $[0, \tau_n]$ with bounded energy, and a sequence $R_n \to \infty$ such that 
\EQ{
\lim_{n \to \infty} \frac{1}{\tau_n} \int_0^{\tau_n} \int_0^{R_n\tau_n} \abs{ \p_t u_n(t, r)}^2 \, r^{D-1} \, \ud r \, \ud t  = 0, 
}
one can find a new sequence $1 \ll r_n \ll R_n$ and a sequence of times $s_n \in [0, \tau_n]$,  so that up to passing to a subsequence of the $\bs u_n$,  we have $\lim_{n \to \infty} \bs \de_{r_n \tau_n}( \bs u_n(s_n)) = 0$. Here  $\bs \de_{R}(\bs u)$ is a local (up to scale $R$) version of the distance function $\bfd$. 

We give a caricature of the no-return analysis, pointing the reader to the technical arguments in Sections~\ref{sec:decomposition},~\ref{sec:conclusion} for the actual arguments.  
We would like to integrate~\eqref{eq:v'} over intervals $[a_n, b_n]$ with $a_n, b_n \to T_*$ such that $\bfd(a_n), \bfd(b_n) \ll 1$ but contain some subinterval $[c_n, d_n] \subset [a_n, b_n]$ on which $\bfd(t) \simeq 1$; such intervals exist under the contradiction hypothesis.  From~\eqref{eq:v'} we obtain, 
\EQ{ \label{eq:vir-ineq} 
\int_{a_n}^{b_n} \int_0^{\rho(t)}\abs{\p_t u(t, r)}^2\,  r^{D-1}\, \ud r \, \ud t \lesssim \rho(a_n) \bfd(a_n) + \rho(b_n) \bfd(b_n) + \int_{a_n}^{b_n}\abs{\Om_{\rho(t)}(\bs u(t))} \, \ud t. 
}
 We consider the choice of $\rho(t)$. One can use the sequential compactness lemma so that choosing $\rho(t)/(d_n - c_n) \gg 1$  we have, 
\EQ{ \label{eq:compact} 
\int_{c_n}^{d_n}\int_0^{\rho(t)}\abs{\p_t u(t, r)}^2 \, r^{D-1} \, \ud r \, \ud t \gtrsim d_n-c_n , 
}
and one can expect that the integral of the error $\int_{c_n}^{d_n} \abs{\Om_{\rho(t)}(\bs u(t))} \, \ud t \ll \abs{d_n-c_n}$ absorbs into the left-hand side by choosing $\rho(t)$ to lie in a region where $\bs u(t)$ has negligible energy.  

To complete the proof one would need to show that the error generated on the intervals $[a_n, c_n]$ and $[d_n, b_n]$ can also be absorbed into the left-hand side, and moreover that the boundary terms $\rho(a_n) \bfd(a_n),  \rho(b_n) \bfd(b_n) \ll d_n - c_n$. For this, we require a more careful choice of the intervals $[a_n, b_n]$ and placement of the cut-off $\rho(t)$, which 
motivates the notion of \emph{collision intervals} introduced  in Section~\ref{ssec:proximity}. These allow us to distinguish between ``interior'' bubbles that come into collision, and ``exterior'' bubbles that stay coherent throughout the intervals $[a_n, b_n]$, and to ensure we place the cutoff in the region between the interior and exterior bubbles. 

Given $K \in \{1, \dots, N\}$, 
we say that an interval $[a, b]$ is a collision interval with parameters $0<\eps< \eta$ and $N-K$ exterior bubbles for some $1 \le K \le N$, if $\bfd(a),  \bfd(b) \le \eps$, there exists a $c \in [a, b]$ with $\bfd(c) \ge \eta$, and a curve $r = \rho_K(t)$ outside of which $\bs u(t) - \bs u^*(t)$ is within $\eps$ of an $N-K$-bubble in the sense of~\eqref{eq:d-intro} (a localized version of $\bfd(t)$); see Defintion~\ref{def:collision}.  We now define $K$ to be the \emph{smallest} non-negative integer for which there exists $\eta>0$, a sequence $\eps_n \to 0$,  and sequences $a_n, b_n \to T_*$, so that $[a_n, b_n]$ are collision intervals with parameters $\eps_n, \eta$ and $N-K$ exterior bubbles, and we write $[a_n, b_n] \in \calC_K( \eps_n, \eta)$; see Section~\ref{ssec:proximity} for the proof that $K$ is well-defined and $\ge 1$, under the contradiction hypothesis.  

We revisit~\eqref{eq:vir-ineq} on a sequence of collision intervals $[a_n, b_n] \in \calC_K( \eps_n, \eta)$. Near the endpoints $a_n, b_n$,  $\bs u(t) - \bs u^*(t)$ is close to an $N$-bubble configuration and we denote the interior scales, which will come into collision, by $\vec \lam = ( \lam_1, \dots, \lam_K)$ and the exterior scales, which stay coherent, by $\vec \mu = ( \vec \mu_{K+1}, \dots, \vec \mu_N)$. We assume for simplicity in this discussion that the collision intervals have only a single subinterval $[c_n, d_n]$ as above, and that $\bfd(t)$ is sufficiently small on the intervals $[a_n, c_n]$ and $[d_n, b_n]$ so that the interior scales are well defined (via modulation theory) there. We call $[a_n, c_n], [d_n, b_n]$ \emph{modulation intervals} and $[c_n, d_n]$ \emph{compactness intervals}. 

 The scale of the $K$th bubble $\lam_K(t)$ plays an important role and must be carefully tracked. We will need to also make sense of this scale on the compactness intervals, where the bubble itself may lose its shape from time to time. We do this by energy-norm considerations; see Definition~\ref{def:mu}. Crucially, the minimality of $K$ can be used to ensure that the intervals $[c_n, d_n]$ as above satisfy $d_n -c_n \simeq \max\{\lam_K(c_n), \lam_K(d_n)\}$; see Lemma~\ref{lem:cd-length}. Thus the first terms on the right-hand-side of~\eqref{eq:vir-ineq} can be absorbed using~\eqref{eq:compact} by ensuring $\rho(a_n) = o(\eps_n^{-1}) \lam_K(a_n), \rho(b_n) = o(\eps_n^{-1}) \lam_K(b_n)$ if we can additionally prove that  the scale $\lam_K(t)$ does not change much on the modulation intervals. Note that our choice of cut-off  will  satisfy $\lam_K(t) \ll \rho(t) \ll \mu_{K+1}(t)$. 
 
 We  must also absorb the errors $(\int_{a_n}^{c_n}+ \int_{d_n}^{b_n})|\Om_{\rho(t)}(\bs u(t))| \, \ud t  \lesssim (\int_{a_n}^{c_n}+ \int_{d_n}^{b_n}) \bfd (t) \, \ud t$ on the modulation intervals. Here we perform a refined modulation analysis on the interior bubbles, which allows us to track the growth of $\bfd(t)$ through a collision of (possibly) many bubbles. Roughly, up to scale $\rho(t)$, $\bs u(t)$ looks like a $K$-bubble, and using the implicit function theorem we define modulation parameters $\vec \iota$,  $\vec \lam(t)$, and error $\bs g(t)$ with 
 \EQ{
 \bs u(t, r) = \bs\calW( \vec \iota, \vec \lam(t); r) + \bs g(t, r), \mif r \le \rho(t),\quad \La\Lam W_{\lam_j(t)} \mid g(t) \Ra = 0, \mfor \, \,  j = 1,\dots, K, 
 }
 where $\Lam:= r\p_r + \frac{D-2}{2}$ is the generator of the $H$-invariant scaling (note that for $D=4, 5, 6$ the decomposition is slightly different due to the slow decay of $\Lam W$). 
The orthogonality conditions and an expansion of the nonlinear energy of $\bs u(t)$ up to scale $\rho(t)$  lead to the coercivity estimate, 
 \EQ{
 \| \bs g(t) \|_{\E} + \sum_{j \neq \calS} \Big( \frac{\lam_{j}(t)}{\lam_{j+1}(t)} \Big)^{\frac{D-2}{4}} \lesssim \max_{i \in \calS} \Big( \frac{\lam_i(t)}{\lam_{i+1}(t)} \Big)^{\frac{D-2}{4}} + \max_{1 \le i \le K}\abs{a^{\pm}_i(t)} +  o_n(1)  \simeq \bfd(t) + o_n(1), 
 }
where $\calS = \{ j \in 1, \dots, K-1 \, : \, \iota_j = \iota_{j+1}\}$ captures the non-alternating bubbles (which experience an attractive interaction force). The terms $a_j^{\pm}(t)$ on the right-hand side above are, roughly speaking, the projections of $\bs g(t)$ onto the unstable/stable directions related to the unique, simple negative eigenvalue associated to the linearization about $W$.  The  $o_n(1)$ term comes from errors due to the presence  of the radiation $\bs u^*$ in the region $r \lesssim \rho(t) \ll t$.  In fact, since $\bfd(t)$ grows out of the modulation intervals we can absorb these errors into $\bfd(t)$ by enlarging the parameter $\eps_n$ and requiring the lower bound $\bfd(t) \ge \eps_n$ on the modulation intervals. 

The growth of $\bfd(t)$ is then captured by the dynamics of adjacent bubbles with the same sign, or by the dynamics along the unstable/stable directions $a_{j}^{\pm}(t)$. In the case when the dynamics is driven by bubble interactions, precise information enters at the level of $\lam_{j}''(t)$, since~\eqref{eq:nlw} is second order.  However, it is not clear how to derive useful estimates from the equation for $\lam''(t)$ obtained by twice differentiating the orthogonality conditions. To cancel terms with critical size, but indeterminate sign, we introduce a localized virial correction to $\lam_j'\simeq - \iota_j  \| \Lam W \|_{L^2}^{-2} \lam_j^{-1}\La \Lam W_{\lam_j} \mid \dot g\Ra$, defining 
\EQ{
\beta_j(t) = - \iota_j \| \Lam W \|_{L^2}^{-2}\La \Lam W_{\U{\lam_j(t)}} \mid \dot g(t)\Ra  -   \| \Lam W \|_{L^2}^{-2}\ang{ \uln A( \lam_j(t)) g(t) \mid \dot g(t)}, 
}
where $\U A(\lam)$ is a truncated (to scale $\lam$) version of  $\U{\Lam} = r \p_r + \frac{D}{2}$, the generator of $L^2$ scaling. Roughly, we show in Sections~\ref{ssec:ref-mod} and~\ref{ssec:bub-dem},  that $(\lam_j(t), \beta_j(t), a_{j}^{\pm}(t))$  satisfy a system of differential inequalities that can be used to control the growth of $\bfd(t)$ until the solution exits a modulation interval. All the while, the $K$th scale $\lam_K(t)$ does not move much, and we obtain bounds of the form
$\int_{a_n}^{c_n}  \bfd(t)\ud t\leq C_0\big(\bfd(a_n)^{\min\big(1, \frac{4}{D-2}\big)} \lambda_K(a_n) + \bfd(c_n)^{\min\big(1, \frac{4}{D-2}\big)} \lambda_K(c_n)\big)$, and an analogous bound on the interval $[d_n, b_n]$ (see the ``ejection'' Lemma~\ref{lem:ejection}). Thus the errors can be absorbed into the left-hand side of~\eqref{eq:vir-ineq} and we obtain a contradiction. 
In dimensions $D \ge 5$,  this proof follows closely the scheme from~\cite{JL6} together with an elegant ``weighted sum" trick from~\cite[Section 6]{DKM9}, which simplifies some of the ODE analysis from~\cite{JL6}; see Section~\ref{ssec:bub-dem}. The analysis in dimension $D=4$ is more complicated, due to the fact that the modulation inequalities for the $j$th scale are only valid on subintervals where the ratio $(\lam_{j}(t)/\lam_{j+1}(t))^{\frac{D-2}{4}}$ is comparable to $\bfd(t)$, and thus a weighted sum trick involving the dynamics of all the bubbles at once does not seem to apply.  For this special case we introduce an induction scheme together with the notion of an ``ignition condition'', (see Definition~\ref{def:ignition}) which identifies the most relevant controllable index $j$ on a given subinterval of the modulation interval; see Appendix~\ref{sec:D=4}. 

While other aspects of the proof adapt readily to dimension $D=3$, this refined analysis of the modulation parameters introduces significant complications due to the slow decay of $\Lam W$.

A similar, but simpler 
refined  modulation  
analysis was performed in~\cite{JL1}.  
The use of  refinements to modulation parameters to obtain dynamical control was introduced by the first author in the context of a two-bubble construction for $NLS$  in~\cite{JJ-APDE}. The notion of localized virial corrections to modulation parameters  was first introduced by Rapha\"el and Szeftel in~\cite{RaSz11} in a different context.

\subsection{Notational conventions}
%
%
The nonlinear energy is denoted $E$, $\cE$ is the energy space.

Given a function $\phi(r)$ and $\lambda>0$, we denote by $\phi_{\lam}(r) = \lam^{-\frac{D-2}{2}}\phi(r/ \lam)$, the $H$-invariant re-scaling, and by $\phi_{\U{\lam}}(r) = \lam^{-\frac{D}{2}} \phi(r/ \lambda)$ the $L^2$-invariant re-scaling. We denote by $\Lam :=r \p_r + \frac{D-2}{2}$ and $\ULam := r \partial r +\frac{D}{2}$ the infinitesimal generators of these scalings. We denote $\ang{\cdot\mid\cdot}$
the radial $L^2(\bR^2)$ inner product given by \eqref{eq:inner-prod}.

We denote by $f(u) := \abs{u}^{\frac{4}{D-2}} u$ the nonlinearity in \eqref{eq:nlw}.
We let $\chi$ be a smooth cut-off function, supported in $r \leq 2$ and equal $1$ for $r \le 1$.

The general rules we follow giving names to various objects are:
\begin{itemize}
\item index of an infinite sequence: $n$
\item sequences of small numbers: $\gamma, \delta, \epsilon, \zeta, \eta, \theta$
\item scales of bubbles and quantities describing the spatial scales: $\lambda, \mu, \nu, \xi, \rho$;
in general we call $\lambda$ the scale of the interior bubbles and $\mu$ the exterior ones
(once these notions are defined)
\item moment in time: $t, s, \tau, a, b, c, d, e, f$
\item indices in summations: $ i, j, \ell$
\item time intervals: $I, J$
\item number of bubbles: $K, M, N$
\item signs are denoted $\iota$ and $\sigma$
\item boldface is used for pairs of elements related to the Hamiltonian structure; an arrow is used for vectors (finite sequences) in other contexts.
\end{itemize}
We call a ``constant'' a number which depends only on the dimension $D$ and the number of bubbles $N$.
Constants are denoted $C, C_0, C_1, c, c_0, c_1$. We write $A \lesssim B$ if $A \leq CB$ and $A \gtrsim B$ if $A \geq cB$.
We write $A \ll B$ if $\lim_{n\to \infty} A / B = 0$.

For any sets $X, Y, Z$ we identify $Z^{X\times Y}$ with $(Z^Y)^X$, which means that
if $\phi: X\times Y \to Z$ is a function, then for any $x \in X$ we can view $\phi(x)$ as a function $Y \to Z$
given by $(\phi(x))(y) := \phi(x, y)$.

\section{Preliminaries}

\subsection{Virial identities} 

 We have the following virial identities.  
 \begin{lem}[Virial identities] \label{lem:vir} 
 Let $\bs u(t) \in \E$ be a solution to~\eqref{eq:nlw} on an open time interval $I$ and $\rho: I \to (0, \infty)$
 a Lipschitz function. Then for almost all $t \in I$, 
 \EQ{\label{eq:vir-1}
 \frac{\ud}{\ud t} \ang{  \p_t u(t) \mid \chi_{\rho(t)} \, r \p_r u(t)}  &= - \frac{D}{2} \int_0^\infty (\p_t u(t, r))^2\chi_{\rho(t)}(r) \, r^{D-1} \, \ud r  \\
 &\quad + \frac{D-2}{2}  \int_0^\infty \Big[(\p_r u(t, r))^2 - \abs{u(t, r)}^{\frac{2D}{D-2}} \Big]\chi_{\rho(t)}(r) \, r^{D-1} \, \ud r \\
&\quad  + \Om_{1, \rho(t)}(\bs u(t)), 
 }
 and, 
 \EQ{
\label{eq:vir-2}
\frac{\ud}{\ud t} \ang{  \p_t u(t) \mid \chi_{\rho(t)} \,u(t)}  &= \int_0^\infty (\p_t u(t, r))^2\chi_{\rho(t)}(r) \, r^{D-1} \, \ud r \\
& \quad - \int_0^\infty\Big[(\p_r u(t, r))^2 - \abs{u(t, r)}^{\frac{2D}{D-2}} \Big]\chi_{\rho(t)}(r) \, r^{D-1} \, \ud r  - \Om_{2, \rho(t)}(\bs u(t)) 
 }
 where 
 \EQ{ \label{eq:Om-rho-def} 
 \Om_{1, \rho(t)}(\bs u(t)) &:= - \frac{\rho'(t)}{\rho(t)} \int_0^\infty \p_t u(t, r) r \p_r u(t, r)   (r \p_r \chi)(r/\rho(t))  \,r^{D-1}\,  \ud r \\
 & \quad - \frac{1}{2}\int_0^\infty  \Big( (\p_t u(t, r))^2 + (\p_r u(t, r))^2 - \frac{D-2}{D} \abs{u(t, r)}^{\frac{2D}{D-2}} \Big)   (r \p_r \chi)(r/\rho(t)) r^{D-1} \,  \ud r, \\
  \Om_{2, \rho(t)}(\bs u(t))&:=- \frac{\rho'(t)}{\rho(t)} \int_0^\infty \p_t u(t, r)  u(t, r)   (r \p_r \chi)(r/\rho(t))  \,r^{D-1}\,  \ud r \\
  &\quad - \int_0^\infty \p_r u(t, r) \frac{u(t, r)}{r} (r \p_r \chi)(r/\rho(t))   \, r^{D-1} \, \ud r . 
 }

 \end{lem}
 \begin{proof}
The proof is a direct computation along with an approximation argument for fixed $t \in I$, assuming $\rho$ is differentiable at $t$.
  \end{proof}
  
  \begin{rem}\label{rem:virial} In practice we will make use of the following two linear combinations of the identities~\eqref{eq:vir-1} and~\eqref{eq:vir-2}. 
  \EQ{\label{eq:virial} 
  \frac{\ud}{\ud t}   \Big\langle  \p_t u(t) \mid \chi_{\rho(t)} \, \Big(r \p_r u(t) + \frac{D-2}{2} u(t)\Big)\Big\rangle & = - \int_0^\infty (\p_t u(t, r))^2\chi_{\rho(t)}(r) \, r^{D-1} \, \ud r \\
  &\quad + \Om_{1, \rho(t)}(\bs u(t)) + \frac{D-2}{2} \Om_{2, \rho(t)}(\bs u(t))
  }
  and, 
  \EQ{ \label{eq:jk} 
   \frac{\ud}{\ud t}   \Big\langle  \p_t u(t) \mid \chi_{\rho(t)} \, \Big(r \p_r u(t) + \frac{D}{2} u(t)\Big)\Big\rangle & = - \int_0^\infty \Big[(\p_r u(t, r))^2 - \abs{u(t, r)}^{\frac{2D}{D-2}} \Big]\chi_{\rho(t)}(r) \, r^{D-1} \, \ud r \\
   &\quad + \Om_{1, \rho(t)}(\bs u(t)) + \frac{D}{2} \Om_{2, \rho(t)}(\bs u(t))
  }
  Note that the multiplier $(r \p_r + \frac{D-2}{2}) u$ in the first indentity~\eqref{eq:virial} corresponds to the generator of $\dot H^1$-invariant dilations $x \cdot \nabla u$. 
  \end{rem}

\subsection{Local Cauchy theory} \label{sec:Cauchy}

In the remainder of this section, we follow the presentation in \cite[Section 2 and Appendix A]{Rod16-adv}.

Given that we are restricting our attention to radially symmetric functions $v: \R^D \to \R$, we will often abuse notation slightly, writing $v = v(r)$ and denoting, for $p\ge 1$, 
\EQ{
\| f \|_{L^p( \R^D)} := \left( \int_0^\infty \abs{ f (r)}^p \, r^{D-1} \, \ud r \right)^{\frac{1}{p}}. 
}

For $0 < s < 1$ and $1 < p < \infty$, $\dot B_{p, 2}^s = \dot B_{p, 2}^s(\bR^D)$ denotes the homogeneous Besov space with norm
\begin{equation}
\|u\|_{\dot B_{p,2}^s} := \bigg(\sum_{j\in \bZ}2^{2js}\|P_j u\|_{L^p}^2\bigg)^\frac 12,
\end{equation}
where $P_j$ are the Littlewood-Paley projections. We recall that if $s < D/p$, then $\dot B_{p, 2}^s$ is a Banach space,
see \cite[Theorem 2.25]{BaChDa11}, and that the homogeneous Besov norms can be equivalently defined in terms of finite differences:
\begin{equation}
\|u\|_{\dot B_{p,2}^s} \simeq \bigg(\int_{\bR^D}|y|^{-D-2s}\|u(\cdot + y) - u\|_{L^p}^2\ud y\bigg)^\frac 12,
\end{equation}
see \cite[Theorem 2.36]{BaChDa11}.

For any time interval $I \subset \bR$, we introduce the Strichartz-type spaces and norms
\begin{align}
S(I) &:= L^\frac{2(D+1)}{D-2}(I\times \bR^D), \\
W(I) &:= L^\frac{2(D+1)}{D-1}\Big(I; \dot B_{\frac{2(D+1)}{D-1}, 2}^\frac 12(\bR^D)\Big) 
\end{align}
We denote $S\lin(t)$ the free wave propagator, in other words for all $\bs v_0 = (v_0, \dot v_0)$ we have
\begin{equation}
S\lin(t)\bs v_0 = \bigg(\cos(t|\grad|)v_0 + \frac{\sin(t|\grad|)}{|\grad|}\dot v_0,
-|\grad|\sin(t|\grad|)v_0 + \cos(t|\grad|)\dot v_0\bigg).
\end{equation}
We say that $\bs u$ is a solution of \eqref{eq:nlw} on a time interval $I \owns 0$ with initial data $\bs u_0 \in \cE$ if
\begin{itemize}
\item $\bs u \in C^0(I; \cE)$,
\item $\|u\|_{S(J)} + \|u\|_{W(J)} < \infty$ for all compact intervals $J \subset I$,
\item $\bs u$ satisfies the Duhamel formula
\begin{equation}
\bs u(t) = S\lin(t)\bs u_0 + \int_0^t S\lin(t - s)(0, f(u(s)))\ud s.
\end{equation}
\end{itemize}

Local well-posedness was obtained by Ginibre and Velo~\cite{GV89}, who used a slightly different but equivalent notion of solution; see also~\cite{Pecher88, Struwe88, Grillakis90}.  We use the versions in \cite{KM08, BCLPZ13}. Key to the proof are Strichartz estimates for the wave equation; see, Lindblad, Sogge~\cite{LinS}, and Ginibre, Velo~\cite{GiVe95}.
\begin{lem}[Cauchy theory in $\E$]\emph{\cite[Theorem 2.7]{KM08} \cite[Theorem 3.3]{BCLPZ13}} \label{lem:Cauchy}  There exists $\de_0>0$ and functions $\eps,  C: [0, \infty) \to (0, \infty)$ with $\eps(\de) \to 0 \mas \de \to 0$,  with the following properties.
 Let $A>0$ and $\bs u_0 = (u_0, u_1) \in \E$ with $\| \bs u_0 \|_{\E} \le A$.  Let $I \ni 0$ be an open interval such that 
\EQ{
\| S\lin(t) \bs u_0 \|_{\calS(I)} = \de \le \de_0  
}
Then there exists a unique solution $\bs u(t)$ to~\eqref{eq:nlw} in the space $C^0(I; \E) \cap S(I) \cap W(I)$ with initial data $\bs u(0) = \bs u_0$. The solution $\bs u(t)$ satisfies the bounds $\| u \|_{\calS(I)} \le C(A) \eps(\de)$, and $ \| \bs u \|_{L^\infty_t(I; \E)} \le C(A)$. To each solution $\bs u(t)$ to~\eqref{eq:nlw} we can associate a maximal interval of existence $0 \in I_{\max}(\bs u) = (-T_-, T_+)$ such that for each compact subinterval $I' \subset I_{\max}$ we have $\| u \|_{S(I')} < \infty$ and, if, say $T_+<\infty$, then $\lim_{T\to T_+} \| u \|_{S( [0, T))} = \infty$. 

The data to solution map is continuous in the following sense. Let $\bs u_0 \in \E$ and let $\bs u(t) \in \E$ be the unique solution to~\eqref{eq:nlw} with data $\bs u_0$, and let  $T_0< T_+( \bs u_0)$. Then for every $\eps>0$ there exists $\de >0$ with the following property:  for all $\bs v_0 \in \E$ and all $T \le T_0$ with $\| \bs u_0 - \bs v_0 \|_{\E}  + T < \de$ we have $T< T_+(\bs v_0)$ and  $\sup_{t \in [0, T]} \| \bs u(t) - \bs v(t)\|_{\E} < \eps$, where $\bs v(t)$ is the unique solution to~\eqref{eq:nlw} associated to $\bs v_0$. 

The completeness of wave operators holds: there exists $\eps_0$ small enough so that if $\bs u_0 \in \E$ satisfies $\|\bs u_0\|_{\E} < \eps_0$, the solution $\bs u(t)$ given above is defined globally in time, satisfies the bound, 
\EQ{ \label{eq:small-norm} 
\sup_{t \in \R}\| \bs u(t) \|_{\E} + \| u \|_{S(\R)}+ \| u \|_{W(\R)} \lesssim \|\bs u_0\|_{\E}, 
}
and scatters in the following sense: there exist free waves $\bs u_\Lin^{\pm}(t) \in \E$  such that 
\EQ{ \label{eq:scattering} 
\| \bs u(t) - \bs u\lin^{\pm}(t) \|_{\E} \to 0 \mas  t \to \pm \infty
}
Conversely, the existence of wave operators holds, i.e.,  for any solution $\bs v_{\Lin}(t) \in \E$ to the free linear equation, there exists a unique, global-in-forward time solution $\bs u(t) \in \E$ to~\eqref{eq:nlw} such that~\eqref{eq:scattering} holds as $t \to \infty$. An analogous statement holds for negative times. 
\end{lem} 

\begin{rem} 
For dimensions $D \ge 6$ the continuous dependence on the initial data is not stated explicitly as part of ~\cite[Theorem 3.3]{BCLPZ13}, but it does follow from their proof; see for example \cite[Remark 4.3]{BCLPZ13}. 
\end{rem}

%

The following lemma is  a consequence of the local Cauchy theory. 
 
 \begin{lem}(Propagation of small $\E$ norm)  \label{lem:prop-small-E} There exists $\de, C>0$ with the following properties. Let $I \ni 0$ be a time interval and let $\bs v(t) \in \E$ be the solution to~\eqref{eq:nlw} on $I$ with initial data $\bs v(0) = \bs v_0$. If, 
 \EQ{
 \| \bs v_0 \|_{\E} \le \de, 
 } 
 then $0 \le  E( \bs v_0) \le C\de^2$ and 
 \EQ{
\sup_{t \in I} \| \bs v(t) \|_{\E}^2 \le C E( \bs v) \le C \de^2. 
 }
 \end{lem} 
 
 \begin{proof}
 Abusing notation and identifying $\bs v_0 \in \E$ with $\bs v_0 \in \dot H^1 \times L^2( \R^D)$ we can express the nonlinear energy of $\bs v_0$ as
 \EQ{
 E( \bs v_0) = \frac{1}{2} \| \dot v_0 \|_{L^2(\R^D)}^2 +  \frac{1}{2} \| \na v_0 \|_{L^2(\R^D)}^2  - \frac{D-2}{2D} \| v_0 \|_{L^{\frac{2D}{D-2}}}^{\frac{2D}{D-2}(\R^D)}. 
 }
 It is clear that $E( \bs v) = E( \bs v_0) \le C \| \bs v_0\|_{\E}^2 \le C\de^2$. Since $\frac{2D}{D-2} >2$, by the Sobolev inequality, $\| v \|_{L^{\frac{2D}{D-2}}(\R^D)} \le C_1 \| \na v \|_{L^2(\R^D)}$  together  with  Hardy's inequality,  $\| \abs{x}^{-1} v \|_{L^2(\R^D)} \le C_2 \| \na v \|_{L^2(\R^D)}^2$, which hold for all $v \in \dot H^1( \R^D)$, we see that by taking $\de>0$ small enough we can find $C>0$ so that 
 \EQ{
 E( \bs v) = E( \bs v_0) \ge C^{-1} \| \bs v_0 \|_{\E}^2 
 }
 The remaining conclusion now follows from the estimate~\eqref{eq:small-norm} restricted to the time interval $I$. 
 \end{proof} 
 
 Using the finite speed of propagation and the previous lemma, one obtains the following localized version. 
 
 \begin{lem}(Propagation of small localized $\E$ norm)  \label{lem:prop-small-E-local} 
 There exists $\de, C>0$ with the following properties. Let $I \ni 0$ be a time interval and let $\bs u(t) \in \E$ be a solution to~\eqref{eq:nlw} on $I$ with initial data $\bs u(0) = \bs u_0$. Let $0< r_1< r_2$. Suppose that 
 \EQ{
  \| \bs u_0 \|_{\E(\frac{ r_1}{4}, 4r_2)} \le \de .
 }
 Then, 
 \EQ{
 \| \bs u(t)  \|_{\E(r_1 + \abs{t}, r_2 - \abs{t})} \le C \de, 
 }
 for all $t \in I$ such that $r_1 + 2 \abs{t} < r_2$. 
 
 \end{lem} 
 
 \begin{proof} 
Let $\fy(r)$ be a smooth cut-off function such that $\fy(r) = 1$ if $r \in [\frac{1}{2} r_1, 2 r_2]$, $\fy(r)= 0$ if $r \in (0, \frac{1}{4} r_1] \cup [4 r_2, \infty)$ and such that $\abs{\p_r \fy (r)} \le 4 r_1^{-1}$ for $r \in [\frac{1}{4} r_1, \frac{1}{2} r_1]$ and $\abs{\p_r \fy (r)} \le 4 r_2^{-1}$ for $r \in [2 r_2, 4r_2]$. Setting $\bs v_0 := \fy \bs u_0$ it follows from the definition of the local $\E$-norm in~\eqref{eq:local-E-norm} that $\|\bs v_0\|_{\E} \le C \|\bs u_0\|_{\E(\frac{ r_1}{4}, 4r_2)}$ for some constant $C>0$ independent of $\bs u_0, r_1, r_2$. Taking $\de>0$ sufficiently small we may apply Lemma~\ref{lem:prop-small-E} to the solution $\bs v(t)$ with initial data $\bs v(0) = \bs v_0$ The conclusion then follows by finite speed of propagation, which ensures that $\bs u(t, r) = \bs v(t, r)$ for all $(t, r) \in I \times (0, \infty)$ with $r \in (r_1 + \abs{t}, r_2 - \abs{t})$ and $r_1 + 2 \abs{t} < r_2$. 
 \end{proof} 
 
 \begin{lem}[Short time evolution close to $W$]  \label{lem:evolve-W} Let $\iota \in \{-1, 1\}$. There exists $\de_0>0$ and a function $\eps_0: [0, \de_0] \to [0, \infty)$ with $\eps_0( \de) \to 0$ as $\de \to 0$ with the following properties. 
 Let $\bs v_0 \in \E$ and let $\bs v(t)$ denote the unique solution to~\eqref{eq:nlw} with $\bs v_0(0) = \bs v_0$. Let $\mu_0, T_0>0$ and suppose that 
 \EQ{
 \| \bs v_0 -  \iota \bs W_{\mu_0} \|_{\E} + \frac{T_0}{\mu_0}   = \de \le \de_0
 }
 Then, $T_0< T_{+}( \bs v_0)$ and
 \EQ{
 \sup_{t \in [0, T_0]} \| \bs v(t) - \iota \bs W_{\mu_0} \|_{\E} <  \eps_0( \de)  
 }
 
 \end{lem} 
 
 \begin{proof} By rescaling we may assume $\mu_0 = 1$. The result is then a particular case of the local Cauchy theory, in particular the continuity of the data to solution map at $\bs W$.  
 \end{proof} 
 
 We also require the following localized version. 
  \begin{lem}[Localized short time evolution close to $W$]  \label{lem:evolve-W-local} Let $\iota \in \{-1, 1\}$.
 There exists $\de_0>0$ and a function $\eps_0: [0, \de_0] \to [0, \infty)$ with $\eps_0( \de) \to 0$ as $\de \to 0$ with the following properties. 
 Let $\bs u_0 \in \E$, $T_0 < T_+(\bs u_0)$, and let $\bs u(t)$ denote the unique solution to~\eqref{eq:nlw} with $\bs u_0(0) = \bs u_0$. Let $\mu_0>0$, $0 < r_1 < r_2 < \infty$  and suppose that 
 \EQ{
 \| \bs u_0 -  \iota \bs W_{\mu_0} \|_{\E( \frac{1}{4} r_1, 4 r_2)} + \frac{T_0}{\mu_0}  =   \de \le \de_0
 }
 Then, 
 \EQ{
 \| \bs u(t) - \iota \bs W_{\mu_0} \|_{\E( r_1 + t, r_2 - t)} <  \eps_0( \de)  
 }
 for all $0 < t \le T_0$ such that $r_1 + 2 \abs{t} < r_2$. 
  \end{lem}

\begin{proof} Let $\fy(r)$ be as in the proof of Lemma~\ref{lem:prop-small-E-local} and define $\bs v_0 := \fy \bs u_0 + (1- \fy) \bs W_{\mu_0}$. By taking $\de_0$ sufficiently small we see that $\bs v_0, \mu_0, T_0$ satisfy the hypothesis of Lemma~\ref{lem:evolve-W}. The conclusion then follows from the finite speed of propagation. 
\end{proof} 

We will make use of the following consequence of the previous four lemmas. 
 
 \begin{lem}
\label{lem:evol-of-Q}
If  $\iota_n \in \{-1, 0, 1\}$, $0 < r_n \ll \mu_n \ll R_n$, $0 < t_n \ll \mu_n$ and $\bs u_{n}$ a sequence of
solutions of \eqref{eq:nlw} such that $\bs u_n(t)$ is defined for $t \in [0, t_n]$ and
\begin{equation}
\lim_{n\to\infty}\|\bs u_n(0) - \iota_n \bs W_{\mu_n}\|_{\cE(\frac 14 r_n, 4 R_n)} = 0,
\end{equation}
then
\begin{equation}
\lim_{n\to\infty}\sup_{t\in [0, t_n]}\|\bs u_n(t) - \iota_n \bs W_{\mu_n}\|_{\cE(r_n + t, R_n-t)} = 0.
\end{equation}
\end{lem}
\begin{proof}
This is a direct consequence of Lemma~\ref{lem:prop-small-E-local} when $\iota_n = 0$ and Lemma~\ref{lem:evolve-W-local} when $\iota_n \in \{-1, 1\}$. 
%
\end{proof}

%

%
%

\subsection{Profile decomposition}

The linear  profile decomposition of Bahouri and G\'erard~\cite{BG} is an essential ingredient in the study of solutions to~\eqref{eq:nlw}; see also~\cite{BrezisCoron, Gerard, Lions1, Lions2,  MeVe98}. 


\begin{lem}[Linear profile decomposition]\emph{\cite{BG}}  \label{lem:pd}  Let  $\bs u_n$ be a bounded sequence in $\E$, i.e.,  $\limsup_{n \to \infty} \|\bs u_n\|_{\E} <\infty$. Then, after passing to a subsequence, there exists sequences $ \lam_{n, j} \in (0, \infty)$, and  $t_{n, j} \in \R$  and finite energy free waves $\bs v\lin^{j} \in \E$ such that for each $J \ge 1$, 
\EQ{
\bs u_n &=  \sum_{j =1}^J \Big( \lam_{n, j}^{-\frac{D-2}{2}}v_{\Lin}^j \big(  \frac{-t_{n, j}}{\lambda_{n, j}}, \frac{ \cdot}{\lambda_{n, j}} \big) , \lambda_{n, j}^{-\frac{D}{2}} \p_t v_{\Lin}^j\big( \frac{-t_{n, j}}{\lambda_{n, i}}, \frac{ \cdot}{\lam_{n, j}} \big)\Big) + \bs w_{n, 0}^J( \cdot)
}
where, denoting by $\bs w_{n, \Lin}^J(t)$ the free wave  with initial data $\bs w_{n, 0}^J$, the following hold: 
\begin{itemize} 
\item for each $j$,  either $t_{n, j} = 0$ for all $n$ or $\lim_{n\to \infty} \frac{-t_{n, j}}{\lam_{n, i}}  = \pm \infty$. If $t_{n, j} = 0$ for all $n$, then one of $\lam_{n, j} \to 0$, $\lam_{n, j} = 1$ for all $n$, or $\lam_{n, j} \to \infty$ as $n \to \infty$, holds; 
\item the scales $\lam_{n, j}$ and times $t_{n, j}$ satisfy, 
\EQ{
\frac{\lam_{n, j}}{\lam_{n, j'}} + \frac{ \lam_{n, j'}}{\lam_{n, j}} + \frac{ \abs{ t_{n, j} - t_{n, j'}}}{\lam_{n, j}} \to \infty \mas n \to \infty; 
}
\item the error term $\bs w_{n}^J$ satisfies, 
\EQ{
&( \lam_{n, j}^{\frac{D-2}{2}}w_{n, \Lin}^J( t_{n, j}, \lambda_{n, j} \cdot), \lam_{n, i}^{\frac{D}{2}} \p_t w_{n, \Lin}^J( t_{n, j}, \s_{n, j} \cdot)) \rightharpoonup 0\in \E \mas n \to \infty
}
for each $J \ge 1$, each $j \in \N$,  and vanishes strongly in the sense that 
\EQ{
\lim_{J \to \infty} \limsup_{n \to \infty} \Big(  \| w_{n, \Lin}^J \|_{L^\infty_{t}L^{\frac{2D}{D-2}}(\R \times \R^D)} + \| w_{n, \Lin} \|_{\calS(\R)} \Big)  = 0; 
}
\item the following  pythagorean decomposition of the free energy holds:  for each $J \ge 1$, 
\EQ{ \label{eq:pyth} 
\| \bs u_n\|_{\E}  &=  \sum_{j=1}^J \|  v^j\lin (-t_{n,j}/ \lam_{n, j}),  \p_t v^j\lin  (-t_{n,j}/ \lam_{n, j}))\|_{\E}  +\| \bs w_{n}^J \|_{\E} + o_n(1) 
} 
as $n \to \infty$. 
 
\end{itemize} 

\end{lem} 

\begin{rem} 
We call  the triplets $(\bs v\lin^i, \lam_{n, j}, t_{n, j})$ \emph{profiles}. Following Bahouri and G\'erard~\cite{BG} we refer to  the profiles  $(\bs v\lin^j, \lam_{n, j}, 0)$ as \emph{centered}, to  the profiles $(\bs v\lin^j , \lam_{n, j}, t_{n, j})$ with $- t_{n, j}/ \lam_{n, j} \to \infty$ as $n \to \infty$ as \emph{outgoing},  and those with $- t_{n, j}/ \lam_{n, j} \to -\infty$ as \emph{incoming}. 
\end{rem} 

\begin{rem} 
In Section~\ref{sec:compact} we implicitly make use of nonlinear profile decompositions, in addition to Lemma~\ref{lem:pd}, when we invoke arguments from~\cite{DKM3erratum}. We refer the reader to~\cite[Proposition 2.1]{DKM3erratum} for the statement. 
\end{rem}

\subsection{Multi-bubble configurations}

In this section we study properties of finite energy maps near a multi-bubble configuration, and we record several properties of the ground state $W$.

The operator $\LL_{\calW}$ obtained by linearization of~\eqref{eq:nlw} about an $M$-bubble configuration $\bs \calW(\vec \iota, \vec \lam)$ is given by, 
\EQ{  \label{eq:LW-def} 
\LL_{\calW} \, g := \uD^2 E_{\bfp}(\calW( \vec\iota, \vec \lam)) g = - \De g - f'(\calW( \vec\iota, \vec \lam) )g, 
}
where $f(z) := \abs{u}^{\frac{4}{D-2}} u$ and $f'( z) = \frac{D+2}{D-2} \abs{z}^{\frac{4}{D-2}}$. Given $\bs g = (g,\dot g) \in \E$, 
\EQ{
\La \uD^2 E(\bs\calW(\vec \iota, \vec \lam)) \bs g \mid \bs g \Ra =  \int_0^\infty \Big(\dot g(r)^2 + (\p_r g(r))^2 -  f'(\calW( \vec \iota, \vec \lam)) g(r)^2 \, \Big) r \ud r. 
}
An important instance of the operator $\LL_{\calW}$ is given by linearizing \eqref{eq:nlw} about a single copy of the ground state $ \calW(  \vec\iota, \vec \lam) = W_{\lam}$. In this case we use the short-hand notation, 
\EQ{ \label{eq:LL-def} 
\LL_{\lam} := -\De  - f'(W_\lam) 
}
We write $\calL := \calL_1$.

We define the infinitesimal generators of $\dot H^1$-invariant dilations by $\Lam$ and in the $L^2$-invariant case we write $\ULam$, which are given by 
\EQ{
\Lam := r \p_r + \frac{D-2}{2}, \quad \ULam := r\p_r + \frac{D}{2}
}
We have 
\EQ{
\Lam W(r) = \Big( \frac{D-2}{D} - \frac{r^2}{ 2D}  \Big) \Big( 1 + \frac{r^2}{D(D-2)}\Big)^{-\frac{D}{2}}
}
Note that both $W$ and $\Lam W$ satisfy,  
\EQ{
\abs{ W(r)} , \, \abs{\Lam W(r) } \simeq 1 \mif r \le 1,  \mand \abs{W(r)}, \,  \abs{ \Lam W(r)} \simeq r^{-D + 2} \mif r \ge 1
}
In fact, 
\EQ{ \label{eq:W-leading} 
W(r) &= 1 + O(r^2) \mif r \ll 1 \\
& = (D(D-2))^{\frac{D-2}{2}} r^{-D+2} + O( r^{-D}) \mif r \gg 1. 
}
and, 
\EQ{ \label{eq:LamW-leading} 
\Lam W(r) & = \frac{D-2}{2} + O(r^2) \mif r \ll 1 \\
& = - \frac{(D(D-2))^{\frac{D}{2}}}{2D} r^{-D+2} + O( r^{-D}) \mif r \gg 1
}
When $D=4$ we will use the extra decay in 
\EQ{\label{eq:D4-magic} 
\ULam \Lam W(r) = (r \p_r + 2) \Lam W(r) = \frac{1}{4} \frac{ 3r^2 -8}{ (1+ \frac{1}{8}r^2)^3} \simeq r^{-4} \mas r \to \infty \mif D=4
}
We note that if $D\ge 5$ then $\ang{\ULam \Lam W \mid \Lam W} = 0$. If $D=4$ then $\ang{\ULam \Lam W \mid \Lam W} = 32$. 

We will use the following computations, 
\EQ{ \label{eq:interaction} 
\frac{D+2}{D-2} &\int_0^\infty \Lam W(r) W(r)^{\frac{4}{D-2}} \, r^{D-1} \, \ud r = -\frac{D-2}{2D} ( D(D-2))^{\frac{D}{2}}\\
\frac{D+2}{D-2}( D(D-2))^{\frac{D-2}{2}}&\int_0^\infty \Lam W(r) W(r)^{\frac{4}{D-2}}  r \, \ud r = \frac{D-2}{2D} ( D(D-2))^{\frac{D}{2}}\\
\| \Lam W \|_{L^2}^2 &= \frac{2(D^2-4)(D(D-2))^{\frac{D}{2}}}{D^2(D-4)} \frac{\Gamma(1+ \frac{D}{2})}{\Gamma(D)} \mif D \ge 5 . 
}
If 
$D=4$, 
\EQ{ \label{eq:LamW-L2-4} 
\int_0^R (\Lam W(r))^2 r^3 \, \ud r = 16 \log R + O(1) \mas R \to \infty.
}
If $\s \ll 1$ we have, 
\EQ{
\ang{ \Lam W_{\U{\s}} \mid \Lam W} \lesssim \s^{\frac{D-4}{2}} \mif D \ge 5.
}
For any $R>0$, 
\EQ{
\ang{ \chi_{R \sqrt{\s}} \Lam W_{\U\sigma}  \mid \Lam W} \lesssim R^2 \s^{\frac{D -2}{2}} \mif D \ge 4. 
}

Next we discuss the spectral properties of $\LL$. Importantly, 
\EQ{
\LL ( \Lam W) = \frac{\ud}{\ud \lam}\rest_{\lam =1} ( -\De W_\lam - f(W_{\lam}) ) = 0
}
and thus if $D \ge 5$, $W \in L^2$ is a zero energy eigenvalue for $\LL$ and a threshold resonance if $D= 4$. In fact, $\{ (f, 0) \in \E \mid \LL f =0\} = \textrm{span}\{\bs W\}$ (see~\cite[Proposition 5.5]{DM08}). In addition to this fact, it was also shown in~\cite[Proposition 5.5]{DM08} that $\LL$ has a unique negative simple eigenvalue that we denote by $-\kappa^2<0$ (we take $\kappa>0$ 
We denote the associated eigenfunction by $\calY$ normalized in $L^2$ so that $\| \calY \|_{L^2} = 1$. By elliptic regularity $\calY$ is smooth, and by Agmon estimates it decays exponentially. Using that $\LL$ is symmetric we deduce that $\ang{\calY \mid \Lam W} = 0$. 

We follow the notations and set-up in~\cite[Section 3]{JJ-AJM}. Define
\EQ{
\bs \calY^- := ( \frac{1}{\kappa} \calY, - \calY), \quad \bs \calY^+ := ( \frac{1}{\kappa} \calY, \calY) 
}
and, 
\EQ{
\bs \al^-= \frac{\kappa}{2} J \bs \calY^+ = \frac{1}{2}( \kappa \calY, - \calY), \quad \bs \al^+:= -\frac{\kappa}{2} J \bs \calY^- = \frac{1}{2}( \kappa \calY, \calY) 
}
Recalling that $J \circ \uD^2 E( \bs W) = \pmat{0 & \textrm{Id} \\ - \LL & 0}$ we see that 
\EQ{
J \circ \uD^2 E( \bs W) \bs \calY^- = - \kappa \bs \calY^-, \mand J \circ \uD^2 E( \bs W) \bs \calY^+ = \kappa \bs \calY^+
}
and for  all $\bs h \in \E$, 
\EQ{
\ang{ \bs \al^- \mid J \circ \uD^2 E( \bs W) \bs h} = - \kappa\ang{ \bs \al^- \mid \bs h} , \quad \ang{ \bs \al^+ \mid J \circ \uD^2 E( \bs W) \bs h} = \kappa\ang{ \bs \al^+ \mid \bs h} . 
}
We view $\bs \al^{\pm} $ as linear forms on $\E$ and we note that $\ang{\bs \al^- \mid \bs \calY^-} = \ang{\bs \al^+ \mid \bs \calY^+} = 1$ and $\ang{\bs \al^- \mid \bs \calY^+} = \ang{\bs \al^+ \mid \bs \calY^-}  = 0$. For $\lam>0$ the rescaled versions of these objects are defined as, 
\EQ{
\bs\calY_{\lam}^-:= ( \frac{1}{\kappa} \calY_\lam, - \calY_{\U\lam}), \quad \bs \calY^+_\lam := ( \frac{1}{\kappa} \calY_{\lam}, \calY_{\U \lam} ) 
}
and, 
\EQ{ \label{eq:al-def} 
\bs \al^-_{\lam} := \frac{\kappa}{2\lam} J \bs \calY^+_\lam = \frac{1}{2}( \frac{\kappa}{\lam} \calY_{\U \lam} , - \calY_{\U\lam}), \quad \bs \al^+_\lam:= -\frac{\kappa}{2\lam} J \bs \calY^- = \frac{1}{2}( \frac{\kappa}{\lam} \calY_{\U\lam}, \calY_{\U\lam}) 
}
These choices of scalings ensure that $\ang{\bs \al^-_\lam \mid \bs \calY^-_\lam} = \ang{\bs \al^+_\lam \mid \bs \calY^+_\lam} = 1$. We have, 
\EQ{
J \circ \uD^2 E( \bs W_\lam) \bs \calY^-_\lam = - \frac{\kappa}{\lam} \bs \calY^-_\lam, \mand J \circ \uD^2 E( \bs W_\lam) \bs \calY^+_\lam = \frac{\kappa}{\lam} \bs \calY^+_\lam
}
and for  all $\bs h \in \E$, 
\EQ{ \label{eq:alpha-pair} 
\ang{ \bs \al^-_\lam \mid J \circ \uD^2 E( \bs W_\lam) \bs h} = - \frac{\kappa}{\lam}\ang{ \bs \al^-_\lam \mid \bs h} , \quad \ang{ \bs \al^+_\lam \mid J \circ \uD^2 E( \bs W_\lam) \bs h} = \frac{\kappa}{\lam}\ang{ \bs \al^+_\lam \mid \bs h} . 
}


 
We define a smooth non-negative function $\calZ \in C^{\infty}(0, \infty) \cap L^1((0, \infty), r^{D-1}\, \ud r)$ as follows. First if $D\ge 7$ we simply define 
 \EQ{ \label{eq:Z-def} 
 \calZ(r) :=  \Lam W(r)  \mif D \ge 7 
 }
and note that  
 \EQ{\label{eq:ZQ-7}
 \ang{ \calZ \mid \Lam W} >0   \mand \ang{\calZ \mid \calY} = 0,  
 }
 In fact the precise form of $\calZ$ is not so important, rather only the properties in~\eqref{eq:ZQ-7} and that it has sufficient decay and regularity. As $\Lam W \not \in \dot{H}^{-1}$ for $D \le 6$ we cannot take $\calZ = \Lam W$. Rather if $4 \le D \le 6$ we fix any $\calZ \in C^{\infty}_0(0, \infty)$ so that  
 \EQ{\label{eq:ZQ} 
  \ang{ \calZ \mid \Lam W} >0   \mand \ang{\calZ \mid \calY} = 0. 
 }
We record the following localized coercivity lemma from~\cite{JJ-AJM}.

\begin{lem}[Localized coercivity for $\LL$] \emph{\cite[Lemma 3.8]{JJ-AJM}}  \label{l:loc-coerce} 
Fix $D \ge 4$.  There exist uniform constants $c< 1/2, C>0$ with the following properties. Let $\bs g = (g, 0)  \in \E$. Then, 
\EQ{ \label{eq:L-coerce}
\ang{ \LL g \mid g} \ge c  \| \bs g \|_{\E}^2  - C\ang{ \calZ \mid g}^2  - C\ang{ \calY \mid g}^2
}
If $R>0$ is large enough then,  
\EQ{ \label{eq:L-loc-R} 
(1-2c)&\int_0^{R} (\p_r g(r))^2  \, r^{D-1} \ud r +  c \int_{R}^\infty (\p_r g(r))^2  \, r^{D-1} \ud r  - \int_0^\infty f'(W(r)) g(r)^2 \, r^{D-1}\, \ud r  \\
& \ge  - C\ang{ \calZ \mid g}^2 - C\ang{ \calY \mid g}^2. 
}
If $r>0$ is small enough, then
\EQ{ \label{eq:L-loc-r} 
(1-2c)&\int_r^{\infty} (\p_r g(r))^2  \, r \ud r +  c \int_{0}^r (\p_r g(r))^2  \, r \ud r  -  \int_0^\infty f'(W(r)) g(r)^2 \, r^{D-1}\, \ud r  \\
& \ge  - C\ang{ \calZ \mid g}^2 - C\ang{ \calY \mid g}^2. 
}
\end{lem}

As a consequence, (see for example~\cite[Proof of Lemma 3.9]{JJ-AJM} for an analogous argument in the case of two bubbles) we have the following coercivity property of $\LL_{\calW}$. 

\begin{lem} \label{lem:D2E-coerce}  Fix $D \ge 4$, $M \in \N$. There exist $\eta, c_0>0$ with the following properties. Consider the subset of $M$-bubble configurations $\bs \calW( \vec\iota, \vec \lam)$ for $\vec \iota \in \{-1, 1\}^M$, $\vec \lam \in (0, \infty)^M$ such that, 
\EQ{ \label{eq:lam-ratio} 
\sum_{j =1}^{M-1} \Big( \frac{\lam_j}{\lam_{j+1}} \Big)^{\frac{D-2}{2}} \le \eta^2. 
}
Let $\bs g \in \E$ be such that 
\EQ{
0 = \La \calZ_{\U{\lam_j}} \mid g\Ra  \mfor j = 1, \dots M. 
}
for $\vec \lam$ as in~\eqref{eq:lam-ratio}. Then, 
\EQ{ \label{eq:coercive} 
\frac{1}{2}\La \uD^2 E( \calW(  \vec \iota, \vec \lam)) \bs g \mid \bs g\Ra + 2 \sum_{j = 1}^M \big(\La \al_{\lam_j}^- \mid \bs g\Ra^2 + \La \al_{\lam_j}^+ \mid \bs g\Ra^2 \big) \ge c_0 \| \bs g \|_{\E}^2. 
} 
\end{lem}

\begin{lem}  \label{lem:M-bub-energy} Fix $D\ge4,  M \in \N$. 
For any $\te>0$, there exists $\eta>0$ with the following property. Consider the subset of $M$-bubble  configurations $\bs \calW(\iota, \vec \lam)$
such that 
\EQ{
\sum_{j =1}^{M-1} \Big( \frac{ \lam_{j}}{\lam_{j+1}} \Big)^{\frac{D-2}{2}} \le \eta.
}
Then, 
\EQ{
  \Big|  E( \bs\calW( \vec \iota, \vec \lam))  - M E( \bs W) +  \frac{(D(D-2))^{\frac{D}{2}}}{D} \sum_{j =1}^{M-1} \iota_j \iota_{j+1}  \Big( \frac{ \lam_{j}}{\lam_{j+1}} \Big)^{\frac{D-2}{2}}  \Big| \le \te \sum_{j =1}^{M-1} \Big( \frac{ \lam_{j}}{\lam_{j+1}} \Big)^{\frac{D-2}{2}} .
}
Moreover, there exists a uniform constant $C>0$ such that for any $\bs g = (g, 0) \in \E$, 
\EQ{
\abs{\ang{ D E_{\bfp}( \calW(m, \vec \iota, \vec \lam)) \mid g} } \le C \| (g, 0) \|_{\E} \sum_{j =1}^M \Big( \frac{\lam_{j}}{\lam_{j+1}} \Big)^{\frac{D-2}{2}} . 
}
\end{lem} 

\begin{proof} 
This is an explicit computation analogous to~\cite[Lemma 2.22]{JL6}.   
\end{proof}

%
%

The following modulation lemma plays an important role in our analysis. Before stating it, we define a proximity function to $M$-bubble configurations. 


 \begin{defn} \label{def-d} Fix $M$ as in Definition~\ref{def:multi-bubble} and let $\bs v \in \E$.  Define, 
\EQ{ \label{eq:d-def} 
\bfd( \bs v) := \inf_{\vec \iota, \vec \lam}  \Big( \| \bs v - \bs \calW( \vec \iota, \vec \lam) \|_{\E}^2 + \sum_{j =1}^{M-1} \Big( \frac{\lam_{j}}{\lam_{j+1}} \Big)^{\frac{D-2}{2}} \Big)^{\frac{1}{2}}.
}
where the infimum is taken over all vectors $\vec \lam = (\lam_1, \dots, \lam_M) \in (0, \infty)^M$ and all $\vec \iota = \{ \iota_1, \dots, \iota_M\} \in \{-1, 1\}^M$. 
\end{defn}


\begin{lem}[Static modulation lemma] \label{lem:mod-static} Let $D \ge 4$ and $M \in \N$. 
There exists $\eta, C>0$ with the following properties.  
Let $\te>0$,  and let $\bs v \in  \calE$   be such that 
\EQ{ \label{eq:v-M-bub} 
\bfd( \bs v)  \le \eta, \mand E( \bs v) \le ME( \bs Q) + \te^2, 
}
Then, there exists a unique choice of $\vec \lam = ( \lam_1, \dots, \lam_M) \in  (0, \infty)^M$, $\vec\iota \in \{-1, 1\}^M$, and $g \in  H$, such that setting $\bs g = (g, \dot v)$, we have 
\EQ{ \label{eq:v-decomp} 
   \bs v &=  \bs \calQ( m, \vec \iota, \vec \lam) + \bs g, \quad 
   0  = \La \calZ_{\U{\lam_j}} \mid g\Ra , \quad \forall j = 1, \dots, M, 
   }
   along with the estimates, 
\EQ{  \label{eq:g-bound-0}
\bfd( \bs v)^2 &\le \|  \bs g \|_{\E}^2  + \sum_{j =1}^{M-1} \Big( \frac{\lam_{j}}{\lam_{j+1}} \Big)^{\frac{D-2}{2}}  \le C \bfd( \bs v)^2,
}
Defining the unstable/stable components of $\bs g$  by, 
\EQ{
a_{j}^{\pm} := \La \bs \al_{\lam_j}^\pm \mid \bs g\Ra
}
we additionally have the estimates, 
\EQ{\label{eq:g-bound-A} 
   \|  \bs g \|_{\E}^2 + \sum_{j \not \in \calS}   \Big( \frac{ \lam_j}{ \lam_{j+1} }\Big)^{\frac{D-2}{2}} & \le C  \max_{ j \in \calS} \Big( \frac{ \lam_j}{ \lam_{j+1} }\Big)^{\frac{D-2}{2}} +  \max_{i \in \{1, \dots, M\}, \pm} | a_i^{\pm}|^2 + \te^2, 
}
where $\calS  := \{ j \in \{ 1, \dots, M-1\} \, : \, \iota_j  = \iota_{j+1} \}$. 
\end{lem}  

\begin{rem} 
Note that the scaling in the definition $\bs \al_{\lam_j}^{\pm}$ is chosen so that $| a_j^\pm| \lesssim \| \bs g\|_{\E}$, see~\eqref{eq:al-def}. 
\end{rem} 

\begin{rem} \label{rem:IFT}  We  use the following, less standard, version of the implicit function theorem in the proof of Lemma~\ref{lem:mod-static}. 

 \emph{
Let $X, Y, Z$ be Banach spaces,  $(x_0, y_0) \in X \times Y$,  and $\de_1, \de_2>0$. Consider a mapping  $G: B(x_0, \de_1) \times B(y_0, \de_2) \to Z$,  continuous in $x$ and $C^1$ in $y$. Assume $G(x_0, y_0) = 0$,  $(D_y G)(x_0, y_0)=: L_0$ has bounded inverse $L_0^{-1}$, and  
\EQ{ \label{eq:IFT-cond} 
&\| L_0 - D_y G(x, y) \|_{\LL(Y, Z)} \le \frac{1}{3 \| L_0^{-1} \|_{\LL(Z, Y)}}  \\ 
& \| G(x, y_0) \|_Z \le \frac{\de_2}{ 3 \| L_0^{-1} \|_{\LL(Z, Y)}},  
}
for all $\| x - x_0 \|_{X} \le \de_1$ and $\| y - y_0 \|_Y \le \de_2$. 
Then, there exists a continuous function $\si: B(x_0, \de_1)  \to B(y_0, \de_2)$ such that for all $x \in B(x_0, \de_1)$, $y = \si(x)$ is the unique solution of $G(x, \si(x)) = 0$ in $B(y_0, \de_2)$. 
}

This is proved in the same way as the usual implicit function theorem, see, e.g.,~\cite[Section 2.2]{ChowHale}. The essential point is that the bounds~\eqref{eq:IFT-cond} give uniform control on the size of the open set where the Banach contraction mapping theorem is applied. 
\end{rem} 

\begin{proof}[Proof of Lemma~\ref{lem:mod-static}] The proof is a standard argument and is very similar to~\cite[Proof of Lemma 3.1]{JL1} and~\cite[Proof of Lemma 2.24]{JL6}. We refer to those papers for details and here only sketch the distinction in the estimate~\eqref{eq:g-bound-A} where the stable/unstable directions enter. 
To prove the estimate~\eqref{eq:g-bound-A} we expand the nonlinear energy of $\bs v$, 
\EQ{
M E( \bs Q) &+ \te^2 \ge E( \bs v) = E( \bs \calQ( m, \vec \iota, \vec \lam) + \bs g) \\
& = E( \bs \calQ( m, \vec \iota, \vec \lam) ) + \ang{ \uD E( \bs \calQ(m, \vec \iota, \vec \lam)) \mid \bs g}  + \frac{1}{2} \ang{ \uD^2 E( \bs \calQ(m, \vec \iota, \vec \lam))\bs g  \mid \bs g} + O(\| \bs g \|_{\E}^3) 
}
and apply the conclusions of Lemma~\ref{lem:D2E-coerce}, in particular the estimate~\eqref{eq:coercive} and Lemma~\ref{lem:M-bub-energy}. This completes the proof. 
\end{proof} 

\begin{lem}  \label{lem:bub-config} Let $D \ge 4$. 
 There exists $\eta>0$ sufficiently small with the following property. Let  $M, L \in \N$,  $\vec\iota \in \{-1, 1\}^M, \vec \sigma \in \{-1, 1\}^L$, $\vec \lam \in (0, \infty)^M, \vec \mu \in (0, \infty)^L$,   and $\bs w = (w, 0)$ be such that $\| \bs w \|_{\E}  < \infty$ and, 
 \begin{align} 
 \|\bs w  - \bs \calW(  \vec \iota, \vec \lam)\|_{\E}^2  + \sum_{j =1}^{M-1} \Big(\frac{\lam_j}{\lam_{j+1}} \Big)^{k} &\le \eta,  \label{eq:M-bub} \\
 \|\bs w  - \bs \calW( \vec \sigma , \vec \mu)\|_{\E}^2 +  \sum_{j =1}^{L-1} \Big(\frac{\mu_j}{\mu_{j+1}} \Big)^{k} &\le \eta. \label{eq:L-bub} 
 \end{align} 
 Then, $M = L$, $\vec \iota = \vec \sigma$. Moreover, for every $\te>0$ the number $\eta>0$ above can be chosen small enough so that 
 \EQ{ \label{eq:lam-mu-close} 
\max_{j = 1, \dots M} | \frac{\lam_j}{\mu_j} - 1 | \le  \te.
 }
\end{lem}

\begin{proof}[Proof of Lemma~\ref{lem:bub-config}] 


Let $g_{\lam} := w - \calW(  \vec \iota, \vec \lam)$ and $g_\mu := w - \calW(  \vec \s, \vec \mu)$. By expanding the nonlinear potential energy we have, 
\EQ{
E_{\bfp}(w) = E_{\bfp}( \calW(\vec \iota, \vec \lam) ) + \La DE_{\bfp}(\calW(\vec \iota, \vec \lam)) \mid g_\lam\Ra + O(\| (g_\lam, 0) \|_{\E}^2) . 
}
Choosing $\eta>0$ small enough so that Lemma~\ref{lem:M-bub-energy} applies, we see that 
\EQ{
M E( \bs Q) - C \eta \le E_{\bfp}( w) \le M E( \bs Q) +C \eta,
}
for some $C>0$. By an identical argument, 
\EQ{
L E( \bs Q) - C \eta \le E_{\bfp}( w) \le L E( \bs Q) +C \eta.
}
It follows that $M=L$. Next, we prove that $\eta>0$ can be chosen small enough to ensure that $\vec \iota = \vec \sigma$. Suppose not, then we can find a sequence $\bs w_n = (w_n, 0)$ with $\| \bs w_n\|_{\E} \le C$, and sequences $\vec \iota_n, \vec \s_n, \vec \lam_n, \vec \mu_n$ so that, 
\EQ{
\| \bs w_n  - \bs \calW(  \vec \iota_n, \vec \lam_n)\|_{\E}^2  + \sum_{j =1}^{M-1} \Big(\frac{\lam_{n,j}}{\lam_{n, j+1}} \Big)^{k} & = o_n(1)  \mas n \to \infty,\\
 \|\bs w_n  - \bs \calW( \vec \sigma_n , \vec \mu_n)\|_{\E}^2 +  \sum_{j =1}^{M-1} \Big(\frac{\mu_{n,j}}{\mu_{n, j+1}} \Big)^{k} &= o_n(1) \mas n \to \infty,
}
but with $\vec \iota_n \neq \vec  \s_n$ for every $n$. 
 Passing to a subsequence we may assume that there exists an index $j_0  \in \{ 1, \dots, M\}$ such that $ \iota_{j, n} = \sigma_{j, n}$ for every $j > j_0$ and every $n$ and $\iota_{j_0, n} \neq \s_{j_0, n}$ for every $n$. We first observe that, 
\EQ{ \label{eq:diff-sign} 
\| \bs\calW(  \vec \iota_n, \vec \lam_n) - \bs\calW( \vec \sigma_n , \vec \mu_n) \|_{\E} \le \|\bs w_n  - \bs \calW(  \vec \iota_n, \vec \lam_n)\|_{\E} +   \|\bs w_n  - \bs \calW( \vec \sigma_n , \vec \mu_n)\|_{\E}  = o_n(1) .
}
We first show that $j_0<M$. 
Assume for contradiction that $j_0 = M$. Then, we may assume that $\iota_{n, M} = 1$,  $\sigma_{n, M} = -1$ and $\lam_{n, M} > \mu_{n, M}$ for all $n$. It follows that there exists a constant $c>0$ for which 
\EQ{
 \, \, \Big|\calW(  \vec \iota_n, \vec \lam_n) - \calW( \vec \sigma_n , \vec \mu_n) \Big| \ge \frac{c}{\lam_{n, M}^{\frac{D-2}{2}}}  \quad \forall r \in [\lam_{n, M}, 2 \lam_{n, M}],
}
for all $n$ large enough. 
But then, 
\EQ{
\| \bs \calW(  \vec \iota_n, \vec \lam_n) - \bs  \calW( \vec \sigma_n , \vec \mu_n) \|_{\E}^2 \ge  \int_{\lam_{n, M}}^{2 \lam_{n, M}} \frac{c^2}{\lam_{n, M}^{D-2}}r^{D-2}  \, \frac{\ud r}{r}  \ge \frac{c^2}{D-2}( 2^{D-2} -1),
}
for all sufficiently large $n$, which contradicts~\eqref{eq:diff-sign}. So $\iota_{1, n} = \sigma_{n, 1}$ for all $n$. Thus $j_0<M$. By a nearly identical argument we can show that we must have $|\lam_{n, j} / \mu_{n, j}- 1| = o_n(1)$ for all $j > j_0$.  Next, again we may assume (after passing to a subsequence) that  $\lam_{n, j_0} > \mu_{n, j_0}$.  It follows again that for all sufficiently large $n$ we have, 
\EQ{
\abs{\calW(\vec \iota_n, \vec \lam_n) - \calW(\vec \sigma_n , \vec \mu_n) } \ge \frac{c}{\lam_{n, j_0}^{\frac{D-2}{2}}}  \quad \forall r \in [\lam_{n, j_0}, 2 \lam_{n, j_0}],
}
which again yields a contradiction. Hence we must have $\vec \iota = \vec \sigma$. 

Finally, we prove~\eqref{eq:lam-mu-close}. Suppose ~\eqref{eq:lam-mu-close} fails. Then there exists $\te_0>0$ and sequences $\vec \lam_{n}, \vec \mu_n$ such that 
\EQ{
\| \calW(  \vec \iota_n, \vec \lam_n) - \calW( \vec \iota_n , \vec \mu_n) \|_H = o_n(1),
} 
but 
\EQ{ \label{eq:diff-scale} 
\sup_{j =1, \dots, M} | \lam_{n, j}/ \mu_{n, j} - 1|\ge \te_0, 
}
 for all $n$. We arrive at a contradiction following  the same logic as before. 
\end{proof}

We require the following lemma, which gives the nonlinear interaction force between bubbles.  Given an $M$-bubble configuration, $\calW(\vec \iota, \vec \lam)$, we set 
\EQ{ \label{eq:fi-def} 
f_{\bfi}( \vec \iota, \vec \lam) :=  f( \calW( \vec \iota, \vec \lam)) - \sum_{j=1}^M \iota_j f( W_{\lam_j})  
}

\begin{lem} \label{lem:interaction} Let $D \ge 4$, $M \in\N$. For any $\theta>0$ there exists $\eta>0$ with the following property. Let $\bs \calW( \vec \iota, \vec \lam)$ be an $M$-bubble  configuration with 
\EQ{
 \sum_{j =0}^{M} \Big( \frac{ \lam_{j}}{\lam_{j+1}} \Big)^{\frac{D-2}{2}} \le \eta, 
 }
under the convention that $\lam_0 = 0$, $\lam_{M+1} = \infty$. Then, 
we have, 
\begin{multline} 
\Big|  \ang{ \Lam W_{\lam_j} \mid f_{\bfi}( \vec \iota, \vec \lam)}  - \iota_{j-1}\frac{D-2}{2D} (D(D-2))^{\frac{D}{2}} \Big( \frac{ \lam_{j-1}}{\lam_j} \Big)^{\frac{D-2}{2}}  + \iota_{j+1} \frac{D-2}{2D} (D(D-2))^{\frac{D}{2}}\Big( \frac{ \lam_{j}}{\lam_{j+1}} \Big)^{\frac{D-2}{2}}  \Big| \\
\le \theta \Big( \Big( \frac{ \lam_{j-1}}{\lam_j} \Big)^{\frac{D-2}{2}}  +   \Big( \frac{ \lam_{j}}{\lam_{j+1}} \Big)^{\frac{D-2}{2}}\Big) 
\end{multline} 
where here $f_{\bfi}( \vec \iota, \vec \lam) $ is defined in~\eqref{eq:fi-def}.
\end{lem} 
\begin{proof}
This is an explicit computation analogous to the one in~\cite[Lemma 2.27]{JL6}. 
\end{proof}


\section{Localized sequential bubbling} \label{sec:compact} 

The goal of this section is to prove a localized sequential bubbling lemma for sequences of solutions to~\eqref{eq:nlw}  with vanishing averaged kinetic energy on a (relatively) expanding region of space. The main result, and the arguments used to prove it are in the spirit of the main theorems of Duyckaerts, Kenig, and Merle in~\cite{DKM3}. To prove the compactness lemma in all space dimensions via a unified approach, we use the virial inequalities of Jia and Kenig to obtain vanishing of the error instead of the channels-of-energy type arguments from~\cite{DKM3, CKLS3, Rod16-adv}, which in those works was limited to either odd space dimensions or the subset of even space dimensions that satisfy $D = 0\mod 4$.

To state the compactness lemma, we define a localized distance function, 
\EQ{ \label{eq:delta-def} 
\bs \de_R( \bs u) :=  \inf_{ M,  \vec \iota, \vec \lam}  \Big( \| (u - \calW( \vec \iota, \vec \lam), \dot u) \|_{\E( r \le R)}^2  + \sum_{j = 1}^{M} \Big(\frac{ \lam_j}{ \lam_{j+1}}\Big)^{\frac{D-2}{2}} \Big)^{\frac{1}{2}}. 
}
where  the infimum above is taken over all $M \in\{0, 1, 2, \dots\}$,  and all  vectors $\vec \iota\in \{-1, 1\}^M, \vec \lam \in (0, \infty)^M$, and here we use the convention that the last scale $\lam_{M+1} = R$.

\begin{lem}[Compactness Lemma] \label{lem:compact}  Let $\rho_n>0$ be a sequence of positive numbers and let $\bs u_n(t) \in\E$ be a sequence of solutions to~\eqref{eq:nlw}  on the time intervals $[0, \rho_n]$ such that 
\EQ{\label{eq:II-compact} 
\limsup_{n \to \infty} \sup_{t \in [0, \rho_0]}\|\bs u_n(t)\|_{\E}< \infty.
} 
Suppose there exists a sequence $R_n \to \infty$  such that,
\EQ{
\lim_{n \to \infty} \frac{1}{ \rho_n} \int_0^{\rho_n} \int_0^{ \rho_n R_n} \abs{\p_t u_n(t, r)}^2 \, r^{D-1} \, \ud r\,  \ud t
 = 0.
}
Then, up to passing to a subsequence of the $\bs u_n$,  there exists a time sequence $t_n \in [0, \rho_n]$ and a sequence $r_n \le R_n$ with $r_n \to \infty$ such that 
\EQ{
\lim_{n \to \infty} \bs \de_{r_n\rho_n}( \bs u_n(t_n))  = 0.
}

\end{lem} 

\begin{rem}\label{rem:seq} 
In fact, the prove provides a sequence $t_n \in [0, \rho_n]$, $r_n \le R_n$ with $r_n \to \infty$, a non-negative  integer $M$ independent of $n$, scales $\vec\lam_n \in (0, \infty)^M$,  and a vector of signs $\vec \iota \in \{-1, 1\}^M$ (also independent of $n$), such that 
\EQ{
\lim_{n \to \infty} \Big( \| (\bs u(t_n) - \bs \calW( \vec \iota, \vec \lam_n) \|_{\E( r \le r_n \rho_n)}^2  + \sum_{j = 1}^{M} \Big(\frac{ \lam_{n, j}}{ \lam_{n, j+1}}\Big)^{\frac{D-2}{2}} \Big)^{\frac{1}{2}} = 0. 
}
\end{rem} 


\subsection{Technical lemmas} 

The proof of Lemma~\ref{lem:compact} requires two Real Analysis results, which we address first.  
\begin{lem}
\label{lem:sequences}
If $a_{k, n}$ are positive numbers such that $\lim_{n\to \infty}a_{k, n} = \infty$ for all $k \in \bN$,
then there exists a sequence of positive numbers $b_n$ such that $\lim_{n\to \infty} b_n = \infty$
and $\lim_{n\to \infty} a_{k, n} / b_n = \infty$ for all $k \in \bN$.
\end{lem}
\begin{proof}
For each $k$ and each $n$ define $\ti a_{k, n} = \min\{ a_{1, n}, \dots, a_{k, n}\}$. Then the sequences $\ti a_{k, n} \to \infty$ as $n \to \infty$ for each $k$, but also satisfy $\ti a_{k, n} \le a_{k, n}$ for each $k, n$, as well as $\ti a_{j, n} \le \ti a_{k, n}$ if $j>k$. Next, choose a strictly increasing sequence $\{n_k \}_k \subset \N$ such that $\ti a_{k, n} \ge k^2$ as long as $n \ge n_k$. For $n$ large enough, let $b_n \in \bN$ be determined by the condition
$n_{b_n} \leq n < n_{b_n + 1}$. Observe that $b_n \to \infty$ as $n \to \infty$. Now fix any $\ell \in \N$ and let $n$ be such that $b_n > \ell$. We then have
\begin{equation}
a_{\ell, n} \geq \ti a_{\ell, n} \ge \ti a_{b_n, n} \ge  b_n^2  \gg b_n.
\end{equation}
Thus the sequence $b_n$ has the desired properties. 
\end{proof}
If $f: [0, 1] \to [0, +\infty]$ is a measurable function, we denote by 
\EQ{
Mf(\tau) := \sup_{I \ni \tau; I \subset [0, 1]} \frac{1}{\abs{I}} \int_I f(t) \, \ud t 
}
its Hardy-Littlewood maximal function. Recall the weak-$L^1$ boundedness estimate
\begin{equation}
\label{eq:weak-max}
|\{\tau \in [0, 1]: Mf(\tau) > \alpha\}| \leq \frac{3}{\alpha}\int_0^1f(t)\ud t, \qquad\text{for all }\alpha > 0,
\end{equation}
see \cite[Section 2.3]{Muscalu-Schlag}.

\begin{lem}
\label{lem:maximal}
Let $f_n$ be a sequence of continuous positive functions defined on $[0, 1]$ such that $\lim_{n \to \infty}\int_0^1 f_n(t)\ud t = 0$ and let $g_n$ be a uniformly bounded sequence of real-valued continuous functions on $[0, 1]$ such that
$\limsup_{n\to\infty} \int_0^1 g_n(t)\ud t \leq 0$.
Then there exists a sequence $t_n \in [0, 1]$ such that
\begin{equation}
\lim_{n\to\infty} Mf_n(t_n) = 0, \qquad \limsup_{n\to \infty}g_n(t_n) \leq 0.
\end{equation}
\end{lem}
\begin{proof}
Let $\alpha_n$ be a sequence such that $\int_0^1 f_n(t)\ud t \ll \alpha_n \ll 1$. Let $A_n :=  \{t \in [0, 1]: Mf_n(t) \le \alpha_n\}$. 
By \eqref{eq:weak-max}, $\lim_{n\to\infty}|A_n| = 1$. Since $g_n$ is uniformly bounded, we have
\begin{equation}
\int_{[0, 1]\setminus A_n}|g_n(t)|\ud t \lesssim |[0, 1] \setminus A_n| \to 0,
\end{equation}
which implies
\begin{equation}
\limsup_{n\to\infty}\int_{A_n}g_n(t)\ud t \leq 0.
\end{equation}
It suffices to take $t_n \in A_n$ such that $g_n(t_n) \leq |A_n|^{-1}\int_{A_n}g_n(t)\ud t$.
\end{proof}

\subsection{Proof of the compactness lemma} 

\begin{proof}[Proof of Lemma~\ref{lem:compact}] 
Rescaling we may assume that $\rho_n = 1$ for each $n$.

\textbf{Step 1.}
We claim that there exist $\sigma_n \in [0, \frac 13]$, $\tau_n \in [\frac 23, 1]$ , and a sequence $R_{1,n} \le R_n$ with $R_{1, n} \to \infty$ as $n \to \infty$  such that
\EQ{
\label{eq:jia-kenig-0}
\lim_{n\to \infty}\int_{\sigma_n}^{\tau_n} \Bigg[\int_0^\infty\Big( \p_r^2 u_n  +& \frac{D-1}{r} \p_r u_n - \abs{u_n}^{\frac{4}{D-2}} u_n \Big)\Big( r \p_r u_n + \frac{D}{2} u_n\Big) \chi_{R_{1, n}} \, r^{D-1} \, \ud r  \\
& \qquad \qquad+ \int_0^\infty \p_t u_n \Big( r \p_r \p_t u_n +  \frac{D}{2} \p_t u_n \Big) \chi_{R_{1, n}} \, r^{D-1} \, \ud r \Bigg] \ud t = 0,
}
where $\chi$ is a smooth cut-off function equal $1$ on $[0, \frac 12]$, with support in $[0, 1]$. Here and later in the argument the second term in the integrand in~\eqref{eq:jia-kenig-0} is to be interpreted as the expression obtained after integration by parts, which is well defined due to the finiteness of the energy. 

Since 
\EQ{
\lim_{n\to \infty}\int_0^\frac 13 \int_0^{R_n}(\partial_t u_n)^2\,r\vd r = 0 \mand \lim_{n\to \infty}\int_\frac 23^1 \int_0^{R_n}(\partial_t u_n)^2\,r\vd r = 0
}
 there exist $\sigma_n \in [0, \frac 13]$,  $\tau_n \in [\frac 23, 1]$ and a sequence $R_{1, n} \to \infty$ such that, 
\EQ{ \label{eq:sigma-tau} 
\lim_{n\to\infty}R_{1, n}  \int_0^{R_n}(\partial_t u(\sigma_n))^2\,r\vd r = 0 \mand 
\lim_{n\to\infty}  R_{1, n} \int_0^{R_n}(\partial_t u(\tau_n))^2\,r\vd r = 0
}
For $t \in [\sigma_n, \tau_n]$, we have the following Jia-Kenig virial identity; see~\cite[Lemma 2.2 and Lemma 2.6]{JK}.
\EQ{\label{eq:jia-kenig-1}
\frac{\ud}{\ud t} \La \p_t u_n & \mid (r \p_r u_n + \frac{D}{2} u_n) \chi_{R_{1, n}}\Ra = \int_0^\infty \p_t u_n \Big( r \p_r \p_t u_n +  \frac{D}{2} \p_t u_n \Big) \chi_{R_{1, n}} \, r^{D-1} \, \ud r  \\
& + \int_0^\infty\Big( \p_r^2 u_n  + \frac{D-1}{r} \p_r u_n - \abs{u_n}^{\frac{4}{D-2}} u_n \Big)\Big( r \p_r u_n + \frac{D}{2} u_n\Big) \chi_{R_{1, n}} \, r^{D-1} \, \ud r 
}
By the Cauchy-Schwarz inequality, the assumption~\eqref{eq:II-compact}  and~\eqref{eq:sigma-tau}, we see that 
\begin{equation}
\lim_{n\to \infty}\int_0^\infty \big(|\partial_t u_n(\sigma_n)|| r \p_r u_n(\sigma_n)+ \frac{D}{2}u_n(\sigma_n)|
+ |\partial_t u_n(\tau_n)|| r \p_r u_n(\tau_n) + \frac{D}{2}u(\tau_n))|\big)\chi(\cdot/R_{1, n})\,r\vd r = 0.
\end{equation}
Integrating~\eqref{eq:jia-kenig-1} between $\sigma_n$ and $\tau_n$, and using the above, we obtain \eqref{eq:jia-kenig-0}.

\textbf{Step 2.}
We rescale again so that $[\sigma_n, \tau_n]$ becomes $[0, 1]$. We apply 
Lemma~\ref{lem:maximal}, to 
\EQ{
f_n(t)&:= \int_0^{R_n}\abs{\p_t u_n(t, r)}^2 \, r^{D-1} \, \ud r, \\
g_n(t) &:= - \int_0^\infty\Big( \p_r^2 u_n(t)  + \frac{D-1}{r} \p_r u_n(t) - \abs{u_n(t)}^{\frac{4}{D-2}} u_n(t) \Big)\Big( r \p_r u_n(t) + \frac{D}{2} u_n(t)\Big) \chi_{R_{1, n}} \, r^{D-1} \, \ud r \\
&\quad - \int_0^\infty \p_t u_n(t) \Big( r \p_r \p_t u_n(t) +  \frac{D}{2} \p_t u_n(t) \Big) \chi_{R_{1, n}} \, r^{D-1} \, \ud r 
}
(integrating by parts we see that $g_n$ is a uniformly bounded sequence of continuous functions)
and we find a sequence $\{t_n\} \in [0, 1]$ such that we have vanishing of the maximal function of the local kinetic energy,
\EQ{ \label{eq:maximal} 
&\lim_{ n \to \infty} \sup_{I \ni t_n;  I \subset [0, 1]}  \frac{1}{\abs{I}} \int_I \int_0^{R_n} | \p_t u_n(t, r)|^2 \, r \, \ud r \, \ud t = 0,  
\\
\textrm{and} \, \,\,  & \lim_{n \to \infty} \int_0^{R_n} | \p_t u_n(t_n, r)|^2 \, r \, \ud r \, \ud t = 0, 
}
and also pointwise vanishing of a localized Jia-Kenig virial functional, 
\EQ{
\label{eq:jia-kenig}
\limsup_{n\to \infty}\Bigg( - \int_0^\infty\Bigg[\Big( \p_r^2 u_n(t_n)  &+ \frac{D-1}{r} \p_r u_n(t_n) - \abs{u_n(t)}^{\frac{4}{D-2}} u_n(t_n) \Big)\Big( r \p_r u_n(t_n) + \frac{D}{2} u_n(t_n)\Big) \\
&\quad + \p_t u_n(t_n) \Big( r \p_r \p_t u_n(t_n) +  \frac{D}{2} \p_t u_n(t_n) \Big)\Bigg] \chi_{\ti R_{ n}} \, r^{D-1} \, \ud r  \Bigg) \leq 0. 
}
for any sequence $\ti R_{n} \le R_{1, n} \le R_n$ with $\ti R_n \to \infty$ as $n \to \infty$. 
We emphasize the  conclusion from the first steps is the existence of the sequence $t_n$ such that~\eqref{eq:maximal} and~\eqref{eq:jia-kenig} hold.   

\textbf{Step 3.} 
Now that we have chosen the sequence $t_n \in[0, 1]$, we may, after passing to a subsequence, assume that $t_n \to t_0 \in [0, 1]$. 


We apply Lemma~\ref{lem:pd} to the sequence $\bs u_n(t_n)$, obtaining profiles $(\bs v^j\lin, t_{n, j}, \lambda_{n, j})$, and $\bs w_{n, 0}^J$,  so that,  using the notation, 
\EQ{
\bs  v_{\Lin, n}^j(0) :=  \big( \lambda_{n, j}^{-\frac{D-2}{2}} v_{\Lin}^i (  \frac{-t_{n, i}}{\lam_{n, j}}, \frac{ \cdot}{\lam_{n, j}} ) , \lambda_{n, j}^{-\frac{D}{2}} \p_t v_{\Lin}^j (  \frac{-t_{n, j}}{\s_{n, j}}, \frac{ \cdot}{\s_{n, j}} )\big), 
}
 we have
\EQ{ \label{eq:profiles1} 
\bs u_n(t_n) &=  \sum_{j =1}^J \bs  v_{\Lin, n}^j(0) + \bs w_{n, 0}^J
}
satisfying the conclusions of Lemma~\ref{lem:pd}. 
We refer to the profiles 
$(\bs v\lin^j(0), t_{n, j}, \s_{n, j})$ with $t_{n, j} = 0$ for all $n$ as \emph{centered} profiles (here the subscript $\lin$ on $\bs v\lin^j$ is superfluous).  We refer to  the profiles $(\bs v\lin^j(0), t_{n, j}, \s_{n, j})$ with $-t_{n, j}/ \s_{n, j} \to \pm \infty$ as \emph{outgoing/incoming} profiles. 

\textbf{Step 4.}(Centered profiles at large scales) 
At each step, we will impose conditions on the choice of the ultimate choice of sequence $r_n \to \infty$. We divide the indices associated to centered profiles into two sets, 
\EQ{
\calJ_{c, 0}&:= 
 \{ j \in \N \mid t_{n, j} = 0 \, \, \forall n, \mand \lim_{n \to \infty}\lambda_{n, j} < \infty \},\\
\calJ_{c, \infty}&:= 
 \{ j \in \N \mid t_{n, j} = 0 \, \, \forall n, \mand \lim_{n \to \infty}\lambda_{n, j} = \infty \}.
}
Using Lemma~\ref{lem:sequences} we choose a sequence $ r_{ 0, n} \to \infty$ so that $r_{ 0, n} \ll R_n, \lam_{n, j}$ for each $\lam_{n, j}$ with $j \in \calJ_{c, \infty}$.
By construction we have, 
\EQ{\label{eq:Jc-infty}
\lim_{n \to \infty} \| (\lambda_{n, i}^{-\frac{D-2}{2}}v\lin^j( 0, \cdot/\lambda_{n, j}), \lambda_{n, i}^{-\frac{D}{2}} \dot v^j \lin (\cdot/ \lambda_{n, j})\|_{\E(0, r_{0, n})} = 0
}
for any of the indices $j\in \calJ_{c, \infty}$. 

%
%


\textbf{Step 5.}(Incoming/outgoing profiles  with $\lim_{n \to \infty}\abs{t_{n, j}} = \infty$) We next treat profiles $(\bs v^j\lin, t_{n, j}, \lambda_{n,i})$ that satisfy, 
\EQ{
-\frac{t_{n, j}}{\s_{n, j}} \to \pm \infty. 
}
Up to passing to a subsequence of $\bs u_n(t_n)$ we may assume that $-t_{n, j}  \to t_{\infty} \in [- \infty, \infty]$. Consider again two sets of indices, 
\EQ{
\calJ_{\Lin, 0} &:= \{ j \in \N \mid -\frac{t_{n, j}}{\s_{n, j}} \to \pm \infty \mand \abs{ t_{n, j}} \to t_{\infty, j} \in \R \}, \\ 
\calJ_{\Lin, \infty} &:= \{ j \in \N \mid -\frac{t_{n, j}}{\s_{n, j}} \to \pm \infty \mand \abs{ t_{n, j}} \to \infty \}.
}
We impose additional restrictions on the sequence $r_n$. We require that $r_n \le \frac{1}{2} \abs{t_{n, j}}$ for each sequence $t_{n, j}$ in $\calJ_{\Lin, \infty}$. So at this stage, we again use Lemma~\ref{lem:sequences} to choose a sequence $r_{1, n} \to \infty$ such that $r_{1, n} \le r_{0, n}$ and $r_{1, n} \le \frac{1}{2} \abs{t_{n, j}}$ for each sequence $t_{n, j}$ in $\calJ_{\Lin, \infty}$. 

Since $\bs v^j\lin$ is a free wave we know that it asymptotically concentrates all of its energy near the light-cone. In fact, 
\EQ{
\lim_{s \to \pm \infty}  \| \bs v^i\lin(s) \|_{\E(r \le \frac{1}{2} \abs{s})} = 0 .
} 
which is proved in~\cite[Lemma 4.1]{DKM1} in odd space dimensions and is a direct consequence of~\cite[Theorem 5] {CKS} in even space dimensions. 

Thus, if $j\in \calJ_{\Lin, \infty}$ and as long as $ r_{1, n} \le \frac{1}{2} \abs{t_{n, j}}$ for $n$ large enough,  we see that $\lambda_{n, j}^{-1} r_{1, n} \le \frac{1}{2} \lambda_{n,j}^{-1} \abs{t_{n, j}}$ and thus 
\EQ{\label{eq:Jl-infty} 
\| \bs v^j_{\Lin}(-t_{n, j}/ \lambda_{n, j}) \|_{\E( r \le r_{1, n} \s_{n, j}^{-1})} \to 0 \mas n \to \infty. 
}
by the above. 

\textbf{Step 6.}(Incoming/outgoing profiles with $\lim_{n \to \infty}\abs{t_{n, j}} < \infty$) Next, we consider profiles $(\bs v^j\lin, t_{n, j}, \lambda_{n,j})$ such that 
\EQ{
-\frac{t_{n, j}}{\lambda_{n, i}} \to \pm \infty \mand  -t_{n, j} \to t_{\infty, j} \in \R
}
that is, those in  $\calJ_{\Lin, 0}$. We  note that $\lambda_{n, i} \to 0$ as $n \to \infty$ for each $i \in \calJ_{\Lin, 0}$. We claim that any such profile must satisfy $\bs v^i\lin  \equiv 0$. We use the argument given in~\cite[Erratum]{DKM3erratum}, modulo a few technicalities which reduce our situation to the one considered there. 


We claim that there exists a new sequence $\sqrt{r_{1, n}} \leq  r_{2, n} \leq r_{1, n}$ such that
\begin{equation}\label{eq:r2} 
\lim_{n\to\infty}\sup_{t \in [0, 1]}\| \bs u_n(t)\|_{\E( A_n^{-1}r_{2, n},  A_n  r_{2, n})} = 0
\end{equation}
for some $1 \ll A_n \ll  r_{2, n}$. 
By Lemma~\ref{lem:prop-small-E-local} it suffices to have
\begin{equation}
\lim_{n\to\infty} \| \bs u_n(0)\|_{\E( A_n^{-1} r_{2, n},   A_n  r_{2, n})} = 0, 
\end{equation}
and then replace $A_n$ by its quarter, for example.

Let $A_n$ be the largest integer such that $A_n^{2A_n} \leq \sqrt{r_{1,n}}$. Obviously, $1 \ll A_n \ll \sqrt{r_{1,n}}$.
For $l \in \{0, 1, \ldots, A_n - 1\}$, set $R_n^{(l)} := A_n^{2l}\sqrt{r_{1, n}}$, so that $A_n^{-1}R_n^{(l+1)} = A_nR_n^{(l)}$, thus
\begin{equation}
\sum_{l=0}^{A_n - 1}\| \bs u_n(0)\|_{\E( A_n^{-1}R_n^{(l)}, A_nR_n^{(l)})} \leq \|\bs u_n(0)\|_{\E}.
\end{equation}
Since all the terms of the sum are positive, there exists $l_0 \in \{0, 1, \ldots, A_n - 1\}$ such that $r_{2, n} := R_n^{(l_0)}$ satisfies
\begin{equation}
\|\bs u_n(0)\|_{\E( A_n^{-1}r_{2, n},  A_n r_{2, n})} \leq A_n^{-1}\| \bs u_n(0)\|_{\E} \to 0.
\end{equation}
proving~\eqref{eq:r2} 

 We now pass to a new sequence of maps $\bs{\ti u}_n$ with vanishing average kinetic energy on the whole space. 
Indeed, define 
\EQ{
\bs{\ti u}_n(0) = \chi_{2r_{2, n}} \bs u_n(0)
}
Denoting by $\bs{\ti u}_n(t)$ the solutions to~\eqref{eq:nlw} with data $\bs{\ti u}_n(0)$ on the interval $t \in [0, 1]$ we can use the finite speed of propagation, the local Cauchy theory, and ~\eqref{eq:r2} to deduce that 
\EQ{ \label{eq:tiu-fsp} 
\bs{\ti u}_n(t, r) = \bs u_n(t, r) \mif r \le r_{2, n}, \mand  \limsup_{n \to \infty} \sup_{t \in [0, 1]} \| \bs {\ti u}_n(t) \|_{\E( r_{2, n}, \infty)} = 0. 
}
From~\eqref{eq:profiles1} we have, 
\EQ{
\bs{\ti u}_n(t_n) &=  \sum_{ 1 \le j \le J; j \in \calJ_{c, 0} \cup \calJ_{\Lin, 0}} \bs v_{\Lin, n}^j(0)    + \chi_{2 r_{2, n}}\bs{ w}_{n, 0}^J \\
&\quad + \sum_{ 1 \le j \le J; j \in \calJ_{c, 0} \cup \calJ_{\Lin, 0}} (\chi_{2 r_{2, n}} -1) \bs v_{\Lin, n}^j(0) + \sum_{ 1 \le j \le J; j \in \calJ_{c, \infty} \cup \calJ_{\Lin, \infty}} \chi_{2 r_{2, n}}  \bs v_{\Lin, n}^j(0) 
}
We claim that in fact $\bs{\ti u}_n(t_n)$ admits a profile decomposition in the sense of Lemma~\ref{lem:pd} of the form, 
\EQ{ \label{eq:profiles2}
\bs{\ti u}_n(t_n) = \sum_{ 1 \le j \le J; j \in \calJ_{c, 0} \cup \calJ_{\Lin, 0}} \bs v_{\Lin, n}^j(0)    + \bs{\ti w}_{n, 0}^J 
}
with the same profiles $(\bs v_{\Lin}^j, \lambda_{n, j}, t_{n, j})$ as in the decomposition for $\bs u_n(t_n)$ and where the error above $\bs{\ti w}_{n, 0}^J $ satisfies, 
\EQ{
\bs{\ti w}_{n, 0}^J  = \chi_{2 r_{2, n}} \bs w_{n, 0}^J + o_n(1) \mas n \to \infty
}
The expansion~\eqref{eq:profiles2} and the $o_n(1)$ above is justified as follows: let $\eps>0$ and use~\eqref{eq:pyth} to find $J_0 >0$ such that 
\EQ{
\sum_{j >J_0} \| \bs v_{\Lin, n}^j(0) \|_{\E}  \le \eps
}
Using~\eqref{eq:Jc-infty}, ~\eqref{eq:Jl-infty} we see that, 
\EQ{
\sum_{j \le J_0, j \in \calJ_{c, \infty} \cup \calJ_{\Lin, \infty}} \| \chi_{2r_{2, n}} \bs v_{\Lin, n}^j(0)  \|_{\E}^2 \to 0 \mas n \to \infty. 
}
Using the same logic used to deduce~\eqref{eq:Jc-infty}, ~\eqref{eq:Jl-infty} we have (since $r_{2, n} \to \infty$), 
\EQ{
 \sum_{ 1 \le j \le J_0; j \in \calJ_{c, 0} \cup \calJ_{\Lin, 0}}\| (1- \chi_{2 r_{2, n}}) \bs v_{\Lin, n}^j(0) \|_{\E}^2 \to 0 \mas n \to \infty.
}  
from which the vanishing of the $o_n(1)$ term follows. It remains to deduce the vanishing properties of the error $\bs{\ti w}_{n, 0}^J$, which follow directly from~\cite[Claim A.1 and Lemma 2.1]{DKM1} in the odd dimensional case and~\cite[Lemma 10 and 11]{CKS} in the case of even dimensions.  

Finally, we can use
use~\eqref{eq:maximal} and~\eqref{eq:tiu-fsp} to see that,  
\EQ{\label{eq:maximal-ti}
&\lim_{ n \to \infty} \sup_{I \ni t_n;  I \subset [0,  1]}  \frac{1}{\abs{I}} \int_I \int_0^{\infty} | \p_t \ti u_n(t, r)|^2 \, r \, \ud r \, \ud t = 0,\\  
&\| \p_t {\ti u}_n(t_n) \|_{L^2} \to 0 \mas n \to \infty. 
}
Then following the exact argument in~\cite[Erratum]{DKM3erratum}, but applied to $\bs{\ti u}_n(t_n)$ we conclude that the set $\calJ_{\Lin, 0}$ is empty, i.e., all of the profiles $(\bs v^j\lin, \lambda_{n, j}, t_{n, j})$ with $j \in \calJ_{\Lin, 0}$ satisfy, $\bs v\lin^j \equiv 0$.

\textbf{Step 7.}(Centered profiles at bounded  scales)  
To recap, we now have 
\EQ{ \label{eq:profiles3}
\bs{\ti u}_n(t_n) = \sum_{ 1 \le j \le J; j \in \calJ_{c, 0} } \bs v_{\Lin, n}^j(0)    + \bs{\ti w}_{n, 0}^J 
}
where $\bs{\ti u}_n(t_n)$ satisfies~\eqref{eq:tiu-fsp} and all of the profiles $(\bs v_{\Lin}^j,  \lambda_{n, j}, 0) $ have $t_{n, j} = 0$ and $\lam_{n, j} \lesssim 1$ for all $n, j$.  Moreover, we have the vanishing in~\eqref{eq:maximal-ti}.  We can now apply the exact same arguments of Duyckaerts, Kenig, and Merle~\cite[Proofs of Corollary 4.1 and Corollary 4.2]{DKM3erratum} (see also the identical arguments by Rodriguez in~\cite[Proof of ]{Rod16-adv} and Jia and Kenig~\cite[Proof of ]{JK}) to deduce that in fact either 
\EQ{
\bs v_{\Lin, n}^j(0) =  \iota_j \bs W_{\lam_{n, j}} \mor \bs v_{\Lin, n}^j(0) =  \bs 0
}
for $\iota_j \in \{-1, 1\}$ for each $j \in \calJ_{c, 0}$. By~\eqref{eq:pyth} there can only be finitely many of these profiles that are non-trivial, and thus after reordering the indices we can find $K_0 \ge 0$ and $\lam_{n, 1} \ll \lam_{n, 2} \ll \dots \ll \lam_{n, K_0} \lesssim 1$  such that 
\EQ{ \label{eq:profiles4} 
\bs{\ti u}_n(t_n) = \sum_{ 1 \le j \le K_0}\iota_j \bs W_{\lam_{n, j}}  + \bs{\ti w}_{n, 0} 
}
where $\bs{\ti w}_{n, 0}  := \bs{\ti w}_{n, 0}^{K_0}$. We note that the error $\bs{\ti w}_{n, 0}$ satisfies, 
\EQ{ \label{eq:tiw-vanishing} 
\limsup_{n \to \infty} \Big( \| \ti w_{n, 0} \|_{L^{\frac{2D}{D-2}}} + \|\dot{\ti w}_{n, 0} \|_{L^2} \Big) = 0 
}
where the $L^2$ vanishing of $\dot{\ti w}_{n, 0}$ follows from~\eqref{eq:maximal-ti} and the decomposition~\eqref{eq:profiles4}.

\textbf{Step 8.}(Vanishing properties of the error)  
We now select the final sequence by choosing $r_n \to \infty$ so that 
\EQ{ \label{eq:r_n-def} 
r_n \le\frac{1}{2} \min\{ r_{2, n},  R_{1, n}\}, \quad \lim_{n \to \infty} \sup_{t \in [0, 1]}\| \bs{\ti u}_n(t) \|_{\E( A_n^{-1} r_{n}, A_n r_n)} = 0
}
where $R_{1, n}$ is as in Steps 1, 2 and where $1 \ll A_n \ll r_n$. The existence of such a sequence follows from the same logic as in Step 6. By~\eqref{eq:jia-kenig} we have 
\EQ{
\label{eq:jia-kenig1}
\limsup_{n\to \infty}\int_0^\infty\Bigg[\Big( \p_r^2 u_n(t_n)  &+ \frac{D-1}{r} \p_r u_n(t_n) - \abs{u_n(t)}^{\frac{4}{D-2}} u_n(t_n) \Big)\Big( r \p_r u_n(t_n) + \frac{D}{2} u_n(t_n)\Big) \\
&\quad + \p_t u_n(t_n) \Big( r \p_r \p_t u_n(t_n) +  \frac{D}{2} \p_t u_n(t_n) \Big)\Bigg] \chi_{r_{ n}} \, r^{D-1} \, \ud r  \leq 0. 
}
Integration by parts above, we obtain, 
\EQ{
\limsup_{n \to \infty} \Bigg(\int_0^\infty \Big[(\p_r u(t_n))^2 - \abs{u(t_n)}^{\frac{2D}{D-2}} \Big]\chi_{r_n} \, r^{D-1} \, \ud r  - \Om_{1, r_n}(\bs u(t_n)) - \frac{D}{2} \Om_{2, r_n}(\bs u(t_n)) \Bigg) \le 0
  }
  where $\Om_{1, r_n}, \Om_{2, r_n}$ are defined in~\eqref{eq:Om-rho-def}. Using~\eqref{eq:r_n-def} along with Sobolev embedding and Hardy's inequality we obtain the vanishing of the errors terms $\Om_{j, r_n}(\bs{\ti u}(t_n))$ above and we conclude that, 
  \EQ{
  \limsup_{n \to \infty} \int_0^\infty \Big[(\p_r u(t_n))^2 - \abs{u(t_n)}^{\frac{2D}{D-2}} \Big]\chi_{r_n} \, r^{D-1} \, \ud r  \le 0 
  }
Due to~\eqref{eq:profiles4},  the orthogonality of the profiles (i.e., $\lam_{n, 1} \ll \lam_{n, 2} \ll \dots \lam_{n, K_0}$), the fact that $\lam_{n, j}/r_n \to 0$ for each $j$ (if $K_0 >0$), and  the fact that the Jia-Kenig functional vanishes at $W$, i.e., 
\EQ{
\int_0^\infty \Big[( \p_r W)^2 - \abs{W}^{\frac{2D}{D-2}}\Big] \, r^{D-1} \, \ud r = 0
}
we can conclude that 
\EQ{
\limsup_{n \to \infty} \int_0^\infty \Big[(\p_r \ti w_{0, n})^2 - \abs{\ti w_{0, n}}^{\frac{2D}{D-2}} \Big]\chi_{r_n} \, r^{D-1} \, \ud r  \le 0
}
But then we may use~\eqref{eq:tiw-vanishing} to see that in fact
\EQ{
\lim_{n \to \infty} \int_0^\infty \abs{\ti w_{0, n}}^{\frac{2D}{D-2}} \chi_{r_n} \, r^{D-1} \, \ud r  = 0
}
Using this estimate in the previous line we conclude that 
\EQ{
\lim_{n \to \infty} \int_0^\infty (\p_r \ti w_{0, n})^2\chi_{r_n} \, r^{D-1} \, \ud r  = 0.  
}
Lastly, by~\eqref{eq:profiles4} and the fact that $\bs u_n(t_n, r) = \bs{\ti u}_n(t_n, r)$ if $r \le r_n$, we see that $\lim_{n \to \infty} \bs \de_{r_n}( \bs u_n(t_n))  = 0$, completing the proof. 
\end{proof} 

\section{The sequential decomposition }

In this section we sketch the proof Theorem~\ref{thm:seq}, the sequential decomposition. We view this result as the consequence of three main ingredients: (1) the identification of the linear radiation $\bs u^*$, (2) a proof that no energy can concentrate in the self-similar region of the light cone, and (3) the compactness lemma proved in the previous section.

\subsection{Identification  of the radiation} \label{ssec:rad} 

The results in this subsection were proved by Duyckaerts, Kenig, and Merle in~\cite{DKM1, DKM3} in the case $D=3$. Following their approach, analogous results were obtained in~\cite{CKLS3} in dimension $D=4$,~\cite{JK} in dimension $D=6$, and in~\cite{Rod16-adv} for all odd $D \ge 5$. The case of even  dimensions $D>6$ follows from an identical argument as the one used in~\cite{Rod16-adv}.  

\begin{prop}(Radiation in case of finite time blow-up)\emph{\cite[Theorem 3.2]{DKM1}}  \label{prop:rad-blowup} Let $\bs u(t) \in \E$ be a solution to~\eqref{eq:nlw} defined on the time  interval $I = (0, T]$ for some $T >0$ and blowing up in the type-II sense as $t \searrow 0$, that is, such that   
\EQ{
\sup_{t \in (0, T]} \| \bs u(t) \|_{\E} < \infty. 
}
Then, there exists $\bs u_0^* \in \E$ such that 
\EQ{
\bs u(t) &\rightharpoonup \bs u_0^* \quad \textrm{weakly in} \, \,  \E \mas t \to 0,\\ 
\| \fy( \bs u(t) - \bs u_0^* )\|_{\E} &\to 0 \mas t \to 0 , 
}
where the latter limit holds for any $\fy \in C_0^{\infty}(0, \infty)$. Moreover, the solution  $\bs u^*(t) \in \E$  of~\eqref{eq:nlw} with initial data $\bs u^*_0$ is defined on the interval $[0, T_0]$ and satisfies, 
\EQ{
\bs u(t, r) &= \bs u^*(t, r)  \mif r \ge t , \quad \forall t \in (0, T_0],   \\
\lim_{t \to 0} E( \bs u(t) - \bs u^*(t)) &= E( \bs u) - E(\bs u^*). 
}

\end{prop}

\begin{prop} (Radiation for a global-in-time solution)\emph{\cite[Corollary 3.9]{DKM3}, \cite[Proposition 4.1]{CKLS3}}  \label{prop:rad-global} Let $\bs u(t) \in \E$ be a solution to~\eqref{eq:nlw} defined on the time  interval $I = [T, \infty)$ for some $T\ge 0$ and such that 
\EQ{
\sup_{t \in [T, \infty)]} \| \bs u(t) \|_{\E} < \infty. 
}
Then, there exists a free wave $\bs u^*\lin(t) \in \E$ such that 
\EQ{
S(-t) \bs u(t) &\rightharpoonup \bs u^*\lin(0) \quad \textrm{weakly in}\, \,  \E  \mas t \to \infty,  \\
\textrm{and} \, \, \forall R \in \R, \quad \lim_{t \to \infty} \int_{t-R}^\infty &\Big[( \p_t( u(t)- u^*\lin(t)))^2 + (\p_r (u(t) -  u^*\lin(t)))^2 + \frac{ (u(t) - u^*\lin(t) )^2}{r^2} \Big] \, r^{D-1} \, \ud r  = 0. 
}
Moreover, there exists $T_0>0$ such that if we denote by $\bs u^*(t) \in \E$ the unique solution to~\eqref{eq:nlw} such that 
\EQ{
\| \bs u^*(t) - \bs u^*\lin(t) \|_{\E} \to 0 \mas t \to \infty, 
}
then, 
\EQ{
\lim_{t \to \infty}E( \bs u(t) - \bs u^*(t)) &= E( \bs u) - E( \bs u^*). 
}
\end{prop} 

\begin{rem}
We note that the proof of  Proposition~\ref{prop:rad-blowup} given in~\cite{DKM1} was given only in dimensions $D= 3, 4, 5$ (and for non-radially symmetric data), but it generalizes in a straight-forward way to higher spaces dimensions using the local Cauchy theory from Lemma~\ref{lem:Cauchy}. The proof of Proposition~\ref{prop:rad-global} is given in dimension $D=3$ in~\cite{DKM3} and was generalized to dimension $D=4$ in~\cite{CKLS3} using  technical tools related to profile decompositions in even space dimensions proved by C\^ote, Kenig, and Schlag in~\cite{CKS}. It was proved in all odd space dimensions in~\cite{Rod16-adv}. Again the proofs given in those references generalize to all even space dimensions, using~\cite{CKS} . 
\end{rem}

\subsection{Non-concentration of self-similar energy }\label{ssec:ssim} 

In this section we sketch the proof that finite energy solutions cannot concentrate linear energy in the self-similar region of the cone. As a consequence of this fact and virial identities, we deduce the vanishing of the averaged kinetic energy in the cone. The proof in this section closely follows the arguments given in~\cite{CKLS3} and~\cite{JK}, which in turn follow the scheme developed by Christodoulou and Tahvildar-Zadeh~\cite{CTZcpam} and Shatah and Tahvildar-Zadeh~\cite{STZ92} in the context of equivariant wave maps. We make one minor observation here, namely that the reductions performed in~\cite{CKLS3, JK} from the $D$-dimensional radially symmetric NLW~\eqref{eq:nlw} to a wave maps-type equation in two-space dimensions work equally well in all space dimensions $D \ge 3$, and thus the arguments from~\cite{JK} (which generalized the Shatah, Tahvildar-Zadeh arguments to cover all finite energy solutions) apply directly\footnote{We note that the published version of~\cite{CKLS3} contained a gap in the proof of the corresponding results, as the arguments used to deduce Proposition~\ref{prop:stz-global} and Proposition~\ref{prop:stz-global} in that paper were performed only for smooth solutions. This gap was closed by an argument of the first author and was included in an appendix to~\cite{CKLS1-erratum}. An earlier argument by Jia and Kenig from~\cite{JK} can also be used to close the gap in~\cite{CKLS3}, and we refer to their approach here.}.

\begin{prop}[No self-similar concentration for blow-up solutions] \label{prop:stz-blowup} \emph{\cite[Theorem 2.1]{JK}} 
Let $\bs u(t) \in \E$ be a solution to~\eqref{eq:nlw} defined on the time  interval $I = (0, T]$ for some $T >0$ and blowing up in the type-II sense as $t \searrow 0$, that is, such that   
\EQ{
\sup_{t \in (0, T]} \| \bs u(t) \|_{\E} < \infty. 
}
Then, for any $\al \in (0, 1)$, 
\EQ{
\lim_{t\searrow 0} \int_{\al t}^t \Big[ ( \p_t u(t, r))^2 + (\p_r u(t, r))^2 + \frac{(u(t, r))^2}{r^2} \Big] \, r^{D-1} \, \ud r = 0. 
}
\end{prop} 

\begin{prop}[No self-similar concentration for global solutions] \emph{\cite[Theorem 2.4]{JK}}  \label{prop:stz-global} Let $\bs u(t) \in \E$ be a solution to~\eqref{eq:nlw} defined on the time  interval $I = [T, \infty)$ for some $T\ge 0$ and such that 
\EQ{
\sup_{t \in [T, \infty)]} \| \bs u(t) \|_{\E} < \infty. 
}
Then, for any $\al \in (0, 1)$, 
\EQ{
\lim_{R \to \infty} \limsup_{t \to \infty} \int_{\al t}^{t -R} \Big[ ( \p_t u(t, r))^2 + (\p_r u(t, r))^2 + \frac{(u(t, r))^2}{r^2} \Big] \, r^{D-1} \, \ud r = 0.
}

\end{prop}

\begin{cor}[Time-averaged vanishing of kinetic energy for blow-up solutions] \emph{\cite[Lemma 2.2]{JK}} \label{cor:kin-blowup}   Let $\bs u(t) \in \E$ be a solution to~\eqref{eq:nlw} satisfying the hypothesis of Proposition~\ref{prop:stz-blowup}.  Then,
\EQ{
\lim_{\tau \searrow 0} \frac{1}{\tau} \int_0^\tau \int_0^t (\p_t u(t, r))^2 \, r^{D-1} \, \ud r \, \ud t  = 0
}
\end{cor}

\begin{cor}[Time-averaged vanishing of kinetic energy for global solutions] \emph{\cite[Lemma 2.6]{JK}} \label{cor:kin-global}  Let $\bs u(t) \in \E$ be a solution to~\eqref{eq:nlw} satisfying the hypothesis of Proposition~\ref{prop:stz-global}.  Then,
\EQ{
 \lim_{\tau \to \infty} \frac{1}{\tau} \int_0^\tau \int_0^{\frac{t}{2}} (\p_t u(t, r))^2 \, r^{D-1} \, \ud r \, \ud t  = 0
}
\end{cor}

\begin{rem} 
The proofs of Proposition~\ref{prop:stz-blowup} and Proposition~\ref{prop:stz-global} in~\cite{CKLS3, JK} are done for the cases $D=4, 6$ and are based on the following reduction, which we generalize to cover all dimensions $D \ge 3$.  Let  
$$
k:= \frac{D-2}{2}.$$ Given $\bs u(t) \in \E$, set, 
\EQ{ \label{eq:psi-def} 
\bs \psi(t, r):= (r^k u(t, r), r^k\p_t u(t, r)). 
}
We see that $\bs u(t)$ solves~\eqref{eq:nlw} if and only if $\bs \psi(t, r)$ solves, 
\EQ{
\p_t^2 \psi - \p_r^2 \psi - \frac{1}{r} \p_r \psi + \frac{ k^2 - \abs{ \psi}^{\frac{2}{k}} }{r^2} \psi  = 0 
}
which bears enough structural similarities to the equivariant wave maps equation that the main elements of the arguments from~\cite{CTZcpam, STZ92} carry over. The key feature for our purposes, is that 
\EQ{
F_k(\psi)&:= \frac{1}{2}\abs{\psi}^2 \Big( k^2 - \frac{k}{k+1} \abs{ \psi}^{\frac{2}{k}} \Big)
}
is positive when $\abs{\psi(t, r)}$ sufficiently small and hence so is the flux density, 
\EQ{
\frac{1}{2} ( \p_t \psi(t, r) - \p_r \psi(t, r))^2 + \frac{F_k(\psi(t, r))}{r^2}  > 0. 
}
Up to changing the values of some constants, the line-by-line arguments in~\cite[Proof of Theorem 2.1]{JK} and~\cite[Proof of Theorem 2.4]{JK} are valid in any dimension $D \ge 4$ with $\bs \psi$ defined as in~\eqref{eq:psi-def}. 
\end{rem} 

\begin{rem} 
The proof of Corollary~\ref{cor:kin-blowup} follows from the virial identity~\eqref{eq:virial} with the cutoff  at $\rho(t) = t/2$ together with Proposition~\ref{prop:stz-blowup}. The exact argument in~\cite[Proof of Lemma 2.2, in particular Eq. (2.66)]{JK} applies in our setting as well. The proof of Corollary~\ref{cor:kin-global} is similar, using now Proposition~\ref{prop:stz-global}, and follows from the exact argument in~\cite[Proof of Lemma 2.4, second displayed equation on page 1552]{JK}. 

\end{rem} 

\subsection{The sequential decomposition}\label{ssec:seq} 

In this section we deduce Theorem~\ref{thm:seq}, the sequential decomposition as a consequence of the Compactness Lemma~\ref{lem:compact} and the collection of results from earlier in this section. 

In the remainder of the paper we unify the blow-up and global-in-time settings by making the following conventions. Consider a finite solution  $\bs u(t) \in \E$ on its maximal time of existence. We assume that either $\bs u(t)$ blows up in backwards time at $T_-=0$ and is defined on an interval $I_*:=(0, T_0]$, or $\bs u(t)$ is global in forward time and defined on the  interval $I_* := [T_0, \infty)$ where in both cases $T_0>0$. We let $T_* := 0$ in the blow-up case and $T_* := \infty$ in the global case. We assume that $\bs u(t)$ exhibits type II behavior in that, 
\EQ{ \label{eq:typeII-2} 
\lim_{t \to T_*} \| \bs u(t) \|_{\E} < \infty. 
}

First we complete the proof of Theorem~\ref{thm:radiation}. 

\begin{proof}[Proof of Theorem~\ref{thm:radiation}]
We let $\bs u^*(t)$ be defined as in Proposition~\ref{prop:rad-blowup} in the case $T_* = 0$ and as in Propostiion~\ref{prop:rad-global} in the case $T_* = \infty$. 
If $T_* = 0$ the conclusions of Theorem~\ref{thm:radiation} are a direct consequence of Proposition~\ref{prop:rad-blowup} and Proposition~\ref{prop:stz-blowup}. If $T_* = +\infty$ we first note that for any $\al \in (0, 1)$, 
 \EQ{
\| \bs u^*(t) \|_{\E(0, \al t)} \le \| \bs u^*(t) - \bs u^*\lin(t) \|_{\E} + \| \bs u^*\lin(t) \|_{\E( 0, \al t)} \to 0 \mas t \to T_*
}
where the vanishing of the last term above is due the asymptotic concentration of free waves near the light cone; see~\cite[Lemma 4.1]{DKM1} for odd $D$ and~\cite[Theorem 5]{CKS} for even $D$. Now apply Proposition~\ref{prop:rad-global} and Proposition~\ref{prop:stz-global}. 
\end{proof}

\begin{proof}[Proof of Theorem~\ref{thm:seq}]
By Corollary~\ref{cor:kin-blowup} if $T_* = 0$ or Corollary~\ref{cor:kin-global} if $T_* = \infty$ we have, 
\EQ{
\lim_{\tau \to T_*} \frac{1}{\tau} \int_0^\tau \int_0^{\frac{t}{2}} (\p_t u(t, r))^2 \, r^{D-1} \, \ud r \, \ud t  = 0.
}
We claim there exists a sequence $\tau_n \to T_*$ such that, 
\EQ{ \label{eq:tau_n} 
\lim_{ n \to \infty} \sup_{0< \sigma < \tau_n} \frac{1}{\sigma} \int_{\tau_n}^{\tau_n + \sigma}\int_0^{\frac{t}{2}} (\p_t u(t, r))^2 \, r^{D-1} \, \ud r \, \ud t  = 0.
}
We show that the above is a consequence of the classical maximal function estimate~\eqref{eq:weak-max}. Indeed, define
\EQ{
\phi(t) =  \int_0^{\frac{t}{2}} (\p_t u(t, r))^2 \, r^{D-1} \, \ud r , \quad \Psi(\tau): = \frac{1}{\tau} \int_0^\tau \phi(t) \, \ud t.
}
Then~\eqref{eq:tau_n} reduces to the following claim: If $\Psi(\tau)  \to 0$ as $\tau \to T_*$, then there exists at least one sequence of times $\tau_n \to T_*$ such that $M\phi(\tau_n) \to 0$ as $n \to \infty$. Now considering intervals $J_n = (0, 1/n]$ if $T_* = 0$ or $J_n = [n/2, n]$ if $T_+ = \infty$ apply the maximal function estimate~\eqref{eq:weak-max} with $\al_n = 6\Psi(n^{-1})$ if $T_* = 0$ or $\al_n = 6\Psi(n)$ if $T_* = \infty$, noting that in both cases $\al_n \to 0$ as $n \to \infty$. 
\EQ{ 
\Big| \{ t \in J_n \, : \, M\phi(t) > \al_n \} \Big| \le \frac{3}{\al_n} \int_{J_n} \phi(t) \, \ud t  \le \frac{1}{2} | J_n| . 
}
This means that $M\phi(t) \le \al_n  \to 0$ for half of the points in $J_n$, from which we select the sequence $\tau_n\to T_*$.  

Next, let $\rho(t)$ be as in Theorem~\ref{thm:radiation}. With $\tau_n$ as in~\eqref{eq:tau_n} we set 
\EQ{
\rho_n := \sup_{t \in [\tau_n, \tau_n + \rho(\tau_n)]}\rho(t) \ll  \tau_n. 
}
It follows from~\eqref{eq:tau_n} that 
\EQ{
\lim_{n \to \infty} \frac{1}{\rho_n} \int_{\tau_n}^{\tau_n + \rho_n} \int_0^{\frac{\tau_n}{2}} (\p_t u(t, r))^2 \, r^{D-1} \, \ud r \, \ud t  = 0.
}
Next defining $
\bs u_n(s, r):= \bs u(\tau_n + s, r)$ and changing variables above we obtain a sequence of solutions $\bs u_n$ defined on intervals $[0, \rho_n]$ such that, 
\EQ{
\lim_{n \to \infty} \frac{1}{\rho_n} \int_{0}^{\rho_n} \int_0^{\rho_n \frac{\tau_n}{2\rho_n}} (\p_s u_n(s, r))^2 \, r^{D-1} \, \ud r \, \ud s  = 0.
}
We can now apply the Compactness Lemma~\ref{lem:compact}  since  $R_n:= \frac{\tau_n}{2 \rho_n} \to  \infty$  as $n \to \infty$. We obtain sequences $s_n \in [0, \rho_n]$ and $1 \ll r_n \ll \frac{\tau_n}{2 \rho_n}$ for which $\bs \delta_{\rho_n r_n}(\bs u_n(s_n)) \to 0$ as $n \to \infty$. Passing back to the original variables we set $t_n = \tau_n + s_n$ and we have, 
$\bs \de_{r_n \rho_n}( \bs u( t_n)) \to 0$ as $n \to \infty$. From~\eqref{eq:delta-def} (and examining the proof of The Compactness Lemma~\ref{lem:compact}, see Remark~\ref{rem:seq}) we obtain an integer $K_0 \ge 0$, and scales $\lam_{n, 1} \ll \lam_{n, 2}  \ll \dots \ll \lam_{n, K_0} \lesssim \rho_n \ll  t_n$, and a vector of signs $\vec \iota \in \{-1, 1\}^{K_0}$  such that 
\EQ{
\| \bs u(t_n) - \bs \calW(\vec \iota, \vec \lam_n) \|_{\E( 0, r_n \rho_n)}  \to 0 \mas n \to \infty 
}
Note that by construction $\rho(t_n) \ll r_n \rho(t_n) \ll  r_n \rho_n \ll t_n$. Thus, from Theorem~\ref{thm:radiation} we have, 
\EQ{
\| \bs u(t_n) - \bs u^*(t_n) \|_{\E( \rho_n)} \to 0 \mand \| \bs u^*(t_n) \|_{\E(0, r_n \rho_n)}  \to 0 \mas n \to \infty
}
Combining the two last displayed equations completes the proof. 
\end{proof}

\section{Decomposition of the solution and collision intervals} \label{sec:decomposition} 

In the final two sections we prove Theorem~\ref{thm:main} for equivariance classes $D \ge 6$. We reserve the cases $D=4, 5$ for the appendix, as these low dimensions require a few technical modifications stemming from the slower decay of $W(r)$ as $r \to \infty$.  

\subsection{Proximity to a multi-bubble and collisions}
\label{ssec:proximity}
For the remainder of the paper we fix a solution
$\bs u(t) \in \E$ of \eqref{eq:nlw},
defined on the time interval $I_*=(0, T_0]$
in the blow-up case and on $I_*=[T_0, \infty)$ in the global case, for some $T_0 > 0$. 
We set $T_* := \infty$ in the global case and $T_* := 0$ in the blow-up case and we assume, 
\EQ{
\lim_{t \to T_*} \| \bs u(t) \|_{\E} < \infty.
}

Let $\bs u^*(t)$ be the radiation as defined Proposition~\ref{prop:rad-blowup} and Proposition~\ref{prop:rad-global}.  We will use crucially the fact that the radiation is given in continuous time. 
Note that combining the results of Proposition~\ref{prop:rad-blowup} and Proposition~\ref{prop:stz-blowup} in the blow-up case and Proposition~\ref{prop:rad-global} together with Proposition~\ref{prop:stz-global} give a function $\rho: I_* \to (0, \infty)$ such that
\EQ{ \label{eq:rho-def} 
&\lim_{t \to T_*} \big((\rho(t) / t)^{\frac{D-2}{2}} + \|\bs u(t) -  \bs u^*(t) \|_{\cE(\rho(t), \infty)}^2 \big) = 0.
}
We also note that
\EQ{ \label{eq:u*-interior} 
\lim_{t \to T_*} \|  \bs u^*(t)\|_{\E( 0, \al t)}  = 0,
}
for any $\al \in (0, 1)$.

By Theorem~\ref{thm:seq} there exists a time sequence $t_n \to T_*$ and an integer $N \ge 0$, which we now fix,  such that $\bs u(t_n) - \bs u^*(t_n)$ approaches an $N$-bubble as $n \to \infty$.  Roughly, our goal is to show that on the region $r \in (0, \rho(t))$, the solution $\bs u(t)$ approaches a continuously modulated $N$-bubble, noting that the radiation $\bs u^*(t)$ is negligible in this region. By convention, we will set $\lambda_{N+1}(t) := t$ to be the ``scale'' of the radiation and $\lambda_0(t) := 0$. Our argument requires the following localized version of the distance function to a multi-bubble.

\begin{defn}[Proximity to a multi-bubble]
\label{def:proximity}
For all $t \in I$, $\rho \in (0, \infty)$, and $K \in \{0, 1, \ldots, N\}$, we define
the \emph{localized multi-bubble proximity function} as
\begin{equation}
\bfd_K(t; \rho) := \inf_{\vec \iota, \vec\lam}\bigg( \| \bs u(t) - \bs u^*(t) - \bs\calW( \vec\iota, \vec\lambda) \|_{\cE(\rho, \infty)}^2 + \sum_{j=K}^{N}\Big(\frac{ \lam_{j}}{\lam_{j+1}}\Big)^{\frac{D-2}{2}} \bigg)^{\frac{1}{2}},
\end{equation}
where $\vec\iota := (\iota_{K+1}, \ldots, \iota_N) \in \{-1, 1\}^{N-K}$, $\vec\lambda := (\lambda_{K+1}, \ldots, \lambda_N) \in (0, \infty)^{N-K}$, $\lambda_K := \rho$ and $\lambda_{N+1} := t$.

The \emph{multi-bubble proximity function} is defined by $\bfd(t) := \bfd_0(t; 0)$.
\end{defn}

\begin{rem} 
We emphasize that if $\bfd_K(t; \rho)$ is small, this means that $ \bs u(t) - \bs u^*(t)$ is close to $N-K$ bubbles in the exterior region  $r \in (\rho, \infty)$. 
\end{rem} 
We can now rephrase Theorem~\ref{thm:seq} in this notation: there exists a monotone sequence $t_n \to T_*$ such that
\begin{equation}
\label{eq:dtn-conv}
\lim_{n \to \infty} \bfd(t_n) = 0.
\end{equation}
Even though this fact is certainly a starting point of our analysis,
it will turn out that we cannot use it as a black box. Rather, we need to examine the proof
and use more precise information provided by the analysis in~\cite{DKM3} (this is done in Section~\ref{sec:compact}). 

We state and prove some simple consequences of the set-up above.
We always assume $N \geq 1$, since the pure radiation case $N = 0$ (in fact,
also the case $N = 1$) is a consequence of the sequential decomposition (as observed by Duyckaerts, Kenig, and Merle in~\cite[Theorem 2, Theorem 5, Corollary 6]{DKM3}). 

Next, a direct consequence of~\eqref{eq:rho-def} is that $\bs u(t)- \bs u^*(t)$ always approaches a $0$-bubble in some exterior region. With $\rho_N(t) = \rho(t)$ given by~\eqref{eq:rho-def} the following lemma is  immediate from the conventions of Definition~\ref{def:proximity} 
\begin{lem}
\label{lem:conv-rhoN}
There exists a function $\rho_N: I \to (0, \infty)$ such that
\begin{equation}
\label{eq:conv-rhoN}
\lim_{t\to T_*}\bfd_N(t; \rho_N(t)) = 0.
\end{equation}
\end{lem}

Theorem~\ref{thm:main} will be a quick consequence of showing that,
\begin{equation}
\label{eq:dt-conv}
\lim_{t \to T_*} \bfd(t) = 0.
\end{equation}
The approach which we adopt in order to prove~\eqref{eq:dt-conv} it is to study colliding bubbles.
A collision is defined as follows.
\begin{defn}[Collision interval]
\label{def:collision}
Let $K \in \{0, 1, \ldots, N\}$. A compact time interval $[a, b] \subset I_*$ is a \emph{collision interval}
with parameters $0 < \epsilon < \eta$ and $N - K$ exterior bubbles if
\begin{itemize}
\item $\bfd(a) \leq \epsilon$ and $\bfd(b) \leq \epsilon$,
\item there exists $c \in (a, b)$ such that $\bfd(c) \geq \eta$,
\item there exists a function $\rho_K: [a, b] \to (0, \infty)$ such that $\bfd_K(t; \rho_K(t)) \leq \epsilon$
for all $t \in [a, b]$.
\end{itemize}
In this case, we write $[a, b] \in \calC_K(\epsilon, \eta)$.
\end{defn}
\begin{defn}[Choice of $K$]
\label{def:K-choice}
We define $K$ as the \emph{smallest} nonnegative integer having the following property.
There exist $\eta > 0$, a decreasing sequence $\epsilon_n \to 0$
and sequences $(a_n), (b_n)$ such that $[a_n, b_n] \in \calC_K(\epsilon_n, \eta)$ for all $n \in \{1, 2, \ldots\}$.
\end{defn}
\begin{lem}[Existence of $K \ge 1$]
\label{lem:K-exist}
If \eqref{eq:dt-conv} is false, then $K$ is well defined and $K \in \{1, \ldots, N\}$.
\end{lem}

\begin{rem} 
The fact that $K \ge 1$ means that at least one bubble must lose its shape if~\eqref{eq:dt-conv} is false.
\end{rem} 
\begin{proof}[Proof of Lemma~\ref{lem:K-exist}]
Assume \eqref{eq:dt-conv} does not hold, so that there exist $\eta > 0$ and a monotone sequence $s_n \to T_*$ such that
\begin{equation}
\bfd(s_n) \geq \eta, \qquad\text{for all }n.
\end{equation}
We claim that there exist sequences $(\epsilon_n), (a_n), (b_n)$ such that $[a_n, b_n] \in \calC_N(\epsilon_n, \eta)$.
Indeed, \eqref{eq:dtn-conv} implies that there exist $\epsilon_n \to 0$, $a_n \leq s_n$ and $b_n \geq s_n$
such that $\bfd(a_n) \leq \epsilon_n$ and $\bfd(b_n) \leq \epsilon_n$. Note that $a_n \to T_*$ and $b_n \to T_*$.
Let $\rho_N: [a_n, b_n] \to (0, \infty)$ be the function given by Lemma~\ref{lem:conv-rhoN},
restricted to the time interval $[a_n, b_n]$.
Then \eqref{eq:conv-rhoN} yields
\begin{equation}
\lim_{n\to\infty}\sup_{t\in[a_n, b_n]}\bfd_N(t; \rho_N(t)) = 0.
\end{equation}
Upon adjusting the sequence $\epsilon_n$, we obtain that all the requirements of Definition~\ref{def:collision}
are satisfied for $K = N$.

We now prove that $K \geq 1$. Suppose $K = 0$. The definition of a collision interval yields
$\bfd_0(c_n;  \rho_n) \leq \epsilon_n$ for some sequence $\rho_n \ge 0$, and at the same time $\bfd(c_n) \geq \eta$ for some $\eta>0$. Without loss of generality we may assume that $c_n$ is a time at which $\eta \le \bfd(c_n)  \le 2 \eta$  for each $n$, and we may assume further that $\eta>0$ is small relative to $\| \bs W \|_{\E}$. 
We show that this is impossible. 

First, by Theorem~\ref{thm:seq} we know that 
\EQ{\label{eq:en-id} 
E(\bs u) = N E( \bs W) + E( \bs u^*).
}
On the other hand, since $\bs d_0(c_n, \rho_n) \le \eps_n$ we can find parameters, $\rho_n \ll  \lam_{n, 1} \ll \dots \ll \lam_{n, N} \ll  \rho(c_n) \ll c_n$ and signs $\vec{\iota}_n$ such that 
\EQ{ \label{eq:tig-small} 
 \| \bs u(c_n) - \bs u^*(c_n) - \bs \calW( \vec{ \iota}_n, \vec{ \lam}_n) \|_{\E( \rho_n, \infty)}^2 + \sum_{j =0}^N\Big( \frac{  \lam_{n, j}}{\lam_{n, j+1}}\Big)^{\frac{D-2}{2}} \lesssim \eps_n^2. 
}
Using the above along with~\eqref{eq:u*-interior}, Lemma~\ref{lem:M-bub-energy}, and the asymptotic orthogonality of the various parameters we have, 
\EQ{
E( \bs u(c_n); \rho_n, \infty) = N E( \bs W) + E( \bs u^*) + o_n (1) \mas n \to \infty
}
Using the above along with~\eqref{eq:en-id} we conclude that,
\EQ{
E( \bs u(c_n);0,  \rho_n) = o_n (1)  \mas n \to \infty
}
Now, let $\bs v_n:= \bs u(c_n) \chi_{\frac{1}{2} \rho_n}$. We have shown that $E( \bs v_n) = o_n(1)$. We claim that we must have $ \| \bs v_n \|_{\E} \simeq \eta$, which would give a contradiction with the critical Sobolev inequality, since $\eta>0$ can be chosen small. To prove the claim, find parameters $\vec \s_n, \vec \mu_n$ such that 
\EQ{
\eta \simeq \bfd(c_n) \simeq  \Big(\| \bs u(c_n) - \bs u^*(c_n) - \bs \calW( \vec \s_n, \vec \mu_n) \|_{\E}^2 + \sum_{j=0}^N \Big( \frac{  \mu_{n, j}}{\mu_{n, j+1}}\Big)^{\frac{D-2}{2}} \Big)^{\frac{1}{2}}
} 
An application of Lemma~\ref{lem:bub-config} (taking $\eta>0$ smaller if needed) together with the above and~\eqref{eq:tig-small} yields that $\vec \s_n = \vec \iota_n$ and moreover that $\rho_n \ll \mu_{n, 1} \ll \dots \mu_{n, N} \ll \rho(c_n) \ll c_n$. In fact, we have $| \lam_{n, j}/ \mu_{n, j} - 1| \lesssim \theta(\eta)$ for $\theta(\eta)$ as in Lemma~\ref{lem:bub-config}  and thus by~\eqref{eq:tig-small} we have,  
\EQ{
\| \bs u(c_n) - \bs u^*(c_n) - \bs \calW( \vec{ \s}_n, \vec{ \mu}_n) \|_{\E( \rho_n, \infty)}^2 + \sum_{j =0}^N\Big( \frac{  \mu_{n, j}}{\mu_{n, j+1}}\Big)^{\frac{D-2}{2}} = o_n(1) 
}
Since $\rho_n \ll \mu_{n, 1}$ and $\mu_{n, N} \ll \rho(c_n)$ we also have, 
\EQ{
\| \bs u^*(c_n) + \bs \calW( \vec{ \iota}_n, \vec{ \mu}_n) \|_{\E(0, \rho_n)} = o_n(1), 
}
From the previous three displayed equations we can conclude that $ \| \bs v_n \|_{\E} \simeq \eta$, proving the claim, and establishing the contradiction. 
\end{proof}

In the remaining part of the paper, we argue by contradiction, fixing $K$ to be the number provided by Lemma~\ref{lem:K-exist}.
We also let $\eta, \epsilon_n, a_n$ and $b_n$ be some choice of objects satisfying the requirements of Definition~\ref{def:K-choice}.
We fix choices of signs and scales for the $N-K$ ``exterior'' bubbles provided by Definition~\ref{def:proximity} in the following lemma.  

\begin{rem} \label{rem:collision} 
For each collision interval there exists a time $c_n \in [a_n, b_n]$ with $\bfd(c_n) \ge \eta$ and we may assume without loss of generality that $\bfd(a_n) = \bfd(b_n) = \eps_n$ and $\bfd(t) \ge \eps_n$ for each $t \in [a_n, b_n]$. Indeed, given some initial choice of $[a_n, b_n] \in \calC_K(\eta, \eps_n)$, we can find $a_n \le \ti a_n < c_n$ and $c_n < \ti b_n \le b_n$ so that $\bfd(a_n) = \bfd(b_n) = \eps_n$ and $\bfd(t) \ge \eps_n$ for each $t \in [\ti a_n, \ti b_n]$.  Just set $a_n \le \ti a_n := \inf\{t \le c_n \mid \bfd(t) \ge \eps_n \}$ and similarly for $\ti b_n$. 

Similarly, give some initial choice $\eps_n \to 0, \eta>0$ and intervals $[a_n, b_n] \in \calC_K( \eta, \eps_n)$ we are free to ``enlarge'' $\eps_n$ by choosing some other sequence $\eps_n \le \ti \eps_n  \to 0$, and new collision subintervals $[\ti a_n, \ti b_n]  \subset [a_n, b_n] \cap \calC_{K}(\eta, \ti \eps_n)$ as in the previous paragraph. We will enlarge our initial choice of $\eps_n$ in this fashion several times over the course of the proof. 
\end{rem} 




\begin{lem}\label{lem:ext-sign} 
Let $K \ge 1$ be the number given by Lemma~\ref{lem:K-exist}, and let $\eta, \epsilon_n, a_n$ and $b_n$ be some choice of objects satisfying the requirements of Definition~\ref{def:K-choice}. Then there exists a sequence $\vec\sigma_n \in \{-1, 1\}^{N-K}$,  a  function $\vec \mu = (\mu_{K+1}, \dots, \mu_N)  \in C^1(\cup_{n \in \N} [a_n, b_n] ;  (0, \infty)^{N-K})$, a sequence $\nu_n \to 0$, and a sequence $m_n \in \Z$, so that defining the function, 
\EQ{ \label{eq:nu-def} 
\nu:\cup_{n \in \N} [a_n, b_n] \to (0, \infty), \quad  \nu(t):= \nu_n \mu_{K+1}(t),  
}
we have, 
\EQ{ \label{eq:nu-prop} 
\lim_{n \to \infty}\sup_{t \in [a_n, b_n]} \Big(\bfd_K(t; \nu(t)) + \| \bs u(t)\|_{\E( \nu(t) \le r \le  2 \nu(t))} \Big) = 0, 
}
and defining  $\bs w(t), \bs h(t)$ for  $t \in \cup_n [a_n, b_n]$ by 
\EQ{ \label{eq:w_n} 
\bs w(t)= (1 - \chi_{\nu(t)})( \bs u(t) - \bs u^*(t))   &=   \sum_{j = K+1}^N  \s_{n, j} \bs W_{\mu_{ j}(t)}  + \bs h(t), 
}
we have, $\bs w(t) ,\bs h(t) \in \E$, and 
\EQ{ \label{eq:mu_n} 
\lim_{n \to \infty} \sup_{t \in [a_n, b_n]} \Big(\| \bs h(t) \|_{\E}^2 + \Big(\frac{\nu(t)}{ \mu_{K+1}(t)} \Big)^{\frac{D-2}{2}} +  \sum_{j = K+1}^{N} \Big( \frac{\mu_{ j}(t)}{\mu_{j+1}(t)} \Big)^{\frac{D-2}{2}} \Big) = 0,   
}
with the convention that $\mu_{N+1}(t) = t$. Finally, $\nu(t)$ satisfies the estimate, 
\EQ{ \label{eq:nu'} 
\lim_{n \to \infty} \sup_{t \in [a_n, b_n]}\abs{\nu'(t) } = 0.
}

\end{lem}

\begin{rem} 
One should think of $\nu(t)$ as the scale that separates the $N-K$ ``exterior'' bubbles, which are defined continuously on the union of the collision intervals $[a_n, b_n]$  from the $K$ ``interior'' bubbles that are coherent at the endpoints of $[a_n, b_n]$, but come into collision somewhere inside the interval and lose their shape. In the case $K =N$, there are no exterior bubbles,  $\mu_{K+1}(t) = t$ and $\nu_n \to 0$ is chosen using~\eqref{eq:rho-def}. 
\end{rem}

\begin{proof}
By Definition~\ref{def:proximity} for each $n$ we can find scales $\rho_K(t) \ll  \ti \mu_{K+1}(t)  \ll \dots \ll \ti  \mu_{N}(t) \ll t $ and signs $\vec \s(t) \in \{-1, 1\}^{N-k}$  for $t \in [a_n, b_n]$, such that defining $\bs {h}_{\rho_K}(t)$ for $r \in ( \rho_K(t), \infty)$ by 
\EQ{
\bs u(t) - \bs u^*(t) = \bs \calW ( \vec \s(t), \vec{\ti  \mu}(t)) + \bs{h}_{\rho_K}(t) 
}
we have, 
\EQ{ \label{eq:timu-small} 
\bfd(t; \rho_K(t)) \simeq \| \bs{ h}_{\rho_K}(t) \|_{\E( \rho_K(t), \infty)}^2 + \sum_{j =K}^N \Big( \frac{\ti \mu_{j}(t)}{ \ti \mu_{j+1}(t)} \Big)^{\frac{D-2}{2}}  \lesssim \eps_n^2 , 
}
keeping the convention $\ti \mu_K(t) := \rho_K(t), \ti \mu_{N+1}(t) := t$.   
Using $\lim_{n \to \infty} \sup_{t \in [a_n, b_n]}\bfd_K( t; \rho_K(t)) = 0$ and the fact that 
\EQ{ \label{eq:ext-bubble-en} 
\lim_{n \to \infty} \sup_{t \in [a_n, b_n]} \| \bs \calW (\vec \s(t), \vec{ \ti \mu}(t)) \|_{\E(\al_n \ti \mu_{K+1}(t) \le r \le  \be_n \ti \mu_{K+1}(t))} = 0, 
}
for any two sequence $\al_n \ll \be_n \ll 1$, 
we can choose a sequence $\nu_n \to 0$ with 
\EQ{\label{eq:timu-vanish} 
\rho_{K}(t) \le \nu_n \ti \mu_{K+1}(t), \mand  \lim_{n \to \infty}\sup_{t \in [a_n, b_n]} \| \bs u(t) - \bs u^*(t)\|_{\E( \frac{1}{4}\nu_n \ti \mu_{K+1}(t) \le  4 \nu_n \ti   \mu_{K+1}(t))} =  0, 
}
and define $\ti \nu(t)= \nu_n \ti \mu_{K+1}(t)$. 
Thus, defining $\bs {\ti w}(t) , \bs{\ti h}(t) \in \E$  for $t \in \cup_n [a_n, b_n]$, by   
\EQ{ \label{eq:tiw-def} 
\bs{\ti w}(t)&:= (1 - \chi_{\ti \nu(t)})( \bs u(t) - \bs u^*(t))   =   \sum_{j = K+1}^N  \s_{ j}(t) \bs W_{\ti \mu_{ j}(t)}  + \bs{\ti h}(t)  
}
we have using~\eqref{eq:timu-small}, 
\EQ{\label{eq:tih-eps} 
\sup_{t \in [a_n,  b_n]} \Big(\| \bs{\ti  h}(t)\|_{\E}^2 + \sum_{j =K}^N \Big( \frac{\ti \mu_{j}(t)}{ \ti \mu_{j+1}(t)} \Big)^k \Big)\le \theta_n^2. 
}
for some sequence $\te_n \to 0$.  
We invoke Lemma~\ref{lem:bub-config} and continuity of the flow to conclude that for each $n$, the sign vector $\vec \s(t) = \vec \s_n$ is independent of $t \in [a_n, b_n]$,   and the functions $\ti \mu_{K+1}(t), \dots,\ti  \mu_{N}(t)$ can be adjusted to be continuous functions of $t$. However, in the next sections we require differentiability of the function $\ti \mu_{K+1}(t)$, so we must modify it slightly. 

Given a vector $\vec \mu(t) = (\mu_{K+1}(t), \dots \mu_N(t))$, set, 
\EQ{
\bs w( t, \vec \mu(t)) := (1- \chi_{\nu_n\mu_{K+1}(t)})( \bs u(t) - \bs u^*(t)) 
}
Fixing $t$ and  suppressing it in the notation, and setting up for an argument as in the proof of Lemma~\ref{lem:mod-static}, define 
\EQ{
F(h, \vec \mu) := h  - ( w( \cdot,  \vec{\ti \mu}) -  \calW( \vec \s_n, \vec{\ti \mu}) ) + w( \cdot, \vec \mu) - \calW( \vec \s_n, \vec \mu) 
}
and note that $F(0, \vec{\ti \mu}) = 0$. Moreover, 
\EQ{
\| F(h, \vec \mu)\|_H \lesssim \| h \|_{H} + \sum_{j= K+1}^N \abs{ \frac{\mu_j}{\ti \mu_j} - 1} 
}
Define, 
\EQ{
G( h, \vec \mu) := \Big( \frac{1}{\mu_{K+1}} \ang{ \calZ_{\U{\mu_{K+1}}} \mid F(h, \vec \mu)}, \dots , \frac{1}{\mu_N} \ang{ \calZ_{\U{\mu_N}} \mid F(h, \vec \mu)} \Big)
}
and thus $G(0, \vec{\ti \mu}) = (0, \dots, 0)$. Following the same scheme as the proof of Lemma~\ref{lem:mod-static} we obtain via Remark~\ref{rem:IFT} a mapping  $\varsigma: B_H(0; C_0 \te_n) \to (0, \infty)^{N-K}$ such that for each $h \in B_H(0; C_0 \te_n)$ we have 
\EQ{
\abs{ \varsigma_j(h)/ \ti \mu_j - 1} \lesssim \te_n
}
and such that
\EQ{
G( h, \vec \mu) = 0 \Longleftrightarrow \vec \mu = \varsigma(h)
}
Using~\eqref{eq:tih-eps} we define
\EQ{
h:= F(\ti h, \varsigma(\ti h)), \quad \vec \mu:= \varsigma(\ti h)
}
By construction we then have, 
\EQ{
\bs w(t, \vec \mu(t)) &= (1- \chi_{\nu(t)})( \bs u(t) - \bs u^*(t))   = \bs \calW(  \vec \s_n, \vec \mu(t)) + \bs h(t) 
}
for $\nu(t):= \nu_n \mu_{K+1}(t)$, 
and for each $j = K+1, \dots, N$, 
\EQ{ \label{eq:h-small} 
\sup_{t \in [a_n, b_n]} \Big(\| \bs h(t) \|_{\E}^2 + \sum_{j=K}^N \Big( \frac{\mu_{j}(t)}{\mu_{j+1}(t)} \Big)^k\Big) \lesssim  \te_n^2 , \quad 
0 = \ang{ \calZ_{\U{\mu_j(t)}} \mid h(t)} 
}
Note that~\eqref{eq:nu-prop} follows from the above and from~\eqref{eq:rho-def}. 
The point is that we can now use orthogonality conditions above to deduce the differentiability of $\mu(t)$.  Indeed, noting the identity, 
\EQ{
\p_t h(t) &= \p_t w(t, \vec \mu(t)) - \p_t \calW( \vec \s_n, \vec \mu(t)) \\
& = \frac{\mu_{K+1}'(t)}{\mu_{K+1}(t)}  \Lam \chi(\cdot/ \nu(t)) \big( u(t) - u^*(t)) + \dot h(t) + \sum_{j=K+1}^N \s_{n, j} \mu_j'(t) \Lam W_{\U{\mu_j(t)}} , 
}
differentiation of the $j$th orthogonality condition for $h(t)$ gives for each $j = K+1, \dots, N$
\EQ{ \label{eq:mu-sys} 
&\s_{n, j} \mu_j'(t) \ang{ \calZ \mid \Lam W} + \sum_{i \neq j} \s_{n, i} \mu_i'(t) \ang{ \calZ_{\U{\mu_j(t)}} \mid \Lam Q_{\U {\mu_i(t)}}}  \\
&+  \frac{\mu_{K+1}'(t)}{\mu_{K+1}(t)} \ang{ \calZ_{\U{ \mu_j(t)}}\mid \Lam \chi(\cdot/\nu(t)) \big( u(t) - u^*(t)\big)} - \mu_j'(t) \ang{ [r \Lam \calZ]_{\U{\mu_j(t)}} \mid r^{-1} h} \\
&\quad = - \ang{ \calZ_{\U {\mu_j(t)}} \mid \dot h(t)}, 
}
which, using~\eqref{eq:timu-vanish} and ~\eqref{eq:h-small}, is a diagonally dominant first order differential system for $\vec \mu(t)$. Fix any $t_0 \in \cup_n [a_n, b_n]$ so that~\eqref{eq:h-small} holds at  the initial data $\vec \mu(t_0)$. The existence and uniqueness theorem  gives a unique solution $\vec \mu_{\textrm{ode}} \in C^1(J)$ for $J \ni t_0$ a sufficiently small neighborhood. As the scales were uniquely defined using the implicit function theorem at each fixed $t$ and the solution of the ODE preserves the orthogonality conditions, we must have $\vec \mu(t) = \vec \mu_{\textrm{ode}}(t)$ must agree.  Hence $\vec \mu(t) \in C^1$. 
Finally, inverting~\eqref{eq:mu-sys} we obtain the estimates, 
\EQ{
\abs{ \mu_j'(t)} \lesssim \| \dot h \|_{L^2} \lesssim \te_n
}
Using the above with $j = K+1$ yields~\eqref{eq:nu'}. 
This completes the proof. 
\end{proof}

We require a few additional facts related to the scale $\nu(t)$. 
Observe that if $t_n \in [a_n, b_n]$ and $\nu(t_n) \leq R_n \ll \mu_{K+1}(t_n)$, then
\begin{equation}
\label{eq:en-R-2R}
\lim_{n\to \infty}\|\bs u(t_n)\|_{\cE(R_n, 2R_n)} = 0.
\end{equation}
Also, if $\mu_n$ is a positive sequence such that $\lim_{n \to \infty}\bs \delta_{\mu_n}(t_n) = 0$, then
\begin{equation}
\label{eq:en-mu-2mu}
\lim_{n \to \infty}\|\bs u(t_n)\|_{\cE(\frac 12 \mu_n, \mu_n)} = 0.
\end{equation}

Importantly, this choice of $\nu(t)$ give us a way of relating the localized distance $\bs \delta_R$ from Section~\ref{sec:compact} with the global distance $\bfd$ on collision intervals. 
\begin{lem}
\label{lem:trapping-of-d}
There exists a constant $\eta_0 > 0$ having the following property.
Let $t_n \in [a_n, b_n]$ and let $\mu_n$ be a positive sequence satisfying the conditions:
\begin{enumerate}[(i)]
\item $\lim_{n\to\infty}\frac{\mu_n}{\mu_{K+1}(t_n)} = 0$,
\item $\mu_n \geq \nu(t_n)$ or $\|\bs u(t_n)\|_{\cE(\mu_n, \nu(t_n))} \leq \eta_0$,
\item $\lim_{n \to \infty}\bs\delta_{\mu_n}(t_n) = 0$.
\end{enumerate}
Then $\lim_{n\to\infty}\bfd(t_n) = 0$.
\end{lem}
\begin{proof}
Let $R_n$ be a sequence such that $\mu_n \ll R_n \ll \mu_{K+1}(t_n)$.
Without loss of generality, we can assume $R_n \geq \nu(t_n)$, since it suffices to replace $R_n$ by $\nu(t_n)$
for all $n$ such that $R_n < \nu(t_n)$.
Let $M_n, \vec\iota_n, \vec\lambda_n$ be parameters such that
\begin{equation}
\label{eq:conv-delta-iii}
\| u(t_n) - \calW( \vec \iota_n, \vec \lam_n) \|_{H( r \le \mu_n)}^2 + \| \dot u(t_n) \|_{L^2(r \le \mu_n)}^2 + \sum_{j = 1}^{M_n-1} \Big(\frac{ \lam_{n, j}}{ \lam_{n, j+1}}\Big)^k \to 0,
\end{equation}
which exist by the definition of the localized distance function \eqref{eq:delta-def}. Set
\begin{align}
\bs u_n^{(i)} &:= \chi_{\frac 12 \mu_n}\bs u(t_n), \\
\bs u_n^{(o)} &:= (1 - \chi_{R_n})\bs u(t_n), \\
\bs u_n^{(m)} &:= \bs u(t_n) - \bs u_n^{(i)} - \bs u_n^{(o)}.
\end{align}
Invoking \eqref{eq:en-mu-2mu}, we have from \eqref{eq:conv-delta-iii} that
\begin{equation}
\label{eq:conv-uni}
\lim_{n\to \infty}\|\bs u_n^{(i)} - \bs\calW(\vec \iota_n, \vec \lam_n)\|_\cE = 0.
\end{equation}
Assumption (ii), together with \eqref{eq:en-mu-2mu} and \eqref{eq:en-R-2R}, yields
\begin{equation}
\label{eq:small-unm}
\|\bs u_n^{(m)}\|_\cE \leq 2\eta_0, \qquad\text{for all }n\text{ large enough.}
\end{equation}
We also have, again using \eqref{eq:en-mu-2mu} and \eqref{eq:en-R-2R},
\begin{equation}
\label{eq:pyth-unimo}
\limsup_{n\to \infty} \big|E(\bs u(t_n)) - E(\bs u_n^{(i)})  - E(\bs u_n^{(m)}) - E(\bs u_n^{(o)})\big| = 0.
\end{equation}
Since $\lim_{n\to\infty}E(\bs u_n^{(o)}) = (N-K)E(\bs W) + E(\bs u^*)$ and $0 \leq E(\bs u_n^{(m)}) \leq 2\eta_0$, the convergence above yields $M_n = K$ and $\lim_{n\to \infty}E(\bs u_n^{(m)}) = 0$.
Using Sobolev embedding,  we get $\lim_{n\to \infty}\|\bs u_n^{(m)}\|_{\cE} = 0$,
and the result follows.
\end{proof}

\subsection{Basic modulation}
\label{ssec:mod}
On some subintervals of the collision interval $[a_n, b_n]$, mutual interactions between the bubbles
dominate the evolution of the solution. We justify the \emph{modulation inequalities}
allowing to obtain explicit information on the solution on such time intervals.
We stress that in our current approach the modulation concerns only the bubbles from $1$ to $K$.


\begin{lem}[Basic modulation, $D \geq 6$] \label{lem:mod-1}
There exist $C_0, \eta_0 > 0$ and a sequence $\zeta_n \to 0$ such that the following is true.

Let $J \subset [a_n, b_n]$ be an open time interval
such that $\bfd(t) \leq \eta_0$ for all $t \in J$.
Then, there exist $\vec\iota \in \{-1, 1\}^K$ (independent of $t \in J$), modulation parameters $\vec\lam \in C^1(J; (0, \infty)^K)$, 
and $\bs g(t) \in \cE$ satisfying, for all $t \in J$,
\begin{align}  \label{eq:u-decomp} 
\chi( \cdot/ \nu(t))\bs u(t)   &= \bs \calW( \vec\iota, \vec\lambda(t)) + \bs g(t), \\
0 &= \big\la \calZ_{\U{\lam_j(t)}} \mid g(t) \big\ra,  \label{eq:g-ortho} 
\end{align} 
where $\nu(t)$ is as in~\eqref{eq:nu-def}. Define the stable/unstable  components $a_j^-(t)$, $a_j^+(t)$ of $\bs g(t)$ by 
\begin{equation}
a_j^\pm(t) := \big\la \bs\alpha_{\lambda_j(t)}^\pm \mid  \bs g(t) \big\ra,
\end{equation}
where $\bs\alpha_\lambda^\pm$ is as in~\eqref{eq:al-def}. 

The estimates, 
\begin{align}  
C_0^{-1}\bfd(t) - \zeta_n &\leq \|\bs g(t) \|_{\cE} +  \sum_{j=1}^{K-1} \Big( \frac{ \lam_{j}(t)}{\lam_{j+1}(t)} \Big)^{\frac{D-2}{4}} 
\leq C_0\bfd(t) + \zeta_n,  \label{eq:d-g-lam} 
\end{align} 
and
\begin{align} 
\|\bs g(t) \|_{\cE} +  \sum_{j \not \in \calS} \Big( \frac{ \lam_{j}(t)}{\lam_{j+1}(t)} \Big)^{\frac{D-2}{4}} &\leq C_0 \max_{j \in \calS} \left( \frac{ \lam_{j}(t)}{\lam_{j+1}(t)}\right)^{\frac{D-2}{4}}  + \max_{1 \leq i \leq K}|a_i^\pm(t)|+ \zeta_n, \label{eq:g-upper} 
\end{align}
hold, where 
\EQ{\label{eq:calA-def} 
\calS := \big\{j \in \{1, \ldots, K-1\}: \iota_j = \iota_{j+1}\big\}. 
} 
Moreover,  for all $j \in \{1, \ldots, K\}$ and $t \in J$, 
\EQ{ 
\abs{ \lam_j'(t)} &\leq  C_0\|\dot g(t) \|_{L^2} + \zeta_n.   \label{eq:lam'} 
}
and, 
\begin{multline}
\Big| \iota_j \lam_j'(t) +  \frac{1}{ \ang{ \calZ \mid W}} \La \calZ_{\U{\lam_j(t)}} \mid \dot g(t)\Ra\Big| \\ \leq C_0\| \bs g(t) \|_{\cE}^2 
+ C_0 \bigg(\Big( \frac{\lam_{j}(t)}{\lam_{j+1}(t)}\Big)^{\frac{D-4}{2}} + \Big(\frac{ \lam_{j-1}(t)}{\lam_{j}(t)}\Big)^{\frac{D-4}{2}}\bigg) \| \dot g(t) \|_{L^2} + \zeta_n,  \label{eq:lam'-refined}
\end{multline} 
where, by convention, $\lambda_0(t) = 0, \lam_{K+1}(t) = \infty$ for all $t \in J$. Finally,  we have 
 \begin{equation}
\label{eq:dta}
\Big| \dd t a_j^\pm(t) \mp \frac{\kappa}{\lambda_j(t)}a_j^\pm(t) \Big| \leq \frac{C_0}{\lambda_j(t)}\bfd(t)^2 + \frac{\zeta_n}{\lam_j}. 
\end{equation}
\end{lem}

\begin{rem} 
We stress that the scaling in $\bs \al^{\pm}_\lam$ is $\dot H^{-1} \times L^2$-invariant so as to ensure that 
\EQ{ \label{eq:al_j-est}
|a_j^{\pm}(t)| \lesssim \| \bs g(t) \|_{\E}. 
}
We also remark that everything in Lemma~\ref{lem:mod-1}, except for the estimates~\eqref{eq:lam'-refined} and~\eqref{eq:dta}, holds without change in the lower dimensions $D = 4, 5$, and we refer to the appendix for suitable modifications of $\lam_j(t), a_{j}^{\pm}(t)$ in those cases. 
\end{rem}

\begin{proof}[Proof of Lemma~\ref{lem:mod-1}]

\textbf{Step 1:}(The decomposition~\eqref{eq:u-decomp} and the estimates~\eqref{eq:d-g-lam} and~\eqref{eq:g-upper})

First, observe that by Lemma~\ref{lem:ext-sign}, 
\EQ{ \label{eq:E>nu} 
 \sup_{t \in [a_n, b_n]} |E( \bs u(t) - \bs u^*(t) ; \nu(t),  \infty)  - (N-K) E( \bs W)|  = o_n(1) \mas n\to \infty 
}
Since $E( \bs u) = E( \bs u^*) + N E( \bs Q)$ it follows from the above along with~\eqref{eq:nu-prop}, ~\eqref{eq:rho-def}, and~\eqref{eq:u*-interior} that 
\EQ{ \label{eq:E<nu} 
\sup_{t \in [a_n, b_n]}  |E(\bs u(t); 0, 2\nu(t)) - K E( \bs Q)| = o_n(1)  \mas n \to \infty
}
Using continuity of the flow, the fact that $\bfd(t) \le \eta_0$ on $J$,  Lemma~\ref{lem:bub-config}, and by taking $\eta_0>0$ small enough,  we obtain continuous functions $\vec{\ti  \lam}(t) = ( \ti \lam_1(t), \dots, \ti \lam_N(t))$  and signs $\vec {\iota}$ independent of $t \in J$, so that 
\EQ{
\bs u(t) - \bs u^*(t) &=  \bs \calW( \vec {\iota},  \vec{\ti \lam}(t)) + \ti{\bs g}(t),
} 
and, 
\EQ{ \label{eq:tilam} 
\bfd(t)^2 \le \| \ti {\bs g}(t)  \|_{\E}^2 + \sum_{j =1}^{N} \Big(\frac{  \ti \lam_j(t)}{\ti \lam_{j+1}(t)}\Big)^{\frac{D-2}{2}}  \le 4 \bfd(t)^2.
}
with as usual the convention that $\ti \lam_{N+1}(t) = t$. Recalling the properties of $\bs w(t) := (1-\chi_{\nu(t)})( \bs u(t) - \bs u^*(t))$ from Lemma~\ref{lem:ext-sign}, in particular~\eqref{eq:nu-def} and~\eqref{eq:mu_n}, and using Lemma~\ref{lem:bub-config} we see from the above that we must have, 
\EQ{ \label{eq:ti-lam-K+1} 
\Big( \frac{\nu(t)}{ \ti \lam_{K+1}(t)} \Big)^{\frac{D-2}{2}}  \lesssim \bfd(t)^2 + o_n(1) \mas n \to \infty,
} 
Using similar logic along with~\eqref{eq:nu-prop} we see that we also have, 
\EQ{\label{eq:ti-lam-K} 
\Big(\frac{ \ti \lam_K(t)}{ \nu(t)}\Big)^{\frac{D-2}{2}}\lesssim \bfd(t)^2 + o_n(1) \mas n \to \infty
}
Together, the previous two lines mean, roughly speaking, that there are $K$ bubbles to the left of the curve $\nu(t)$ and $N-K$ bubbles to the right of the curve $\nu(t)$. 

 
 For the purposes of this argument we denote by 
 \EQ{ \label{eq:v_n-w_n-def} 
 \bs v(t) &:= \bs u(t) \chi_{\nu(t)}, \quad 
  \bs w(t) := ( \bs u(t) - \bs u^*(t)) ( 1- \chi_{\nu(t)}) ,
 }
We may express $\bs v(t)$  on $J \subset [a_n, b_n]$ as follows, 
 \EQ{
 \bs v (t)  &=  \sum_{j=1}^K \iota_j \bs W_{\ti \lam_j(t)}   
- ( 1- \chi_{\nu(t)}) \sum_{j=1}^K  \iota_j  \bs W_{\ti \lam_j(t)}     + \chi_{\nu(t)} \sum_{j=K+1}^N  \iota_j  \bs W_{\ti \lam_j(t)}  + \chi_{\nu(t)} \bs u^*(t) + \chi_{\nu(t)} \ti {\bs g} (t).
 }
 Using~\eqref{eq:rho-def} along with~\eqref{eq:tilam} and ~\eqref{eq:ti-lam-K}  we see that, 
 \EQ{ \label{eq:v_n-upper} 
 \|  \bs v (t)  - \sum_{j=1}^K  \iota_j  \bs W_{\ti \lam_j(t)}  \|_{\E}^2 + \sum_{j =1}^K  \Big(\frac{  \ti \lam_j(t)}{\ti \lam_{j+1}(t)}\Big)^{\frac{D-2}{2}}  \lesssim \bfd(t)^2 + o_n(1) \mas n \to \infty.
 }
 This means that 
 \EQ{
 \bfd( \bs v(t)) \lesssim \bfd(t) + o_n(1) \mas n \to \infty
 }
 where $\bfd(\bs v)$ is as in the notation of Lemma~\ref{lem:mod-static}. By taking $\eta_0>0$ small enough, and $n$ large enough, we may apply Lemma~\ref{lem:mod-static}, (as well as Lemma~\ref{lem:bub-config}, which ensures the signs $\vec \iota$ stays fixed) at each $t \in J$, to obtain unique $\bs g(t) \in \E$, $\vec \lam (t) \in (0, \infty)^K$ so that 
 \EQ{ \label{eq:lam-def} 
 \bs v(t) &=  \bs \calW( \vec \iota,  \vec \lam(t)) + \bs g(t),  , \quad
 0 = \La\calZ_{\U{\lam_j(t)}} \mid g(t)\Ra, \quad \forall j = 1, \dots, K, 
 }
 where in this formula $\vec \iota, \vec\lam$ are $K$-vectors, i.e., $\vec \iota = ( \iota_1, \dots, \iota_K)$, $\vec \lam(t) = ( \lam_1(t), \dots, \lam_K(t))$.  We note the estimate, 
 \EQ{ \label{eq:d_K} 
 \bfd( \bs v_n(t))^2 \le \| \bs g(t)\|_{\E}^2 + \sum_{j =1}^{K-1} \Big( \frac{ \lam_{j}(t)}{\lam_{j+1}(t)} \Big)^{\frac{D-2}{2}}  +\Big( \frac{\lam_K(t)}{ \nu(t)} \Big)^{\frac{D-2}{2}} &\le 4 \bfd( \bs v(t))^2 + o_n(1)  \\
&  \lesssim \bfd(t)^2 + o_n(1), 
 }
 as $n \to \infty$. Next, using~\eqref{eq:E<nu} we see that 
 \EQ{
 E( \bs v)  \le K E( \bs Q) + o_n(1).
 }
 Therefore, the estimate~\eqref{eq:g-bound-A} from Lemma~\ref{lem:mod-static} applied here yields, 
 \EQ{
 \| \bs g(t)  \|_{\E}^2 \lesssim  \sup_{j \in \calS} \Big( \frac{ \lam_{j}(t)}{\lam_{j+1}(t)} \Big)^{\frac{D-2}{2}}  + \max_{1 \le i \le K} | a_j^\pm(t)| + o_n(1) 
 }
 where $\calS = \{ j \in \{1, \dots, K-1\} \, : \, \iota_j = \iota_{j+1} \}$,  proving~\eqref{eq:g-upper}. 
 
Next, we prove the lower bound in~\eqref{eq:d-g-lam}. Note the identity, 
 \EQ{ \label{eq:u-u^*-1} 
\bs u(t) - \bs u^*(t) &=  \bs v(t)   +  \bs w(t)  - \chi_{\nu(t)} \bs u^*(t) \\
&=     \sum_{j=1}^K \iota_j  W_{\lam_j(t)}  + \sum_{j= K+1}^N \s_{n , j}\bs W_{\mu_{ j}(t)}   + \bs g(t) + \bs h(t) -  \chi_{\nu_n(t)} \bs u^*(t) 
 }
 which follows from~\eqref{eq:v_n-w_n-def} and~\eqref{eq:w_n}.
 First we prove that $(\iota_{K+1}, \dots, \iota_N) = ( \sigma_{K+1}, \dots, \sigma_N)$. From~\eqref{eq:w_n} and~\eqref{eq:mu_n} we see that 
\EQ{
\| \bs w(t) - \sum_{j= K+1}^N \s_{n , j}\bs W_{\mu_{j}(t)}  \|_{\E}^2 + \Big(\frac{ \nu(t)}{ \mu_{ K+1}(t)}\Big)^{\frac{D-2}{2}} + \sum_{j = K+1}^N \Big( \frac{ \mu_{ j}(t)}{\mu_{j+1}(t)} \Big)^{\frac{D-2}{2}}  = o_n(1) \mas n \to \infty .
}
On the other hand, we see from~\eqref{eq:ti-lam-K+1} that, 
\EQ{
 \| \bs w(t)  - \sum_{j= K+1}^N \iota_j \bs W_{\ti \lam_{j}(t)}  \|_{\E}^2 + \Big(\frac{ \nu(t)}{ \ti \lam_{ K+1}(t)}\Big)^{\frac{D-2}{2}} + \sum_{j = K+1}^N \Big( \frac{ \ti \lam_{j}(t)}{\ti \lam_{ j+1}(t)} \Big)^{\frac{D-2}{2}}  \lesssim \bfd(t)^2 + o_n(1) .
}
Hence, using Lemma~\ref{lem:bub-config} we see that for any $\te_0>0$ we may take $\eta_0>0$ small enough so that $(\iota_{K+1}, \dots, \iota_N) = ( \sigma_{K+1}, \dots, \sigma_N)$, and in addition we have 
\EQ{ \label{eq:tilam-mu} 
\abs{\frac{ \ti \lam_{j}(t)}{\mu_{n, j}(t)} - 1} \le \te_0 \quad \forall j = K+1, \dots, N.
}
The above, together with~\eqref{eq:mu_n}  implies that 
\EQ{
\sum_{j=K+1}^{N} \Big(\frac{  \ti \lam_j(t)}{\ti \lam_{j+1}(t)}\Big)^{\frac{D-2}{2}}  =  o_n(1) \mas n \to \infty .
}
We may thus rewrite~\eqref{eq:u-u^*-1} as 
 \EQ{ \label{eq:u-u*-1} 
\bs u(t) - \bs u^*(t) 
&=      \sum_{j=1}^K \iota_j \bs W_{\lam_j(t)}  + \sum_{j= K+1}^N \iota_ j\bs W_{\mu_{ j}(t)}   + \bs g(t) + \bs h(t) -  \chi_{\nu(t)} \bs u^*(t) 
 }
Noting that    
\EQ{
\sup_{t \in[a_n, b_n]}\|  \bs u^*(t) \chi_{\nu(t)} \|_{\E} = o_n(1) \mas n \to \infty, 
}
the previous line together with~\eqref{eq:d_K} and~\eqref{eq:mu_n} imply 
that, 
\EQ{
\bfd(t)^2  \lesssim \bfd( \bs v(t))^2 + o_n(1) \lesssim \| \bs g(t)\|_{\E}^2 + \sum_{j =1}^{K-1} \Big( \frac{ \lam_{j}(t)}{\lam_{j+1}(t)} \Big)^{\frac{D-2}{2}}  + o_n(1) \mas n \to \infty, 
}
which proves the lower bound in~\eqref{eq:d-g-lam}. 

\textbf{Step 2:}(The dynamical estimates~\eqref{eq:lam'}, ~\eqref{eq:lam'-refined}, and~\eqref{eq:dta}) Momentarily assuming that $\vec \lam \in C^1(J)$ (we will justify this assumption below)  we record the computations, 
\EQ{
\p_t v(t) &=  \dot g(t)   - \frac{\nu'(t)}{\nu(t)}  (r \p_r \chi)( \cdot/ \nu(t))u(t)  ,   \quad 
\p_t \calW( \vec \iota, \vec \lam(t)) =   - \sum_{j=1}^K \iota_j \lam_{j}'(t) \Lam W_{\U{\lam_j(t)}} , 
}
which lead to the expression, 
\EQ{ \label{eq:dt-g-1} 
\p_t g(t) = \dot g(t)  + \sum_{j=1}^K \iota_j \lam_{j}'(t) \Lam W_{\U{\lam_j(t)}}  - u(t) \frac{\nu'(t)}{\nu(t)}  (r \p_r \chi)( \cdot/ \nu(t)) . 
}
We differentiate the orthogonality conditions~\eqref{eq:g-ortho} for each $j = 1, \dots, K$,  
\EQ{
0 &= -\frac{\lam_j'}{\lam_j} \ang{ \ULam \calZ_{\U{\lam_j}} \mid g} + \ang{ \calZ_{\U{\lam_j}} \mid \p_t g}  \\
& = -\frac{\lam_j'}{\lam_j} \ang{ \ULam \calZ_{\U{\lam_j}} \mid g} + \ang{ \calZ_{\U{\lam_j}} \mid \dot g} + \sum_{\ell=1}^K  \iota_{\ell} \lam_{\ell}' \ang{ \calZ_{\U{\lam_{j}}} \mid \Lam W_{\U{\lam_{\ell}} }}  - \frac{\nu'}{\nu} \ang{ \calZ_{\U{\lam_j}} \mid  u  (r \p_r \chi)( \cdot/ \nu) }, 
}
which we rearrange into the system, 
\begin{multline}  \label{eq:lam'-system} 
\iota_j \lam_j' \Big( \ang{\calZ \mid \Lam W }  -  \lam_j^{-1}\La  \ULam \calZ_{\U{\lam_j}} \mid   g\Ra \Big) + \sum_{i \neq j}  \iota_{i} \lam_{i}' \ang{ \calZ_{\U{\lam_{j}}} \mid \Lam W_{\U{\lam_{i}} }}  \\
=  - \ang{ \calZ_{\U{\lam_j}} \mid \dot  g} +  \frac{\nu'}{\nu} \ang{ \calZ_{\U{\lam_j}} \mid  u  (r \p_r \chi)( \cdot/ \nu) }. 
\end{multline} 
This is a diagonally dominant system, hence invertible, and we arrive at the estimate, 
\EQ{ \label{eq:lam_j'} 
\abs{ \lam_j'} & \lesssim \| \dot g \|_{L^2} + o_n(1)  \quad j = 1, \dots, K, 
}
after noting the estimates, 
\EQ{
\abs{ \ang{ \calZ_{\U{\lam_j}} \mid \dot g}} & \lesssim \|\dot g \|_{L^2}  \\
\abs{ \frac{\nu'(t)}{\nu(t)} \ang{ \calZ_{\U{\lam_j}} \mid  u(t)  (r \p_r \chi)( \cdot/ \nu(t)) }} &\lesssim \abs{ \nu'}\frac{\lam_j}{\nu} \| r^{-1}  u(t) (r \p_r \chi)( \cdot/\nu(t)) \|_{L^2}  = o_n(1) , 
}
where the last line follows from ~\eqref{eq:nu'}. Lastly, we note that the system~\eqref{eq:lam'-system} implies that $\vec \lam(t)$ is a $C^1$ function on $J$. Indeed, arguing as in the end of the proof of Lemma~\ref{lem:ext-sign}, let $t_0 \in J$ be any time and let $\vec \lam(t_0)$ be defined as in~\eqref{eq:lam-def}. Using the smallness~\eqref{eq:d_K} at time $t_0$, the system~\eqref{eq:lam'-system} admits a unique $C^1$ solution $\vec\lam_{\textrm{ode}}(t)$ in a neighborhood of $t_0$. Due to the way the system~\eqref{eq:lam'-system} was derived, the orthogonality conditions in~\eqref{eq:lam-def} hold with $\vec \lam_{\textrm{ode}}(t)$. Since $\vec\lam(t)$ was obtained uniquely via the implicit function theorem, we must have $\vec \lam(t) = \vec \lam_{\textrm{ode}}(t)$, which means that $\vec \lam(t)$ is $C^1$. 

The estimates~\eqref{eq:lam'-refined} are immediate from~\eqref{eq:lam'-system} using~\eqref{eq:lam_j'} along with the estimates, 
\EQ{
 \ang{ \calZ_{\U{\lam_{j}}} \mid \Lam Q_{\U{\lam_{i}} }}  & \lesssim \begin{cases} \Big( \frac{\lam_{j}}{\lam_i} \Big)^{\frac{D}{2}} \mif  j < i \\ \Big( \frac{\lam_{i}}{\lam_j} \Big)^{\frac{D-4}{2}} \mif j>i \end{cases}  \\
  \Big| \lam_j^{-1}\ang{ \U\Lam \calZ_{\U{\lam_j}} \mid  g} \Big|  &\lesssim  \| g \|_H. 
}
using here that $D -4 \ge \frac{D-2}{2}$ as long as $D \ge 6$. 

Lastly, we consider the estimates~\eqref{eq:dta}. We first write the equation for $\bs g(t)$.  
\EQ{
\p_t \bs g = \p_t \bs v  - \p_t  \bs \calW( \vec \iota, \vec \lam) 
}
Noting that 
\EQ{
\p_t \bs v &= \chi( \cdot/ \nu) \p_t \bs u - \frac{\nu'}{\nu} (r \p_r \chi)( \cdot/ \nu) \bs u  \\
& = \chi( \cdot/ \nu)  J \circ \uD E( \bs u) - \frac{\nu'}{\nu} (r \p_r \chi)( \cdot/ \nu) \bs u \\
& =  J \circ \uD E( \chi(\cdot/ \nu) \bs u) + \Big(\chi( \cdot/ \nu)  J \circ \uD E( \bs u) - J \circ \uD E( \chi(\cdot/ \nu) \bs u)\Big) - \frac{\nu'}{\nu} (r \p_r \chi)( \cdot/ \nu) \bs u 
}
we arrive at, 
\EQ{ \label{eq:g-eq-ham} 
\p_t \bs g &= J \circ \uD E( \bs  \calW(\vec \iota, \vec \lam) + \bs g) - \p_t  \bs \calW( \vec \iota, \vec \lam) \\
&\quad + \Big(\chi( \cdot/ \nu)  J \circ \uD E( \bs u) - J \circ \uD E( \chi(\cdot/ \nu) \bs u)\Big) - \frac{\nu'}{\nu} (r \p_r \chi)( \cdot/ \nu) \bs u 
}
 We compute, 
\EQ{
\frac{\ud}{\ud t} a_{j}^- = \La \p_t  \bs \al_{\lam_j}^- \mid \bs g\Ra + \La \bs \al_{\lam_j}^- \mid \p_t \bs g \Ra
}
Expanding the first term on the right gives, 
\EQ{
\La \p_t  \bs \al_{\lam_j}^- \mid \bs g\Ra  &= \frac{\kappa}{2}\La \p_t( \lam_j^{-1} \calY_{\U{\lam_j}}) \mid g \Ra + \frac{1}{2} \La \p_t( \calY_{\U{\lam_j}}) \mid \dot g \Ra \\
& = -\frac{\kappa}{2}\frac{\lam_j'}{\lam_j}\La  \lam_j^{-1} \calY_{\U{\lam_j}} + \frac{1}{\lam_j} ( \ULam \calY)_{\U{\lam_j}} \mid g \Ra - \frac{1}{2} \frac{\lam_j'}{\lam_j} \La  (\ULam  \calY)_{\U{\lam_j}}) \mid \dot g \Ra 
}
and thus, 
\EQ{
\Big|\La \p_t  \bs \al_{\lam_j}^- \mid \bs g\Ra \Big| \lesssim \frac{1}{\lam_j}( \bfd(t)^2 + o_n(1)) . 
}
We use~\eqref{eq:g-eq-ham} to expand the second term, 
\EQ{ \label{eq:a'-exp} 
\La \bs \al_{\lam_j}^- \mid \p_t \bs g \Ra &= \La \bs \al_{\lam_j}^- \mid J \circ \uD^2 E( \bs  \calW(\vec \iota, \vec \lam)) \bs g  \Ra  \\
&\quad +\La \bs \al_{\lam_j}^- \mid J \circ \Big( \uD E( \bs  \calW(\vec \iota, \vec \lam) + \bs g)  - \uD^2 E( \bs  \calW(\vec \iota, \vec \lam)) \bs g\Big)\Ra \\
&\quad -  \La \bs \al_{\lam_j}^- \mid \p_t \bs \calW( \vec \iota, \vec \lam) \Ra \\
&\quad + \La \bs \al_{\lam_j}^- \mid  \Big(\chi( \cdot/ \nu)  J \circ \uD E( \bs u) - J \circ \uD E( \chi(\cdot/ \nu) \bs u)\Big)  \Ra \\
&\quad - \frac{\nu'}{\nu}\La \bs \al_{\lam_j}^- \mid  (r \p_r \chi)( \cdot/ \nu) \bs u\Ra
}
By~\eqref{eq:alpha-pair} the first term on the right gives the leading order, 
\EQ{
\La \bs \al_{\lam_j}^- \mid J \circ \uD^2 E( \bs  \calW(\vec \iota, \vec \lam)) \bs g  \Ra = -\frac{\kappa}{\lam_j} a_j^-.
}
Next, we expand, 
\EQ{
\La \bs \al_{\lam_j}^- \mid J \circ \Big( \uD E( \bs  \calW(\vec \iota, \vec \lam) + \bs g)  &- \uD^2 E( \bs  \calW(\vec \iota, \vec \lam)) \bs g\Big)\Ra  \\
&= -\frac{1}{2} \La \calY_{\U \lam_j} \mid f( \calW( \vec \iota, \vec \lam) + g) - f( \calW( \vec \iota, \vec \lam))  - f'( \calW( \vec \iota, \vec \lam) )g \Ra \\
&\quad - \frac{1}{2} \La \calY_{\U \lam_j} \mid f( \calW( \vec \iota, \vec \lam))  - \sum_{i =1}^K \iota_i f(W_{\lam_i}) \Ra.
}
The first line satisfies, 
\EQ{
\Big| \La \calY_{\U \lam_j} \mid f( \calW( \vec \iota, \vec \lam) + g) - f( \calW( \vec \iota, \vec \lam))  - f'( \calW( \vec \iota, \vec \lam) )g \Ra \Big| \lesssim \frac{1}{\lam_j} (\bfd(t)^2 + o_n(1)).
}
Noting that $f( \calW( \vec \iota, \vec \lam))  - \sum_{i =1}^K \iota_i f(W_{\lam_i}) = f_{\bfi}(\vec \iota, \vec \lam)$, the same argument used to prove Lemma~\ref{lem:interaction} gives, 
\EQ{
\Big| \La \calY_{\U \lam_j} \mid  f_{\bfi}(\vec \iota, \vec \lam) \Ra\Big|  \lesssim \frac{1}{\lam_j} \Big( \Big(\frac{\lam_{j}}{\lam_{j+1}} \Big)^{\frac{D-2}{2}} + \Big(\frac{\lam_{j-1}}{\lam_{j}} \Big)^{\frac{D-2}{2}} \Big) \lesssim \frac{1}{\lam_j} (\bfd(t)^2 + o_n(1)).
}
Consider now the third line in~\eqref{eq:a'-exp}. 
 \EQ{
 -  \La \bs \al_{\lam_j}^- \mid \p_t \bs \calW( \vec \iota, \vec \lam) \Ra &=  \frac{\kappa}{2} \iota_j \frac{ \lam_j' }{\lam_j} \La \calY_{\U{\lam_j}}  \mid  \Lam W_{\U{\lam_j}} \Ra  + \sum_{i \neq j}  \iota_i \frac{\kappa}{2} \frac{\lam_i' }{\lam_j} \La \calY_{\U{\lam_j}}  \mid \Lam W_{\U{\lam_i}} \Ra  \\
 &= \sum_{i \neq j}  \iota_i \frac{\kappa}{2} \frac{\lam_i' }{\lam_j} \La \calY_{\U{\lam_j}}  \mid \Lam W_{\U{\lam_i}} \Ra 
 }
where in the last equality we used the vanishing $\La \calY \mid \Lam W \Ra$. Noting the estimates, 
\EQ{
\Big| \La \calY_{\U{\lam_j}}  \mid \Lam W_{\U{\lam_i}} \Ra  \Big| \lesssim \begin{cases} \Big(\frac{\lam_{i}}{\lam_j} \Big)^{\frac{D-4}{2} } \mif i < j \\\Big(\frac{\lam_{i}}{\lam_j} \Big)^{\frac{D}{2} } \mif i>j 
\end{cases} 
}
Using here the fact that $D \ge 6$, we obtain, 
\EQ{
\Big| \La \bs \al_{\lam_j}^- \mid \p_t \bs \calW( \vec \iota, \vec \lam) \Ra \Big| \lesssim  \frac{1}{\lam_j} (\bfd(t)^2 + o_n(1)).
}
Finally, using~\eqref{eq:nu-prop} and~\eqref{eq:nu'} we see that the last two lines of~\eqref{eq:a'-exp} satisfy, 
\EQ{
\Big| \La \bs \al_{\lam_j}^- \mid  \Big(\chi( \cdot/ \nu)  J \circ \uD E( \bs u) - J \circ \uD E( \chi(\cdot/ \nu) \bs u)\Big)  \Ra  \Big|  \lesssim  \frac{1}{\lam_j} o_n(1),  \\
\Big| \frac{\nu'}{\nu}\La \bs \al_{\lam_j}^- \mid  (r \p_r \chi)( \cdot/ \nu) \bs u\Ra \Big| \lesssim \frac{1}{\lam_j} o_n(1). 
}
This completes the proof. 
%
%
%
\end{proof}

\subsection{Refined modulation}  \label{ssec:ref-mod} 

Next, our goal is to gain precise dynamical control of the modulation parameters in the spirit of \cite{JJ-APDE, JL1}.
The idea is to construct a virial correction to the modulation parameters  (see~\eqref{eq:beta-def}). The idea of adding a correction term based on underlying symmetries (in our case scaling) to modulation parameters originates in Rapha\"el and Szeftel~\cite[Proposition 4.3]{RS}. 
We start by  finding suitable truncation of the function $\frac{1}{2} r^2$, 
similar to~\cite[Lemma 3.10]{JJ-AJM}. Since here we may have arbitrary number of bubbles, we need to localize this function both away from $r =0$ and away from $r =\infty$. 


\begin{lem} \label{lem:q} 
For any $c > 0$ and $R > 1$ there exists a function $q = q_{c, R} \in C^4((0, \infty))$
having the following properties:
\begin{enumerate}[(P1)]
\item $q(r) = \frac 12 r^2$ for all $x \in \bR^D$ such that $r \in [R^{-1}, R]$,
\item there exists $\wt R > 0$ (depending on $c$ and $R$)
such that $q(r) = \tx{const}$ for $r \geq \wt R$ and $q(r) = \tx{const}$ for $r \leq \wt R^{-1}$,
\item $|q'(r)| \lesssim r$ and $|q''(r)| \lesssim 1$ for all $r \in(0, \infty)$,
with constants independent of $c$ and $R$,
\item $ q''(r) + \frac{D-1}{r} q'(r) \ge - c$ for all $r \in (0, \infty)$. 
\item $|\Delta^2 q(r)| \leq c r^{-2}$.
\item $\Big| \Big( \frac{q'(r)}{r} \Big)' \Big| \le c r^{-1}$ for all $r>0$. 
\end{enumerate}
\end{lem}
\begin{proof} See~\cite[Proof of Lemma 4.13]{JL6}. The exact same function can be used here. 
\end{proof}

\begin{defn}[Localized virial operator]
For each $\lam>0$ we set
\begin{align}
A(\lambda)g(r) &:= q'\big(\frac{r}{\lambda}\big)\cdot \p_r g(r) + \frac{D-2}{2D}\frac{1}{\lam} \De q \big(\frac{r}{ \lam}\big) g(r), \label{eq:opA}\\
\uln A(\lambda)g(r) &:=
 q'\big(\frac{r}{\lambda}\big)\cdot\p_r g(r) + \frac{1}{2}\frac{1}{\lam} \De q \big(\frac{r}{ \lam}\big) g(r) \label{eq:opA0}
\end{align}
where here $\De = \p_r^2 + \frac{D-1}{r}\p_r $. These operators depend on $c$ and $R$ as in Lemma~\ref{lem:q}. 
\end{defn}
Note the similarity between $A$ and $\frac{1}{\lambda} \Lam$ and between $\U A$ and $\frac{1}{\lam} \ULam$. For technical reasons we introduce the space 
\EQ{
X:= \{ \bs g \in \E \mid  (\p_r g, \dot g) \in \E\}.
}

\begin{lem}[Localized virial estimates]  \emph{\cite[Lemma 3.12]{JJ-AJM}} \label{lem:A} 
For any $c_0>0$ there exist $c_1, R_1>0$, so that for all $c, R$ as Lemma~\ref{lem:q} with $c< c_1$, $R> R_1$ the operators $A(\lambda)$ and $A_0(\lambda)$ defined in~\eqref{eq:opA} and~\eqref{eq:opA0} have the following properties:

  \begin{itemize}[leftmargin=0.5cm]
    \item the families $\{A(\lambda): \lam > 0\}$, $\{\U A(\lambda): \lam> 0\}$, $\{\lambda\partial_\lambda A(\lambda): \lam > 0\}$
      and $\{\lambda\partial_\lambda \U A(\lambda): \lambda > 0\}$ are bounded in $\mathscr{L}(\dot H^1; L^2)$, with the bound depending only on the choice of the function $q(r)$,
    \item  
   Let $\bs g_1 = \bs \calW(  \vec \iota, \vec \lam)$ be an $M$-bubble configuration and let $\bs g_2 \in X$. Then, for each $\lam_j$, 
       $j \in \{1, \dots, M\}$ we have 
      \EQ{ \label{eq:A-by-parts} 
      \La A( \lam_j) g_1 \mid (f( g_1 + g_2) - f(g_1) - f'(g_1) g_2 \Ra = - \La A( \lam_j) g_2 \mid f( g_1 + g_2) - f(g_1) \Ra
      }
    \item For all $\bs g \in X$ we have  
\EQ{
        \label{eq:pohozaev}
        \ang{\U A(\lambda)g \mid -\De g} \ge  - \frac{c_0}{\lambda}\| \bs g\|_{\E}^2 + \frac{1}{\lambda}\int_{R^{-1} \lam}^{R\lambda}  (\partial_r g)^2 \, r^{D-1} \, \ud r, 
        }
        \item  For $\lam, \mu >0$ with either $\lam/\mu \ll 1$ or $\mu/\lam \ll 1$, 
\begin{align} 
      \label{eq:L0-A0}
      \|\ULam \Lambda W_\uln{\lambda} - \U A(\lambda)\Lambda W_{\lambda}\|_{L^2} &\leq c_0, \\
      \label{eq:L-A}
      \|\big(\frac{1}{\lam} \Lambda  - A(\lambda)\big) W_\lambda\|_{L^{\frac{2D}{D-2}}} &\leq \frac{c_0}{\lambda},  \\
    \| A(\lam) W_\mu \|_{L^{\frac{2D}{D-2}}} + \| \U A(\lam) W_\mu \|_{L^{\frac{2D}{D-2}}}  &\lesssim  \frac{1}{\lam}  \min \{ (\lam/ \mu)^{\frac{D-2}{2}},  (\mu/ \lam)^{\frac{D-2}{2}} \}  \label{eq:A-mismatch} \\
     \| A(\lam) \Lam W_\mu \|_{L^2} + \| \U  A(\lam) \Lam W_\mu \|_{L^2} & \lesssim \min \{ (\lam/ \mu)^{\frac{D-2}{2}},  (\mu/ \lam)^{\frac{D-2}{2}} \} \label{eq:A-mismatch-2}
     \end{align} 
     
     \item  Lastly, the following localized coercivity estimate holds. Fix any smooth function $\calZ \in L^2 \cap X$ such that $\ang{\calZ \mid \Lam W} >0$ and $\ang{ \calZ \mid \calY} = 0$. For any $\bs g \in \E, \lam>0$ with $\ang{g \mid \calZ_{\U \lam}} = 0$, 
     \EQ{ \label{eq:coercive-virial}
    \frac{1}{\lambda}\int_{R^{-1} \lam}^{R\lambda}  (\partial_r g)^2 r^{D-1} \, \ud r -    \frac{1}{\lam} \int_{0}^\infty  \frac{1}{D} \De q\big( \frac{r}{\lam} \big)  f'(W_{\lam}) g^2  \, r^{D-1}\, \ud r  
   & \ge - \frac{c_0}{\lam} \| \bs g  \|_{\E}^2 \\
    &\quad -\frac{C_0}{\lam} \La\frac{1}{\lam}\calY_{\U \lam}  \mid  g\Ra^2 .
     }
  \end{itemize}
\end{lem}

\begin{proof}
See~\cite[Lemma 3.12]{JJ-AJM} for the proof for $D =6$ and~\cite[Lemma 4.13]{JL6} for the case of $k$-equivariant wave maps.  That argument generalizes in a straightforward way to all $D \ge 4$. 
\end{proof}

The modulation parameters $\vec \lam(t)$ defined in Lemma~\ref{lem:mod-1} are imprecise proxies for the dynamics in the cases $3 \le D \le 6$  due to the fact that the orthogonality conditions were imposed relative to $\calZ \neq \Lam Q$ (note that we will treat the cases $D=3, 4, 5$ in the appendix). Indeed, we use~\ref{eq:g-ortho} primarily to ensure coercivity, and thus the estimate~\eqref{eq:g-upper}, as well as the differentiability of $\vec\lam(t)$. To access the dynamics of~\eqref{eq:nlw} we  introduce a correction $\vec \xi(t)$ defined as follows. For each $t \in J \subset [a_n, b_n]$ as in Lemma~\ref{lem:mod-1} set, 
\EQ{\label{eq:xi-def} 
\xi_j(t) := \begin{cases} \lam_j(t) \mif D \ge 7 \\ \lam_j(t) - \frac{\iota_j}{\| \Lam W \|_{L^2}^2} \La \chi(  \cdot/ L\lam_j(t)) \Lam W_{\U{\lam_j(t)}} \mid g(t) \Ra  \mif  D= 6\end{cases} 
}
for each $j = 1, \dots, K-1$, and where $L>0$ is a large parameter to be determined below. (Note that for $j =K$ we only require the brutal estimate~\eqref{eq:lam'}). We require yet another modification, since the dynamics  of~\eqref{eq:nlw} truly enter after taking two derivatives of the modulation parameters and it is not clear how to derive useful estimates from the expression for $\xi_j''(t)$. So we introduce a refined modulation parameter, which we view as a subtle correction to $\xi_j'(t)$.  For each $t \in J \subset [a_n, b_n]$ as in Lemma~\ref{lem:mod-1} and for each $j \in \{1, \dots, K\}$ define, 
\begin{equation} \label{eq:beta-def} 
\beta_j(t) :=- \frac{\iota_j}{ \| \Lam W \|_{L^2}^2} \La \Lam W_{\U{\lam_j(t)}} \mid \dot g(t)\Ra  -   \frac{1}{ \| \Lam W \|_{L^2}^2} \ang{ \uln A( \lam_j(t)) g(t) \mid \dot g(t)}.
\end{equation}
The function $\be_j(t)$ is identical to the function $\beta_j(t)$ in~\cite{JL6} and similar to the function called $b(t)$ in~\cite{JL1}. 

%
%

\begin{lem}[Refined modulation] 
\label{cor:modul}
Let $D \ge 6$ and $c_0\in (0 ,1)$.  There exist $\eta_0=\eta(c_0)>0$, $L_0 = L_0(c_0)$, as well as $c= c(c_0)>0, R=R(c_0)>1$ as in Lemma~\ref{lem:q}, a constant $C_0>0$, and a decreasing sequence $ \de_n \to 0$ so that the following is true. 
Let  $J \subset [a_n, b_n]$ be an open time interval with 
\EQ{
\de_n \le \bfd(t) \le \eta_0 
}
for all $t \in J$. Let $\calS :=  \{ j \in \{1, \dots, K-1\} \mid \iota_{j}  = \iota_{j+1} \}$.  Then, for all $t \in J$, 
\EQ{\label{eq:g-bound} 
\| \bs g(t) \|_{\E} + \sum_{i \not \in \calS}  \big(\lambda_i(t) / \lambda_{i+1}(t)\big)^{\frac{D-2}{4}} \le C_0 \max_{ i  \in \calS}  \big(\lambda_i(t) / \lambda_{i+1}(t)\big)^{\frac{D-2}{4}} + \max_{1 \leq i \leq K}|a_i^\pm(t)|, 
}
and, 
\begin{equation}
\label{eq:d-bound-nlw}
\frac{1}{C_0}\bfd(t) \leq \max_{i \in \cS}\big(\lambda_i(t) / \lambda_{i+1}(t)\big)^{\frac{D-2}{4}}
+ \max_{1 \leq i \leq K}|a_i^\pm(t)| \leq C_0 \bfd(t).
\end{equation}
Moreover, for all $j \in \{1, \dots, K-1 \}$,  $t \in J$, and $L \ge L_0$, 
\begin{equation}\label{eq:xi_j-lambda_j} 
|\xi_j(t) / \lambda_j(t) - 1| \leq c_0, 
\end{equation}
\begin{equation}\label{eq:lam_j'-beta_j} 
|\xi_j'(t) - \beta_j(t)| \leq c_0 \bfd(t) 
\end{equation}
and,  
\EQ{ \label{eq:beta_j'} 
 \beta_{j}'(t) &\ge  \Big( \iota_j \iota_{j+1}\om^2 -   c_0\Big) \frac{1}{\lam_{j}(t)} \left( \frac{\lam_{j}(t)}{\lam_{j+1}(t)} \right)^{\frac{D-2}{2}} +    \Big(-\iota_j \iota_{j-1}\om^2 -  c_0\Big) \frac{1}{\lam_{j}(t)} \left( \frac{\lam_{j-1}(t)}{\lam_{j}(t)} \right)^{\frac{D-2}{2}}   \\
&\quad  -  \frac{c_0}{\lam_j(t)}  \bfd(t)^2 - \frac{C_0}{\lam_j(t)} \Big( (a_j^+(t))^2 + (a_j^-(t))^2 \Big),   
}
where, by convention, $\lambda_0(t) = 0, \lam_{K+1}(t) = \infty$ for all $t \in J$, and $\om^2>0$ is defined by 
\EQ{\label{eq:omega-def} 
\omega^2 = \om^2(D): = \frac{D-2}{2D} (D(D-2))^{\frac{D}{2}}\| \Lam W \|_{L^2}^{-2}>0.
}
Finally, for each $j \in \{1, \dots, K\}$, 
\EQ{ \label{eq:dta-D} 
\Big| \dd t a_j^\pm(t) \mp \frac{\kappa}{\lambda_j(t)}a_j^\pm(t) \Big| \leq \frac{C_0}{\lambda_j(t)}\bfd(t)^2. 
}
\end{lem}

\begin{rem}
\label{rem:deltan}
Without loss of generality (upon enlarging $\eps_n$) we can assume that $\eps_n \ge \de_n$ 
so that Lemma~\ref{cor:modul} can always be applied on the time intervals
$J \subset [a_n, b_n]$ as long as $\bfd(t) \leq \eta_0$ on $J$.  
\end{rem}

Before proving Lemma~\ref{cor:modul} we rewrite the equation satisfied by $\bs g(t)$ in~\eqref{eq:g-eq-ham} in components as follows, 
\EQ{ \label{eq:g-eq} 
\p_t g(t) &= \dot g(t)  + \sum_{j=1}^K \iota_j \lam_{j}'(t) \Lam W_{\U{\lam_j(t)}}  + \phi( u(t), \nu(t))   \\
\p_t \dot g(t) &=  - \LL_{\calW} g +f_{\bfi}( \iota, \vec \lam) + f_{\bfq}( \vec \iota,  \vec \lam, g) + \dot \phi( u(t), \nu(t)), 
}
%
where 
\EQ{\label{eq:h-def} 
\phi(u, \nu) &= - u \frac{\nu'}{\nu} ( r \p_r \chi)( \cdot/ \nu) \\
\dot \phi( u, \nu) &= - \p_t u \frac{\nu'}{\nu} ( r \p_r \chi)( \cdot / \nu) - (r^2 \De \chi)( \cdot/ \nu)  r^{-2} u  \\
&\quad - 2 \frac{1}{r} (r \p_r \chi)( \cdot/ \nu)  \p_r u + \chi( \cdot/ \nu) f(u) - f( \chi( \cdot/ \nu) u)
}
which we note are supported in $r  \in ( \nu, \infty)$, 
\EQ{
f_{\bfi}(\vec \iota,  \vec \lam) &:= f \Big( \sum_{i = 1}^K \iota_i W_{\lam_i}\Big) - \sum_{i =1}^K \iota_i f(W_{\lam_i})\\
f_{\bfq}( \vec \iota, \vec \lam, g)& = f\Big( \sum_{i = 1}^K \iota_i W_{\lam_i} + g\Big)  - f\Big( \sum_{i = 1}^K \iota_i W_{\lam_i} \Big)  - f'\Big( \sum_{i = 1}^K \iota_i W_{\lam_i}\Big) g  . 
}
The subscript  $\bfi$ above stands for ``interaction'' and $\bfq$ stands for ``quadratic.'' 
For $\|\bs g \|_{\E} \le 1$, the term $f_{\bfq}( \vec \iota, \vec \lam)$ satisfies, 
\EQ{ \label{eq:fq-quadratic} 
\| f_{\bfq}( \vec \iota, \vec \lam) \|_{L^{\frac{2D}{D+2}}} \lesssim \| \bs g \|_{\E}^2 + \|\bs g\|_{\cE}^{\frac{D+2}{D-2}}. 
}
The proof of~\eqref{eq:fq-quadratic} follows the fact that the pointwise estimates, 
\EQ{
| f(x_1 + x_2) - f(x_1) - f'(x_1) x_2| \le \abs{z_2}^2 \mif D =6
}
and if $D \ge 7$, 
\EQ{ \label{eq:fq-ineq} 
| f(x_1 + x_2) - f(x_1) - f'(x_1) x_2| \le  \begin{cases}  \abs{x_2}^{\frac{D+2}{D-2}} \quad \forall x_1, x_2 \in \R  \\  \abs{x_1}^{-\frac{D-6}{D-2}} \abs{x_2}^2 \mif x_1 \neq 0 \end{cases} 
}
See~\cite[Lemma 2.1]{JJ-APDE} for the proof of the previous two estimates. 

In one instance it will be convenient to write the equation for $\dot g$ as follows, 
\EQ{ \label{eq:g-eq-alt} 
\p_t \dot g = \De g + f_{\bfi}(\vec \iota,  \vec \lam)  + \ti  f_{\bfq}( \vec \iota, \vec \lam, g)  + \dot \phi(u, \nu) 
}
with, 
\EQ{\label{eq:ti-fq-def} 
\ti f_{\bfq}( \vec \iota, \vec \lam, \vec g) := f\Big( \sum_{i = 1}^K \iota_i W_{\lam_i} + g\Big)  - f\Big( \sum_{i = 1}^K \iota_i W_{\lam_i} \Big) 
} 

\begin{proof}[Proof of Lemma~\ref{cor:modul} ]
First, we prove the estimates~\eqref{eq:g-bound} and~\eqref{eq:d-bound-nlw}.  Let $\zeta_n$ be the sequence given by Lemma~\ref{lem:mod-1} and let $\de_n$ be any sequence such that $\zeta_n/ \de_n \to 0$ as $n \to \infty$. Using Lemma~\ref{lem:mod-1}, estimate~\eqref{eq:g-bound} follows from~\eqref{eq:g-upper}  and the estimate~\eqref{eq:d-bound-nlw} follows from~\eqref{eq:d-g-lam}. 

Note also that with this choice of $\de_n$ and ~\eqref{eq:g-bound}, the estimate~\eqref{eq:lam'} leads to, 
\EQ{ \label{eq:lam'-new} 
\abs{\lam_j'(t)} \lesssim \bfd(t) 
}

Next, we treat~\eqref{eq:xi_j-lambda_j}, which is only relevant in the case $D =6$. From~\eqref{eq:xi-def} we see that, 
\EQ{
|\xi_j / \lam_j -1| &\lesssim |  \lam_j^{-1}\La \chi( \cdot/ L\lam_j)\Lam W_{\U{\lam_j}} \mid g \Ra | \\
& \lesssim \| g \|_{L^3}  \lam_j^{-1} \| \chi(L \cdot/ \lam_j) \Lam W_{\U{\lam_j}}  \|_{L^{\frac{3}{2}}}   \lesssim  (\log L)^{\frac{2}{3}}  \| \bs g \|_{\E}   
}
which is small by taking $\eta_0$ sufficiently small (after $L$ is fixed below).  

Next we compute $\xi_j'(t)$. For $D=6$, from~\eqref{eq:xi-def} we have 
\EQ{ \label{eq:xi'-exp} 
\xi_j' &=  \lam_j'  - \frac{\iota_j}{\| \Lam W \|_{L^2}^2}  \La \chi( \cdot/ L\lam_j ) \Lam W_{\U{\lam_j}} \mid \p_t g  \Ra
  \\
  &\quad + \frac{\iota_j}{\| \Lam W \|_{L^2}^2}\frac{ \lam_j' }{\lam_j} \La   (r \p_r \chi)( \cdot/  L\lam_j)  \Lam W_{\U{\lam_j}} \mid g\Ra + \frac{\iota_j}{\| \Lam W \|_{L^2}^2} \frac{\lam_{j}'}{\lam_j} \La \chi( \cdot/ L\lam_j) \ULam \Lam W_{\U{\lam_j}} \mid g\Ra . 
}
We examine each of the terms on the right above. The last two terms are negligible. Indeed, using $\| g \|_{L^3(\R^6)} \lesssim \| \bs g \|_{\E}$, 
\EQ{
\Big| \frac{ \lam_j'}{\lam_j} \La  ( r\p_r \chi)( \cdot/ L\lam_j)  \Lam W_{\U{\lam_j}} \mid g\Ra \Big| &\lesssim \abs{ \lam_j'}   \| g \|_{L^{3}}  \Big(\int_{2^{-1}L}^{2L} |\Lam W(r)|^{\frac{3}{2}} \,r^5 \, \ud r\Big)^{\frac{2}{3}} \\
& \lesssim  \bfd(t)^2 , 
}
and, 
\EQ{
\Big|\frac{\lam_{j}'}{\lam_j} \La \chi( \cdot/ L\lam_j) \ULam \Lam W_{\U{\lam_j}} \mid g\Ra \Big| & \lesssim \abs{ \lam_j'} \|g \|_{L^3} \int_0^{2L} |\ULam\Lam W(r)|^{\frac{3}{2}}\, r^5 \, \ud r  \lesssim (1+ \log( L)) \bfd(t)^2.  
}
Using~\eqref{eq:g-eq} in the second term in~\eqref{eq:xi'-exp} gives 
\EQ{
- \frac{\iota_j}{\| \Lam W \|_{L^2}^2}  \La \chi( \cdot/ L\lam_j ) \Lam W_{\U{\lam_j}} \mid \p_t g  \Ra&= - \frac{\iota_j}{\| \Lam W \|_{L^2}^2}  \La \chi( \cdot/ L\lam_j ) \Lam W_{\U{\lam_j}} \mid \dot g  \Ra  \\
&\quad -  \frac{\iota_j}{\| \Lam W \|_{L^2}^2}  \La\chi( \cdot/ L\lam_j ) \Lam W_{\U{\lam_j}} \mid \sum_{i=1}^K \iota_i \lam_{i}' \Lam W_{\U{\lam_i}} \Ra\\
&\quad - \frac{\iota_j}{\| \Lam W \|_{L^2}^2}  \La\chi( \cdot/ L\lam_j ) \Lam W_{\U{\lam_j}} \mid \phi( u, \nu)  \Ra. 
}
The first term on the right satisfies, 
\EQ{
- \frac{\iota_j}{\| \Lam W \|_{L^2}^2}  \La \chi( \cdot/ L\lam_j ) \Lam W_{\U{\lam_j}} \mid \dot g  \Ra   &=  -\frac{\iota_j}{\| \Lam W \|_{L^2}^2}  \La  \Lam W_{\U{\lam_j}} \mid \dot g  \Ra  +  \frac{\iota_j}{\| \Lam W \|_{L^2}^2}  \La (1-\chi( \cdot /L\lam_j )) \Lam W_{\U{\lam_j}} \mid \dot g  \Ra  \\
& =  -\frac{\iota_j}{\| \Lam W \|_{L^2}^2}  \La  \Lam W_{\U{\lam_j}} \mid \dot g  \Ra  + o_L(1)\|\bs g \|_{\E}. 
}
where the $o_L(1)$ term can be made as small as we like by taking $L>0$ large. 
Using~\eqref{eq:lam'-new}, the second term yields, 
\EQ{
& -  \frac{\iota_j}{\| \Lam W \|_{L^2}^2}  \La \chi( \cdot/ L\lam_j ) \Lam W_{\U{\lam_j}} \mid \sum_{i=1}^K \iota_i \lam_{i}' \Lam W_{\U{\lam_i}} \Ra 
 = - \lam_j'  \\  
 - \sum_{i \neq j} \frac{\iota_j\iota_i \lam_{i}'}{\| \Lam W \|_{L^2}^2} &  \La\chi( \cdot/ L\lam_j ) \Lam W_{\U{\lam_j}} \mid   \Lam W_{\U{\lam_i}} \Ra
  +   \frac{  \lam_{j}' }{\| \Lam W \|_{L^2}^2}  \La (1-  \chi( \cdot/ L\lam_j )) \Lam W_{\U{\lam_j}} \mid \Lam W_{\U{\lam_j}} \Ra \\
 & = - \lam_j' + O(( \lam_{j-1}/ \lam_j) + (\lam_j/ \lam_{j+1}) + o_L(1))  \bfd(t). 
} 
Finally, the third term vanishes due to the fact that for each $j< K$,  $L\lam_j  \ll \lam_K \ll \nu$, and hence
\EQ{
  \La \chi( \cdot/ L\lam_j ) \Lam W_{\U{\lam_j}} \mid \phi( u, \nu)  \Ra  = 0. 
}
Plugging all of this back into~\eqref{eq:xi'-exp} we obtain, 
\EQ{ \label{eq:xi'-est} 
\Big|\xi_j'(t) + \frac{\iota_j}{\| \Lam W \|_{L^2}^2}  \La  \Lam W_{\U{\lam_j}} \mid \dot g  \Ra  \Big| \le c_0 \bfd(t) . 
}
for $D =2$, after fixing $L>0$ sufficiently large. The same estimate for $D \ge 7$, i.e., when $\xi_j'(t) = \lam_j'(t)$, is immediate from~\eqref{eq:lam'-refined} since in this case we take $\calZ = \Lam W$. Thus~\eqref{eq:xi'-est} holds for all $D \ge 6$. 
The estimate~\eqref{eq:lam_j'-beta_j} is then immediate from~\eqref{eq:xi'-est}, the definition of $\be_j$,  and the estimate, 
\EQ{
\Big| \frac{1}{ \| \Lam W \|_{L^2}^2} \ang{ \uln A( \lam_j) g \mid \dot g} \Big| \lesssim \| \bs g \|_{\E}^2 , 
}
which follows from the first bullet point in Lemma~\ref{lem:A}. 

We prove~\eqref{eq:beta_j'}. We compute, 
\EQ{ \label{eq:beta'-1} 
\beta_j' &=  \frac{\iota_j}{ \| \Lam W \|_{L^2}^2}  \frac{\lam_j'}{\lam_j} \La \ULam \Lam W_{\U{\lam_j}} \mid \dot g \Ra  - \frac{\iota_j}{ \| \Lam W \|_{L^2}^2} \La \Lam W_{\U{\lam_j}} \mid \p_t  \dot g \Ra\\
&\quad -   \frac{1}{ \| \Lam W \|_{L^2}^2} \frac{\lam_j'}{\lam_j} \ang{ \lam_j \p_{\lam_j} \uln A( \lam_j) g \mid \dot g} -   \frac{1}{ \| \Lam W \|_{L^2}^2} \ang{ \uln A( \lam_j) \p_t g \mid \dot g} -    \frac{1}{ \| \Lam W \|_{L^2}^2} \ang{ \uln A( \lam_j) g \mid \p_t \dot g}.
}
Using~\eqref{eq:g-eq} we arrive at the expression, 
\EQ{
-\La \Lam W_{\U{\lam_j}} \mid \p_t  \dot g \Ra &=  \ang{ \Lam W_{\U{\lam_j}} \mid  (\LL_{\calW}- \LL_{\lam_j}) g} - \ang{ \Lam W_{\U{\lam_j}} \mid f_{\bfi}( \iota, \vec \lam) }    \\
&\quad   - \ang{ \Lam W_{\U{\lam_j}} \mid f_{\bfq}(\vec \iota,  \vec\lam, g)}  - \ang{ \Lam W_{\U{\lam_j}} \mid  \dot \phi( u, \nu)}, 
}
where in the first term on the right we used that $\LL_{\lam_j} \Lam W_{\U {\lam_j} } = 0$. Using~\eqref{eq:g-eq}  we obtain, 
\EQ{
- &\ang{ \uln A( \lam_j) \p_t g \mid \dot g}  = - \ang{ \uln A( \lam_j) \dot g \mid \dot g} -   \sum_{i=1}^K \iota_i \lam_{i}'  \ang{ \uln A( \lam_j) \Lam W_{\U{\lam_i}} \mid \dot g} -  \ang{ \uln A( \lam_j) \phi( u, \nu) \mid \dot g}\\
& = - \iota_j \lam_j'  \ang{ \uln A( \lam_j) \Lam W_{\U{\lam_j}} \mid \dot g} - \sum_{i \neq j} \iota_i \lam_{i}'  \ang{ \uln A( \lam_j) \Lam W_{\U{\lam_i}} \mid \dot g} -  \ang{ \uln A( \lam_j) \phi( u, \nu) \mid \dot g}
}
where we used that $ \ang{ \uln A( \lam_j) \dot g \mid \dot g} =0$. Finally, using~\eqref{eq:g-eq-alt} we have, 
\EQ{
-\ang{ \uln A( \lam_j) g \mid \p_t \dot g} & =  -\ang{ \uln A( \lam_j) g \mid \De g } - \ang{ \uln A( \lam_j) g \mid f_{\bfi}( \iota, \vec \lam) } \\
& \quad  - \ang{ \uln A( \lam_j) g \mid \ti  f_{\bfq}( \vec \iota,  \vec \lam, g)}-  \ang{ \uln A( \lam_j) g \mid \dot  \phi( u, \nu)}.
}
Plugging these back into~\eqref{eq:beta'-1} and rearranging we have,
\EQ{ \label{eq:beta'-exp} 
\| \Lam Q \|_{L^2}^2 \beta_j' &=  -  \frac{\iota_j }{\lam_j}  \ang{ \Lam W_{\lam_j} \mid f_{\bfi}(  \iota, \vec \lam) } -  \ang{ \uln A( \lam_j) g \mid \De g }  + \ang{ (A(\lam_j) - \uln A( \lam_j)) g \mid \ti  f_{\bfq}( \vec \iota,  \vec \lam, g)} \\
&\quad + \ang{ \Lam W_{\U{\lam_j}} \mid ( \LL_{\calW} - \LL_{\lam_j}) g} + \iota_j \frac{\lam_j' }{\lam_j}\ang{ \big( \frac{1}{\lam_j} \ULam - \U{A}( \lam_j) \big) \Lam W_{\lam_j} \mid \dot g} \\
&\quad    - \ang{ {A}(\lam_j) \sum_{i =1}^K \iota_i W_{\lam_i} \mid f_{\bfq}( \vec \iota,  \vec\lam, g)}   - \ang{ A( \lam_j) g \mid \ti  f_{\bfq}( \vec \iota,  \vec \lam, g)} \\
&\quad    + \iota_j \ang{ ({A}( \lam_j) -  \frac{1}{\lam_j}\Lam) W_{\lam_j} \mid f_{\bfq}( \vec \iota,  \vec\lam, g)}  -    \frac{\lam_j'}{\lam_j} \ang{ \lam_j \p_{\lam_j} \uln A( \lam_j) g \mid \dot g} \\
&\quad + \sum_{ i \neq j} \iota_i \ang{ {A}(\lam_j) W_{\lam_i} \mid  f_{\bfq}( \vec \iota,  \vec\lam, g)} \\ 
&\quad - \sum_{i \neq j} \iota_i \lam_{i}'  \ang{ \uln A( \lam_j) \Lam W_{\U{\lam_i}} \mid \dot g}   - \ang{ \uln A( \lam_j) g \mid f_{\bfi}( \iota, \vec \lam) } \\
&\quad  - \iota_j \ang{ \Lam W_{\U{\lam_j}} \mid  \dot \phi( u, \nu)} -  \ang{ \uln A( \lam_j) \phi( u, \nu) \mid \dot g} -  \ang{ \uln A( \lam_j) g \mid \dot  \phi( u, \nu)}
}
We examine each of the terms on the right-hand side above. The leading order contribution comes from the first term, i.e., by Lemma~\ref{lem:interaction} 
\EQ{
-   \frac{\iota_j }{\lam_j\| \Lam W \|_{L^2}^{2}}  \ang{ \Lam W_{\lam_j} \mid f_{\bfi}(  \iota, \vec \lam) } = - (\om^2 + O( \eta_0^2))  \frac{ \iota_j \iota_{j+1}}{\lam_j}  \Big( \frac{\lam_j}{ \lam_{j+1}}  \Big)^{\frac{D-2}{2}} + (\om^2 + O( \eta_0^2)) \frac{ \iota_j \iota_{j-1}}{\lam_j}\Big( \frac{\lam_{j-1}}{ \lam_{j}}  \Big)^{\frac{D-2}{2}}
}
The second  and third terms together will have a sign, up to an acceptable error. First, using~\eqref{eq:pohozaev} we have, 
\EQ{
 \ang{ \uln A( \lam_j) g \mid -\De g } \ge -\frac{c_0}{\lam_j} \| g \|_{H}^2 + \frac{1}{\lam_j} \int_{R^{-1} \lam_j}^{R \lam_j}  ( \p_r g)^2  \, r^{D-1} \ud r
}
To treat the third term, we start by using the definition~\eqref{eq:ti-fq-def} to observe the identity, 
\EQ{\label{eq:tif-exp} 
\ti  f_{\bfq}( \vec \iota, \vec \lam, g) = f_{\bfq}( \vec \iota, \vec \lam, g) + \Big( f'( \calW(\vec \iota, \vec \lam)) - f'(W_{\lam_j})  \Big) g + f'(W_{\lam_j}) g .
}
The first two terms above contribute acceptable errors. Indeed, using~\eqref{eq:fq-quadratic}, 
\EQ{
\Big|\La(A(\lam_j) - A_{\frac{1}{2}}( \lam_j) )g  \mid  f_{\bfq}( \vec \iota, \vec \lam, g)\Ra\Big| \lesssim \frac{1}{\lam_j} \big(\| \bs g \|_{\E}^3 +
\|\bs g\|_{\E}^{\frac{2D}{D-2}}\big), 
}
and since $\| f'( \calW(\vec \iota, \vec \lam)) - f'(W_{\lam_j})\|_{L^{\frac{D}{2}}}  \lesssim \eta_0$, we have,  
\EQ{
\Big|\La\lam_j^{\frac{1}{2}}(A(\lam_j) - A_{\frac{1}{2}}( \lam_j) )g  \mid (f'( \calW(\vec \iota, \vec \lam)) - f'(W_{\lam_j}) )g\Ra\Big| \le \frac{c_0}{\lam_j} \| \bs g \|_{\E}^2.
}
Putting this together with the fact that $(A(\lam_j) - \U A(\lam_j))  g  = -\frac{1}{D\lam_j} \De q(\cdot/ \lam_j)   g$ we have, 
\EQ{
\La(A(\lam_j) - A_{\frac{1}{2}}( \lam_j) )g  \mid  \ti  f_{\bfq}( \vec \iota, \vec \lam, g)\Ra \ge - \frac{1}{D\lam_j} \int_0^\infty \De q( r/ \lam_j) f'(W_{\lam_j}) g^2 \, r^{D-1} \, \ud r - \frac{c_0}{\lam_j}  \| \bs g \|_{\E}^2 .
}

We show that the remaining terms contribute acceptable errors. For the fourth term a direct calculation gives,  
\EQ{
\| \Lam W_{\lam_j}( f'( \calW( \vec\iota, \vec\lam)-f'(W_{\lam_j}))\|_{L^{\frac{2D}{D+2}}}\lesssim \Big(\frac{\lam_j}{\lam_{j+1}}\Big)^{\frac{D-2}{2}} + \Big(\frac{\lam_{j-1}}{\lam_{j}}\Big)^{\frac{D-2}{2}}
}
and hence, 
\EQ{
\abs{ \ang{ \Lam W_{\U{\lam_j}} \mid ( \LL_{\calW} - \LL_{\lam_j}) g}} 
& \lesssim \frac{1}{\lam_j} \| g \|_H  \Big( \Big( \frac{ \lam_j}{\lam_{j+1}} \Big)^{\frac{D-2}{2}} + \Big( \frac{ \lam_{j-1}}{\lam_{j}} \Big)^{\frac{D-2}{2}}\Big).
}
By~\eqref{eq:L0-A0} along with~\eqref{eq:lam'} we have, 
\EQ{
\abs{\iota_j \frac{\lam_j' }{\lam_j}\ang{ \big( \frac{1}{\lam_j} \ULam - \U{A}( \lam_j) \big) \Lam W_{\lam_j} \mid \dot g}} \lesssim \frac{c_0}{\lam_j} \bfd(t) \| \bs g \|_{\E} . 
}
For the sixth term on the right-hand side of~\eqref{eq:beta'-exp} we note that 
\EQ{
A( \lam_j) \sum_{i =1}^K \iota_i W_{\lam_i} = A( \lam_j) \calW(  \vec \iota, \vec \lam), 
}
and hence we may apply~\eqref{eq:A-by-parts} with $g_1 = \calW(  \vec \iota, \vec \lam)$ and $g_2 = g$ to conclude that 
\EQ{
   \La {A}(\lam_j) \sum_{i =1}^K \iota_i W_{\lam_i} \mid f_{\bfq}(\vec \iota,  \vec\lam, g)\Ra  + \ang{ A( \lam_j) g \mid \ti  f_{\bfq}( \vec \iota,  \vec \lam, g)}  =0,  
}
which takes care of the sixth and seventh terms. 
Next we consider the eighth term. We claim the estimate, 
\EQ{ \label{eq:8-term} 
\abs{ \ang{ ({A}( \lam_j) -  \frac{1}{\lam_j}\Lam) W_{\lam_j} \mid f_{\bfq}( \vec \iota,  \vec\lam, g)} } \le \frac{c_0}{ \lam_j}   \bfd(t)^2 . 
}
When $D=6$  this follows directly from~\eqref{eq:L-A} and ~\eqref{eq:fq-quadratic}.  
For dimensions $D \ge 7$ the brutal estimate~\eqref{eq:fq-quadratic} is not sufficient and we  require a more careful analysis, based on the point-wise estimate~\eqref{eq:fq-ineq}. First, recalling the definition of $A(\lam_j)$ we note that 
\EQ{ \label{eq:fq-regions} 
\Big|\La ({A}( \lam_j) -  \frac{1}{\lam_j}\Lam) W_{\lam_j} \mid f_{\bfq}( \vec \iota,  \vec\lam, g) \Ra\Big| &\lesssim \frac{1}{\lam_j} \int_{0}^{\frac{1}{R} \lam_j} |\Lam W_{\lam_j}| |f_{\bfq}(\vec \iota, \vec\lam, g)| \, r^{D-1} \ud r  \\
&\quad + \frac{1}{\lam_j} \int_{R \lam_j}^{\infty}  |\Lam W_{\lam_j}| |f_{\bfq}(\vec \iota, \vec\lam, g)| \, r^{D-1} \ud r 
}
For the first integral on the right we introduce an auxiliary large parameter $L>0$ and divide the integral into two regions  $ r \in (0, L \sqrt{\lam_{j-1} \lam_j})$ and $r \in [   L \sqrt{\lam_{j-1} \lam_j}, R^{-1} \lam_j]$. In the first region we use the first estimate in~\eqref{eq:fq-ineq} for $x_1 = \calW(\vec\iota, \vec \lam; r)$ and $x_1 = g( \cdot, r)$ to obtain
\EQ{
\frac{1}{\lam_j} \int_{0}^{L \sqrt{\lam_{j-1}{\lam_j}}} |\Lam W_{\lam_j}| |f_{\bfq}(\vec \iota, \vec\lam, g)| \, r^{D-1} \ud r  &\lesssim  \frac{1}{\lam_j}  \int_0^{L \sqrt{\lam_{j-1}{\lam_j}}}  |\Lam W_{\lam_j}| \abs{g}^{\frac{D+2}{D-2}} \, r^{D-1} \, \ud r  \\
& \lesssim \frac{1}{\lam_j} \Big( \int_0^{L \sqrt{\lam_{j-1}{\lam_j}}}  |\Lam W_{\lam_j}|^{\frac{2D}{D-2}} \, r^{D-1} \, \ud r \Big)^{\frac{D-2}{2}} \| \bs g \|_{\E}^{\frac{D+2}{D-2}}\\
&\lesssim \frac{1}{\lam_j} L^{\frac{D-2}{2}} \Big( \frac{\lam_{j-1}}{\lam_j} \Big)^{\frac{D-2}{4}} \| \bs g \|_{\E}^{\frac{D+2}{D-2}} \le \frac{c_0}{\lam_j} \bfd(t)^2 
}
by ensuring $\eta_0$ is sufficiently small relative to $L$. Next we observe that $L>0$ can be taken sufficiently large so that the point-wise estimate, 
\EQ{ \label{eq:calW-lower} 
\Big| \calW( \vec \iota, \vec \lam; r)|  \gtrsim \lam_{j}^{-\frac{D-2}{2}} 
}
holds uniformly in $r \in [   L \sqrt{\lam_{j-1} \lam_j}, R^{-1} \lam_j]$. Using the second inequality in~\eqref{eq:fq-ineq} we then have, 
\EQ{
\frac{1}{\lam_j} \int_{L \sqrt{\lam_{j-1}{\lam_j}}}^{\frac{1}{R} \lam_j} |\Lam W_{\lam_j}| |f_{\bfq}(\vec \iota, \vec\lam, g)| \, r^{D-1} \ud r  &\lesssim\frac{1}{\lam_j}  \int_{L \sqrt{\lam_{j-1}{\lam_j}}}^{\frac{1}{R} \lam_j} |\Lam W_{\lam_j}| |\calW(\vec\iota, \vec \lam; r)|^{-\frac{D-6}{D-2}} \abs{g}^2 \, r^{D-1} \, \ud r  \\
&\lesssim \frac{1}{\lam_j} \frac{1}{R^2}  \int_0^{\frac{1}{R}\lam_j} \frac{g^2}{r^2} \, r^{D-1} \, \ud r  \le \frac{c_0}{\lam_j} \bfd(t)^2 
}
where the last line follows from taking $R$ sufficiently large. The analysis of the second integral in~\eqref{eq:fq-regions} is analogous, this time dividing the region of integration $r \in [R \lam_j, \infty]$ into the regions $r \in [R\lam_j, L^{-1} \sqrt{\lam_j \lam_{j+1}}]$ and $r \in (L^{-1} \sqrt{\lam_j \lam_{j+1}}, \infty)$, and  using the point-wise estimate 
\EQ{
\Big| \calW( \vec\iota, \vec \lam; r)| \gtrsim \lam_j^{\frac{D-2}{2}} r^{-(D-2)} 
}
 in the region $r \in  [R\lam_j, L^{-1} \sqrt{\lam_j \lam_{j+1}}]$, which holds as long as $L$ is taken sufficiently large. 
This proves~\eqref{eq:8-term}.

Using the first bullet point in Lemma~\ref{lem:A} and~\eqref{eq:lam'} we estimate the ninth term as follows, 
\EQ{
 \abs{ \frac{\lam_j'}{\lam_j} \ang{ \lam_j \p_{\lam_j} \uln A( \lam_j) g \mid \dot g}} & \lesssim \frac{1}{\lam_j} \bfd(t) \| \bs g \|_{\E}^2  . 
 } 
Next, using~\eqref{eq:A-mismatch} and~\eqref{eq:fq-quadratic} we have, 
\EQ{
 \Big| \sum_{ i \neq j} \iota_i \ang{ {A}(\lam_j) W_{\lam_i} \mid  f_{\bfq}( \vec \iota,  \vec\lam, g)}  \Big| \le \frac{c_0}{ \lam_j} \bfd(t)^2 . 
}
 An application of~\eqref{eq:A-mismatch} and~\eqref{eq:lam'} gives 
 \EQ{
 \sum_{i \neq j} \abs{\lam_{i}'  \ang{ \uln A( \lam_j) \Lam Q_{\U{\lam_i}} \mid \dot g} }  \le   \frac{c_0}{ \lam_j}  \| g \|_{H}^2 . 
}
Next, consider the twelfth term. Using the first bullet point in Lemma~\ref{lem:A}, and in particular the spatial localization of $\U A( \lam_j)$ we obtain 
\EQ{
\abs{ \ang{ \uln A( \lam_j) g \mid f_{\bfi}( \iota, \vec \lam) } } \lesssim \| g \|_H \| f_{\bfi}( \iota, \vec \lam) \|_{ L^2(\ti R^{-1} \lam_j \le r \le \ti R \lam_j)} . 
}
Using the estimate, 
\EQ{
\| f_{\bfi}( \iota, \vec \lam)  \|_{L^2(\ti R^{-1} \lam_j \le r \le \ti R \lam_j)} \lesssim \frac{1}{\lam_j}  \Big( \frac{\lam_j}{\lam_{j+1} }\Big)^{\frac{D-2}{2}} +   \frac{1}{\lam_j} \Big( \frac{\lam_{j-1}}{\lam_{j}} \Big)^{\frac{D-2}{2}}. 
}
We obtain
\EQ{
\abs{ \ang{ \uln A( \lam_j) g \mid f_{\bfi}(  \iota, \vec \lam) } } \lesssim \frac{1}{\lam_j}\| g \|_H\Big( \Big( \frac{\lam_j}{\lam_{j+1} }\Big)^{\frac{D-2}{2}} +   \frac{1}{\lam_j} \Big( \frac{\lam_{j-1}}{\lam_{j}} \Big)^{\frac{D-2}{2}} \Big). 
}
Finally, we treat the last line of~\eqref{eq:beta'-exp}. First, using Lemma~\ref{lem:ext-sign} and the definition of $\dot\phi$ in~\eqref{eq:h-def} we have
\EQ{ \label{eq:outer-error} 
\abs{ \ang{ \Lam W_{\U{\lam_j}} \mid  \dot \phi( u, \nu)}} & \lesssim  \frac{1}{\lam_j} \Big( \frac{ \lam_j}{\nu} \Big)^{\frac{D-2}{2}}  \| \bs u(t)\|_{\E(\nu(t), 2 \nu(t))}  \lesssim \frac{\theta_n}{\lam_j}\Big( \frac{ \lam_j}{\nu} \Big)^{\frac{D-2}{2}}. 
}
for some sequence $\theta_n \to 0$ as $n \to \infty$. The last two terms in~\eqref{eq:beta'-exp} vanish due to the support properties of $\U A( \lam_j), \phi(u, \nu), \dot \phi(u, \nu)$ and the fact that $\lam_j \le \lam_K \ll \nu$. 

Combining these estimates in~\eqref{eq:beta'-exp} we obtain the inequality, 
\EQ{
\beta' &\ge  \Big({-} \iota_j \iota_{j+1}\om^2 -   c_0\Big) \frac{1}{\lam_{j}} \left( \frac{\lam_{j}}{\lam_{j+1}} \right)^{\frac{D-2}{2}} +    \Big(\iota_j \iota_{j-1}\om^2 -  c_0\Big) \frac{1}{\lam_{j}} \left( \frac{\lam_{j-1}}{\lam_{j}} \right)^{\frac{D-2}{2}}  \\
&   + \frac{1}{\lam_j} \int_{R^{-1} \lam_j}^{R \lam_j}  ( \p_r g)^2  \, r \ud r - \frac{1}{D\lam_j} \int_0^\infty \De q( r/ \lam_j) f'(W_{\lam_j}) g^2 \, r^{D-1} \, \ud r   -  c_0 \frac{ \bfd(t)^2}{ \lam_{j}}  
}
Finally, we use the localized coercivity estimate~\eqref{eq:coercive-virial} on the second line above along with the estimates $\La \frac{1}{\lam_j} \calY_{\U{\lam_j}} \mid g \Ra^2 \lesssim (a_j^+)^2 + (a_j^-)^2$ to see that 
\EQ{
 \frac{1}{\lam_j} \int_{R^{-1} \lam_j}^{R \lam_j}  ( \p_r g)^2  \, r \ud r - \frac{1}{D\lam_j} \int_0^\infty \De q( r/ \lam_j) f'(W_{\lam_j}) g^2 \, r^{D-1} \, \ud r \ge -   c_0 \frac{ \| g \|_H^2}{ \lam_{j}} - \frac{C_0}{\lam_j} \Big( (a_j^+)^2 + (a_j^-)^2 \Big) 
}
 This  completes the proof of~\eqref{eq:beta_j'}. 
 
 Finally, ~\eqref{eq:dta-D} follows from~\eqref{eq:dta} and our choice of $\de_n$.
\end{proof}


Finally, we prove that, again by enlarging $\epsilon_n$, we can control
the error in the virial identity, see Lemma~\ref{lem:vir}, by $\bfd$. 
\begin{lem}
\label{lem:virial-error}
There exist $C_0, \eta_0 > 0$ depending only on $D$ and $N$
and a decreasing sequence $\eps_n \to 0$ such that
\begin{equation}
|\Om_{1, \rho(t)}(\bs u(t)) + \frac{D-2}{2} \Om_{2, \rho(t)}(\bs u(t))| \leq C_0\bfd(t) 
\end{equation}
for all $t \in [a_n, b_n]$ such that $\eps_n \le \bfd(t) \leq \eta_0$,
$\rho(t) \leq \nu(t)$ and $|\rho'(t)| \leq 1$.
\end{lem}
\begin{proof}
Since $\lim_{n\to\infty}\sup_{t\in[a_n, b_n]}\|\bs u(t)\|_{\cE(\nu(t), 2\nu(t))} = 0$, Lemma~\ref{lem:mod-1} yields
\begin{equation}
\|\bs u(t) - \bs\calW( \vec\iota, \vec\lambda(t)) - \bs g(t)\|_{\cE(0, 2\nu(t))} \to 0, \qquad\text{as }n \to \infty.
\end{equation}
Using Remark~\ref{rem:deltan},~\eqref{eq:g-bound} and~\eqref{eq:d-bound-nlw} we have $\|\bs g(t)\|_\cE \lesssim \bfd(t)$, hence, after choosing $\eps_n \to 0$ sufficiently large, it suffices to check that
\begin{equation}
|\Omega_{1, \rho(t)}(\bs\calW( \vec\iota, \vec\lambda(t))) +  \frac{D-2}{2}\Omega_{2, \rho(t)}(\bs\calW( \vec\iota, \vec\lambda(t)))|
\leq C_0\bfd(t),
\end{equation}
which in turn will follow from
\EQ{
\Big| \int_0^\infty  \Big( \frac{1}{2} ( \p_r \calW( \vec \iota, \vec \lam))^2 + \frac{D-2}{2D} |\calW( \vec \iota, \vec \lam)|^{\frac{2D}{D-2}} + \frac{D-2}{2} \p_r \calW(\vec\iota, \vec \lam) \frac{\calW(\vec \iota, \vec \lam)}{r}  \Big) \, (r \p_r \chi)( r/ \rho) r^{D-1} \, \ud r  \Big| \\
\le C_0 \bfd(t)
}
Noting the identity, 
\EQ{
\frac{1}{2} ( \p_r W(r))^2 + \frac{D-2}{2D} |W(r)|^{\frac{2D}{D-2}} + \frac{D-2}{2} \p_r W(r) \frac{W(r)}{r}  = 0 
}
it suffices to estimates the cross terms in the integral above, and the desired bound follows from an explicit computation. 
\end{proof} 


\section{Conclusion of the proof} \label{sec:conclusion} 

\subsection{The scale of the $K$-th bubble}

As mentioned in the Introduction, the $K$-th bubble is of particular importance.
We introduce a function $\mu$ which is well-defined on every $[a_n, b_n]$,
and of size comparable with $\lambda_K$ on time intervals where the solution approaches a multi-bubble configuration.
\begin{defn}[The scale of the $K$-th bubble] \label{def:mu} 
Fix $\kappa_1 > 0$ small enough. For all $t \in I$, we set
\begin{equation}
\label{eq:mu-def}
\mu(t) := \sup\big\{r \leq \nu(t): \|\bs u(t)\|_{\cE(r, \nu(t))} = \kappa_1\big\}.
\end{equation}
\end{defn}
Note that, if $\kappa_1$ is sufficiently small, then $K > 0$ implies
$\|\bs u(t)\|_{\cE(0, \nu(t))} \geq 2\kappa_1$,
hence $\mu(t)$ is a well-defined finite positive number for all $t \in I$.
Since in the definition of $\mu(t)$ we can restrict to rational $r$, $\mu$ is a measurable function. Even if $\mu$ is not necessarily a continuous function, still $\mu(t)$
is well-defined for each individual value of $t$.
We stress that $\mu(t) \leq \nu(t)$ for all $n$ large enough and $t \in [a_n, b_n]$,
thus $\mu(t) \ll \mu_{K+1}(t)$ as $n \to \infty$.

We also introduce a specific ``regularization'' of $\mu$.
For a given collision interval $[a_n, b_n]$, we set
\begin{equation}
\label{eq:mustar-def}
\mu_*: [a_n, b_n]: \to (0, \infty), \qquad \mu_*(t) := \inf_{s \in [a_n, b_n]}\big(4\mu(s) + |s-t|\big).
\end{equation}
We choose not to include $n$ in the notation. We stress that $\mu_*$ depends on $n$,
which will be known from the context.
\begin{lem}
\label{lem:mustar}
The function $\mu_*$ has the following properties:
\begin{enumerate}[(i)]
\item its Lipschitz constant is $\leq 1$,
\item there exist $\delta, \kappa_2 > 0$ and $n_0 \in \bN$ depending on $\kappa_1$
such that $t \in [a_n, b_n]$ with $n \geq n_0$ and $\bfd(t) \leq \delta$
imply $\kappa_2 \lambda_K(t) \leq \mu^*(t) \leq \kappa_2^{-1}\lambda_K(t)$,
where $\lambda_K(t)$ is the modulation parameter defined in Lemma~\ref{lem:mod-1},
\item if $t_n \in [a_n, b_n]$, $1 \ll r_n \ll \mu_{K+1}(t_n)/\mu_*(t_n)$ and $\lim_{n \to \infty}\bs\de_{r_n\mu_*(t_n)}(t_n) = 0$, then $\lim_{n\to \infty}\bfd(t_n) = 0$.
\end{enumerate}
\end{lem}
\begin{proof}
Statement (i) is clear.

Recall that $4\mu(t) \leq 4\nu(t)$, hence the definition of $\mu_*(t)$ yields
\begin{equation}
\mu_*(t) = \inf\big\{4\mu(s) + |s-t|: s \in [a_n, b_n] \cap [t - 4\nu(t), t + 4\nu(t)]\big\} .
\end{equation}
Let $s \in [a_n, b_n] \cap [t - 4\nu(t), t + 4\nu(t)]$.
By the definition of $\mu$ and \eqref{eq:en-R-2R}, we have, for $n$ large enough,
\begin{equation}
\|\bs u(t)\|_{\cE(\mu(s), 16\nu(s))} \leq 2\kappa_1,
\end{equation}
hence Lemma~\ref{lem:prop-small-E-local} yields
\begin{equation}
\|\bs u(t)\|_{\cE(4\mu(s) + |t - s|, 4\nu(s) - |t - s|)} \leq C\kappa_1.
\end{equation}
From \eqref{eq:nu'} and $|s - t| \leq 4\nu(t)$, we deduce that $\nu(t) \leq 4\nu(s) - |t - s|$, thus
\begin{equation}
\|\bs u(t)\|_{\cE(4\mu(s) + |t - s|, \nu(t))} \leq C\kappa_1.
\end{equation}
Taking the supremum with respect to $s \in [a_n, b_n] \cap [t - 4\nu(t), t + 4\nu(t)]$, we obtain
\begin{equation}
\label{eq:right-to-mustar}
\|\bs u(t)\|_{\cE(\mu_*(t), \nu(t))} \leq C\kappa_1.
\end{equation}
Statement (iii) now follows from Lemma~\ref{lem:trapping-of-d} with $\mu_n = r_n\mu_*(t_n)$,
provided that we choose $\kappa_1 \leq \eta_0 / C$.

Let $\kappa_2 > 0$ be such that
\begin{equation}
\|W\|_{H(r \geq (4\kappa_2)^{-1})} \leq \frac 12\kappa_1, \qquad \|W\|_{H(r \geq \kappa_2)} \geq 2C\kappa_1,
\end{equation}
where $C$ is the constant in \eqref{eq:right-to-mustar}.
It is clear that $\frac 14\mu_*(t) \leq \mu(t)$, hence \eqref{eq:mu-def} yields
\begin{equation}
\|\bs u(t)\|_{\cE(\frac 14\mu_*(t), \nu(t))} \geq \kappa_1.
\end{equation}
Thus, in order to prove that $\mu_*(t) \leq \kappa_2^{-1}\lambda_K(t)$, it suffices to check that
\begin{equation}
\|\bs u(t)\|_{\cE((4\kappa_2)^{-1}\lambda_K(t), \nu(t))} < \kappa_1.
\end{equation}
We use \eqref{eq:u-decomp}. By \eqref{eq:d-g-lam}, $\|\bs g\|_\cE \ll 1$
when $\delta \ll 1$ and $n_0 \gg 1$. Thus, it suffices to see that
\begin{align}
\|\bs\calW(\vec\iota, \vec\lambda)\|_{\cE((4\kappa_2)^{-1}\lambda_K(t), \nu(t))} < \kappa_1,
\end{align}
whenever $\sum_{j=1}^K \lambda_j(t)/\lambda_{j+1}(t) \ll 1$,
which is obtained directly from the definitions of $\bs \calW$ and $\kappa_2$.

Similarly, using \eqref{eq:right-to-mustar}, we will have $\mu_*(t) \geq \kappa_2 \lambda_K(t)$ if we can prove that
\begin{equation}
\|\bs u(t)\|_{\cE(\kappa_2\lambda_K(t), \nu(t))} > C\kappa_1.
\end{equation}
But the last bound follows from
\begin{align}
\|\bs\calW(\vec\iota, \vec\lambda)\|_{\cE(\kappa_2\lambda_K(t), \nu(t))} > C\kappa_1
\end{align}
whenever $\sum_{j=1}^K \lambda_j(t)/\lambda_{j+1}(t) \ll 1$.
\end{proof}

Our next goal is to prove that the minimality of $K$ (see Definition~\ref{def:K-choice}) implies a lower bound
on the length of the collision intervals.

\begin{lem}
\label{lem:cd-length}
If $\eta_1 > 0$ is small enough, then for any $\eta \in (0, \eta_1]$ there exist $\epsilon \in (0, \eta)$ and $C_{\bs u} > 0$
having the following property.
If $[c, d] \subset [a_n, b_n]$, $\bfd(c) \leq \epsilon$, $\bfd(d) \leq \epsilon$ and there exists $t_0 \in [c, d]$
such that $\bfd(t_0) \geq \eta$, then
\begin{equation}
\label{eq:cd-length}
d - c \geq C_{\bs u}^{-1}\max(\mu_*(c), \mu_*(d)).
\end{equation}
\end{lem}
\begin{proof}
We argue by contradiction. If the statement is false, then there exist $\eta > 0$, a decreasing sequence $(\epsilon_n)$
tending to $0$, an increasing sequence $(C_n)$ tending to $\infty$
and intervals $[c_n, d_n] \subset [a_n, b_n]$ (up to passing to a subsequence in the sequence of the collision intervals $[a_n, b_n]$) such that
$\bfd(c_n) \leq \epsilon_n$, $\bfd(d_n) \leq \epsilon_n$, there exists $t_n \in [c_n, d_n]$ such that $\bfd(t_n) \geq \eta$
and $d_n - c_n \leq C_n^{-1}\max(\mu_*(c_n), \mu_*(d_n))$.
We will check that, up to adjusting the sequence $\epsilon_n$, $[c_n, d_n] \in \calC_{K-1}(\epsilon_n, \eta)$ for all $n$,
contradicting Definition~\ref{def:K-choice}.

The first and second requirement in Definition~\ref{def:collision} are clearly satisfied.
It remains to construct a function $\rho_{K-1}:[c_n, d_n] \to [0, \infty)$ such that
\begin{equation}
\label{eq:dK-1conv0}
\lim_{n\to\infty}\sup_{t\in [c_n, d_n]}\bfd_{K-1}(t; \rho_{K-1}(t)) = 0.
\end{equation}
Assume $\mu_*(c_n) \geq \mu_*(d_n)$ (the proof in the opposite case is very similar).
Let $r_n$ be a sequence such that $\lambda_{K-1}(c_n) \ll r_n \ll \lambda_K(c_n)$ (recall that $\kappa_2\mu_*(c_n) \leq \lambda_K(c_n) \leq \kappa_2^{-1}\mu_*(c_n)$ and that $\lambda_0(t) = 0$ by convention).
Set $\rho_{K-1}(t) := r_n + (t - c_n)$ for $t \in [c_n, d_n]$.
Recall that $\vec \sigma_n \in \{-1, 1\}^{N-K}$ and $\vec \mu(t) \in (0, \infty)^{N-K}$ are defined in Lemma~\ref{lem:ext-sign}.
Let $\iota_{n}$ be the sign of the $K$-th bubble at time $c_n$, and set
$\wt \sigma := (\iota_n, \vec\sigma_n) \in \{-1, 1\}^{N-(K-1)}$
and $\wt\mu(t) := (\lambda_K(c_n), \vec\mu(t)) \in (0, \infty)^{N-(K-1)}$.
Let $R_n$ be a sequence such that $\nu_n(c_n) \ll R_n \ll \mu_{K+1}(c_n)$.
Applying Lemma~\ref{lem:evol-of-Q} with these sequences $r_n, R_n$ and $\bs u_n(t) := \bs u(c_n + t)$, we obtain
\begin{equation}
\lim_{n\to\infty}\sup_{t\in[c_n, d_n]}\|\bs u(t) - \bs\calW(\wt\sigma_n, \wt\mu(t))\|_{\cE(\rho_{K-1}(t), \infty)} = 0,
\end{equation}
implying \eqref{eq:dK-1conv0}
\end{proof}
\begin{rem}
We denote the constant $C_{\bs u}$ to stress that it depends on the solution $\bs u$
and is obtained in a non-constructive way as a consequence of the assumption that $\bs u$
does not satisfy the continuous time soliton resolution.
\end{rem}
\subsection{Demolition of the multi-bubble} \label{ssec:bub-dem}
This paragraph is devoted to the analysis of the ODE system satisfied by the modulation parameters.
We apply here the ``weighted sum'' trick from \cite[Section 6]{DKM9}.

\begin{lem}
\label{lem:ejection}
Let $D \geq 6$. If $\eta_0$ is small enough,
then there exists $C_0 \geq 0$ depending only on $D$ and $N$ such that the following is true.
If $[t_1, t_2] \subset I_*$ is a finite time interval such that $\bfd(t) \leq \eta_0$ for all $t \in [t_1, t_2]$, then
\begin{gather}
\label{eq:lambdaK-bd}
\sup_{t \in [t_1, t_2]} \lambda_K(t) \leq \frac 43\inf_{t \in [t_1, t_2]} \lambda_K(t), \\
\label{eq:int-d-bd}
\int_{t_1}^{t_2} \bfd(t)\ud t\leq C_0\big(\bfd(t_1)^{\frac{4}{D-2}} \lambda_K(t_1) + \bfd(t_2)^{\frac{4}{D-2}} \lambda_K(t_2)\big).
\end{gather}
\end{lem}
\begin{rem}
Since $\bfd(t) \leq \eta_0$ is small, Lemma~\ref{lem:mustar} (iii) yields $\lambda_K(t) \simeq \mu_*(t)$,
so in the formulation of the lemma we could just as well write $\mu_*$ instead of $\lambda_K$.
\end{rem}
\begin{proof}[Proof of Lemma~\ref{lem:ejection}]
\textbf{Step 1.}
First, we argue that \eqref{eq:lambdaK-bd} follows from \eqref{eq:int-d-bd}.
Without loss of generality, assume $\lambda_K(t_1) \geq \lambda_K(t_2)$.
Since $|\lambda_K'(t)| \lesssim \bfd(t)$, see \eqref{eq:lam'}, and $\bfd(t)^{\frac{4}{D-2}} \lesssim \eta_0^{\frac{4}{D-2}}$ is small, \eqref{eq:int-d-bd} implies
\begin{equation}
\int_{t_1}^{t_2} |\lambda_K'(t)| \ud t \leq \frac 17 \lambda_K(t_1),
\end{equation}
thus $\inf_{t \in [t_1, t_2]} \lambda_K(t) \geq \frac 67\lambda_K(t_1)$
and $\sup_{t \in [t_1, t_2]} \lambda_K(t) \leq \frac 87 \lambda_K(t_1)$, so \eqref{eq:lambdaK-bd} follows. It remains to prove \eqref{eq:int-d-bd}.

\noindent
\textbf{Step 2.}
Let $C_1 > 0$ be a large number chosen below and consider the auxiliary function
\begin{equation}
\phi(t) := \sum_{j \in \cS}2^{-j}\xi_j(t)\beta_j(t) - C_1\sum_{j=1}^K \lambda_j(t)a_j^-(t)^2
+ C_1\sum_{j=1}^K \lambda_j(t) a_j^+(t)^2,
\end{equation}
inspired by the function $A(t)$ from \cite[Section 6]{DKM9}.
We claim that for all $t \in [t_1, t_2]$
\begin{equation}
\label{eq:dtphi-lbound}
\phi'(t) \geq c_2\bfd(t)^2,
\end{equation}
with $c_2 > 0$ depending only on $D$ and $N$. The remaining part of Step 1 is devoted to proving this bound.

Using \eqref{eq:lam_j'-beta_j}, \eqref{eq:dta-D} and recalling that $|\lambda_K'(t)| \lesssim \bfd(t)$, see \eqref{eq:lam'}, we obtain
\begin{equation}
\label{eq:dtphi-lbound-3}
\begin{aligned}
\phi'(t) \geq \sum_{j \in \cS}2^{-j}\beta_j^2(t) + \sum_{j\in\cS}2^{-j}\lambda_j(t)\beta_j'(t)
+ C_1\nu\sum_{j=1}^K \big(a_j^-(t)^2 + a_j^+(t)^2\big) - c_0 \bfd(t)^2.
\end{aligned}
\end{equation}
We focus on the second term of the right hand side. Applying \eqref{eq:beta_j'}, we have
\begin{equation}
\begin{aligned}
\sum_{j\in\cS}\lambda_j(t)\beta_j'(t) &\geq \omega^2\sum_{j\in\cS}2^{-j}\bigg(\iota_j\iota_{j+1}\Big(\frac{\lambda_j(t)}{\lambda_{j+1}(t)}\Big)^\frac{D-2}{2}
-\iota_j\iota_{j-1}\Big(\frac{\lambda_{j-1}(t)}{\lambda_{j}(t)}\Big)^\frac{D-2}{2}\bigg) \\
&-C_2\sum_{k=1}^K\big(a_j^-(t)^2 + a_j^+(t)^2\big) - c_0\bfd(t)^2.
\end{aligned}
\end{equation}
Noting that $\iota_j\iota_{j+1} = 1$ if $j \in \cS$, we rewrite the first sum on the right hand side as
\begin{equation}
\sum_{j\in \cS}2^{-j}\Big(\frac{\lambda_j(t)}{\lambda_{j+1}(t)}\Big)^\frac{D-2}{2}
- \sum_{j \in \cS}2^{-j}\iota_j\iota_{j-1}\Big(\frac{\lambda_{j-1}(t)}{\lambda_{j}(t)}\Big)^\frac{D-2}{2}.
\end{equation}
Splitting the first sum into two equal terms, and shifting the index in the second sum, we obtain
\begin{equation}
\begin{aligned}
\sum_{j\in\cS}2^{-j-1}\Big(\frac{\lambda_j(t)}{\lambda_{j+1}(t)}\Big)^\frac{D-2}{2}
+ \bigg[\sum_{j\in\cS}2^{-j-1}\Big(\frac{\lambda_j(t)}{\lambda_{j+1}(t)}\Big)^\frac{D-2}{2}
- \sum_{j+1\in \cS}2^{-j-1}\iota_{j}\iota_{j+1}\Big(\frac{\lambda_j(t)}{\lambda_{j+1}(t)}\Big)^\frac{D-2}{2}\bigg].
\end{aligned}
\end{equation}
We will check that the number inside the square parenthesis is nonnegative.
To see this, we rewrite it as
\begin{equation}
\sum_{j}2^{-j-1}\Big(\frac{\lambda_j(t)}{\lambda_{j+1}(t)}\Big)^\frac{D-2}{2}
\big[\mathbbm{1}_\cS(j) - \iota_j\iota_{j+1}\mathbbm{1}_\cS(j+1) \big].
\end{equation}
If $j \notin \cS$, then $\iota_j\iota_{j+1} = -1$, hence all the terms in the above sum are nonnegative. We have thus proved that
\begin{equation}
\sum_{j\in\cS}\lambda_j(t)\beta_j'(t) \geq 2^{-N-1}\omega^2\sum_{j\in\cS}\Big(\frac{\lambda_j(t)}{\lambda_{j+1}(t)}\Big)^\frac{D-2}{2} - C_2\sum_{k=1}^K\big(a_j^-(t)^2 + a_j^+(t)^2\big) - c_0\bfd(t)^2.
\end{equation}
Taking $C_1 > C_2 / \nu$ and using \eqref{eq:dtphi-lbound-3} together with \eqref{eq:d-bound-nlw}, we get \eqref{eq:dtphi-lbound}.

\noindent
\textbf{Step 3.}
Since $\phi$ is increasing, it has at most one zero, which we denote $t_3 \in [t_1, t_2]$.
If $\phi(t) > 0$ for all $t \in [t_1, t_2]$, we set $t_3 := t_1$, and if $\phi(t) < 0$
for all $t \in [t_1, t_2]$, then we set $t_3 := t_2$. We will show that
\begin{equation}
\label{eq:int-d-bd-1}
\int_{t_3}^{t_2}\bfd(t)\ud t \leq C_0\bfd(t_2)^{\frac{4}{D-2}}\lambda_K(t_2).
\end{equation}
By the symmetry of the problem, one can similarly bound the integral over $[t_1, t_3]$,
and summing the two we get \eqref{eq:int-d-bd}. Without loss of generality,
we can assume $t_3 < t_2$, since otherwise \eqref{eq:int-d-bd-1} is trivial.

Observe that for all $t \in (t_3, t_2]$ we have
\begin{equation}
\label{eq:phi-lambdaK-bd}
\begin{aligned}
\frac{\phi(t)}{\lambda_K(t)} &\lesssim \sum_{j \in \cS}\frac{\lambda_j(t)}{\lambda_{K}(t)}|\beta_j(t)| + \sum_{j=1}^K \frac{\lambda_j(t)}{\lambda_K(t)}\big(a_j^-(t)^2 + a_j^+(t)^2\big) \\
&\lesssim \sum_{j \in \cS}\frac{\lambda_j(t)}{\lambda_{j+1}(t)}|\beta_j(t)| + \bfd(t)^2
\lesssim \bfd(t)^{\frac{D+2}{D-2}}.
\end{aligned}
\end{equation}
Combining this bound with \eqref{eq:dtphi-lbound}, for all $t \in (t_3, t_2]$ we get
\begin{equation}
\phi'(t) \geq c_2 \big(\phi(t) / \lambda_K(t)\big)^{\frac{D-2}{D+2}} \bfd(t),
\end{equation}
thus
\begin{equation}
\label{eq:dtphi-lbound-2}
\lambda_K(t)^\frac{D-2}{D+2} \big(\phi(t)^{\frac{4}{D+2}}\big)' \gtrsim \bfd(t).
\end{equation}
Using \eqref{eq:phi-lambdaK-bd} and $|\lambda_K'(t)| \lesssim \bfd(t)$, we get
\begin{equation}
\big(\lambda_K(t)^\frac{D-2}{D+2}\phi(t)^\frac{4}{D+2}\big)' - \lambda_K(t)^\frac{D-2}{D+2} \big(\phi(t)^\frac{4}{D+2}\big)' = \frac{D-2}{D+2}\lambda_K'(t)\big(\phi(t)/\lambda_K(t)\big)^\frac{4}{D+2} \gtrsim -\bfd(t)^{\frac{D+2}{D-2}}.
\end{equation}
Since $\bfd(t)^\frac{4}{D-2} \lesssim \eta_0^\frac{4}{D-2}$ is small, \eqref{eq:dtphi-lbound-2} yields
\begin{equation}
\big(\lambda_K(t)^\frac{D-2}{D+2}\phi(t)^\frac{4}{D+2}\big)' \gtrsim \bfd(t)
\end{equation}
which, integrated, gives
\begin{equation}
\int_{t_3}^{t_2}\bfd(t)\ud t \lesssim \lambda_K(t_2)^\frac{D-2}{D+2}\phi(t_2)^\frac{4}{D+2} - \lambda_K(t_3)^\frac{D-2}{D+2}\phi(t_3)^\frac{4}{D+2}
\leq \lambda_K(t_2)^\frac{D-2}{D+2}\phi(t_2)^\frac{4}{D+2}.
\end{equation}
Invoking \eqref{eq:phi-lambdaK-bd}, we obtain \eqref{eq:int-d-bd-1}.
\end{proof}

Starting from now, $\eta_0 > 0$ is fixed so that Lemma~\ref{lem:ejection} holds and Lemma~\ref{lem:cd-length}
can be applied with $\eta = \eta_0$.
We also fix $\epsilon > 0$ to be the value given by Lemma~\ref{lem:cd-length} for $\eta = \eta_0$.
%
Recall that $\bfd(a_n) = \bfd(b_n) = \epsilon_n$ and
$\bfd(t) \geq \epsilon_n$ for all $t \in [a_n, b_n]$.
\begin{lem}
\label{lem:cedf}
There exists $\theta_0 > 0$ such that for any sequence satisfying $\epsilon_n \ll \theta_n \leq \theta_0$
and for all $n$ large enough there exists a partition of the interval $[a_n, b_n]$
\begin{equation}
\begin{aligned}
a_n = e^L_{n, 0} \leq e^R_{n, 0} \leq c^R_{n, 0} \leq d^R_{n, 0} \leq f^R_{n, 0}
\leq f^L_{n, 1} \leq d^L_{n, 1} \leq c^L_{n, 1} \leq e^L_{n, 1} \leq \ldots \leq e^R_{n, N_n} = b_n,
\end{aligned}
\end{equation}
having the following properties.
\begin{enumerate}[(1)]
\item
\label{it:1}
For all $m \in \{0, 1, \ldots, N_n\}$ and $t \in [e_{n, m}^L, e_{n, m}^R]$, $\bfd(t) \leq \eta_0$, and
\begin{equation}
\label{eq:error-on-eLeR}
\int_{e_{n, m}^L}^{e_{n, m}^R}\bfd(t)\ud t \leq C_2 \theta_n^{4/(D-2)} \min(\mu_*(e_{n, m}^L), \mu_*(e_{n, m}^R)),
\end{equation}
where $C_2\geq 0$ depends only on $k$ and $N$.
\item \label{it:2}
For all $m \in \{0, 1, \ldots, N_n-1\}$ and $t \in [e_{n, m}^R, c_{n, m}^R] \cup [f_{n, m}^R, f_{n, m+1}^L] \cup [c_{n, m+1}^L, e_{n, m+1}^L]$, $\bfd(t) \geq \theta_n$.
\item \label{it:3}
For all $m \in \{0, 1, \ldots, N_n-1\}$ and $t \in [c_{n,m}^R, f_{n,m}^R]\cup [f_{n,m+1}^L, c_{n,m+1}^L]$, $\bfd(t) \geq \epsilon$.
\item \label{it:4}
For all $m \in \{0, 1, \ldots, N_n-1\}$, $\bfd(d_{n,m}^R) \geq \eta_0$
and $\bfd(d_{n,m+1}^L) \geq \eta_0$.
\item \label{it:5}
For all $m \in \{0, 1, \ldots, N_n-1\}$, $\bfd(c_{n, m}^R) = \bfd(c_{n, m+1}^L) = \epsilon$.
\item \label{it:6}
For all $m \in \{0, 1, \ldots, N_n-1\}$,
either $\bfd(t) \geq \epsilon$ for all $t \in [c_{n, m}^R, c_{n, m+1}^L]$,
or $\bfd(f_{n, m}^R) = \bfd(f_{n,m+1}^L) = \epsilon$.
\item \label{it:7}
For all $m \in \{0, 1, \ldots, N_n-1\}$,
\begin{equation}
\label{eq:mu-no-change}
\begin{aligned}
\sup_{t \in [e_{n, m}^L, c_{n, m}^R]}\mu_*(t) &\leq 2 \kappa_2^{-2}\inf_{t \in [e_{n, m}^L, c_{n, m}^R]}\mu_*(t), \\
\sup_{t \in [c_{n, m+1}^L, e_{n, m+1}^R]}\mu_*(t)  &\leq 2\kappa_2^{-2}\inf_{t \in [c_{n, m+1}^L, e_{n, m+1}^R]}\mu_*(t).
\end{aligned}
\end{equation}
\end{enumerate}
\end{lem}
\begin{proof}
For all $t_0 \in [a_n, b_n]$ such that $\bfd(t_0) < \eta_0$, let $J(t_0) \subset [a_n, b_n]$ be the union of all the open (relatively in $[a_n, b_n]$) intervals containing $t_0$ on which $\bfd < \eta_0$. Equivalently, we have one of the following three cases:
\begin{itemize}
\item $J(t_0) = (\wt a_n, \wt b_n)$, $t_0 \in (\wt a_n, \wt b_n)$, $\bfd(\wt a_n) = \bfd(\wt b_n) = \eta_0$ and $\bfd(t) < \eta_0$ for all $t \in (\wt a_n, \wt b_n)$,
\item $J(t_0) = [a_n, \wt b_n)$, $t_0 \in [a_n, \wt b_n)$,
$\bfd(\wt b_n) = \eta_0$ and $\bfd(t) < \eta_0$ for all $t \in [a_n, \wt b_n)$,
\item $J(t_0) = (\wt a_n, b_n]$, $t_0 \in (\wt a_n, b_n]$,
$\bfd(\wt a_n) = \eta_0$ and $\bfd(t) < \eta_0$ for all $t \in (\wt a_n, b_n]$.
\end{itemize}
Note that $\theta_n \gg \epsilon_n$ implies $\wt a_n > a_n$ and $\wt b_n < b_n$.
Clearly, any two such intervals are either equal or disjoint.

Consider the set
\begin{equation}
A := \{t \in [a_n, b_n]: \bfd(t) \leq \theta_n\}.
\end{equation}
Since $A$ is a compact set, there exists a finite sequence
\begin{equation}
a_n \leq s_{n, 0} < s_{n, 1} < \ldots < s_{n, N_n} \leq b_n
\end{equation}
such that
\begin{equation}
\label{eq:cover-of-A}
s_{n, m} \in A, \qquad A \subset \bigcup_{m=0}^{N_n} J(s_{n, m}).
\end{equation}
Without loss of generality, we can assume $J(s_{n, m}) \cap J(s_{n, m'}) = \emptyset$ whenever $m \neq m'$ (it suffices to remove certain elements from the sequence).

%
Let $m \in \{0, 1, \ldots, N_n-1\}$. Since $J(s_{n, m}) \cap J(s_{n, m+1}) = \emptyset$, there exists $t \in (s_{n, m}, s_{n, m+1})$ such that
$\bfd(t) \geq \eta_0$. Let $d_{n, m}^R$ be the smallest such $t$, and $d_{n, m+1}^L$
the largest one. Let $c_{n, m}^R$ be the smallest number such that
$\bfd(t) \geq \epsilon$ for all $t \in (c_{n, m}^R, d_{n, m}^R)$.
Similarly, let $c_{n, m+1}^L$ be the biggest number such that
$\bfd(t)\geq \epsilon$ for all $t \in (d_{n, m+1}^L, c_{n, m+1}^L)$.
Next, let $e_{n, m}^R$ be the smallest number such that
$\bfd(t) \geq 2\theta_n$ for all $t \in (e_{n, m}^R, c_{n, m}^R)$.
If we take $\theta_n < \frac{\epsilon}{2}$, then we have $e_{n, m}^R < c_{n, m}^R$.
Since $\bfd(s_{n, m}) \leq \theta_n$, we have $e_{n, m}^R > s_{n, m}$.
Similarly, let $e_{n, m+1}^L$ be the biggest number such that
$\bfd(t)\geq 2\theta_n$ for all $t \in (c_{n, m+1}^L, e_{n, m+1}^L)$
(again, it follows that $e_{n, m+1}^L < s_{n, m+1}$).
Finally, if $\bfd(t) \geq \epsilon$ for all $t \in (d_{n, m}^R, d_{n, m+1}^L)$,
we set $f_{n, m}^R$ and $f_{n, m+1}^L$ arbitrarily,
for example $f_{n, m}^R := d_{n, m}^R$ and $f_{n, m+1}^L := d_{n, m+1}^L$.
If, on the contrary, there exists $t \in (d_{n, m}^R, d_{n, m+1}^L)$
such that $\bfd(t) < \epsilon$, we let $f_{n, m}^R$
be the biggest number such that $\bfd(t)\geq \epsilon$ for all $t \in (d_{n, m}^R, f_{n, m}^R)$, and $f_{n, m+1}^L$ be the smallest number such that
$\bfd(t) \geq \epsilon$ for all $t \in (f_{n, m+1}^L, d_{n, m+1}^L)$.

We check all the desired properties. For all $n \in \{0, 1, \ldots, N_n\}$, we have $e_{n, m}^L \leq s_{n, m} \leq e_{n, m}^R$.
Since $\bfd(e_{n, m}^L) \leq 2\theta_n$ and $\bfd(e_{n, m}^R) \leq 2\theta_n$, the property \ref{it:1} follows from Lemma~\ref{lem:ejection}.
The properties \ref{it:3}, \ref{it:4}, \ref{it:5} and \ref{it:6} follow directly from the construction.
The property \ref{it:2} is now equivalent to the following statement: if $\bfd(t_0) < \theta_n$, then
there exists $m \in \{0, 1, \ldots, N_n\}$ such that $t_0 \in [e_{n, m}^L, e_{n, m}^R]$.
But \eqref{eq:cover-of-A} implies that $t_0 \in J(s_{n, m})$ for some $m$
and, by construction, $\bfd(t) > \theta_n$ for all $t \in J(s_{n, m}) \setminus [e_{n, m}^L, e_{n, m}^R]$,
so we obtain $t \in [e_{n, m}^L, e_{n, m}^R]$. Finally, note that $\bfd(t) \leq \eta_0$
for all $t \in [e_{n, m}^L, c_{n, m}^R] \cup [c_{n, m+1}^L, e_{n, m+1}^R]$, hence, using again Lemma~\ref{lem:ejection},
but on the time intervals $[e_{n, m}^L, c_{n, m}^R]$ and $[c_{n, m+1}^L, e_{n, m+1}^R]$, we deduce the property \ref{it:7}
from \eqref{eq:lambdaK-bd} and Lemma~\ref{lem:mustar} (ii).

\end{proof}

\subsection{End of the proof: virial inequality with a cut-off}
In this section, we conclude the proof, by integrating the virial identity on the time interval $[a_n, b_n]$.
The radius where the cut-off is imposed has to be carefully chosen, which is the object of the next lemma.
\begin{lem}
\label{lem:rho}
There exist $\theta_0 > 0$ and a locally Lipschitz function $\rho : \cup_{n=1}^\infty [a_n, b_n] \to (0, \infty)$
having the following properties:
\begin{enumerate}
\item $\max(\rho(a_n)\|\partial_t u(a_n)\|_{L^2}, \rho(b_n)\|\partial_t u(b_n)\|_{L^2}) \ll \max(\mu_*(a_n), \mu_*(b_n))$ as $n \to \infty$,
\item $\lim_{n \to \infty}\inf_{t\in[a_n, b_n]}\big(\rho(t) / \mu_*(t)\big) = \infty$ and $\lim_{n\to\infty}\sup_{t\in[a_n, b_n]}\big(\rho(t) / \mu_{K+1}(t)\big) = 0$,
\item if $\bfd(t_0) \leq\frac 12\theta_0$, then $|\rho'(t)| \leq 1$ for almost all $t$ in a neighborhood of $t_0$,
\item $\lim_{n\to\infty}\sup_{t\in[a_n, b_n]}|\Omega_{\rho(t)}(\bs u(t))| = 0$.
\end{enumerate}
\end{lem}
\begin{proof}
We will define two functions $\rho^{(a)}, \rho^{(b)}$, and then set $\rho := \min(\rho^{(a)}, \rho^{(b)}, \nu)$.
First, we let
\begin{equation}
\rho^{(a)}(a_n) := \min(R_n\mu_*(a_n), \nu(a_n)),
\end{equation}
where $1 \ll R_n \ll \|\partial_t u(a_n)\|_{L^2}^{-1}$.
Consider an auxiliary sequence
\begin{equation}
\label{eq:deltan-virial-def}
\delta_n := \sup_{t \in [a_n, b_n]}
\|\bs u(t)\|_{\cE(\min(\rho^{(a)}(a_n) + t - a_n, \nu(t)); 2\nu(t))}.
\end{equation}
We claim that $\lim_{n \to \infty} \delta_n = 0$.
Indeed, if $\rho^{(a)}(a_n) + t - a_n \geq \nu(t)$, then it suffices to recall \eqref{eq:en-R-2R}.
In the opposite case, \eqref{eq:nu'} yields $t - a_n \leq \nu(t) \leq \nu(a_n) + o(t - a_n)$,
hence $t - a_n \leq 2\nu(a_n)$. Since we have
\begin{equation}
\lim_{n \to\infty}\|\bs u(t)\|_{\cE(\frac 14\rho^{(a)}(a_n); 4\nu(t))} = 0,
\end{equation}
it suffices to apply Lemma~\ref{lem:prop-small-E-local}.

Let $\theta_0 > 0$ be given by Lemma~\ref{lem:cedf}, and divide $[a_n, b_n]$ into subintervals applying this lemma
for the constant sequence $\theta_n = \theta_0$.
We let $\rho^{(a)}$ be the piecewise affine function such that
\begin{equation}
\dd t\rho^{(a)}(t) := 1\ \text{if }t \in [e_{n, m}^L, e_{n, m}^R],\qquad \dd t\rho^{(a)}(t) := \delta_n^{-\frac 12}\text{ otherwise.}
\end{equation}
We check that $\lim_{n\to\infty}\inf_{t\in[a_n, b_n]}\big(\rho^{(a)}(t) / \mu_*(t)\big) = \infty$.
First, suppose that $t \in [e_{n, m}^R, e_{n, m+1}^L]$ and $t - e_{n, m}^R \gtrsim \mu_*(e_{n, m}^R)$.
Then $\mu_*(t) \leq \mu_*(e_{n, m}^R) + (t - e_{n, m}^R) \lesssim t - e_{n, m}^R$
and $\rho^{(a)}(t) \geq \delta_n^{-\frac 12}(t - e_{n, m}^R)$, so $\rho^{(a)}(t) \gg \mu_*(t)$.

By Lemma~\ref{lem:cd-length}, $e_{n, m+1}^L - e_{n, m}^R \geq C_{\bs u} \mu_*(e_{n, m}^R)$,
so in particular we obtain $\rho^{(a)}(e_{n, m+1}^L)\gg \mu_*(e_{n, m+1}^L)$ for all $m \in \{0, 1, \ldots, N_n - 1\}$.
Note that we also have $\rho^{(a)}(e_{n, 0}^L) = \rho^{(a)}(a_n) \gg \mu_*(a_n) = \mu_*(e_{n, 0}^L)$,
by the choice of $\rho^{(a)}(a_n)$.
Since, by the property (7), $\mu_*$ changes at most by a factor $2\kappa_2^{-2}$ on $[e_{n, m}^L, e_{n, m}^R]$ and $\rho^{(a)}$ is increasing,
we have $\rho^{(a)}(e_{n, m}^R) \gg \mu_*(e_{n, m}^R)$.

Finally, if $t - e_{n, m}^R \leq \mu_*(e_{n, m}^R)$, then $\mu_*(t) \leq 2\mu_*(e_{n, m}^R)$,
which again implies $\rho^{(a)}(t) \gg \mu_*(t)$.

The function $\rho^{(b)}$ is defined similarly, but integrating from $b_n$ backwards.
Properties (1), (2), (3) are clear.
By the expression for $\Omega_{\rho(t)}(\bs u(t))$, see Lemma~\ref{lem:vir},
we have
\begin{equation}
|\Omega_{\rho(t)}(\bs u(t))| \lesssim (1+|\rho'(t)|)\|\bs u(t)\|_{\cE(\rho(t), 2\rho(t))}^2 \lesssim \sqrt{\delta_n} \to 0,
\end{equation}
which proves the property (4).

\end{proof}
We need one more elementary result.
\begin{lem}
\label{lem:subdivision}
If $\mu_*: [a, b] \to (0, \infty)$ is a $1$-Lipschitz function
and $b - a \geq \frac 14 \mu_*(a)$, then there exists a sequence
$a = a_0 < a_1 < \ldots < a_l < a_{l+1} = b$ such that
\begin{equation}
\label{eq:lip-partition}
\frac 14 \mu_*(a_i) \leq a_{i+1} - a_i \leq \frac 34 \mu_*(a_i),\qquad\text{for all }i \in \{1, \ldots, l\}.
\end{equation}
\end{lem}
\begin{proof}
We define inductively $a_{i+1} := a_i + \frac 14\mu_*(a_i)$, as long as $b - a_i > \frac 34 \mu_*(a_i)$.
We need to prove that $b - a_i > \frac 34 \mu_*(a_i)$ implies $b - a_{i+1} > \frac 14\mu_*(a_{i+1})$.

Since $\mu_*$ is $1$-Lipschitz, $\mu_*(a_{i+1}) = \mu_*(a_i + \mu_*(a_i)/4)
\leq \mu_*(a_i) + \mu_*(a_i)/4 = \frac 54 \mu_*(a_i)$, thus
\begin{equation}
b - a_{i+1} = b - a_i - \frac 14 \mu_*(a_i) > \frac 34 \mu_*(a_i) - \frac 14\mu_*(a_i)
> \frac{5}{16}\mu_*(a_i) \geq \frac 14 \mu_*(a_{i+1}).
\end{equation}
\end{proof}
\begin{rem}
\label{rem:subdivision}
Note that \eqref{eq:lip-partition} and the fact that $\mu_*$ is $1$-Lipschitz
imply $\inf_{t \in [a_i, a_{i+1}]}\mu_*(t) \geq \frac 14\mu_*(a_i)$
and $\sup_{t \in [a_i, a_{i+1}]}\mu_*(t) \leq \frac 74\mu_*(a_i)$, thus
\begin{equation}
\frac 17 \sup_{t \in [a_i, a_{i+1}]}\mu_*(t) \leq a_{i+1} - a_i \leq 3\inf_{t \in [a_i, a_{i+1}]}\mu_*(t),
\end{equation}
in other words the length of each subinterval is comparable with
both the smallest and the largest value of $\mu_*$ on this subinterval.
\end{rem}
\begin{lem}
\label{lem:virial-decrease}
Let $\rho$ be the function given by Lemma~\ref{lem:rho} and set
\begin{equation}
\label{eq:fv-def}
\fv(t) := \int_0^\infty\partial_t u(t) r\partial_r u(t) \chi_{\rho(t)}\,r\vd r.
\end{equation}
\begin{enumerate}[1.]
\item There exists a sequence $\theta_n \to 0$ such that the following is true. If $[\wt a_n, \wt b_n] \subset [a_n, b_n]$ is such that
\begin{equation}
\wt b_n - \wt a_n \geq \frac 14 \mu_*(\wt a_n)\quad\text{and}\quad
\bfd(t) \geq \theta_n\text{ for all }t \in [\wt a_n, \wt b_n],
\end{equation}
then
\begin{equation}
\label{eq:intermediate}
\fv(\wt b_n) < \fv(\wt a_n).
\end{equation}
\item 
For any $c, \theta > 0$ there exists $\delta > 0$ such that if $n$ is large enough, $[\wt a_n, \wt b_n] \subset [a_n, b_n]$,
\begin{equation}
c \mu_*(\wt a_n) \leq \wt b_n - \wt a_n \quad\text{and}\quad
\bfd(t) \geq \theta\text{ for all }t \in [\wt a_n, \wt b_n],
\end{equation}
then
\begin{equation}
\label{eq:intermediate-2}
\fv(\wt b_n) - \fv(\wt a_n) \leq  -\delta \sup_{t\in[\wt a_n, \wt b_n]}\mu_*(t).
\end{equation}
\end{enumerate}
\end{lem}
\begin{proof}
By the virial identity, we obtain
\begin{equation}
\label{eq:virial-fin}
\fv'(t)  =
- \int_0^\infty (\partial_t u(t))^2 \chi_{\rho(t)}\,r\vd r + o_n(1).
\end{equation}
We argue by contradiction. If the claim is false, then there exists $\theta >0$
and an infinite sequence $[\wt a_n, \wt b_n] \subset [a_n, b_n]$
(as usual, we pass to a subsequence in $n$ without changing the notation)
such that
\begin{equation}
\wt b_n - \wt a_n \geq \frac 14 \mu_*(\wt a_n)\quad\text{and}\quad
\bfd(t) \geq \theta\text{ for all }t \in [\wt a_n, \wt b_n],
\end{equation}
and
\begin{equation}
\fv(\wt b_n) - \fv(\wt a_n) \geq 0.
\end{equation}

By Lemma~\ref{lem:subdivision}, there exists a subinterval of $[\wt a_n, \wt b_n]$, which we still denote $[\wt a_n, \wt b_n]$, such that
\begin{equation}
\frac 14 \mu_*(\wt a_n) \leq \wt b_n - \wt a_n \leq \frac 34 \mu_*(\wt a_n)\quad\text{and}\quad
\fv(\wt b_n) - \fv(\wt a_n) \geq 0.
\end{equation}
Let $\wt\rho_n := \inf_{t \in [\wt a_n, \wt b_n]}\rho(t)$. From \eqref{eq:virial-fin}, we have
\begin{equation}
\lim_{n\to\infty}\frac{1}{\wt b_n - \wt a_n}\int_{\wt a_n}^{\wt b_n}\int_0^{\frac 12 \wt \rho_n}(\partial_t u(t))^2\,r\vd r = 0.
\end{equation}
By Lemma~\ref{lem:rho}, $\inf_{t\in[\wt a_n, \wt b_n]}\mu_{K+1}(t) \gg \wt\rho_n \gg \inf_{t\in[\wt a_n, \wt b_n]}\mu_*(t)
\simeq \sup_{t\in[\wt a_n, \wt b_n]}\mu_*(t)$, so Lemma~\ref{lem:compact} yields sequences $t_n \in [\wt a_n, \wt b_n]$ and $1 \ll r_n \ll \mu_{K+1}(t_n)/\mu_*(t_n)$ such that
\begin{equation}
\lim_{n\to\infty}\bs\de_{r_n\mu_*(t_n)}(\bs u(t_n)) = 0,
\end{equation}
which is impossible by Lemma~\ref{lem:mustar} (iii). The first part of the lemma is proved.

In the second part, we can assume without loss of generality $\wt b_n - \wt a_n \leq \frac 34 \mu_*(\wt a_n)$.
Indeed, in the opposite case, we apply Lemma~\ref{lem:subdivision}
and keep only one of the subintervals where $\mu_*$ attains its supremum, and on the remaining subintervals we use
\eqref{eq:intermediate}.

After this preliminary reduction, we argue again by contradiction.
If the claim is false, then there exist $c, \theta >0$, a sequence $\delta_n \to 0$
and a sequence $[\wt a_n, \wt b_n] \subset [a_n, b_n]$ (after extraction of a subsequence) such that
\begin{equation}
c\mu_*(\wt a_n) \leq \wt b_n - \wt a_n \leq \frac 34 \mu_*(\wt a_n)\quad\text{and}\quad
\bfd(t) \geq \theta\text{ for all }t \in [\wt a_n, \wt b_n],
\end{equation}
and
\begin{equation}
\fv(\wt b_n) - \fv(\wt a_n) \geq -\delta_n\mu_*(\wt a_n)
\end{equation}
(we use the fact that $\mu_*(\wt a_n)$ is comparable to $\sup_{t\in[\wt a_n, \wt b_n]}\mu_*(t)$, see Remark~\ref{rem:subdivision}).

Let $\wt\rho_n := \inf_{t \in [\wt a_n, \wt b_n]}\rho(t)$. From \eqref{eq:virial-fin}, we have
\begin{equation}
\lim_{n\to\infty}\frac{1}{\wt b_n - \wt a_n}\int_{\wt a_n}^{\wt b_n}\int_0^{\frac 12 \wt \rho_n}(\partial_t u(t))^2\,r\vd r = 0.
\end{equation}
We now conclude as in the first part.
\end{proof}

\begin{proof}[Proof of Theorem~\ref{thm:main}]
Let $\theta_n$ be the sequence given by Lemma~\ref{lem:virial-decrease}, part 1.
We partition $[a_n, b_n]$ applying Lemma~\ref{lem:cedf} for this sequence $\theta_n$.
Note that this partition is different than the one used in the proof of Lemma~\ref{lem:rho}.
We claim that
for all $m \in \{0, 1, \ldots, N_n-1\}$
\begin{align}
\label{eq:intermediate-fin-1}
\fv(c_{n, m}^R) - \fv(e_{n, m}^R) &\leq o_n(1)\mu_*(c_{n, m}^R), \\
\label{eq:intermediate-fin-2}
\fv(f_{n, m+1}^L) - \fv(f_{n, m}^R) &\leq o_n(1)\mu_*(f_{n, m}^R), \\
\label{eq:intermediate-fin-3}
\fv(e_{n, m+1}^L) - \fv(c_{n, m+1}^L) &\leq o_n(1)\mu_*(c_{n, m+1}^L).
\end{align}
Here, $o_n(1)$ denotes a sequence of positive numbers converging to $0$ when $n\to\infty$.
In order to prove the first inequality, we observe that if $c_{n, m}^R - e_{n, m}^R \geq \frac 14 \mu_*(e_{n, m}^R)$,
then \eqref{eq:intermediate} applies and yields $\fv(c_{n, m}^R) - \fv(e_{n, m}^R)  < 0$.
We can thus assume $c_{n, m}^R - e_{n, m}^R \leq \frac 14 \mu_*(e_{n, m}^R) \leq \frac{\kappa_2^{-2}}{2} \mu_*(c_{n, m}^R)$,
where the last inequality follows from Lemma~\ref{lem:cedf}, property \ref{it:7}.
But then \eqref{eq:virial-fin} again implies the required bound.
The proofs of the second and third bound are analogous.

We now analyse the compactness intervals $[c_{n, j}^R, f_{n, j}^R]$ and $[f_{n, j+1}^L, c_{n, j+1}^L]$.
We claim that there exists $\delta > 0$ such that for all $n$ large enough and $m \in \{0, 1, \ldots, N_n\}$
\begin{equation}
\label{eq:compactness-fin}
\fv(c_{n, m+1}^L) - \fv(c_{n, m}^R) \leq -\delta \max(\mu_*(c_{n, m}^R), \mu_*(c_{n, m+1}^L)).
\end{equation}
We consider separately the two cases mentioned in Lemma~\ref{lem:cedf}, property \ref{it:6}.
If $\bfd(t) \geq \epsilon$ for all $t \in [c_{n, m}^R, c_{n, m+1}^L]$, then Lemma~\ref{lem:cd-length}
yields $c_{n, m+1}^L - c_{n,m}^R \geq C_{\bs u}^{-1}\mu_*(c_{n,m}^R)$, so we can apply \eqref{eq:intermediate-2}, which proves \eqref{eq:compactness-fin}.
If $\bfd(f_{n,m}^R) = \epsilon$, then we apply the same argument on the time interval $[c_{n,m}^R, f_{n,m}^R]$ and obtain
\begin{equation}
\label{eq:compactness-fin-4}
\fv(f_{n, m}^R) - \fv(c_{n, m}^R) \leq -\delta \max(\mu_*(c_{n, m}^R), \mu_*(f_{n, m}^R)),
\end{equation}
and similarly
\begin{equation}
\label{eq:compactness-fin-5}
\fv(c_{n, m+1}^L) - \fv(f_{n, m+1}^L) \leq -\delta \max(\mu_*(c_{n, m+1}^L), \mu_*(f_{n,m+1}^L)).
\end{equation}
The bound \eqref{eq:intermediate-fin-2} yields \eqref{eq:compactness-fin}.

Finally, on the intervals $[e_{n, m}^L, e_{n,m}^R]$, for $n$ large enough Lemma~\ref{lem:rho} yields $|\rho'(t)| \leq 1$
for almost all $t$, and Lemma~\ref{lem:virial-error} implies $|\fv'(t)| \lesssim \bfd(t)$.
By Lemma~\ref{lem:cedf}, properties \ref{it:1} and \ref{it:7}, we obtain
\begin{equation}
\label{eq:modulation-fin}
\begin{aligned}
\fv(e_{n,m}^R) - \fv(e_{n,m}^L) &\leq o_n(1)\mu_*(c_{n,m}^R),\qquad\text{for all }m \in \{0, 1, \ldots, N_n-1\}, \\
\fv(e_{n,m}^R) - \fv(e_{n,m}^L) &\leq o_n(1)\mu_*(c_{n,m}^L),\qquad\text{for all }m \in \{1, \ldots, N_n-1, N_n\}.
\end{aligned}
\end{equation}

Taking the sum in $m$ of \eqref{eq:intermediate-fin-1}, \eqref{eq:intermediate-fin-3}, \eqref{eq:compactness-fin} and
\eqref{eq:modulation-fin}, we deduce that there exists $\delta > 0$
and $n$ arbitrarily large such that
\begin{equation}
\fv(b_n) - \fv(a_n) \leq -\delta\max(\mu_*(c_{n,0}^R), \mu_*(c_{n,N_n}^L)).
\end{equation}
But $\mu_*(a_n) \simeq \mu_*(c_{n,0}^R)$ and $\mu_*(b_n) \simeq \mu_*(c_{n,N_n}^L)$, hence
\begin{equation}
\fv(b_n) - \fv(a_n) \leq -\wt\delta\max(\mu_*(a_n), \mu_*(b_n)).
\end{equation}
Lemma~\ref{lem:rho} (1) and \eqref{eq:fv-def} yield
\begin{equation}
|\fv(a_n)| \ll \mu_*(a_n), \qquad |\fv(b_n)| \ll \mu_*(b_n),
\end{equation}
a contradiction which finishes the proof.
\end{proof}

\subsection{Absence of elastic collisions}
\label{ssec:inelastic}
This section is devoted to proving Proposition~\ref{prop:inelastic}
Our proof closely follows Step 3 in our proof of \cite[Theorem 1.6]{JL1}.
\begin{proof}[Proof of Proposition~\ref{prop:inelastic}]
Suppose that a solution $\bs u$ of \eqref{eq:nlw}, defined on its maximal time of existence
$t \in (T_-, T_+)$, is a pure multi-bubble in both time directions in the sense of Definition~\ref{def:pure}, in other words
\begin{equation}
\lim_{t \to T_+}\bfd(t) = 0, \qquad\text{and}\qquad \lim_{t \to T_-}\bfd(t) = 0, 
\end{equation}
and the radiation $\bs u^* = \bs u^*_L$ or $\bs u^* = \bs u^*_0$ in both time directions satisfies $\bs u^* \equiv 0$. 
In this proof, all the $N$ bubbles can be thought of as ``interior'' bubbles
thus, whenever we invoke the results from the preceding sections, it should always
be understood that $K = N$.
Applying Lemma~\ref{lem:mod-static} with $\theta = 0$ and $M = N$,
we obtain from \eqref{eq:g-bound-A} and \eqref{eq:g-bound-0} that
\begin{equation}
\label{eq:pure-g-bound}
\bfd(t)^2 \leq C \bigg(\sum_{j \in \cS}\Big(\frac{\lambda_j(t)}{\lambda_{j+1}(t)}\Big)^\frac{D-2}{2}
+ \sum_{k=1}^K \big(a_k^-(t)^2 + a_k^+(t)^2\big)\bigg).
\end{equation}
Inspecting the proof of Lemma~\ref{lem:mod-1}, it follows that the last inequality
and the fact that $\bs u^* = 0$ imply that Lemma~\ref{lem:mod-1} holds with $\zeta_n = 0$.
Similarly, Lemma~\ref{cor:modul} holds with $\delta_n = 0$.

Let $\eta > 0$ be a small number to be chosen later and $t_+$ be such that
$\bfd(t) \leq \eta$ for all $t \geq t_+$.
Lemma~\ref{lem:ejection} yields
\begin{equation}
\int_{t_+}^{t}\bfd(t)\ud t \leq C_0 \big(\bfd(t_+)^{\frac{4}{D-2}} \lambda_N(t_+),
\bfd(t)^{\frac{4}{D-2}} \lambda_N(t)\big)
\end{equation}
and passing to the limit $t \to T_+$ we get
\begin{equation}
\label{eq:integral-of-d}
\int_{t_+}^{T_+}\bfd(t)\ud t \leq C_0 \bfd(t_+)^{\frac{4}{D-2}} \lambda_N(t_+).
\end{equation}
From the bound $|\lambda_N'(t)| \lesssim \bfd(t)$, see \eqref{eq:lam'} with $\zeta_n = 0$,
together with \eqref{eq:integral-of-d}, implies that $\lim_{t \to T_+}\lambda_N(t)$
is a finite positive number, thus $T_+ = +\infty$.

Analogously, $T_- = -\infty$ and $\lim_{t \to -\infty}\lambda_N(t) \in (0, +\infty)$ exists.

The remaining part of the argument is exactly the same as in \cite{JL1}, but we reproduce it
here for the reader's convenience.

Let $\delta > 0$ be arbitrary. Inspecting the proof of Lemma~\ref{lem:virial-error},
we see that in the present case it holds with $\epsilon_n = 0$, thus for any $R > 0$ we have
$\big|\Omega_{1,R}(\bs u(t)) + \frac{D-2}{2}\Omega_{2,R}(\bs u(t))\big| \leq C_0\bfd(t)$.
From this bound and the estimates above, we obtain existence of $T_1, T_2 \in \bR$ such that
\begin{align}
\int_{-\infty}^{T_1}\Big|\Omega_{1,R}(\bs u(t)) + \frac{D-2}{2}\Omega_{2,R}(\bs u(t))\Big|\ud t &\leq \frac 13 \delta, \\
\int_{T_2}^{+\infty}\Big|\Omega_{1,R}(\bs u(t)) + \frac{D-2}{2}\Omega_{2,R}(\bs u(t))\Big|\ud t &\leq \frac 13 \delta
\end{align}
for any $R > 0$. On the other hand, because of the bound $|\Om_{j,R}(\bs u(t))| \le C_0 \|\bs u(t)\|_{\cE(R, 2R)}$ and since $[T_1, T_2]$ is a finite time interval, for all $R$ sufficiently large we have
\begin{equation}
\int_{T_1}^{T_2}\Big|\Omega_{1,R}(\bs u(t)) + \frac{D-2}{2}\Omega_{2,R}(\bs u(t))\Big|\ud t \leq \frac 13 \delta,
\end{equation}
in other words
\begin{equation}
\int_{\bR}\Big|\Omega_{1,R}(\bs u(t)) + \frac{D-2}{2}\Omega_{2,R}(\bs u(t))\Big|\ud t \leq \delta.
\end{equation}
Integrating the virial identity \eqref{eq:virial} with $\rho(t) = R$ over the real line,
we obtain
\begin{equation}
\int_{-\infty}^{+\infty}\int_0^\infty (\partial_t u(t, r)\chi_R(r))^2\,r^{D-1}\vd r\ud t \leq \delta.
\end{equation}
By letting $R \to +\infty$, we get
\begin{equation}
\int_{-\infty}^{+\infty}\int_0^\infty (\partial_t u(t, r))^2\,r^{D-1}\vd r\ud t \leq \delta,
\end{equation}
which implies the $\bs u$ is stationary since $\delta$ is arbitrary.
\end{proof}

\appendix 

\section{Modifications to the argument in the case $D =5$} 
In this section we outline the technical changes to the arguments in Section~\ref{sec:decomposition} needed to prove Theorem~\ref{thm:main} dimensions $D =5$.
\subsection{Decomposition of the solution}
The set-up in Sections~\ref{ssec:proximity} holds without modification for $D=5$. The number $K \ge1$ is defined as in Lemma~\ref{lem:K-exist}, the collision intervals $[a_n, b_n] \in \calC_K(\eta, \eps_n)$ are as in Definition~\ref{def:K-choice}, and  the sequences of signs $\vec \s_n \in \{-1, 1\}^{N-K}$,  scales $\vec \mu(t) \in (0, \infty)^{N-K}$, and the sequence $\nu_n \to 0$ and  the function $\nu(t) = \nu_n \mu_{K+1}(t)$ are as in Lemma~\ref{lem:ext-sign}. 

 Lemma~\ref{lem:mod-1} also holds with a minor modification to the stable/unstable components. Let $J \subset [a_n, b_n]$ be any time interval on which $\bfd(t) \le \eta_0$, where $\eta_0$ is as in  Lemma~\ref{lem:mod-1}. Let $\vec \iota \in \{ -1, 1\}^K, \vec \lam(t) \in (0, \infty)^K$, $\bs g(t) \in \E$, and $a_{j}^{\pm}(t)$  be as in the statement of Lemma~\ref{lem:mod-1}. Define for each $1 \le j \le K$, the modified stable/unstable components, 
 \EQ{
 \ti a_{j}^{\pm}(t) := \La \bs \al_{\lam_j(t)}^{\pm} \mid \bs g(t) + \sum_{i < j} \iota_i \bs W_{\lam_i(t)} \Ra 
 }
 The estimate~\eqref{eq:dta} will hold for $\ti a_j^{\pm}(t)$ rather than for $a_j^{\pm}(t)$, see~\eqref{eq:a-tia-5} and~\eqref{eq:dta-5} below. 
 We make a similar modification (i.e., removing the interior bubbles from $\bs g(t)$) to the refined modulation parameter $\xi_j(t)$. For each $j \in \{1, \dots, K-1\}$, we set 
 \EQ{
 \xi_j(t) = \lam_j(t) - \frac{\iota_j}{\| \Lam W \|_{L^2}^2} \La \chi( \cdot/ L \lam_j(t)) \Lam W_{\U{\lam_j(t)}} \mid g(t)  +  \sum_{i <j} \iota_i W_{\lam_i(t)} \Ra
 }
 where $L>0$ is a large parameter to be determined below. 
 The refined modulation parameter $\be_j(t)$ requires no modifications and is defined as in~\eqref{eq:beta-def} for all $j \in \{1, \dots, K-1\}$. 

With these definitions, the following analogue of Lemma~\ref{cor:modul} holds. 

%

\begin{prop}[Refined modulation, $D=5$]
\label{prop:modul-5} 
 Let $c_0  \in (0, 1)$. There exists  constants $L_0 = L_0(c_0)>0$, $\eta_0 = \eta_0(c_0)$, as well as  $c= c(c_0)$ and $R = R(c_0)>1$ as in Lemma~\ref{lem:q},  a constant $C_0>0$, and a 
decreasing sequence $\epsilon_n \to 0$ so that the following is true. 

Suppose $L > L_0$ and  $J \subset [a_n, b_n]$ is an open time interval with $\epsilon_n \leq \bfd(t) \le \eta_0$
for all $t \in J$, where $\calS :=  \{ j \in \{1, \dots, K-1\} \mid \iota_{j}  = \iota_{j+1} \}$.  
Then, for all $t \in J$, 
\EQ{\label{eq:g-bound-5} 
\| \bs g(t) \|_{\E} + \sum_{i \not \in \calS}  (\lambda_i(t) / \lambda_{i+1}(t))^\frac 34 \le  \max_{ i  \in \calS}  (\lambda_i(t) / \lambda_{i+1}(t))^\frac 34 +  \max_{1 \leq i \leq K}|a_i^\pm(t)|\\
}
and, 
\EQ{ \label{eq:d-bound-5} 
\frac{1}{C_0} \bfd(t) \le \max_{i \in \calS} (\lambda_i(t) / \lambda_{i+1}(t))^\frac 34  +  \max_{1 \leq i \leq K}|a_i^\pm(t)| \le C_0 \bfd(t) ,
}
and, 
\begin{equation}\label{eq:xi_j-lambda_j-5}
\Big|\frac{\xi_j(t)}{\lambda_{j}(t)} - 1\Big| \le c_0  
\end{equation}
Moreover,  for all $t \in J$, 
\EQ{ \label{eq:xi_j'-5} 
\abs{ \xi'_j(t) } \le C_0 \bfd(t) , 
}
\begin{equation}\label{eq:xi_j'-beta_j-5} 
\Big|\xi_j'(t) - \beta_j(t)\Big| \leq c_0\bfd(t), 
\end{equation}
and,  
\EQ{ \label{eq:beta_j'-5} 
 \beta_{j}'(t) &\ge  \Big( \iota_j \iota_{j+1}\om^2 -   c_0\Big) \frac{1}{\lam_{j}(t)} \left( \frac{\lam_{j}(t)}{\lam_{j+1}(t)} \right)^{\frac{3}{2}} +    \Big({-}\iota_j \iota_{j-1}\om^2 -  c_0\Big) \frac{1}{\lam_{j}(t)} \left( \frac{\lam_{j-1}(t)}{\lam_{j}(t)} \right)^{\frac{3}{2}}   \\
&\quad  -  \frac{c_0}{\lam_j(t)}  \bfd(t)^2 - \frac{C_0}{\lam_j(t)} \Big( (a_j^+(t))^2 + (a_j^-(t))^2 \Big),   
}
where, by convention, $\lambda_0(t) = 0, \lam_{K+1}(t) = \infty$ for all $t \in J$, and $\om^2>0$ is as in~\eqref{eq:omega-def}
Finally, for each $j \in \{1, \dots, K\}$, 
\EQ{ \label{eq:a-tia-5} 
\Big| \ti a_j^\pm(t) - a_j^\pm(t) \Big| \le C_0 \bfd(t)^2
}
and
\EQ{ \label{eq:dta-5} 
\Big| \dd t \ti a_j^\pm(t) \mp \frac{\kappa}{\lambda_j(t)} \ti a_j^\pm(t) \Big| \leq \frac{C_0}{\lambda_j(t)}\bfd(t)^2. 
}
\end{prop}

\begin{proof} 
The estimates~\eqref{eq:g-bound-5} and~\eqref{eq:d-bound-5} follow as in the proofs of the corresponding estimates in Lemma~\ref{cor:modul}. Next, we have, 
\EQ{
\Big|\frac{\xi_j}{\lambda_{j}} - 1\Big|  &\lesssim \frac{1}{\lam_{j}} \| \chi(  \cdot/ L \lam_j ) \Lam W_{\U{\lam_j}} \|_{L^{\frac{10}{7}}} \| g \|_{L^{\frac{10}{3}}} +\sum_{i < j} \frac{1}{\lam_j} \Big| \La \chi(  \cdot/ L \lam_j ) \Lam W_{\U{\lam_j}}  \mid W_{\lam_i} \Ra \Big| \\
&\lesssim L^{\frac{1}{2}} \| \bs g \|_{\E}  +L^2  \sum_{i < j} \Big( \frac{\lam_i}{\lam_j} \Big)^{\frac{3}{2}}
}
which proves~\eqref{eq:xi_j-lambda_j-5} as long as $\eta_0$ is sufficiently small compared to $L$. 

Next, we compute $\xi_j'(t)$. We have, 
\EQ{ \label{eq:xi'-exp-5} 
\xi_j' &=  \lam_j'  - \frac{\iota_j}{\| \Lam W \|_{L^2}^2}  \La \chi( \cdot/ L\lam_j ) \Lam W_{\U{\lam_j}} \mid \p_t g  \Ra  +  \frac{\iota_j}{\| \Lam W \|_{L^2}^2} \La \chi( \cdot/ L\lam_j ) \Lam W_{\U{\lam_j}} \mid  \sum_{i< j} \iota_i \lam_i' \Lam W_{\U \lam_i} \Ra\\
  &\quad + \frac{\iota_j}{\| \Lam W \|_{L^2}^2}\frac{ \lam_j' }{\lam_j} \La   (r \p_r \chi)( \cdot/  L\lam_j)  \Lam W_{\U{\lam_j}} \mid g +   \sum_{i <j} \iota_i W_{\lam_i(t)} \Ra \\
  &\quad + \frac{\iota_j}{\| \Lam W \|_{L^2}^2} \frac{\lam_{j}'}{\lam_j} \La \chi( \cdot/ L\lam_j) \ULam \Lam W_{\U{\lam_j}} \mid g+   \sum_{i <j} \iota_i W_{\lam_i(t)} \Ra . 
}
The last two terms above are acceptable errors. Indeed, 
\EQ{
\Big|&  \frac{ \lam_j' }{\lam_j} \La   (r \p_r \chi)( \cdot/  L\lam_j)  \Lam W_{\U{\lam_j}} \mid g +   \sum_{i <j} \iota_i W_{\lam_i(t)} \Ra \Big| \\
&\lesssim  \frac{|\lam_j'|}{\lam_j} \Big( \|  (r \p_r \chi)( \cdot/  L\lam_j)  \Lam W_{\U{\lam_j}} \|_{L^{\frac{10}{7}}}  \| g \|_{L^{\frac{10}{3}}} +\sum_{i < j}  \Big| \La   (r \p_r \chi)( \cdot/  L\lam_j)  \Lam W_{\U{\lam_j}}  \mid W_{\lam_i} \Ra \Big|\Big) \\
&\lesssim \Big( L^{\frac{1}{2}} \| \bs g \|_{\E}  +L^2  \sum_{i < j} \Big( \frac{\lam_i}{\lam_j} \Big)^{\frac{3}{2}} \Big) \bfd(t) 
}
and similarly, 
\EQ{
\Big|& \frac{\lam_{j}'}{\lam_j} \La \chi( \cdot/ L\lam_j) \ULam \Lam W_{\U{\lam_j}} \mid g+   \sum_{i <j} \iota_i W_{\lam_i(t)} \Ra  \Big| \lesssim \Big( L^{\frac{1}{2}} \| \bs g \|_{\E}  +L^2  \sum_{i < j} \Big( \frac{\lam_i}{\lam_j} \Big)^{\frac{3}{2}} \Big) \bfd(t) 
}
Using~\eqref{eq:g-eq} in the second term in~\eqref{eq:xi'-exp-5} gives 
\EQ{
- \frac{\iota_j}{\| \Lam W \|_{L^2}^2}  \La \chi( \cdot/ L\lam_j ) \Lam W_{\U{\lam_j}} \mid \p_t g  \Ra&= - \frac{\iota_j}{\| \Lam W \|_{L^2}^2}  \La \chi( \cdot/ L\lam_j ) \Lam W_{\U{\lam_j}} \mid \dot g  \Ra  \\
&\quad -  \frac{\iota_j}{\| \Lam W \|_{L^2}^2}  \La\chi( \cdot/ L\lam_j ) \Lam W_{\U{\lam_j}} \mid \sum_{i=1}^K \iota_i \lam_{i}' \Lam W_{\U{\lam_i}} \Ra\\
&\quad - \frac{\iota_j}{\| \Lam W \|_{L^2}^2}  \La\chi( \cdot/ L\lam_j ) \Lam W_{\U{\lam_j}} \mid \phi( u, \nu)  \Ra. 
}
The first term on the right satisfies, 
\EQ{
- \frac{\iota_j}{\| \Lam W \|_{L^2}^2}  \La \chi( \cdot/ L\lam_j ) \Lam W_{\U{\lam_j}} \mid \dot g  \Ra   &=  -\frac{\iota_j}{\| \Lam W \|_{L^2}^2}  \La  \Lam W_{\U{\lam_j}} \mid \dot g  \Ra  +  \frac{\iota_j}{\| \Lam W \|_{L^2}^2}  \La (1-\chi( \cdot /L\lam_j )) \Lam W_{\U{\lam_j}} \mid \dot g  \Ra  \\
& =  -\frac{\iota_j}{\| \Lam W \|_{L^2}^2}  \La  \Lam W_{\U{\lam_j}} \mid \dot g  \Ra  + o_L(1)\|\bs g \|_{\E}. 
}
where the $o_L(1)$ term can be made as small as we like by taking $L>0$ large. 
Using~\eqref{eq:lam'-new}, the second term yields, 
\EQ{
 -  \frac{\iota_j}{\| \Lam W \|_{L^2}^2}  &\La \chi( \cdot/ L\lam_j ) \Lam W_{\U{\lam_j}} \mid \sum_{i=1}^K \iota_i \lam_{i}' \Lam W_{\U{\lam_i}} \Ra 
 \\
 &= - \lam_j'    - \sum_{i \neq j} \frac{\iota_j\iota_i \lam_{i}'}{\| \Lam W \|_{L^2}^2}   \La\chi( \cdot/ L\lam_j ) \Lam W_{\U{\lam_j}} \mid   \Lam W_{\U{\lam_i}} \Ra\\
 &\quad  +   \frac{  \lam_{j}' }{\| \Lam W \|_{L^2}^2}  \La (1-  \chi( \cdot/ L\lam_j )) \Lam W_{\U{\lam_j}} \mid \Lam W_{\U{\lam_j}} \Ra \\
 & = - \lam_j'   - \sum_{i < j} \frac{\iota_j\iota_i \lam_{i}'}{\| \Lam W \|_{L^2}^2}   \La\chi( \cdot/ L\lam_j ) \Lam W_{\U{\lam_j}} \mid   \Lam W_{\U{\lam_i}} \Ra + O( (\lam_j/ \lam_{j+1})^{\frac{5}{2}} + o_L(1))  \bfd(t). 
} 
Finally, the third term vanishes due to the fact that for each $j< K$,  $L\lam_j  \ll \lam_K \ll \nu$, and hence
\EQ{
  \La \chi( \cdot/ L\lam_j ) \Lam W_{\U{\lam_j}} \mid \phi( u, \nu)  \Ra  = 0. 
}
Plugging this all back into~\eqref{eq:xi'-exp-5} gives, 
\EQ{
\Big|\xi_j'  +  \frac{\iota_j}{\| \Lam W \|_{L^2}^2}  \La \chi( \cdot/ L\lam_j ) \Lam W_{\U{\lam_j}} \mid \dot g  \Ra \Big|  \le c_0 \bfd(t) 
}
after first choosing $L$ sufficiently large, and then $\eta_0$ sufficiently small. The estimates~\eqref{eq:xi_j'-5} is immediate, and~\eqref{eq:xi_j'-beta_j-5} now follows as in the proof of~\eqref{eq:lam_j'-beta_j} in Lemma~\ref{cor:modul}. 

The estimate~\eqref{eq:beta_j'-5} is proved exactly as in the proof of~\eqref{eq:beta_j'} in Lemma~\ref{cor:modul}. 

Next, we have, 
\EQ{
\abs{ a_{j}^{\pm} - \ti a_j^{\pm}} &\lesssim  \Big| \La \bs \al_{\lam_j}^{\pm} \mid \sum_{i < j} \iota_i \bs W_{\lam_i} \Ra  \Big| \lesssim \sum_{i < j} \frac{\lam_i}{\lam_j} \Big| \La \calY_{\U{\lam_j}}  \mid W_{\U{\lam_i}} \Ra \Big|  \lesssim \sum_{i < j} \Big(\frac{\lam_i}{\lam_j} \Big)^{\frac{3}{2}} 
}
which proves~\eqref{eq:a-tia-5}. 

Lastly, we prove~\eqref{eq:dta-5}, which is analogous to the proof of~\eqref{eq:dta}, but now using~\eqref{eq:g-bound-5} and~\eqref{eq:d-bound-5}, and noting an extra cancellation of the contribution of the interior bubbles. 
We compute, 
\EQ{ \label{eq:a'-5} 
\frac{\ud}{\ud t} \ti a_{j}^- = \La \p_t  \bs \al_{\lam_j}^- \mid \bs g + \sum_{i <j} \iota_i \bs W_{\lam_i}\Ra + \La \bs \al_{\lam_j}^- \mid \p_t \bs g \Ra - \sum_{i< j} \iota_i\frac{  \kappa}{2} \frac{\lam_i'}{\lam_j}  \La  \calY_{\U{\lam_j}} \mid  \Lam W_{\U{\lam_i}} \Ra
}
Expanding the first term on the right gives, 
\EQ{
\La \p_t  \bs \al_{\lam_j}^- \mid \bs g\Ra  &= \frac{\kappa}{2}\La \p_t( \lam_j^{-1} \calY_{\U{\lam_j}}) \mid g + \sum_{i <j} \iota_i  W_{\lam_i} \Ra + \frac{1}{2} \La \p_t( \calY_{\U{\lam_j}}) \mid \dot g \Ra \\
& = -\frac{\kappa}{2}\frac{\lam_j'}{\lam_j}\La  \lam_j^{-1} \calY_{\U{\lam_j}} + \frac{1}{\lam_j} ( \ULam \calY)_{\U{\lam_j}} \mid g + \sum_{i <j} \iota_i  W_{\lam_i}\Ra - \frac{1}{2} \frac{\lam_j'}{\lam_j} \La  (\ULam  \calY)_{\U{\lam_j}}) \mid \dot g \Ra 
}
and thus, 
\EQ{
\Big|\La \p_t  \bs \al_{\lam_j}^- \mid \bs g\Ra \Big| \lesssim \frac{1}{\lam_j}\bfd(t)^2  . 
}
We use~\eqref{eq:g-eq-ham} to expand the second term, 
\EQ{ \label{eq:a'-exp-5D} 
\La \bs \al_{\lam_j}^- \mid \p_t \bs g \Ra &= \La \bs \al_{\lam_j}^- \mid J \circ \uD^2 E( \bs  \calW(\vec \iota, \vec \lam)) \bs g  \Ra  \\
&\quad +\La \bs \al_{\lam_j}^- \mid J \circ \Big( \uD E( \bs  \calW(\vec \iota, \vec \lam) + \bs g)  - \uD^2 E( \bs  \calW(\vec \iota, \vec \lam)) \bs g\Big)\Ra \\
&\quad -  \La \bs \al_{\lam_j}^- \mid \p_t \bs \calW( \vec \iota, \vec \lam) \Ra \\
&\quad + \La \bs \al_{\lam_j}^- \mid  \Big(\chi( \cdot/ \nu)  J \circ \uD E( \bs u) - J \circ \uD E( \chi(\cdot/ \nu) \bs u)\Big)  \Ra \\
&\quad - \frac{\nu'}{\nu}\La \bs \al_{\lam_j}^- \mid  (r \p_r \chi)( \cdot/ \nu) \bs u\Ra
}
By~\eqref{eq:alpha-pair} the first term on the right gives the leading order, 
\EQ{
\La \bs \al_{\lam_j}^- \mid J \circ \uD^2 E( \bs  \calW(\vec \iota, \vec \lam)) \bs g  \Ra = -\frac{\kappa}{\lam_j} a_j^-.
}
Next, we expand, 
\EQ{
\La \bs \al_{\lam_j}^- \mid J \circ \Big( \uD E( \bs  \calW(\vec \iota, \vec \lam) + \bs g)  &- \uD^2 E( \bs  \calW(\vec \iota, \vec \lam)) \bs g\Big)\Ra  \\
&= -\frac{1}{2} \La \calY_{\U \lam_j} \mid f( \calW( \vec \iota, \vec \lam) + g) - f( \calW( \vec \iota, \vec \lam))  - f'( \calW( \vec \iota, \vec \lam) )g \Ra \\
&\quad - \frac{1}{2} \La \calY_{\U \lam_j} \mid f( \calW( \vec \iota, \vec \lam))  - \sum_{i =1}^K \iota_i f(W_{\lam_i}) \Ra.
}
The first line satisfies, 
\EQ{
\Big| \La \calY_{\U \lam_j} \mid f( \calW( \vec \iota, \vec \lam) + g) - f( \calW( \vec \iota, \vec \lam))  - f'( \calW( \vec \iota, \vec \lam) )g \Ra \Big| \lesssim \frac{1}{\lam_j} (\bfd(t)^2 + o_n(1)).
}
Noting that $f( \calW( \vec \iota, \vec \lam))  - \sum_{i =1}^K \iota_i f(W_{\lam_i}) = f_{\bfi}(\vec \iota, \vec \lam)$, the same argument used to prove Lemma~\ref{lem:interaction} gives, 
\EQ{
\Big| \La \calY_{\U \lam_j} \mid  f_{\bfi}(\vec \iota, \vec \lam) \Ra\Big|  \lesssim \frac{1}{\lam_j} \Big( \Big(\frac{\lam_{j}}{\lam_{j+1}} \Big)^{\frac{3}{2}} + \Big(\frac{\lam_{j-1}}{\lam_{j}} \Big)^{\frac{3}{2}} \Big) \lesssim \frac{1}{\lam_j} \bfd(t)^2 .
}
Consider now the third line in~\eqref{eq:a'-exp-5D}. 
 \EQ{
 -  \La \bs \al_{\lam_j}^- \mid \p_t \bs \calW( \vec \iota, \vec \lam) \Ra &=  \frac{\kappa}{2} \iota_j \frac{ \lam_j' }{\lam_j} \La \calY_{\U{\lam_j}}  \mid  \Lam W_{\U{\lam_j}} \Ra  + \sum_{i \neq j}  \iota_i \frac{\kappa}{2} \frac{\lam_i' }{\lam_j} \La \calY_{\U{\lam_j}}  \mid \Lam W_{\U{\lam_i}} \Ra  \\
 &= \sum_{i \neq j}  \iota_i \frac{\kappa}{2} \frac{\lam_i' }{\lam_j} \La \calY_{\U{\lam_j}}  \mid \Lam W_{\U{\lam_i}} \Ra 
 }
where in the last equality we used the vanishing $\La \calY \mid \Lam W \Ra$. Noting the estimates, 
\EQ{
\Big| \La \calY_{\U{\lam_j}}  \mid \Lam W_{\U{\lam_i}} \Ra  \Big| \lesssim  \Big(\frac{\lam_{i}}{\lam_j} \Big)^{\frac{5}{2}}  \mif i>j  
}
we obtain, 
\EQ{
\Big| - \La \bs \al_{\lam_j}^- \mid \p_t \bs \calW( \vec \iota, \vec \lam) \Ra - \sum_{i \neq j}  \iota_i \frac{\kappa}{2} \frac{\lam_i' }{\lam_j} \La \calY_{\U{\lam_j}}  \mid \Lam W_{\U{\lam_i}} \Ra \Big| \lesssim  \frac{1}{\lam_j} \bfd(t)^2 .
}
Using~\eqref{eq:nu-prop} and~\eqref{eq:nu'} we see that the last two lines of~\eqref{eq:a'-exp-5D} satisfy, 
\EQ{
\Big| \La \bs \al_{\lam_j}^- \mid  \Big(\chi( \cdot/ \nu)  J \circ \uD E( \bs u) - J \circ \uD E( \chi(\cdot/ \nu) \bs u)\Big)  \Ra  \Big|  \lesssim  \frac{1}{\lam_j} o_n(1),  \\
\Big| \frac{\nu'}{\nu}\La \bs \al_{\lam_j}^- \mid  (r \p_r \chi)( \cdot/ \nu) \bs u\Ra \Big| \lesssim \frac{1}{\lam_j} o_n(1). 
}
Plugging this all back into~\eqref{eq:a'-5} and using~\eqref{eq:a-tia-5} we obtain, 
\EQ{
\Big| \ti a_j^- + \frac{\kappa}{\lam_j} \ti a_j^- \Big| \lesssim \frac{1}{\lam_j} (\bfd(t)^2 + o_n(1))
}
This completes the proof after ensuring $\eps_n$ is large enough so that the $o_n(1)$ term above can be absorbed into $\bfd(t)$. 
\end{proof} 

\subsection{Conclusion of the proof}
Using the modulation estimates above, we can prove the following analog of Lemma~\ref{lem:ejection}.
\begin{lem}
\label{lem:ejection-5D}
Let $D = 5$. If $\eta_0$ is small enough,
then there exists $C_0 \geq 0$ depending only on $N$ such that the following is true.
If $[t_1, t_2] \subset I_*$ is a finite time interval such that $\bfd(t) \leq \eta_0$ for all $t \in [t_1, t_2]$, then
\begin{gather}
\label{eq:lambdaK-bd-5D}
\sup_{t \in [t_1, t_2]} \lambda_K(t) \leq \frac 43\inf_{t \in [t_1, t_2]} \lambda_K(t), \\
\label{eq:int-d-bd-5D}
\int_{t_1}^{t_2} \bfd(t)\ud t\leq C_0\big(\bfd(t_1) \lambda_K(t_1) + \bfd(t_2) \lambda_K(t_2)\big).
\end{gather}
\end{lem}
\begin{proof}[Sketch of a proof]
\textbf{Step 1.} is exactly the same as for Lemma~\ref{lem:ejection}.

\noindent
\textbf{Step 2.}
Let $C_1 > 0$ be a large number chosen below and consider the auxiliary function
\begin{equation}
\phi(t) := \sum_{j \in \cS}2^{-j}\xi_j(t)\beta_j(t) - C_1\sum_{j=1}^K \lambda_j(t)\wt a_j^-(t)^2
+ C_1\sum_{j=1}^K \lambda_j(t) \wt a_j^+(t)^2.
\end{equation}
We claim that for all $t \in [t_1, t_2]$
\begin{equation}
\label{eq:dtphi-lbound-5D}
\phi'(t) \geq c_2\bfd(t)^2,
\end{equation}
with $c_2 > 0$ depending only on $N$. The remaining part of Step 1 is devoted to proving this bound.

Using \eqref{eq:xi_j'-beta_j-5}, \eqref{eq:dta-5} and recalling that $|\lambda_K'(t)| \lesssim \bfd(t)$, we obtain
\begin{equation}
\label{eq:dtphi-lbound-3-5D}
\begin{aligned}
\phi'(t) \geq \sum_{j \in \cS}2^{-j}\beta_j^2(t) + \sum_{j\in\cS}2^{-j}\lambda_j(t)\beta_j'(t)
+ C_1\nu\sum_{j=1}^K \big(\wt a_j^-(t)^2 + \wt a_j^+(t)^2\big) - c_0 \bfd(t)^2,
\end{aligned}
\end{equation}
where $c_0 > 0$ can be made arbitrarily small.
We focus on the second term of the right hand side. Like in Step 2. of the proof of Lemma~\ref{lem:ejection},
only using \eqref{eq:beta_j'-5} instead of \eqref{eq:beta_j'}, we obtain
\begin{equation}
\sum_{j\in\cS}\lambda_j(t)\beta_j'(t) \geq 2^{-N-1}\omega^2\sum_{j\in\cS}\Big(\frac{\lambda_j(t)}{\lambda_{j+1}(t)}\Big)^\frac{D-2}{2} - C_2\sum_{k=1}^K\big(a_j^-(t)^2 + a_j^+(t)^2\big) - c_0\bfd(t)^2.
\end{equation}
The bound \eqref{eq:a-tia-5} implies that \eqref{eq:dtphi-lbound-3-5D} holds with with $\wt a_j^\pm$ replaced by $a_j^\pm$.
Taking $C_1 > C_2 / \nu$ and using \eqref{eq:d-bound-5}, we get \eqref{eq:dtphi-lbound-5D}.

\noindent
\textbf{Step 3.}
As in Step 3. of the proof of Lemma~\ref{lem:ejection}, it suffices to check that if $t_3 \in [t_1, t_2]$
and $\phi(t) > 0$ for all $t \in (t_3, t_2]$, then
\begin{equation}
\label{eq:int-d-bd-1-5D}
\int_{t_3}^{t_2}\bfd(t)\ud t \leq C_0\bfd(t_2)\lambda_K(t_2).
\end{equation}

Observe that for all $t \in (t_3, t_2]$ we have
\begin{equation}
\label{eq:phi-lambdaK-bd-5D}
\begin{aligned}
\frac{\phi(t)}{\lambda_K(t)} &\lesssim \sum_{j \in \cS}\frac{\lambda_j(t)}{\lambda_{K}(t)}|\beta_j(t)| + \sum_{j=1}^K \frac{\lambda_j(t)}{\lambda_K(t)}\big(\wt a_j^-(t)^2 + \wt a_j^+(t)^2\big) \\
&\lesssim \sum_{j \in \cS}\frac{\lambda_j(t)}{\lambda_{j+1}(t)}|\beta_j(t)| + \bfd(t)^2
\lesssim \bfd(t)^{\frac 43 + 1} + \bfd(t)^2 \lesssim \bfd(t)^2.
\end{aligned}
\end{equation}
Combining this bound with \eqref{eq:dtphi-lbound-5D}, for all $t \in (t_3, t_2]$ we get
\begin{equation}
\phi'(t) \geq c_2 \sqrt{\phi(t) / \lambda_K(t)} \bfd(t),
\end{equation}
thus
\begin{equation}
\label{eq:dtphi-lbound-2-5D}
\sqrt{\lambda_K(t)} \big(\sqrt{\phi(t)}\big)' \gtrsim \bfd(t).
\end{equation}
Using \eqref{eq:phi-lambdaK-bd-5D} and $|\lambda_K'(t)| \lesssim \bfd(t)$, we get
\begin{equation}
\big(\sqrt{\lambda_K(t)\phi(t)}\big)' - \sqrt{\lambda_K(t)} \big(\sqrt{\phi(t)}\big)' = \frac 12\lambda_K'(t)\sqrt{\phi(t)/\lambda_K(t)} \gtrsim -\bfd(t)^2.
\end{equation}
Since $\bfd(t)$ is small, \eqref{eq:dtphi-lbound-2-5D} yields
\begin{equation}
\big(\sqrt{\lambda_K(t)\phi(t)^\frac{4}{D+2}}\big)' \gtrsim \bfd(t)
\end{equation}
which, integrated, gives
\begin{equation}
\int_{t^3}^{t_2}\bfd(t)\ud t \lesssim \sqrt{\lambda_K(t_2)\phi(t_2)} - \sqrt{\lambda_K(t_3)\phi(t_3)}
\leq \sqrt{\lambda_K(t_2)\phi(t_2)}.
\end{equation}
Invoking \eqref{eq:phi-lambdaK-bd-5D}, we obtain \eqref{eq:int-d-bd-1-5D}.
\end{proof}
Using Lemma~\ref{lem:ejection-5D} in lieu of Lemma~\ref{lem:ejection},
the remaining arguments in Section~\ref{sec:conclusion} hold for $D=5$ without changes.

\section{Modifications to the argument in the case $D = 4$}  \label{sec:D=4}
In this section we outline the changes to the arguments in Section~\ref{sec:decomposition} and  Section~\ref{sec:conclusion} needed to prove Theorem~\ref{thm:main} dimension $D=4$. 

\subsection{Decomposition of the solution} 

The set-up in Sections~\ref{ssec:proximity} holds without modification for $D=4$. The number $K \ge1$ is defined as in Lemma~\ref{lem:K-exist}, the collision intervals $[a_n, b_n] \in \calC_K(\eta, \eps_n)$ are as in Definition~\ref{def:K-choice}, and  the sequences of signs $\vec \s_n \in \{-1, 1\}^{N-K}$,  scales $\vec \mu(t) \in (0, \infty)^{N-K}$, and the sequence $\nu_n \to 0$ and  the function $\nu(t) = \nu_n \mu_{K+1}(t)$ are as in Lemma~\ref{lem:ext-sign}. 

 Lemma~\ref{lem:mod-1} also holds with a minor modification to the stable/unstable components. Let $J \subset [a_n, b_n]$ be any time interval on which $\bfd(t) \le \eta_0$, where $\eta_0$ is as in  Lemma~\ref{lem:mod-1}. Let $\vec \iota \in \{ -1, 1\}^K, \vec \lam(t) \in (0, \infty)^K$, $\bs g(t) \in \E$, and $a_{j}^{\pm}(t)$  be as in the statement of Lemma~\ref{lem:mod-1}. Define for each $1 \le j \le K$, the modified stable/unstable components, 
 \EQ{
 \ti a_{j}^{\pm}(t) := \La \bs \al_{\lam_j(t)}^{\pm} \mid \bs g(t) + \sum_{i < j} \iota_i \bs W_{\lam_i(t)} \Ra 
 }
 Let $L>0$ be a parameter to be fixed below and for each $j \in \{1, \dots, K-1\}$ set, 
\EQ{\label{eq:xi-def-k1} 
\xi_j(t) &:=  \lam_j(t)  - \frac{ \iota_j}{8\log(\frac{\lam_{j+1}(t)}{\lam_j(t)}) }\La \chi_{L\sqrt{\lam_j(t) \lam_{j+1}(t)}} \Lam W_{\U{\lam_j(t)}} \mid g(t) + \sum_{i<j} \iota_iW_{\lam_i(t)}  \Ra,
}
and, 
\begin{equation} \label{eq:beta-def-k1} 
\beta_j(t) := - \iota_j\La \chi_{L\sqrt{\xi_j(t) \lam_{j+1}(t)}} \Lam W_{\U{\lam_j(t)}} \mid \dot g(t)\Ra  -   \ang{ \uln A( \lam_j(t)) g(t) \mid \dot g(t)}  
\end{equation}

\begin{prop}[Refined modulation, $D=4$]
\label{prop:modul-k1} 
 Let $c_0  \in (0, 1)$ and $c_1>0$. There exists  $L_0 = L_0(c_0, c_1)>0$, $\eta_0 = \eta_0(c_0, c_1)$, as well as  $c= c(c_0, c_1)$ and $R = R(c_0, c_1)>1$ as in Lemma~\ref{lem:q},  a constant $C_0>0$ and a 
decreasing sequence $\epsilon_n \to 0$ so that the following is true. 

Suppose $L > L_0$ and  $J \subset [a_n, b_n]$ is an open time interval with $\epsilon_n \leq \bfd(t) \le \eta_0$
for all $t \in J$, where $\calS :=  \{ j \in \{1, \dots, K-1\} \mid \iota_{j}  = \iota_{j+1} \}$.  
Then, for all $t \in J$, 
\EQ{\label{eq:g-bound-k1} 
\| \bs g(t) \|_{\E} + \sum_{i \not \in \calS}  (\lambda_i(t) / \lambda_{i+1}(t))^\frac 12 \le  \max_{ i  \in \calS}  (\lambda_i(t) / \lambda_{i+1}(t))^\frac 12  + \max_{1 \leq i \leq K}|a_i^\pm(t)|, \\
}
and, 
\EQ{ \label{eq:d-bound-k1} 
\frac{1}{C_0} \bfd(t) \le \max_{i \in \calS} (\lambda_i(t) / \lambda_{i+1}(t))^\frac 12  + \max_{1 \leq i \leq K}|a_i^\pm(t)| \le C_0 \bfd(t) .
}
\begin{equation}\label{eq:xi_j-lambda_j-k1}
\Big|\frac{\xi_j(t)}{\lambda_{j+1}(t)} - \frac{\lambda_j(t)}{\lambda_{j+1}(t)}\Big| \leq C_0 \frac{\bfd(t)^2}{\log( \lam_{j+1}/ \lam_j)}, 
\end{equation}
as well as,  
\EQ{ \label{eq:a-tia-4} 
\Big| a_j^{\pm}(t) - \ti a_j^{\pm}(t) \Big| \le C_0\bfd(t)^2,  
}
and, 
\begin{equation}
\label{eq:dta-4}
\Big| \dd t \ti a_j^\pm(t) \mp \frac{\kappa}{\lambda_j(t)}\ti a_j^\pm(t) \Big| \leq \frac{C_0}{\lambda_j(t)}\bfd(t)^2 . 
\end{equation}
Moreover, let $j \in \calS$ be such that for all $t \in J$
\EQ{ \label{eq:max-ratio-ass}
c_1 \bfd(t) \le \Big(\frac{\lam_j(t)}{\lam_{j+1}(t)}\Big)^{\frac{1}{2}}.
}
Then for all $t \in J$, 
\EQ{ \label{eq:xi_j'-k1} 
\abs{ \xi'_j(t) }\Big(\log\Big(\frac{\lam_{j+1}(t)}{\lam_j(t)}\Big)\Big)^{\frac{1}{2}} \le C_0 \bfd(t) , 
}
\begin{equation}\label{eq:xi_j'-beta_j-k1} 
\Big|\xi_j'(t) 8\log(\frac{\lam_{j+1}(t)}{\lam_j(t)})- \beta_j(t)\Big| \leq {C_0} \bfd(t),  
\end{equation}
 and,  
\EQ{ \label{eq:beta_j'-4} 
 \beta_{j}'(t) &\ge  \Big(\iota_j \iota_{j+1}16 -   c_0\Big) \frac{1}{\lam_{j+1}(t)} + \Big({-}\iota_j \iota_{j-1}16 -  c_0\Big) \frac{\lam_{j-1}(t)}{\lam_{j}(t)^2}    \\
 &\quad - \frac{c_0}{\lam_j(t)} \bfd(t)^2  - \frac{C_0}{\lam_{j}(t)} \Big( (a_j^+(t))^2 + (a_j^-(t))^2 \Big) 
}
where, by convention, $\lambda_0(t) = 0, \lam_{K+1}(t) = \infty$ for all $t \in J$.  

\end{prop}

\begin{rem} 
Proposition~\ref{prop:modul-k1} and its proof are nearly identical to~\cite[Proposition A.1]{JL6} and its proof, which treat the case  $k=1$ for the energy-critical equivariant wave map equation. 
\end{rem}


\begin{proof} 
The estimates~\eqref{eq:g-bound-k1} and~\eqref{eq:d-bound-k1} follow as in the proofs of the corresponding estimates in Lemma~\ref{cor:modul}. We next prove~\eqref{eq:xi_j-lambda_j-k1}. From the definition of $\xi_j(t)$, 
\EQ{
\bigg|\frac{\xi_j}{\lambda_{j+1}} - \frac{\lam_j}{\lam_{j+1}}\bigg| &\lesssim \Big|\frac{ 1}{\log(\frac{\lam_{j+1}}{\lam_j}) } \lam_{j+1}^{-1} \La \chi_{L\sqrt{\lam_j \lam_{j+1}}} \Lam W_{\U{\lam_j}} \mid g \Ra \Big| \\
&\quad + \Big| \frac{ 1}{\log(\frac{\lam_{j+1}}{\lam_j}) }\lam_{j+1}^{-1}\La \chi_{L\sqrt{\lam_j \lam_{j+1}}} \Lam W_{\U{\lam_j}} \mid \sum_{i<j} W_{\lam_i}  \Ra \Big|
}
For the first term on the right we have, 
\EQ{
\Big|\frac{ 1}{\log(\frac{\lam_{j+1}}{\lam_j}) } \lam_{j+1}^{-1} \La \chi_{L\sqrt{\lam_j \lam_{j+1}}} \Lam W_{\U{\lam_j}} \mid g \Ra \Big| 
& \lesssim_L \frac{ 1}{\log(\frac{\lam_{j+1}}{\lam_j}) } \|g \|_{H} (\lam_{j}/ \lam_{j+1})^{\frac{1}{2}}  
}
 Next, for any $i <j$ we have, 
\EQ{
\lam_{j+1}^{-1}|\La \chi_{L\sqrt{\lam_j \lam_{j+1}}} \Lam W_{\U{\lam_j}} \mid W_{\lam_i}  \Ra | &\lesssim_L   \frac{\lam_j}{\lam_{j+1}}  \int_0^{L (\lam_{j+1}/ \lam_j)^{\frac{1}{2}}} | \Lam W(r) | |W_{\lam_i/ \lam_j}(r)|\, r^3 \, \ud r\\
&\lesssim_L  \frac{\lam_j}{\lam_{j+1}}\frac{\lam_i}{\lam_j} \log( \lam_{j+1}/ \lam_j)
}
and hence, 
\EQ{
 \Big| \frac{ 1}{\log(\frac{\lam_{j+1}}{\lam_j}) }\lam_{j+1}^{-1}\La \chi_{L\sqrt{\lam_j \lam_{j+1}}} \Lam W_{\U{\lam_j}} \mid \sum_{i<j} W_{\lam_i}  \Ra \Big|  \lesssim_L \frac{\lam_j}{\lam_{j+1}} \sum_{i<j} \frac{\lam_{i}}{\lam_j}
}
and ~\eqref{eq:xi_j-lambda_j-k1} follows. 

Next using~\eqref{eq:g-bound-k1} and~\eqref{eq:lam'} for each $j$, we have
\EQ{ \label{eq:lam'-k1} 
\abs{ \lam_j'}  \lesssim \bfd(t) 
}
We show that in fact $\xi_j'$ satisfies the improved estimate~\eqref{eq:xi_j'-k1}. We compute, 
\EQ{ \label{eq:xi'-exp-k1} 
\xi_j' &=  \lam_j'  - \frac{ \iota_j}{8(\log(\frac{\lam_{j+1}}{\lam_j}))^2 }( \frac{\lam_{j}'}{\lam_j} - \frac{\lam_{j+1}'}{\lam_{j+1}})\La \chi( \cdot/ L\sqrt{\lam_j \lam_{j+1}}) \Lam W_{\U{\lam_j}} \mid g + \sum_{i<j} \iota_i W_{\lam_i} \Ra\\
& \quad + \frac{ \iota_j}{16\log(\frac{\lam_{j+1}}{\lam_j}) }\big( \frac{\lam_j'}{\lam_j} + \frac{\lam_{j+1}'}{\lam_{j+1}}\big)\La (r \p_r \chi)( \cdot/ L\sqrt{\lam_j \lam_{j+1}}) \Lam W_{\U{\lam_j}} \mid g + \sum_{i<j} \iota_i W_{\lam_i} \Ra \\
&\quad +  \frac{ \iota_j}{8\log(\frac{\lam_{j+1}}{\lam_j}) } \frac{\lam_j'}{\lam_j} \La \chi( \cdot/ L\sqrt{\lam_j \lam_{j+1}}) \ULam \Lam W_{\U{\lam_j}} \mid g + \sum_{i<j} \iota_i W_{\lam_i}  \Ra \\
&\quad  -  \frac{ \iota_j}{8\log(\frac{\lam_{j+1}}{\lam_j}) }\La \chi( \cdot/ L\sqrt{\lam_j \lam_{j+1}})\Lam W_{\U{\lam_j}} \mid \p_t g \Ra\\
&\quad +  \sum_{i<j} \frac{ \iota_i \iota_j}{8\log(\frac{\lam_{j+1}}{\lam_j}) } \lam_i' \La \chi( \cdot/ L\sqrt{\lam_j \lam_{j+1}}) \Lam W_{\U{\lam_j}} \mid \Lam W_{\U{\lam_i}} \Ra
}
The second, third, and fourth terms on the right above contribute acceptable errors. Indeed, 
\EQ{
\Big|  \frac{ \iota_j}{8(\log(\frac{\lam_{j+1}}{\lam_j}))^2 }( \frac{\lam_{j}'}{\lam_j} - \frac{\lam_{j+1}'}{\lam_{j+1}})\La \chi( \cdot/ L\sqrt{\lam_j \lam_{j+1}}) \Lam W_{\U{\lam_j}} \mid g + \sum_{i<j} \iota_iW_{\lam_i} \Ra \Big| &\lesssim \frac{ \bfd(t)}{c_1(\log(\frac{\lam_{j+1}}{\lam_j}))^2} \\
\Big|  \frac{ \iota_j}{16\log(\frac{\lam_{j+1}}{\lam_j}) }\big( \frac{\lam_j'}{\lam_j} + \frac{\lam_{j+1}'}{\lam_{j+1}}\big)\La (r \p_r \chi)( \cdot/ L\sqrt{\lam_j \lam_{j+1}}) \Lam W_{\U{\lam_j}} \mid g + \sum_{i<j} \iota_i W_{\lam_i}\Ra  \Big| & \lesssim
\frac{\bfd(t) }{c_1 \log(\frac{\lam_{j+1}}{\lam_j})} \\
\Big|\frac{ \iota_j}{8\log(\frac{\lam_{j+1}}{\lam_j}) } \frac{\lam_j'}{\lam_j} \La \chi( \cdot/ L\sqrt{\lam_j \lam_{j+1}}) \ULam \Lam W_{\U{\lam_j}} \mid g + \sum_{i<j} \iota_i W_{\lam_i} \Ra \Big| &\lesssim \frac{ \bfd(t)^2}{\log(\frac{\lam_{j+1}}{\lam_j})}
}
with the gain in the last line arising from the fact that $\ULam \Lam W \in L^{\frac{4}{3}}$; see~\eqref{eq:D4-magic}. 
The leading order comes from the second to last term in~\eqref{eq:xi'-exp-k1}.  Using~\eqref{eq:g-eq} gives 
\EQ{
 -  \frac{ \iota_j}{8\log(\frac{\lam_{j+1}}{\lam_j}) }\La \chi( \cdot/ L\sqrt{\lam_j \lam_{j+1}}) \Lam W_{\U{\lam_j}} \mid \p_t g \Ra &= -    \frac{ \iota_j}{8\log(\frac{\lam_{j+1}}{\lam_j}) }\La \chi( \cdot/ L\sqrt{\lam_j \lam_{j+1}}) \Lam W_{\U{\lam_j}} \mid \dot g \Ra\\
&\quad -  \lam_j' \frac{ 1}{8\log(\frac{\lam_{j+1}}{\lam_j}) }\La \chi( \cdot/ L\sqrt{\lam_j \lam_{j+1}}) \Lam W_{\U{\lam_j}} \mid   \Lam W_{\U{\lam_j}} \Ra\\
&\quad  -  \frac{ \iota_j}{8\log(\frac{\lam_{j+1}}{\lam_j}) }\La \chi( \cdot/ L\sqrt{\lam_j \lam_{j+1}}) \Lam W_{\U{\lam_j}} \mid \sum_{i \neq j} \iota_i \lam_{i}' \Lam W_{\U{\lam_i}} \Ra\\
&\quad - \frac{ \iota_j}{8\log(\frac{\lam_{j+1}}{\lam_j}) }\La \chi( \cdot/ L\sqrt{\lam_j \lam_{j+1}}) \Lam W_{\U{\lam_j}}  \mid \phi( u, \nu)  \Ra. 
}
We estimate the contribution of each of the terms on the right above to~\eqref{eq:xi'-exp-k1}. The last term above vanishes due to the support properties of $\phi(u, \nu)$. Using~\eqref{eq:LamW-L2-4}, ~\eqref{eq:lam'-k1} on the second term above, gives 
\EQ{
\Big|-\lam_j' \frac{ 1}{8\log(\frac{\lam_{j+1}}{\lam_j}) }\La \chi( \cdot/ L\sqrt{\lam_j \lam_{j+1}}) \Lam W_{\U{\lam_j}} \mid   \Lam W_{\U{\lam_j}} \Ra + \lam_j'| \lesssim\frac{ \bfd(t)}{\log(\frac{\lam_{j+1}}{\lam_j})}
}
which means this terms cancels the term $\lam'$ on the right-hand side of~\eqref{eq:xi'-exp-k1}  up to an acceptable error. Next we write, 
\EQ{
-  \frac{ \iota_j}{8\log(\frac{\lam_{j+1}}{\lam_j}) }\La \chi( \cdot/ L\sqrt{\lam_j \lam_{j+1}}) &\Lam W_{\U{\lam_j}} \mid \sum_{i \neq j} \iota_i \lam_{i}' \Lam W_{\U{\lam_i}} \Ra \\
&=  -\sum_{i<j} \frac{ \iota_i \iota_j}{8\log(\frac{\lam_{j+1}}{\lam_j}) } \lam_i' \La \chi( \cdot/ L\sqrt{\lam_j \lam_{j+1}}) \Lam W_{\U{\lam_j}} \mid \Lam W_{\U{\lam_i}} \Ra\\
&\quad  -  \sum_{i>j} \frac{ \iota_i \iota_j}{8\log(\frac{\lam_{j+1}}{\lam_j}) } \lam_i' \La \chi( \cdot/ L\sqrt{\lam_j \lam_{j+1}}) \Lam W_{\U{\lam_j}} \mid \Lam W_{\U{\lam_i}} \Ra
}
The first term cancels the last term in~\eqref{eq:xi'-exp-k1}. For the second term we estimate, if $i>j$, 
\EQ{
|\La \chi( \cdot/ L\sqrt{\lam_j \lam_{j+1}}) \Lam W_{\U{\lam_j}} \mid \Lam W_{\U{\lam_i}} \Ra|  \lesssim \lam_j/ \lam_{j+1}
}
and thus, using~\eqref{eq:lam'-k1} the second term in the previous equation contributes an acceptable error. Plugging all of these estimates back into~\eqref{eq:xi'-exp-k1} gives the estimate, 
\EQ{ \label{eq:xi'-lead} 
\Big| \xi_j' +  \frac{ \iota_j}{8\log(\frac{\lam_{j+1}}{\lam_j}) }\La \chi( \cdot/ L\sqrt{\lam_j \lam_{j+1}}) \Lam W_{\U{\lam_j}} \mid \dot g \Ra \Big| \lesssim \frac{\bfd(t)}{c_1\log(\frac{\lam_{j+1}}{\lam_j})} 
}
Using~\eqref{eq:g-bound-k1} and $\|\chi( \cdot/ L\sqrt{\lam_j \lam_{j+1}})\Lam W_{\U{\lam_j}}  \|_{L^2} \lesssim_L (\log(\frac{\lam_{j+1}}{\lam_j}))^{\frac{1}{2}}$,   we deduce the estimate, 
\EQ{
\Big| \frac{ \iota_j}{8\log(\frac{\lam_{j+1}}{\lam_j}) }\La \chi( \cdot/ L\sqrt{\lam_j \lam_{j+1}}) \Lam W_{\U{\lam_j}} \mid \dot g \Ra\Big| \lesssim_L \frac{\bfd(t)}{(\log(\frac{\lam_{j+1}}{\lam_j}))^{\frac{1}{2}}} 
}
which completes the proof of~\eqref{eq:xi_j'-k1}. 

Next we compare $\be_j$  and $2\xi_j' \log( \lam_{j+1}/ \lam_j)$. Using~\eqref{eq:beta-def-k1} we have, 
\EQ{
\Big|  \ang{ \uln A( \lam_j(t)) g(t) \mid \dot g(t)} \Big| \lesssim \| \bs g \|_{\E}^2 \lesssim  \bfd(t)^2 , 
}
We also note the estimate
\EQ{
\Big| \La ( \chi( \cdot/ L\sqrt{\lam_j \lam_{j+1}}) - \chi( \cdot/ L\sqrt{\lam_j \lam_{j+1}}) \Lam Q_{\U{\lam_j}} \mid \dot g \Ra\Big|  \le \frac{1}{c_1^2} \bfd(t)^2  . 
}
which is a consequence of~\eqref{eq:xi_j-lambda_j-k1}. Using~\eqref{eq:xi'-lead} the estimate~\eqref{eq:xi_j'-beta_j-k1} follows. 

Finally, the proof of the estimate~\eqref{eq:beta_j'-4} is nearly identical to the argument used to prove~\eqref{eq:beta_j'}, differing only in a  few places where the cut-off $\chi_{L \sqrt{\xi_j \lam_{j+1}}}$ is  involved. Arguing as in the proof of ~\eqref{eq:beta_j'} we arrive at the formula, 
\EQ{ \label{eq:beta'-exp-k1} 
\beta_j' &=  -  \frac{\iota_j }{\lam_j}  \ang{ \Lam W_{\lam_j} \mid f_{\bfi}( \iota, \vec \lam) } +  \ang{ \uln A( \lam_j) g \mid -\De g }   + \ang{ (A(\lam_j) - \uln A( \lam_j)) g \mid \ti  f_{\bfq}( \vec \iota,  \vec \lam, g)} \\
&\quad + \ang{ \chi( \cdot/ L \sqrt{\xi_j \lam_{j+1}}) \Lam W_{\U{\lam_j}} \mid ( \LL_{\calW} - \LL_{\lam_j}) g} + \iota_j \frac{\lam_j' }{\lam_j}\ang{ \big( \frac{1}{\lam_j} \ULam - \U{A}( \lam_j) \big) \Lam W_{\lam_j} \mid \dot g} \\
&\quad    - \ang{ {A}(\lam_j) \sum_{i =1}^K \iota_i W_{\lam_i} \mid f_{\bfq}( \vec \iota,  \vec\lam, g)}   - \ang{ A( \lam_j) g \mid \ti  f_{\bfq}( \vec \iota,  \vec \lam, g)} \\
&\quad    + \iota_j \ang{ ({A}( \lam_j) -  \frac{1}{\lam_j}\chi( \cdot/ L \sqrt{\xi_j \lam_{j+1}})\Lam) W_{\lam_j} \mid f_{\bfq}( \vec \iota,  \vec\lam, g)}  -    \frac{\lam_j'}{\lam_j} \ang{ \lam_j \p_{\lam_j} \uln A( \lam_j) g \mid \dot g} \\
&\quad + \sum_{ i \neq j} \iota_i \ang{ {A}(\lam_j) W_{\lam_i} \mid  f_{\bfq}( \vec \iota,  \vec\lam, g)}  - \sum_{i \neq j} \iota_i \lam_{i}'  \ang{ \uln A( \lam_j) \Lam W_{\U{\lam_i}} \mid \dot g}   - \ang{ \uln A( \lam_j) g \mid f_{\bfi}(  \iota, \vec \lam) } \\
&\quad  - \iota_j \ang{ \chi( \cdot/ L \sqrt{\xi_j \lam_{j+1}})\Lam W_{\U{\lam_j}} \mid  \dot \phi( u, \nu)} -  \ang{ \uln A( \lam_j) \phi( u, \nu) \mid \dot g} -  \ang{ \uln A( \lam_j) g \mid \dot  \phi( u, \nu)} \\
&\quad + \frac{\iota_j }{\lam_j}  \ang{ (1- \chi( \cdot/ L \sqrt{\xi_j \lam_{j+1}}))\Lam W_{\lam_j} \mid f_{\bfi}(  \iota, \vec \lam) } -  \iota_j \frac{\lam_j' }{\lam_j}\ang{  (1- \chi( \cdot/ L \sqrt{\xi_j \lam_{j+1}}))\ULam \Lam W_{\U{\lam_j}} \mid \dot g}\\
&\quad + \ang{ \LL_{\lam_j} (\chi( \cdot/ L \sqrt{\xi_j \lam_{j+1}})\Lam W_{\U{\lam_j}}) \mid  g}
 + \frac{ \iota_j}{2}\big( \frac{\xi_j'}{\xi_j} + \frac{\lam_{j+1}'}{\lam_{j+1}}\big)\La \Lam \chi( \cdot/ L \sqrt{\xi_j \lam_{j+1}}) \Lam W_{\U{\lam_j}} \mid \dot g \Ra 
}
All but the last four terms above are treated exactly as in the proof of~\eqref{eq:beta_j'}. For the fourth-to-last term a direct computation gives, 
\EQ{
\Big|\frac{\iota_j }{\lam_j}  \ang{ (1- \chi( \cdot/ L \sqrt{\xi_j \lam_{j+1}}))\Lam W_{\lam_j} \mid f_{\bfi}(\iota, \vec \lam) }  \Big| \ll \frac{1}{\lam_j} \Big( \frac{\lam_j}{\lam_{j+1}} + \frac{\lam_{j-1}}{\lam_j} \Big). 
}
For the third-to-last term, we use that $\ULam \Lam W \in L^2$ (see~\eqref{eq:D4-magic}), ~\eqref{eq:lam'-k1},  and~\eqref{eq:g-bound-k1} to deduce that, 
\EQ{
\Big| \iota_j \frac{\lam_j' }{\lam_j}\ang{  (1- \chi( \cdot/ L \sqrt{\xi_j \lam_{j+1}}))\ULam \Lam W_{\U{\lam_j}} \mid \dot g} \Big| \ll \frac{1}{\lam_j} \bfd(t)^2,
}
The size of the constant $L>0$ becomes relevant only in the second-to-last term. Indeed, since $\LL \Lam W = 0$, we have, 
\begin{multline} 
\LL_{\lam_j} (\chi_{L \sqrt{\xi_j \lam_{j+1}}}\Lam W_{\U{\lam_j}}) \\
= \frac{1}{L^2 \xi_j \lam_{j+1}} (\De \chi)( \cdot/ L \sqrt{\xi_j \lam_{j+1}})\Lam W_{\U{\lam_j}} +  \frac{2}{L \sqrt{ \xi_j \lam_{j+1}}} \chi'( \cdot/ L \sqrt{\xi_j \lam_{j+1}}) \frac{(r \p_r\Lam W)_{ \U{\lam_j}}}{r} . 
\end{multline} 
And therefore, using~\eqref{eq:g-bound-k1} and~\eqref{eq:xi_j-lambda_j-k1} we obtain the estimate, 
\EQ{
\Big| \ang{ \LL_{\lam_j} (\chi( \cdot/ L \sqrt{\xi_j \lam_{j+1}})\Lam W_{\U{\lam_j}}) \mid  g} \Big| \lesssim \frac{1}{L} \frac{1}{\lam_{j}} \bfd(t)^2
}
for a uniform constant, independent of $L$. Taking $L>1$ large enough relative to $c_0$ makes this an acceptable error. Finally, for the last term we use the improved estimate~\eqref{eq:xi_j'-k1} for $\xi_j'$ and~\eqref{eq:xi_j-lambda_j-k1} to obtain, 
\EQ{
\abs{ \frac{\xi_j'}{\xi_j} + \frac{\lam_{j+1}'}{\lam_{j+1}}}  \lesssim \frac{1}{\lam_j} \Big( \abs{ \xi_j'} + \frac{\lam_j}{\lam_{j+1}} \abs{ \lam_{j+1}'} \Big) \ll \frac{1}{\lam_j} \bfd(t), 
}
and hence, 
\EQ{
\Big|  \frac{ \iota_j}{2}\big( \frac{\xi_j'}{\xi_j} + \frac{\lam_{j+1}'}{\lam_{j+1}}\big)\La (r \p_r  \chi)( \cdot/ L\sqrt{\xi_j \lam_{j+1}}) \Lam W_{\U{\lam_j}} \mid \dot g \Ra  \Big| \ll\frac{1}{\lam_j} \bfd(t). 
}
This completes the proof of~\eqref{eq:beta_j'-4}. 
Lastly, we note that the estimates~\eqref{eq:a-tia-4} and~\eqref{eq:dta-4} follow from exactly the same arguments used to prove~\eqref{eq:a-tia-5} and~\eqref{eq:dta-5} in Proposition~\ref{prop:modul-5}. 
\end{proof}

We note that Lemma~\ref{lem:virial-error} and its proof remain valid for $D = 4$.

\subsection{Conclusion of the proof}
We have the following analog of Lemma~\ref{lem:ejection}.
\begin{lem}
\label{lem:ejection-4D}
Let $D = 4$. If $\eta_0$ is small enough,
then there exists $C_0 \geq 0$ depending only on $N$ such that the following is true.
If $[t_1, t_2] \subset I_*$ is a finite time interval such that $\bfd(t) \leq \eta_0$ for all $t \in [t_1, t_2]$, then
\begin{gather}
\label{eq:lambdaK-bd-4D}
\sup_{t \in [t_1, t_2]} \lambda_K(t) \leq \frac 43\inf_{t \in [t_1, t_2]} \lambda_K(t), \\
\label{eq:int-d-bd-4D}
\int_{t_1}^{t_2} \bfd(t)\ud t\leq C_0\big(\bfd(t_1) \lambda_K(t_1) + \bfd(t_2) \lambda_K(t_2)\big).
\end{gather}
\end{lem}
Due to the fact that some of the estimates in Proposition~\ref{prop:modul-k1}
hold only under the additional assumption \eqref{eq:max-ratio-ass},
we were not able to adapt to the current setting the proof for $D \geq 5$ given above.
We will provide a different proof, closer to \cite[Section 5]{JL6}.

We introduce below the notion of \emph{ignition condition}.
We stress that the definition which follows is meaningful
for \emph{any} continuous functions, not necessarily the ones given by the modulation.
\begin{defn}\label{def:ignition} 
Let $I$ be a time interval, $K \in \bN$, $\iota_j \in \{{-}1, 1\}$ and $\lambda_j, a_j^-, a_j^+ \in C(I)$
for all $1 \leq j \leq K$. Set
\begin{align}
\cS &:= \{i: 1 \leq i \leq K-1 \text{ and }\iota_i = \iota_{i+1}\}, \\
\bfd_\tx{par}(t) &:= \sqrt{\sum_{i \in \cS}\frac{\lambda_i(t)}{\lambda_{i+1}(t)} +\sum_{1 \leq i \leq K}\big(a_i^+(t)^2+a_i^-(t)^2\big)}.
\end{align}
We say that $\iota_1, \ldots, \iota_K, \lambda_1, \ldots, \lambda_K, a_1^-, \ldots, a_K^-, a_1^+, \ldots, a_K^+$
satisfy the \emph{ignition condition} with parameters $c_1, c_2, C_2 > 0$
if for any $I = [t_1, t_2]$, $t_0 \in I$ and $j\in\{1, \ldots, K\}$ satisfying at least one of the two pairs of conditions:
\begin{itemize}
\item\vspace{1em}~

\vspace*{-3.4em}
\begin{equation}
\label{eq:ignit-induct-1}
\sum_{i \in \cS, i < j}\frac{\lambda_i(t)}{\lambda_{i+1}(t)} + \sum_{i=1}^{j-1}a_i^\pm(t)^2 \leq c_2\bfd_\tx{par}(t)^2, \qquad\text{for all }t \in [t_1, t_2],
\end{equation}
and
\begin{equation}
\label{eq:ignit-induct-2}
a_j^+(t_0)^2 + a_j^-(t_0)^2 \geq c_1\bfd_\tx{par}(t_0)^2,
\end{equation}
\item\vspace{1.5em}~

\vspace*{-3.5em}
\begin{equation}
\label{eq:ignit-induct-3}
\sum_{i \in \cS, i < j-1}\frac{\lambda_i(t)}{\lambda_{i+1}(t)} + \sum_{i=1}^{j-1}a_i^\pm(t)^2 \leq c_2\bfd_\tx{par}(t)^2, \qquad\text{for all }t \in [t_1, t_2],
\end{equation}
and
\begin{equation}
\label{eq:ignit-induct-4}
\iota_{j-1}\iota_{j}\frac{\lambda_{j-1}(t_0)}{\lambda_j(t_0)} \geq c_1\bfd_\tx{par}(t_0)^2,
\end{equation}
\end{itemize}
there is at least one of the bounds:
\begin{equation}
\label{eq:ode-dbd-left}
\int_{t_1}^{t_0} \bfd_\tx{par}(t)\ud t \leq C_2 \bfd_\tx{par}(t_1)\lambda_K(t_1)
\end{equation}
or
\begin{equation}
\label{eq:ode-dbd-right}
\int_{t_0}^{t_2} \bfd_\tx{par}(t)\ud t \leq C_2 \bfd_\tx{par}(t_2)\lambda_K(t_2).
\end{equation}
\end{defn}
\begin{rem}
If the ignition condition is satisfied with given parameters $c_1, c_2, C_2 > 0$,
then it is also satisfied with any parameters $(\wt c_1, \wt c_2, \wt C_2)$
such that $\wt c_1 \geq c_1$, $\wt c_2 \leq c_2$ and $\wt C_2 \geq C_2$.
\end{rem}
\begin{rem}
By convention, $\lambda_0(t) = 0$ for all $t \in [t_1, t_2]$, thus \eqref{eq:ignit-induct-4}
never holds for $j = 1$. Clearly, \eqref{eq:ignit-induct-4} cannot hold either if $j-1\notin \cS$.
\end{rem}

\begin{lem}
\label{lem:ignition-induct}
If $(\iota_1, \ldots, \iota_K, \lambda_1, \ldots, \lambda_K, a_1^-, \ldots, a_K^-, a_1^+, \ldots, a_K^+)$ satisfy the ignition condition
with parameters $c_1, c_2, C_2$ and $I$ is a time interval such that
\begin{equation}
\label{eq:dsharp-small}
\frac{\iota_1\iota_2+1}{2}\frac{\lambda_1(t)}{\lambda_2(t)}
+ a_1^-(t)^2 + a_1^+(t)^2 \leq \frac 12 \min(1,c_2) \bfd_\tx{par}(t)^2\qquad\text{for all }t \in I,
\end{equation}
then $(\iota_2, \ldots, \iota_K, \lambda_2, \ldots, \lambda_K, a_2^-, \ldots, a_K^-, a_2^+, \ldots, a_K^+)$ satisfy the ignition condition
with parameters $(2c_1, \frac 12 \min(1, c_2), 2C_2)$ on the interval $I$.
\end{lem}
\begin{proof}
Set
\begin{align}
\sh\bfd(t) &:= \sqrt{\frac{\iota_1\iota_2+1}{2}\frac{\lambda_1(t)}{\lambda_2(t)}
+ a_1^-(t)^2 + a_1^+(t)^2}, \\
\fl\bfd(t) &:= \sqrt{\sum_{i \in \cS \setminus\{1\}}\frac{\lambda_i(t)}{\lambda_{i+1}(t)} +\sum_{i=2}^K\big(a_i^+(t)^2+a_i^-(t)^2\big)}
= \sqrt{\bfd_\tx{par}(t)^2 - \sh\bfd(t)^2}.
\end{align}
By assumption \eqref{eq:dsharp-small}, we have $\sh\bfd(t)^2 \leq \frac 12\bfd_\tx{par}(t)^2$, which implies $\bfd_\tx{par}(t) \leq \sqrt 2\fl\bfd(t)$.

We verify the ignition condition for $(\iota_2, \ldots, \iota_K, \lambda_2, \ldots, \lambda_K, a_2^-, \ldots, a_K^-, a_2^+, \ldots, a_K^+)$. If $j \geq 2$ and
\begin{equation}
\sum_{i \in \cS, 1< i < j}\frac{\lambda_i(t)}{\lambda_{i+1}(t)} + \sum_{i=2}^{j-1}a_i^\pm(t)^2 \leq \frac 12 \min(1, c_2)\fl\bfd(t)^2, \qquad\text{for all }t \in [t_1, t_2],
\end{equation}
then adding $\sh\bfd(t)^2$ to both sides and using \eqref{eq:dsharp-small},
we get \eqref{eq:ignit-induct-1}.
Also, $a_j^+(t_0)^2 + a_j^-(t_0)^2 \geq 2c_1 \fl\bfd(t_0)$ implies \eqref{eq:ignit-induct-2}.
Since we assume $(\iota_j, \lambda_j, a^-_j, a^+_j)_{j=1}^K$ satisfy the ignition condition, we obtain
at least one of the bounds \eqref{eq:ode-dbd-left}, \eqref{eq:ode-dbd-right}.
Since $\fl\bfd(t) \leq \bfd_\tx{par}(t) \leq \sqrt 2\fl\bfd(t)$, we obtain
the same bound with with $\fl\bfd$ instead of $\bfd_\tx{par}$ and $2C_2$ instead of $C_2$.
The case where bounds \eqref{eq:ignit-induct-3} and \eqref{eq:ignit-induct-4} hold is similar.
\end{proof}

\begin{lem}
\label{lem:ejection-par}
For all $K \in \bN$ and functions $c_2, C_2: (0, c_1^*] \to (0, \infty)$,
increasing and decreasing respectively,
there exist $c_{1, \min}, C_1 > 0$
such that if for all $c_1 \geq c_{1, \min}$, $(\iota_j, \lambda_j, a^-_j, a^+_j)_{j=1}^K$ satisfy the ignition condition with parameters $(c_1, c_2(c_1), C_2(c_1))$ on a time interval $I = [t_1, t_2]$, then
\begin{gather}
\int_I \bfd_\tx{par}(t)\ud t \leq C_1(\lambda_K(t_1)\bfd_\tx{par}(t_1) + \lambda_K(t_2)\bfd_\tx{par}(t_2)). \label{eq:ode-intd}
\end{gather}
\end{lem}
\begin{proof}
Induction with respect to $K$.

\noindent
\textbf{Step 1.} For $K = 1$, we let $c_{1, \min} = \frac 12$.
We will only use the fact that the ignition condition is satisfied with parameters
$c_1 = \frac 12, c_2, C_2$ for some $c_2, C_2 > 0$.
The conditions \eqref{eq:ignit-induct-1} and \eqref{eq:ignit-induct-2} hold for $j = 1$,
thus for all $t_0$ we have \eqref{eq:ode-dbd-left} or \eqref{eq:ode-dbd-right}.
Let
\begin{align}
A_- := \{t_0 \in I: \eqref{eq:ode-dbd-left}\text{ holds}\}, \qquad
A_+ := \{t_0 \in I: \eqref{eq:ode-dbd-right}\text{ holds}\},
\end{align}
so that $A_-$ and $A_+$ are closed sets, and $I = A_- \cup A_+$.
We define
\begin{align}
\sh t_1 := \max\,A_-, \qquad \sh t_2 := \min\,A_+.
\end{align}
We adopt the convention that $\sh t_1 = t_1$ if $A_- = \emptyset$,
and similarly $\sh t_2 = t_2$ if $A_+ = \emptyset$.
With these conventions, we find that $\sh t_1 \geq \sh t_2$.
By the ignition condition, we have
\begin{equation}
\int_{t_1}^{t_2}\bfd_\tx{par}(t)\ud t \leq \int_{t_1}^{\sh t_1}\bfd_\tx{par}(t)\ud t + \int_{\sh t_2}^{t_2}\bfd_\tx{par}(t)\ud t \leq C_2(c_{1, \min})\big(\bfd_\tx{par}(t_1)\lambda_K(t_1) + \bfd_\tx{par}(t_2)\lambda_K(t_2)\big),
\end{equation}
which settles the base case $K = 1$.

\noindent
\textbf{Step 2.}
We continue with the induction step.
Set $\fl c_2(c_1) := \frac 12 \min(1, c_2(c_1/2))$ and $\fl C_2(c_1) := 2C_2(c_1/2)$
for all $c_1 \in (0, c_1^*]$.
Let $\fl c_{1, \min} > 0$ be the number given by the induction hypothesis
(for $K-1$ instead of $K$)
for these functions $\fl c_2$ and $\fl C_2$.
We set $c_3 := \frac 12 \fl c_2(\fl c_{1,\min})$ and $c_{1, \min} := \min\big(\frac 12 \fl c_2(\fl c_{1,\min}), c_2(c_3)\big)$.

Assume $(\iota_j, \lambda_j, a^-_j, a^+_j)_{j=1}^K$ satisfy the ignition condition
for all $c_1 \geq c_{1, \min}$, and let
\begin{equation}
\label{eq:ode2-a1cond}
A := \{t \in [t_1, t_2]: a_1^+(t)^2 + a_1^-(t)^2 \geq c_{1,\min}\bfd_\tx{par}(t)^2\}.
\end{equation}
By the ignition condition, \eqref{eq:ode-dbd-left} or \eqref{eq:ode-dbd-right} holds for all $t_0 \in A$,
with $C_2 = C_2(c_{1, \min})$.
Let
\begin{align}
A_- := \{t_0 \in A: \eqref{eq:ode-dbd-left}\text{ holds}\}, \qquad
A_+ := \{t_0 \in A: \eqref{eq:ode-dbd-right}\text{ holds}\},
\end{align}
so that $A_-$ and $A_+$ are closed sets, and $A = A_- \cup A_+$.
We define
\begin{align}
\sh t_1 := \max\,A_-, \qquad \sh t_2 := \min\,A_+.
\end{align}
We adopt the convention that $\sh t_1 = t_1$ if $A_- = \emptyset$,
and similarly $\sh t_2 = t_2$ if $A_+ = \emptyset$.
It is not excluded either that $\sh t_1 \geq \sh t_2$.

By the ignition condition, we have
\begin{equation}
\int_{t_1}^{\sh t_1}\bfd_\tx{par}(t)\ud t + \int_{\sh t_2}^{t_2}\bfd_\tx{par}(t)\ud t \leq C_2(c_{1, \min})\big(\bfd_\tx{par}(t_1)\lambda_K(t_1) +\bfd_\tx{par}(t_2)\lambda_K(t_2)\big),
\end{equation}
so it remains to consider the interval $(\sh t_1, \sh t_2)$.
Notice that $(\sh t_1, \sh t_2) \subset I \setminus A$.

\noindent
\textbf{Step 3.}
We treat separately the cases $\iota_1\iota_2 = -1$ and $\iota_1\iota_2 = 1$.
In the former case, we set $\fl t_1 = \sh t_1$, $\fl t_2 = \sh t_2$,
and go to the next step.

Assume $\iota_1\iota_2 = 1$ and let
\begin{equation}
\label{eq:ode2-lambda1cond}
B := \Big\{t \in [\sh t_1, \sh t_2]: \frac{\lambda_1(t)}{\lambda_2(t)} \geq c_3\bfd_\tx{par}(t)^2\Big\}.
\end{equation}
We have $c_{1, \min} \leq c_2(c_3)$,
thus $a_1^+(t)^2 + a_1^-(t)^2 \leq c_2(c_3)$ for all $t \in [\sh t_1, \sh t_2]$,
so that the ignition condition implies
that \eqref{eq:ode-dbd-left} or \eqref{eq:ode-dbd-right} holds for all $t_0 \in A$,
with $c_2 = c_2(c_3)$ and $C_2 = C_2(c_3)$.
Let
\begin{align}
B_- := \{t_0 \in B: \eqref{eq:ode-dbd-left}\text{ holds}\}, \qquad
B_+ := \{t_0 \in B: \eqref{eq:ode-dbd-right}\text{ holds}\},
\end{align}
so that $B_-$ and $B_+$ are closed sets, and $B = B_- \cup B_+$.
We define
\begin{align}
\fl t_1 := \max\,B_-, \qquad \fl t_2 := \min\,B_+.
\end{align}
We adopt the convention that $\fl t_1 = \sh t_1$ if $B_- = \emptyset$,
and similarly $\fl t_2 = \sh t_2$ if $B_+ = \emptyset$.
It is not excluded either that $\fl t_1 \geq \fl t_2$.

By the ignition condition, we have
\begin{equation}
\int_{\sh t_1}^{\fl t_1}\bfd_\tx{par}(t)\ud t + \int_{\fl t_2}^{\sh t_2}\bfd_\tx{par}(t)\ud t \leq C_2(c_3)\big(\bfd_\tx{par}(\sh t_1)\lambda_K(\sh t_1) + \bfd_\tx{par}(\sh t_2)\lambda_K(\sh t_2)\big),
\end{equation}
so it remains to consider the interval $(\fl t_1, \fl t_2)$.
Notice that $(\fl t_1, \fl t_2) \subset I \setminus (A\cup B)$.

\noindent
\textbf{Step 4.}
We check that for all $\fl c_1 \geq \fl c_{1,\min}$, $(\iota_j, \lambda_j, a^-_j, a^+_j)_{j=2}^K$
satisfy the ignition condition with parameters $(\fl c_1, \fl c_2(\fl c_1), \fl C_2(\fl c_1))$
on the interval $[\fl t_1, \fl t_2]$. We have
\begin{equation}
\label{eq:param1small-flat}
a_1^+(t)^2 + a_1^-(t)^2 + \frac{\iota_1\iota_2 + 1}{2}\frac{\lambda_1(t)}{\lambda_2(t)} \leq (c_{1,\min} + c_3)\bfd_\tx{par}(t)^2 \leq \fl c_2(\fl c_{1,\min})\bfd_\tx{par}(t)^2.
\end{equation}
We definition of $\fl c_2$, we also have
\begin{equation}
\fl c_2(\fl c_{1,\min}) \leq \frac 12\min(1, c_2(\fl c_{1,\min}/2)) \leq \frac 12 \min(1,c_2(\fl c_1/2)),
\end{equation}
where in the last step we used the fact that $c_2$ is an increasing function.
Thus, by Lemma~\ref{lem:ignition-induct}, $(\iota_j, \lambda_j, a^-_j, a^+_j)_{j=2}^K$ satisfy the ignition condition
with parameters $\fl c_1$, $\frac 12 \min(1, c_2(\fl c_1/2)) = \fl c_2(\fl c_1)$, $2C_2(\fl c_1/2) = \fl C_2(\fl c_1)$.

By the induction hypothesis, we have
\begin{equation}
\int_{\fl t_1}^{\fl t_2}\fl\bfd(t)\ud t \leq \fl C_1\sup_{\fl t_1 \leq t \leq \fl t_2}(\lambda_K(t)\fl\bfd(t)),
\end{equation}
where
\begin{equation}
\fl\bfd(t) := \sqrt{\sum_{i \in \cS \setminus\{1\}}\frac{\lambda_i(t)}{\lambda_{i+1}(t)} +\sum_{i=2}^K\big(a_i^+(t)^2+a_i^-(t)^2\big)}.
\end{equation}
The bound \eqref{eq:param1small-flat} and $\fl c_2(\fl c_{1,\min})\leq 1$ imply
$\bfd_\tx{par}(t) \leq \sqrt 2\fl \bfd(t)$ for all $t \in [\fl t_1, \fl t_2]$,
and the desired bound follows.
\end{proof}
Next, we prove that the modulation parameters satisfy the ignition condition.
\begin{lem}
\label{lem:ignition-lemma}
For any $c_1 \in (0, c_1^*]$ there exist $\eta_0, c_2, C_2 > 0$ such that
the following is true.
Let $\bs u$ be a solution of \eqref{eq:nlw} for $D = 4$, let $\bfd$ be
defined by \eqref{eq:d-intro}, $I$ a time interval such that $\bfd(t) \leq \eta_0$
for all $t \in I$, and let $(\iota_j, \lambda_j, a^-_j, a^+_j)_{j=1}^K$
be the modulation parameters as in Lemma~\ref{lem:mod-1}.
Then $(\iota_j, \lambda_j, a^-_j, a^+_j)_{j=1}^K$
satisfy the ignition condition with parameters $(c_1, c_2, C_2)$ on $I$.
\end{lem}
\begin{proof}
Assume first that \eqref{eq:ignit-induct-1} and \eqref{eq:ignit-induct-2} hold.
We have $|a_j^+(t_0)| \geq |a_j^-(t_0)|$ or $|a_j^+(t_0)| \leq |a_j^-(t_0)|$.
We will show that the former implies \eqref{eq:ode-dbd-right}, and the latter implies \eqref{eq:ode-dbd-left}.
Since the two cases are analogous, we only consider the first one.
To fix ideas, assume $a_j^+(t_0) > 0$, the case $a_j^+(t_0) < 0$
being analogous, so that
\begin{equation}
a_j^+(t_0) \geq \frac{\sqrt{c_1}}{\sqrt 2}\bfd_\tx{par}(t_0)\quad\Rightarrow\quad \wt a_j^+(t_0) \geq \frac{2\sqrt{c_1}}{3}\bfd_\tx{par}(t_0),
\end{equation}
where the last inequality follows from \eqref{eq:a-tia-4}.

Set
\begin{equation}
t_3 := \sup\big\{t \in [t_0, t_2]: \wt a_j^+(t) \geq \frac{\sqrt{c_1}}{4}\bfd_\tx{par}(t)\text{ for all }t \in [t_0, t_3]\big\}.
\end{equation}
If $\eta_0$ is small enough, then \eqref{eq:dta-4} yields
\begin{equation}
\big(\wt a_j^+\big)'(t) \geq \frac{\kappa}{2\lambda_j(t)}\wt a_j^+(t),\qquad\text{for all }t \in [t_0, t_3].
\end{equation}
Integrating, and using again the inequality defining $t_3$, we obtain
\begin{equation}
\label{eq:t3-first-case}
\int_{t_0}^{t_3}\bfd_\tx{par}(t)\ud t \lesssim \bfd_\tx{par}(t_3)\sup_{t\in[t_0, t_3]}\lambda_j(t).
\end{equation}
Since $|\lambda_k'(t)| \lesssim \bfd(t) \lesssim \bfd_\tx{par}(t)$,
for all $k \geq j$ we obtain
\begin{equation}
\label{eq:lambda_k_constant}
\sup_{t\in[t_0, t_3]}\lambda_k(t) \leq (1+c_4)\inf_{t\in[t_0, t_3]}\lambda_k(t),
\end{equation}
where $c_4$ can be made arbitrarily small upon adjusting $\eta_0$.
Similarly, for $k > j$ \eqref{eq:dta-4} yields
\begin{equation}
\big|\big(\wt a_k^\pm\big)'(t)\big| \lesssim \bfd_\tx{par}(t) / \lambda_k(t),
\end{equation}
so \eqref{eq:t3-first-case} together with \eqref{eq:lambda_k_constant} yield
\begin{equation}
|\wt a_k^\pm(t_3) - \wt a_k^\pm(t_0)| \leq c_4 \bfd_\tx{par}(t_3), \qquad\text{for all }k > j.
\end{equation}
Also, \eqref{eq:dta-4} yields
\begin{equation}
\max(0, \dd t|\wt a_j^-(t)|) \lesssim \bfd_\tx{par}(t)^2 / \lambda_j(t),
\end{equation}
thus using again \eqref{eq:t3-first-case} and \eqref{eq:lambda_k_constant}
with $k = j$, we have
\begin{equation}
|\wt a^-_j(t_3)| -|\wt a^-_j(t_0)|  \leq c_4\bfd_\tx{par}(t_3).
\end{equation}
Set
\begin{align}
\sh \bfd(t)^2 &:= \sum_{i \in \cS, i < j}\frac{\lambda_i(t)}{\lambda_{i+1}(t)} + \sum_{i=1}^{j-1}a_i^\pm(t)^2, \\
\fl \bfd(t)^2 &:= \sum_{i \in \cS, i \geq j}\frac{\lambda_i(t)}{\lambda_{i+1}(t)}
+ \sum_{i=j+1}^{K}a_i^\pm(t)^2 + a_j^-(t)^2.
\end{align}
From the bounds above, we obtain $\fl \bfd(t_3)^2 \leq \fl \bfd(t_0)^2 + c_4\bfd_\tx{par}(t_3)^2$, with $c_4$ small. By \eqref{eq:ignit-induct-1},
we have $\sh\bfd(t_3)^2 \leq c_2 \bfd_\tx{par}(t_3)$.
Since $\wt a_j^+$ is increasing on $[t_0, t_3]$, we obtain
\begin{equation}
\begin{aligned}
\bfd_\tx{par}(t_3)^2 &= \sh \bfd(t_3)^2 + \fl\bfd(t_3)^2 + a_j^+(t_3)^2
\leq c_2 \bfd_\tx{par}(t_3)^2 + \fl\bfd(t_0)^2 + c_4\bfd_\tx{par}(t_3)^2
+ a_j^+(t_3)^2 \\
& \leq (c_2 + c_4)\bfd_\tx{par}(t_3)^2 + \big(1 + 9/(4c_1)\big)a_j^+(t_3)^2.
\end{aligned}
\end{equation}
If $c_1, c_2$ and $c_4$ are small enough, this implies $a_j^+(t_3) \geq \frac{\sqrt{c_1}}{2}\bfd_\tx{par}(t_3)$,
thus $t_3 = t_2$ and \eqref{eq:t3-first-case} yields \eqref{eq:ode-dbd-right}.

Assume now that \eqref{eq:ignit-induct-3} and \eqref{eq:ignit-induct-4} hold.
We will prove that $\beta_{j-1}(t_0) \geq 0$ implies \eqref{eq:ode-dbd-right}.
An analogous argument would show that $\beta_{j-1}(t_0) \leq 0$ implies \eqref{eq:ode-dbd-left}.

Set
\begin{equation}
t_3 := \sup\big\{t \in [t_0, t_2]: \xi_{j-1}(t) / \lambda_j(t) \geq \frac{c_1}{4} \bfd_\tx{par}(t)^2\big\}.
\end{equation}
Then, choosing $c_2$ in \eqref{eq:ignit-induct-3} small enough, \eqref{eq:beta_j'-4} yields
\begin{equation}
\beta_{j-1}'(t) \geq 8\lambda_j(t)^{-1}, \qquad\text{for all }t\in [t_0, t_3].
\end{equation}

For $0 < x \ll 1$, set $\Phi(x) := \sqrt{-x\log x}$. Note that
\begin{equation}
\begin{aligned}
\sqrt{x} &\sim \Phi(x) / \sqrt{{-}\log \Phi(x)}, \\
\Phi'(x) &= \frac{\sqrt{-\log x}}{2\sqrt x} + O(({-}x\log x)^{-1/2}) > 0.
\end{aligned}
\end{equation}
With $c_2 > 0$ to be determined, consider the auxiliary function
\begin{equation}
\phi(t) := \beta_{j-1}(t) + c_3\Phi\big(\xi_{j-1}(t)/\lambda_{j}(t)\big).
\end{equation}
The Chain Rule gives
\begin{equation}
\phi'(t) = \beta_{j-1}'(t) + c_3 \frac{\xi_{j-1}'(t)}{\lambda_{j}(t)}\Phi'\Big(\frac{\xi_{j-1}(t)}{\lambda_{j}(t)}\Big).
\end{equation}
By \eqref{eq:xi_j'-k1}, we have $|\xi_{j-1}'(t)| \lesssim (\xi_{j-1}(t) / \lambda_{j}(t))^\frac 12\log({-}\xi_{j-1}(t) / \lambda_{j}(t))^{-1/2}$,
hence we can choose $c_3$ such that
\begin{equation}
\label{eq:f'-lbound-k1}
\phi'(t) \geq 4{\lambda_{j}(t)^{-1}}, \qquad\text{for all }t \in [t_0, t_3].
\end{equation}
If we consider $\wt \phi(t) := \beta_{j-1}(t) + \frac{c_3}{2}\Phi\big(\xi_{j-1}(t)/\lambda_{j}(t)\big)$
instead of $\phi$, then the computation above shows that $\wt \phi$ is increasing.
We have $\wt \phi(t_l) \geq 0$, so $\wt \phi(t) \geq 0$ for all $t \in [t_0, t_3]$,
implying
\begin{equation}
\label{eq:d-ubound-k1}
\bfd(t) \lesssim \sqrt{\xi_{j-1}(t)/\lambda_{j}(t)} \lesssim
\phi(t)/\sqrt{{-}\log \phi(t)}.
\end{equation}

The bound \eqref{eq:f'-lbound-k1} yields
\begin{equation}
\label{eq:fund-thm-for-f-k1}
\begin{aligned}
\big(\lambda_{j}(t)\phi(t)^{2}/\sqrt{{-}\log \phi(t)}\big)' \gtrsim \phi(t)/\sqrt{{-}\log \phi(t)}.
\end{aligned}
\end{equation}
We observe that $|\phi(t)| \lesssim \Phi(\bfd(t)^2)$, hence $\phi(t)^2/\sqrt{-\log \phi(t)} \lesssim \bfd(t)^2\sqrt{{-}\log\bfd(t)}$ and
\begin{equation}
\int_{t_0}^{t_3}\phi(t)/\sqrt{{-}\log \phi(t)}\ud t \lesssim
\lambda_{j}(t_3)\phi(t_3)^{2}/\sqrt{-\log \phi(t_3)} \lesssim \bfd(t_3)^{2}\sqrt{{-}\log\bfd(t_{3})}\lambda_{j}(t_3).
\end{equation}
Thus, \eqref{eq:d-ubound-k1} yields
\begin{equation}
\label{eq:int-d-second}
\int_{t_0}^{t_3}\bfd(t)\ud t \lesssim \bfd(t_3)^{2}\sqrt{{-}\log\bfd(t_{3})}\lambda_{j}(t_3).
\end{equation}
The argument from the first part of the proof yields \eqref{eq:lambda_k_constant}, for all $k \geq j$.
Also, \eqref{eq:dta-4} gives $\big|\big(\wt a_j^\pm\big)'(t)\big| \lesssim
 \lambda_j(t)^{-1}\bfd(t)$, thus using again \eqref{eq:d-ubound-k1}
and \eqref{eq:lambda_k_constant} we get, for all $k \geq j$,
\begin{equation}
|\wt a_k^\pm(t_3) - \wt a_k^\pm(t_0)| \leq c_4\bfd_\tx{par}(t_3)
\end{equation}
with $c_4$ small, since the right hand side of \eqref{eq:int-d-second} is $\ll \bfd(t_3)\lambda_j(t_3)$ if $\eta_0$ is small.
Finally, \eqref{eq:xi_j'-beta_j-k1} and $\beta_{j-1}(t) \geq 0$ imply,
again using \eqref{eq:int-d-second},
\begin{equation}
\label{eq:compare-of-xis}
\frac{\xi_{j-1}(t_3)}{\lambda_j(t_3)} \geq (1 - c_4)\frac{\xi_{j-1}(t_0)}{\lambda_j(t_0)} \quad\Rightarrow\quad \frac{\lambda_{j-1}(t_3)}{\lambda_j(t_3)} \geq (1 - c_4)\frac{\lambda_{j-1}(t_0)}{\lambda_j(t_0)},
\end{equation}
where in the last step we use \eqref{eq:xi_j-lambda_j-k1}.
Set
\begin{align}
\sh \bfd(t)^2 &:= \sum_{i \in \cS, i < j-1}\frac{\lambda_i(t)}{\lambda_{i+1}(t)} + \sum_{i=1}^{j-1}a_i^\pm(t)^2, \\
\fl \bfd(t)^2 &:= \sum_{i \in \cS, i \geq j}\frac{\lambda_i(t)}{\lambda_{i+1}(t)}
+ \sum_{i=j}^{K}a_i^\pm(t)^2.
\end{align}
From the bounds above, we obtain $\fl \bfd(t_3)^2 \leq \fl \bfd(t_0)^2 + c_4\bfd_\tx{par}(t_3)^2$, with $c_4$ small. By \eqref{eq:ignit-induct-1},
we have $\sh\bfd(t_3)^2 \leq c_2 \bfd_\tx{par}(t_3)$.
Applying \eqref{eq:compare-of-xis}, we obtain
\begin{equation}
\begin{aligned}
\bfd_\tx{par}(t_3)^2 &= \sh \bfd(t_3)^2 + \fl\bfd(t_3)^2 + \lambda_{j-1}(t_3) / \lambda_j(t_3)\\
&\leq c_2 \bfd_\tx{par}(t_3)^2 + \fl\bfd(t_0)^2 + c_4\bfd_\tx{par}(t_3)^2
+ \lambda_{j-1}(t_3) / \lambda_j(t_3) \\
& \leq (c_2 + c_4)\bfd_\tx{par}(t_3)^2 + \big((1-c_4)^{-1} + c_1^{-1}\big)\lambda_{j-1}(t_3) / \lambda_j(t_3).
\end{aligned}
\end{equation}
If $c_1, c_2$ and $c_4$ are small enough, this implies $\xi_{j-1}(t_3) / \lambda_j(t_3) \geq \frac{c_1}{2}\bfd_\tx{par}(t_3)^2$,
thus $t_3 = t_2$ and \eqref{eq:int-d-second} yields \eqref{eq:ode-dbd-right}.
\end{proof}
\begin{proof}[Proof of Lemma~\ref{lem:ejection-4D}]
It suffices to prove \eqref{eq:int-d-bd-4D}, and \eqref{eq:lambdaK-bd-4D}
will follow by the same argument as in the proof of Lemma~\ref{lem:ejection}.

For given $c_1 \in (0, c_1^*]$, let $\eta_0 = \eta_0(c_1)$,
$c_2 = c_2(c_1)$ and $C_2 = C_2(c_1)$ be given by Lemma~\ref{lem:ignition-lemma}.
One can always decrease $\eta_0 = \eta_0(c_1)$ and $c_2 = c_2(c_1)$, and increase $C_2 = C_2(c_1)$. Thus, it is now easy to construct inreasing functions
$c_2, \eta_0: (0, c_1^*] \to (0, \infty)$,
and a decreasing function $C_2: (0, c_1^*] \to (0, \infty)$,
such that for all $c_1$ the modulation parameters satisfy
the ignition condition with parameters $(c_1, c_2(c_1), C_2(c_1))$
on any time interval on which $\bfd(t) \leq \eta_0(c_1)$.
We first take a decreasing sequence of values $c_1$
converging to zero, find corresponding sequences of values $c_2, \eta_0$ and $C_2$
which we can assume monotone,
and then set $c_2, \eta_0$ and $C_2$ constant between consecutive values of $c_1$.

The conclusion follows from Lemma~\ref{lem:ejection-par}
and the fact that $\bfd_\tx{par}(t) \simeq \bfd(t)$ for all $t \in I$, see \eqref{eq:d-bound-k1}.
\end{proof}

\bibliographystyle{plain}
\bibliography{researchbib}

\bigskip
\centerline{\scshape Jacek Jendrej}
\smallskip
{\footnotesize
 \centerline{CNRS and LAGA, Universit\'e  Sorbonne Paris Nord}
\centerline{99 av Jean-Baptiste Cl\'ement, 93430 Villetaneuse, France}
\centerline{\email{jendrej@math.univ-paris13.fr}}
} 
\medskip 
\centerline{\scshape Andrew Lawrie}
\smallskip
{\footnotesize
 \centerline{Department of Mathematics, Massachusetts Institute of Technology}
\centerline{77 Massachusetts Ave, 2-267, Cambridge, MA 02139, U.S.A.}
\centerline{\email{alawrie@mit.edu}}
}

\end{document}